\documentclass[11pt]{article}

\usepackage[english]{babel}

\usepackage[letterpaper,top=2cm,bottom=2cm,left=3cm,right=3cm,marginparwidth=1.75cm]{geometry}

\usepackage{amsmath}
\usepackage{graphicx}
\usepackage{amssymb}
\usepackage{wasysym}
\usepackage{bbm}
\usepackage{todonotes}
\usepackage{float}
\usepackage{authblk}
\usepackage{makeidx}
\usepackage[colorlinks=true, allcolors=blue]{hyperref}
\allowdisplaybreaks
\usepackage[nottoc]{tocbibind}
\newcommand\nnfootnote[1]{%
  \begin{NoHyper}
  \renewcommand\thefootnote{}\footnote{#1}%
  \addtocounter{footnote}{-1}%
  \end{NoHyper}
}

\usepackage{amsthm}
\newtheorem{thm}{Theorem}[subsection]
\newtheorem{lemma}[thm]{Lemma}
\newtheorem{prop}[thm]{Proposition}
\newtheorem{cor}[thm]{Corollary}
\theoremstyle{remark}
\newtheorem{rem}[thm]{Remark}
\newtheorem{exmp}[thm]{Example}
\newtheorem{problem}[thm]{Problem}

\theoremstyle{definition}
\newtheorem{definition}[thm]{Definition}

\newcommand{\Z}{\mathbb{Z}}

\newcommand{\norm}[1] {\| #1\|}
\newcommand{\Prob}{\mathcal P}

\newcommand{\R}{\mathbb{R}}
\newcommand{\ent}{\text{ent}}
\newcommand{\Ent}{\text{Ent}}

\newcommand{\threeeven}{\Z^3_{\text{even}}}

\newcommand{\mc}[1]{\mathcal{#1}}
\newcommand{\m}[1]{\mathbb{#1}}

\newcommand{\Vol}{\text{Vol}}
\newcommand{\dd}{\mathrm{d}}

\makeindex
\newcommand{\symindex}[1]{\index{{\textbf{Symbols}}!#1}}
\newcommand{\termindex}[1]{\index{{\textbf{Terms}}!#1}}

\title{Large deviations for the 3D dimer model}
\author[1]{Nishant Chandgotia}
\affil[1]{Tata Institute of Fundamental Research - Centre for Applicable Mathematics}
\author[2]{Scott Sheffield}
\affil[2]{Massachusetts Institute of Technology}
\author[3]{Catherine Wolfram}
\affil[3]{ETH Z\"urich and Yale University}

\date{\today}

\begin{document}

\maketitle 
\begin{abstract}
In 2000, Cohn, Kenyon and Propp studied uniformly random perfect matchings of large induced subgraphs of $\mathbb Z^2$ (a.k.a.\ dimer configurations or domino tilings) and developed a large deviation theory for the associated {\em height functions}. We establish similar results for large induced subgraphs of $\mathbb Z^3$. To formulate these results, recall that a perfect matching on a bipartite graph induces a {\em flow} that sends one unit of current from each even vertex to its odd partner. One can then subtract a ``reference flow'' to obtain a {\em divergence-free flow}. (On a planar graph, the curl-free dual of this flow is the height function gradient.)

We show that the flow induced by a {\em uniformly random} dimer configuration converges in law (when boundary conditions on a bounded $R \subset \mathbb R^3$ are controlled and the mesh size tends to zero) to the deterministic divergence-free flow $g$ on $R$ that maximizes 
    $$\int_{R} \ent\big(g(x)\bigr) \,dx$$
given the boundary data, where $\ent(s)$ is the maximal specific entropy obtained by an ergodic Gibbs measure with mean current $s$. The function $\ent$ is not known explicitly, but we prove that it is continuous and {\em strictly concave} on the octahedron $\mathcal O$ of possible mean currents (except on the edges of $\mathcal O$) which implies (under reasonable boundary conditions) that the maximizer is uniquely determined. We further establish two versions of a large deviation principle, using the integral above to quantify how exponentially unlikely the discrete random flows are to approximate {\em other} deterministic flows.

The planar dimer model is mathematically rich and well-studied, but many of the most powerful tools do not seem readily adaptable to higher dimensions (e.g.\ Kasteleyn determinants, McShane-Whitney extensions, FKG inequalities, monotone couplings, Temperleyan bijections, perfect sampling algorithms, plaquette-flip connectivity, etc.)\ Our analysis begins with a smaller set of tools, which include Hall's matching theorem, the ergodic theorem, non-intersecting-lattice-path formulations, and double-dimer cycle swaps. Several steps that are straightforward in 2D (such as the ``patching together'' of matchings on different regions) require interesting new techniques in 3D.
\end{abstract}

\nnfootnote{MSC2020 Subject Classifications: Primary 60F10, 82B20; Secondary 82B30.
}

\tableofcontents 

\section{Introduction}

\subsection{Overview}
Let $G = (V,E)$ be a bipartite graph. A \textit{dimer cover} (a.k.a.\ {\em perfect matching}) of $G$ is a collection of edges so that every vertex is contained in exactly one edge. Throughout this paper, we will assume that $G$ is an induced subgraph of $\mathbb{Z}^d$. We partition $\mathbb{Z}^d$ into \textit{even} vertices (the sum of whose coordinates is even) and \textit{odd} vertices (the sum of whose coordinates is odd). By convention, we represent an ({\em a priori} undirected) edge $e$ by $(a,b)$ where $a$ is even and $b$ is odd.

We can also take a dual perspective, where each vertex $a = (a_1, a_2, \ldots, a_d)$ is represented by the hypercube $[a_1-\frac12, a_1+\frac12] \times \ldots \times [a_d - \frac12, a_d +\frac12]$ and each matched edge is represented by a ``domino'' which is the union of two adjacent hypercubes. The dominoes are $2 \times 1$ (or $1\times 2$) boxes in 2D and $2 \times 1 \times 1$ (or $1\times 2\times 1$ or $1\times 1\times 2)$ boxes in 3D. From this perspective, the perfect matchings corresponding to the subgraph induced by $V \subset \mathbb Z^d$ correspond to {\em domino tilings} of the region formed by the union of the corresponding cubes. The figure below illustrates a domino tiling of a two-dimensional region called the {\em Aztec diamond}. On the left, the domino corresponding to $(a,b)$ is colored one of four colors, according to the value of the unit-length vector $b-a$. On the right, squares are colored by parity.
\begin{figure}[H]
    \centering
    \includegraphics[scale=0.17]{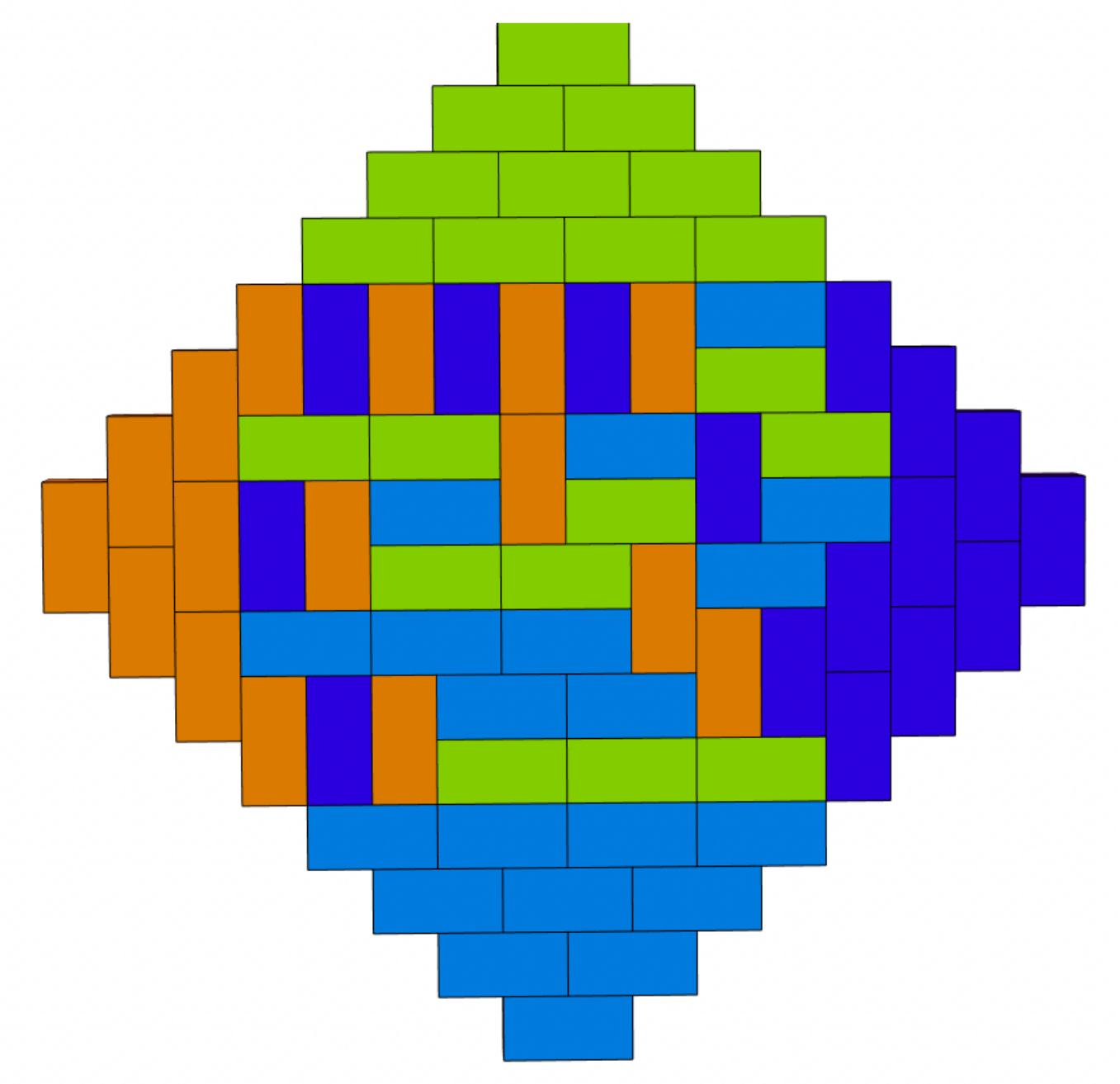}$\qquad\qquad\qquad$\includegraphics[scale=0.21]{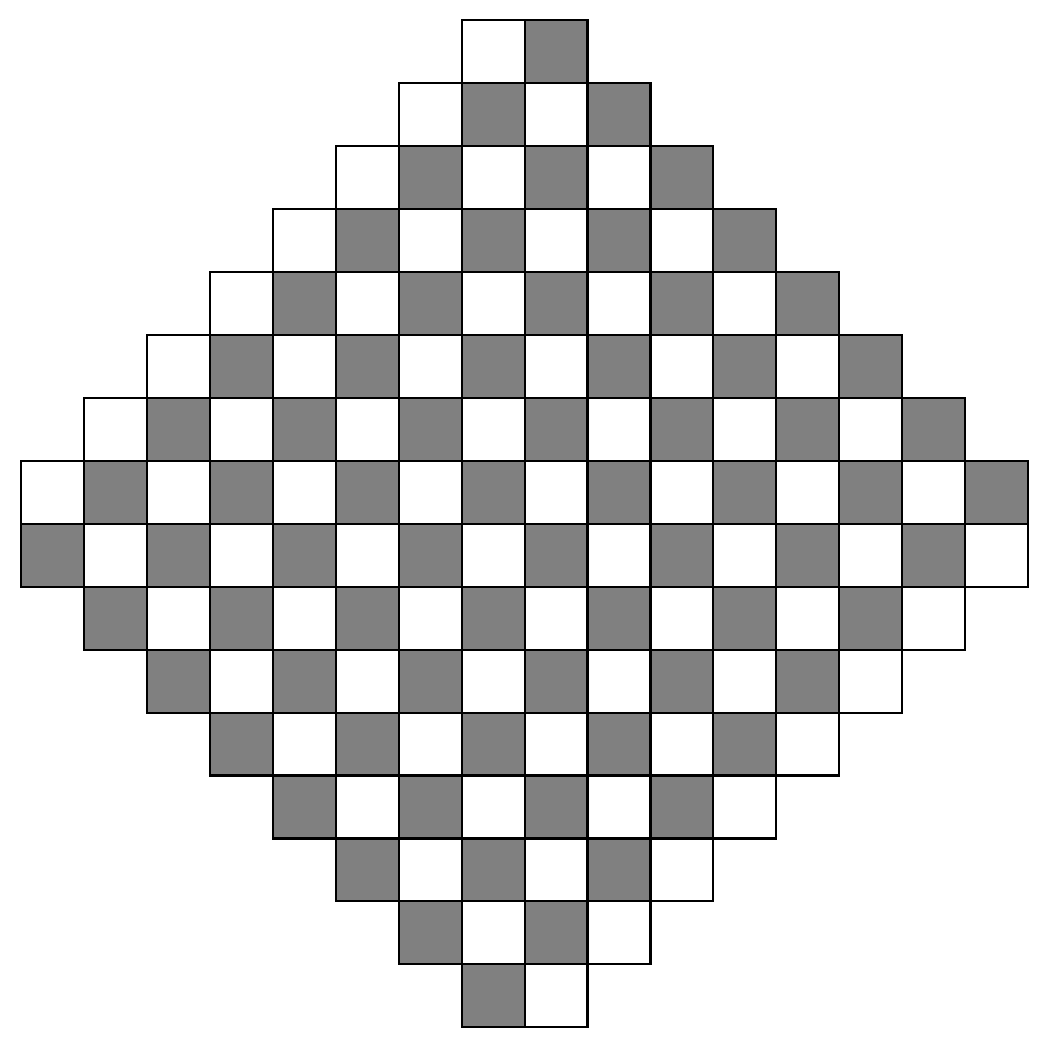} 
\caption{Tiling of an Aztec diamond and bipartite coloring of squares in $\m Z^2$.}
    \label{fig:aztec_diamond}
\end{figure}
In other words, every domino in the tiling on the left contains one square that is black (in the chessboard coloring on the right) and one that is white---and the color of a domino depends on whether its white square lies north, south, east or west of its black square.

\termindex{Chapter 1!brickwork tilings/pattern}Tilings with all dominoes oriented the same way are called {\em brickwork tilings}. There are four brickwork orientations in dimension $2$---and $2d$ brickwork orientations in dimension $d$.
\begin{figure}[H]
    \centering
        \includegraphics[scale=0.17]{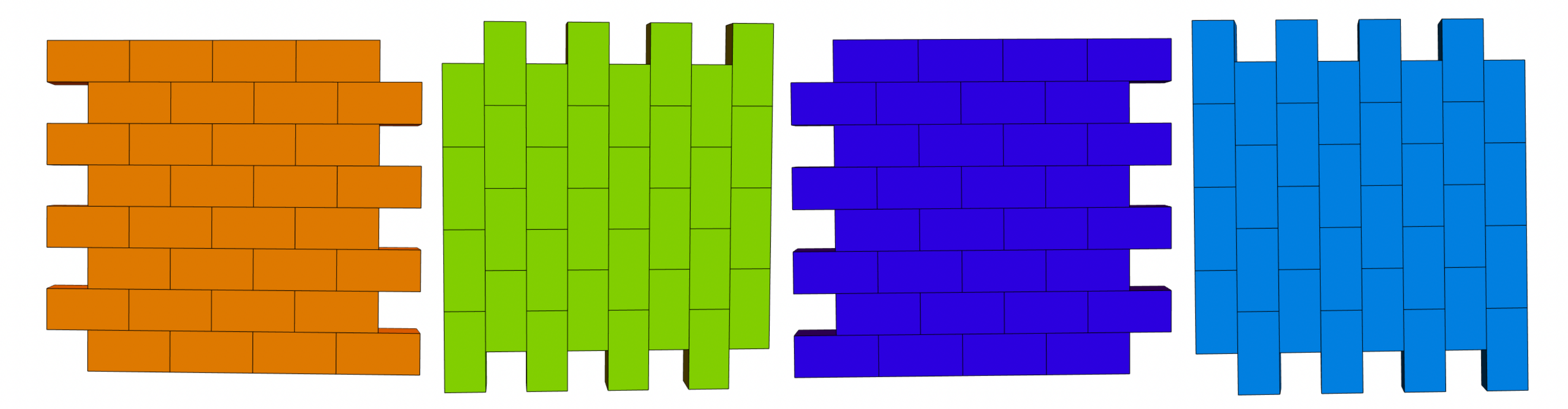}
    \caption{The four brickwork patterns in two dimensions.}
    \label{fig:brickworks}
\end{figure}
A perfect matching $\tau$ of an induced subgraph $R$ of $\mathbb Z^d$ induces an {\em lattice flow} $v_\tau$ that sends one unit of current from every even vertex to the odd vertex it is matched to. If we subtract a ``reference flow'' (which sends a current of magnitude $1/2d$ from each even vertex to each of its $2d$ odd neighbors) we obtain a {\em divergence-free} flow $f_\tau$.  The main problem in this paper is to understand the behavior of the {\em random} divergence-flow $f_\tau$ that corresponds to a $\tau$ chosen uniformly from the set of tilings of a large region, subject to certain boundary conditions.

\subsection{Two-dimensional background}

In two dimensions, the divergence-free flow on $\mathbb Z^2$ described above has a dual flow on the dual graph (obtained by rotating each edge 90 degrees counterclockwise about its center) that is a curl-free flow, and hence can be realized as the discrete gradient of some real-valued function defined on the vertices of the dual graph; see Section \ref{section:tiling_flows}. This function (defined up to additive constant) is called the {\em height function} of the flow. Questions about the random flows associated to random perfect matchings can be equivalently formulated as questions about random height functions. For example, one can ask: when a tiling of a large region is chosen uniformly at random, what does the ``typical'' height function look like?

In 2000, Cohn, Kenyon and Propp studied domino tilings of a domain $R \subset \mathbb R^2$ like the one below, asking what happens in the limit as the mesh size tends to zero and the (appropriately rescaled) height function on the boundary converges to a limiting function \cite{cohn2001variational}. Note that given any tiling $\tau$ that covers $R$ (e.g.\ $\tau$ could be one of the brickwork tilings) one can form a tileable region $R_n$ by restricting to the tiles strictly contained in $R$, and the choice of $R_n$ 
determines how the height function changes along the boundary.
\begin{figure}[H]
    \centering
    \includegraphics[scale=0.45]{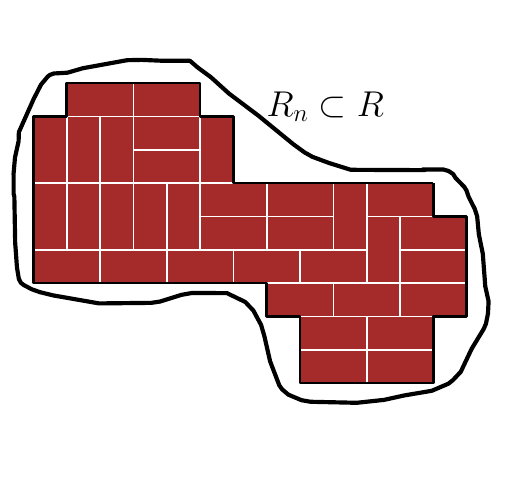}
    \caption{An example of a fixed boundary region $R_n\subset R$ for the LDP in two dimensions.}
    \label{fig:2D_fixed_boundary}
\end{figure}
Cohn, Kenyon and Propp showed that as the mesh size tends to zero (and the rescaled boundary heights converge to some function on $\partial R$) the random height function converges in probability to the unique continuum function $u$ that (given the boundary values) minimizes the integral \begin{equation} \label{eqn::2dent} \int \sigma\big( \nabla u(z) \big) dz, \end{equation} where $\sigma(s)$ is the {\em surface tension} function, which means that $-\sigma(s)$ is the {\em specific entropy} (a.k.a.\ {\em entropy per vertex}) of any ergodic Gibbs measure of slope $s$. More generally, they established a theory of {\em large deviations} by showing how exponentially unlikely the random height function would be to concentrate near any other $u$. Earlier work studied this problem specifically for the \textit{Aztec diamond}, see \cite{CohnElkiesPropp} and \cite{jockusch1998random}.

The proof in \cite{cohn2001variational} used ingredients from the scalar height function theory (McShane-Whitney extensions, monotone couplings, stochastic domination, etc.)\ and the Kasteleyn determinant representation (an exact formula for the entropy function) that do not appear readily adaptable to three dimensions. 

The literature on the two-dimensional dimer model is quite large and we will not attempt a detailed survey here. Introductory overviews with additional references include e.g.\ \cite{kenyon2009lectures} and \cite{gorin2021lectures}.

We remark that fixing the asymptotic height function boundary values on $\partial R$ is equivalent to fixing the asymptotic rate at which current flows though $\partial R$ in the corresponding divergence-free flow. The latter interpretation is the one that extends most naturally to higher dimensions.

\subsection{Three-dimensional setup and simulations}
The goal of this paper is to extend \cite{cohn2001variational} to higher dimensions, where different tools are required. For simplicity and clarity, we focus on 3D, but we expect similar arguments to work in dimensions higher than three. (We discuss possible generalizations and open problems in Section~\ref{sec:open}.) Before we present our main results, we provide a few illustrations. The figure below illustrates a uniformly random tiling of a $10 \times 10 \times 10$ cube, with six colors corresponding to the six orientations. Next to it is the underlying black-and-white checkerboard grid. This figure and the others below were generated by a Monte-Carlo simulation (see Section \ref{sec:uniform sampling}) that is known to be mixing, but whose mixing rate is not known. It was run long enough that the pictures {\em appeared} to stabilize but we cannot quantify how close our samples are to being truly uniform. The efficient {\em exact sampling} algorithms that work in 2D do not have known analogs in 3D.

\begin{figure}[H]
    \centering
    \includegraphics[scale=0.35]{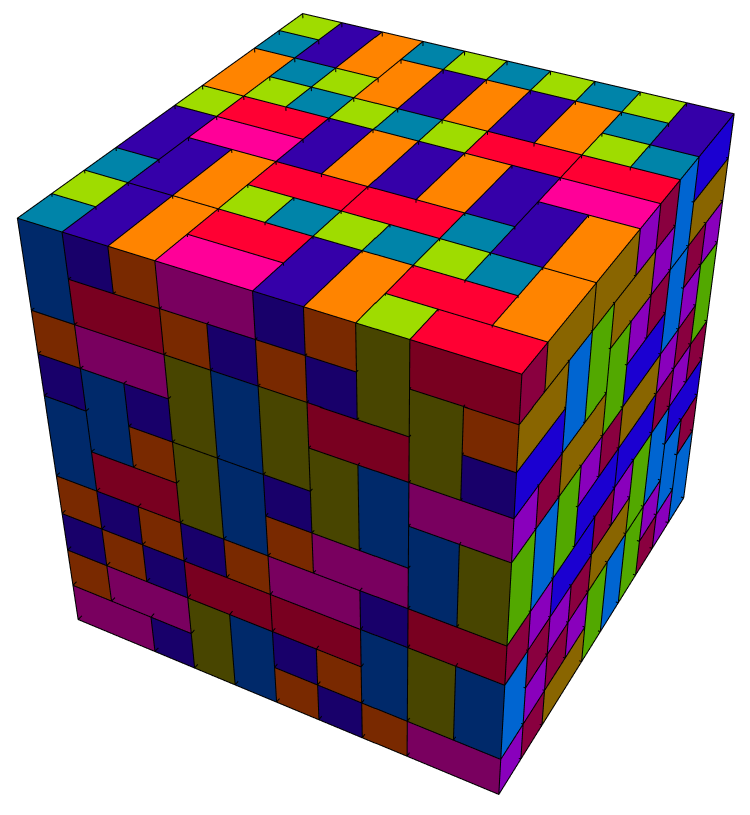}$\qquad \qquad \qquad$ \includegraphics[scale=0.35]{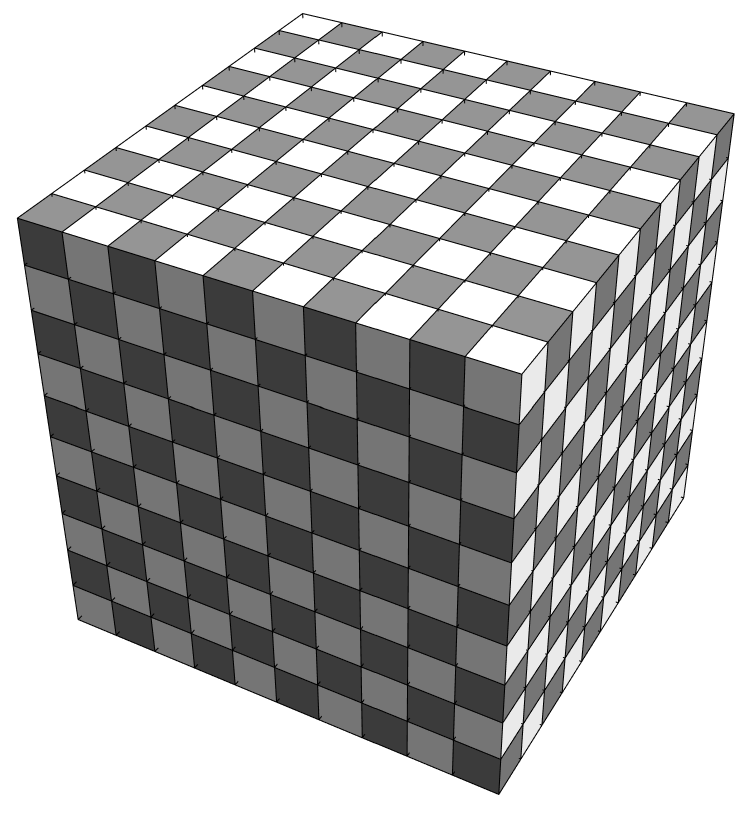}
    \caption{A dimer tiling of the $10\times 10\times 10$ cube and the bipartite coloring of the cubes in $\m Z^3$.}
    \label{fig:cubes!}
\end{figure}
The figure below represents a random tiling $\tau$ of a region $R$ called the {\em Aztec pyramid} (formed by stacking Aztec diamonds of width $2, 4, 6, \ldots, 36$). Next to it is again the underlying black-and-white checkerboard coloring. Recall that (due to the reference flow) the divergence-free flow $f_\tau$ sends a $1/6$ unit of current through each square on the boundary $\partial R$. Such a square divides a cube inside $R$ from a cube outside $R$. The flow is directed {\em into} $R$ if the cube inside $R$ is even, and {\em out of} $R$ if the cube inside $R$ is odd. Of the four triangular faces of the pyramid, two consist entirely of even cubes and the other two consist entirely of odd cubes. This means that $f_\tau$ current enters two opposite triangular faces at its maximal rate, and exits other two triangular faces at its maximal rate, while the net current through the bottom square face is zero (since on the lower boundary, the number of faces bounding even cubes in $R$ equals the number of faces bounding odd cubes in $R$).
\begin{figure}[H]
    \centering
    \includegraphics[scale=0.2]{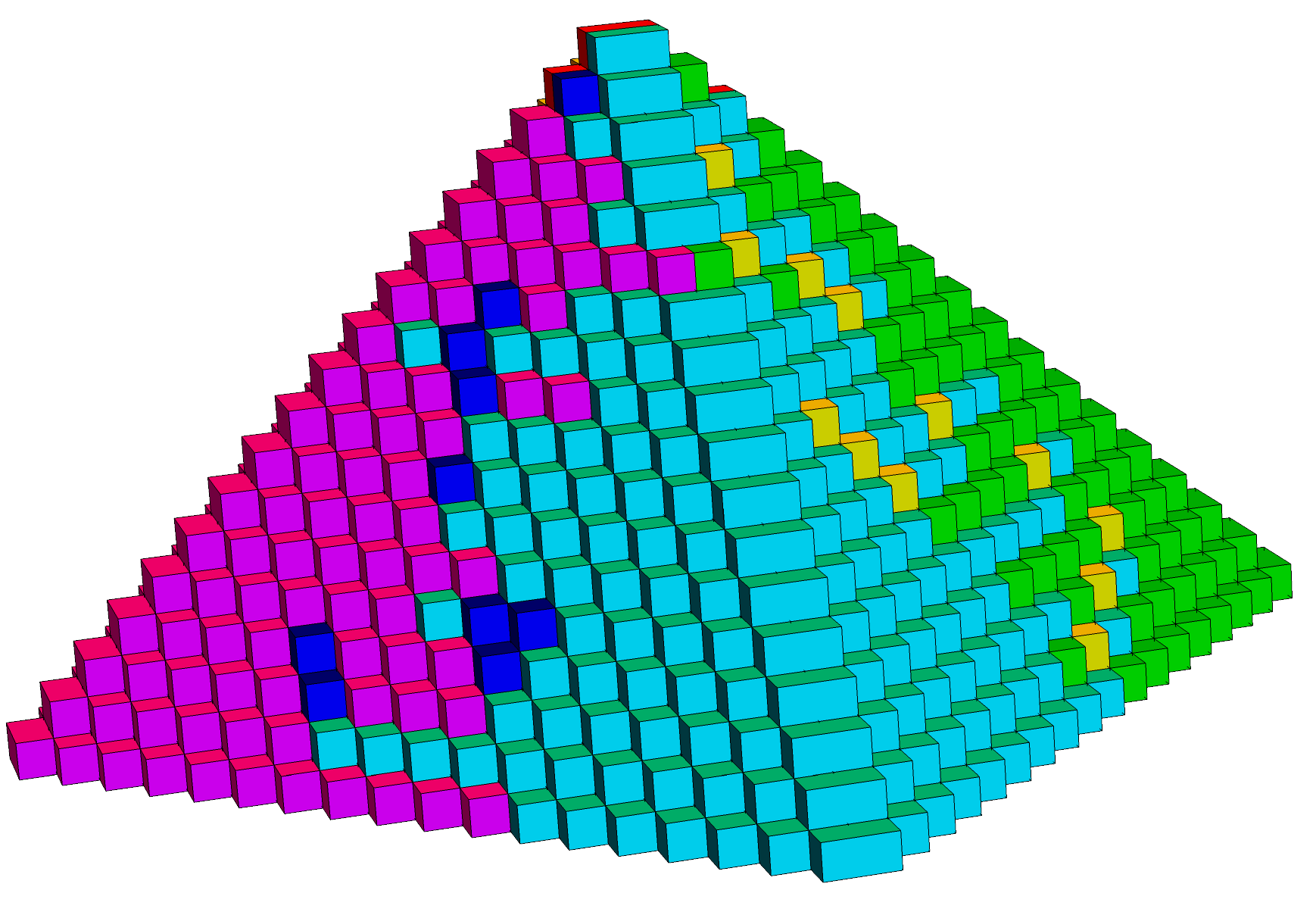}$\qquad\qquad$\includegraphics[scale=0.37]{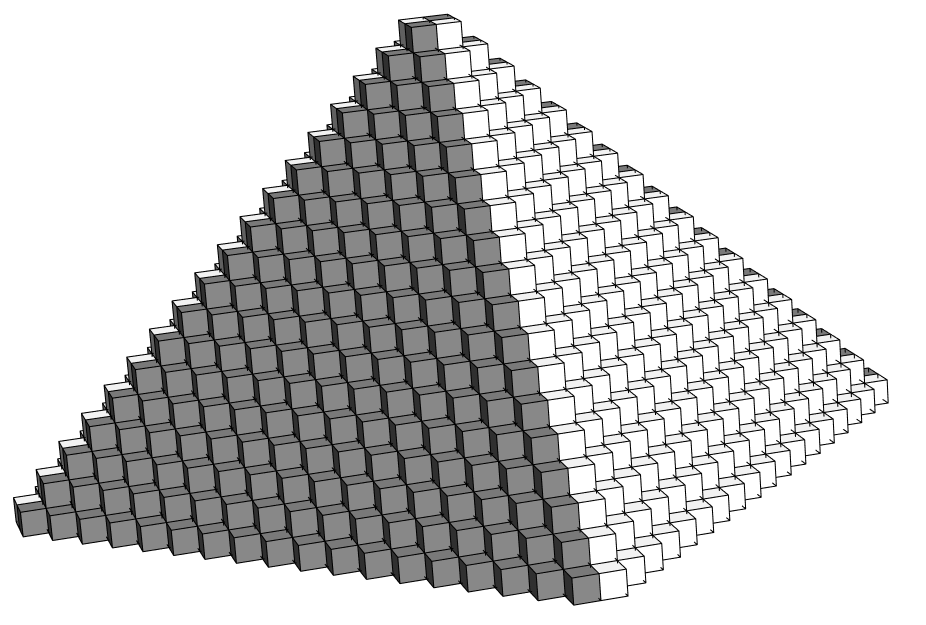}
    \caption{A dimer tiling of an Aztec pyramid and the bipartite coloring of the cubes in $\m Z^3$.}
    \label{fig:pyramids!}
\end{figure}
Below is a larger Aztec pyramid seen from above and from underneath. One can construct a computer animation showing the horizontal cross-sections one at a time. For the three large simulations shown here in Figures~\ref{fig:large_pyramid}, \ref{fig:aztecoctahedron} and \ref{fig:prism}, animations of the slices are available at \href{https://github.com/catwolfram/3d-dimers}{https://github.com/catwolfram/3d-dimers}. In these animations, it appears that each cross section has four ``frozen'' brickwork regions and a roughly circular ``unfrozen'' region, similar to the 2D Aztec diamond.
\vspace{-.15in}
\begin{figure}[H]
    \centering
         \includegraphics[scale=0.50]{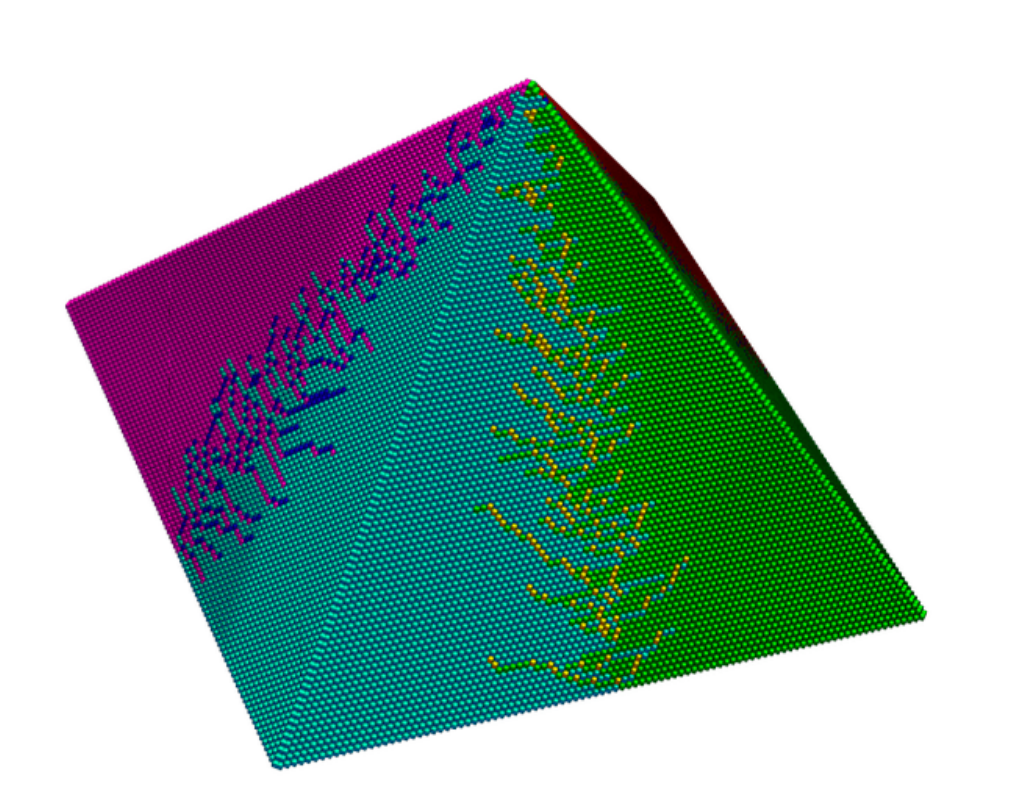}\includegraphics[scale=0.20]{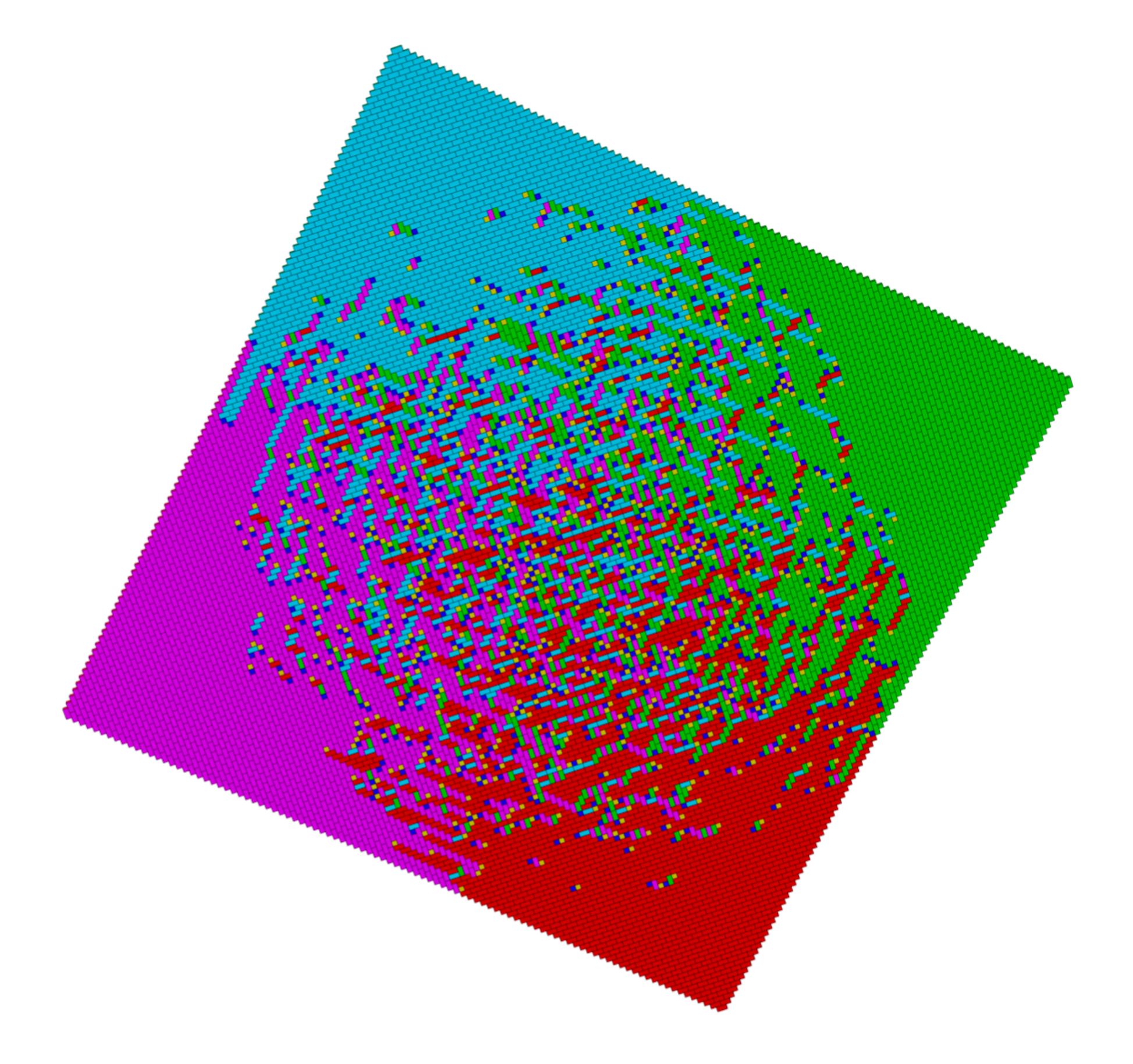}
    \caption{A tiling of a larger pyramid of Aztec diamonds, from the side and from below.}
    \label{fig:large_pyramid}
\end{figure}
\vspace{-.1in}
We now describe the two larger labeled figures. Figure~\ref{fig:aztecoctahedron} illustrates a uniformly random tiling of the {\em Aztec octahedron} formed by gluing two Aztec pyramids along their square face. Four of the eight triangular faces of the octahedron contain only even cubes on their boundary, and the other four contain only odd cubes (and these alternate; distinct faces sharing an edge have opposite parity). In light of this, we can say that the current enters four of the faces at the maximum possible rate and exits the other four faces at the maximum possible rate. In simulations there appear to be twelve frozen regions (one for each {\em edge} of the octahedron) in which one of the six brickwork patterns dominates. (By contrast, tilings of the two dimensional Aztec diamond have four frozen regions, one for each {\em vertex} of the diamond.) Within each brickwork region, current flows at its maximum possible rate from one face (where it enters the octahedron) to an adjacent face (where it exits). Away from these brickwork regions, one sees a mix of colors, with a gradually varying density for each color. These are regions where the magnitude of the current flow is smaller, and the mean current appears to vary continuously across space.

\begin{figure}[!ht] 
\begin{center}
   \includegraphics[scale=0.3]{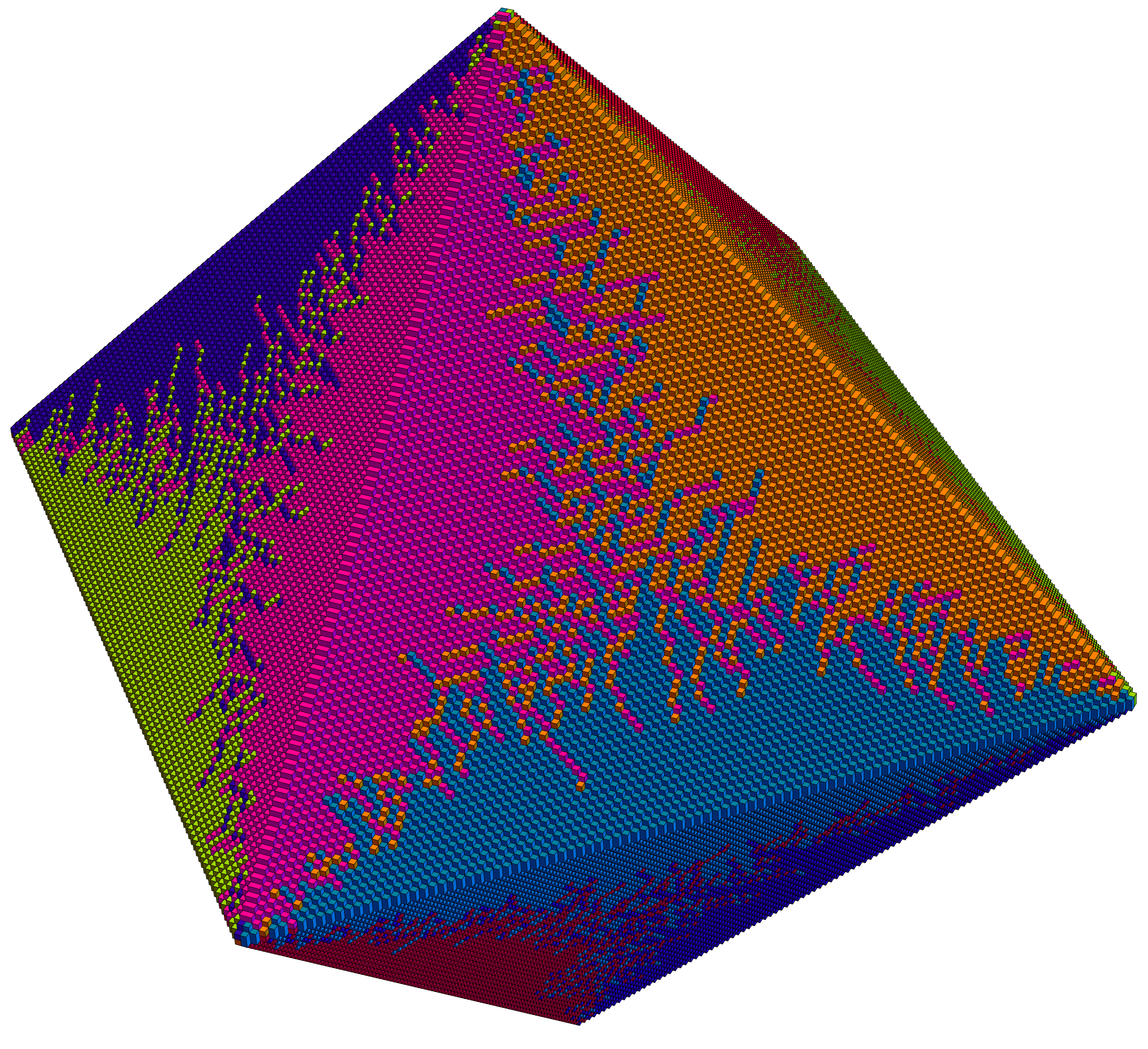}
\end{center}
\caption{Tiling of an Aztec octahedron.}
\label{fig:aztecoctahedron}
\end{figure}

Figure~\ref{fig:prism} illustrates a tiling of the {\em Aztec prism} (formed by stacking Aztec diamonds whose widths alternate between $2n$ and $2n+2$). Again, each slice seems to be frozen outside of a roughly circular region. The width-alternation ensures that each of the four rectangular side faces of the prism has either only even faces or only odd faces exposed. Thus, current enters two of the opposite side faces at its maximal rate, and exits the other two at its maximal rate. The net current flowing through the top and bottom faces is zero. In this figure, and in all of the examples above, the distribution of domino colors in a subset of the tiled region determines the ``mean direction of current flow'' in that subset. We are interested in understanding what the ``typical flow'' looks like in the fine mesh limit.

\begin{figure}[!ht] \begin{center}
\includegraphics[scale=0.2]{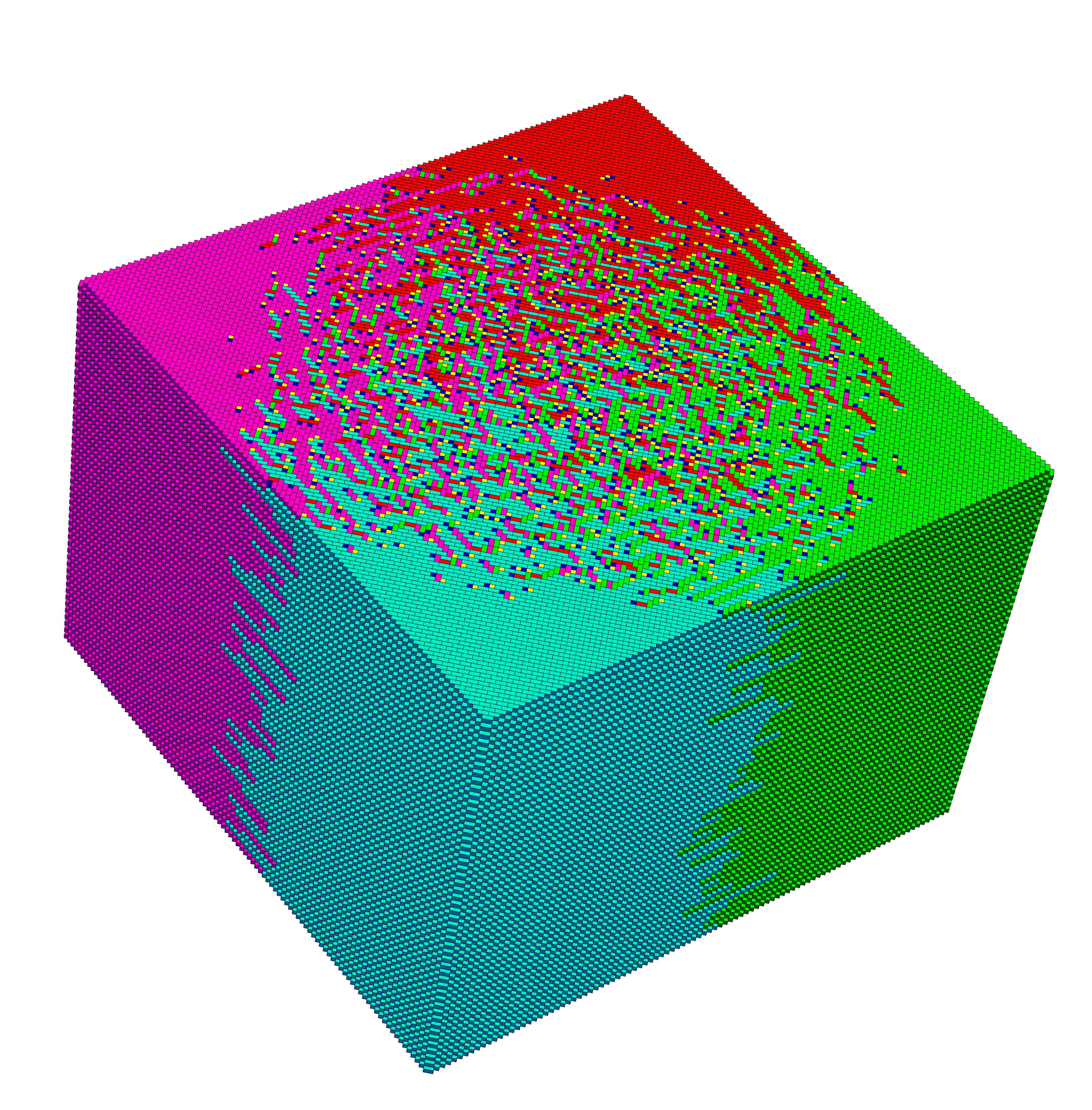}
\end{center}
\vspace{-.3in}
\caption{Tiling of an Aztec prism.}\label{fig:prism}
\end{figure}

\subsection{Main results and methods}\label{sec: main results intro}

The main results of this paper will be two versions of a \textit{large deviation principle} (LDP) for fine-mesh limits of uniformly random dimer tilings of compact regions $R\subset \m R^3$, with some limiting boundary condition. The versions of the LDP we prove differ in how we treat the boundary conditions.

A large deviation principle is a result about a sequence of probability measures $(\rho_n)_{n\geq 1}$ which quantifies the probability of rare events at an exponential scale as $n\to \infty$. More precisely, a sequence of probability measures $(\rho_n)_{n\geq 1}$ on a topological space $(X,\mc B)$ is said to satisfy a large deviation principle (LDP) with rate function $I$ and speed $v_n$ if $I : X \to [0,\infty)$ is a lower semicontinuous function, and for all Borel measurable sets $B$,
\begin{align*}
    -\inf_{x\in B^\circ} I(x) \leq \liminf_{n\to \infty} v_n^{-1} \log \rho_n(B) \leq \limsup_{n\to \infty} v_n^{-1} \log \rho_n(B) \leq -\inf_{x\in \overline{B}} I(x),
\end{align*}
where $B^\circ, \overline{B}$ denote the interior and closure of $B$ respectively. The rate function $I(x)$ is \textit{good} if its sublevel sets $\{x: I(x) \leq a\}$ are compact. Implicit in this definition is a choice of topology $\mc B$ on $X$. A large deviation principle for $(\rho_n)_{n\geq 1}$ implies that random samples from $\rho_n$ are exponentially more likely to be near the minimizers of $I(\cdot)$ as $n\to\infty$. When $I$ is good and has a unique minimizer, this means that random samples from $\rho_n$ \textit{concentrate} as $n\to \infty$ in the sense that if $U$ is any neighborhood of the unique minimizer, then as $n\to \infty$ the probability $\rho_n(X \setminus U)$ tends to zero exponentially quickly in $v_n$. Good references for this subject include \cite{dembo2009large} and \cite{varadhan2016large}. 

To formulate the large deviation principle for dimer tilings, we associate to each tiling $\tau$ of $\m Z^3$ a divergence-free discrete vector field. As mentioned above, we can define a flow $v_\tau$ that moves one unit of flow from the even endpoint to the odd endpoint of each $e$ in $\tau$. Subtracting a reference flow which sends $\frac{1}{6}$ flow along every edge from even to odd produces a divergence-free discrete vector field:
\begin{align*}
    f_\tau(e) = \begin{cases} +5/6 &\quad e\in \tau \\ 
    -1/6 &\quad e\not\in \tau. \end{cases}
\end{align*}
We call this a \textit{tiling flow}, and it will play an analogous role to the height function in two dimensions. The height function in two dimensions is a scalar potential function whose gradient is the {\em dual} of the tiling flow  (i.e., the flow on the dual lattice obtained by rotating each edge 90 degrees counterclockwise about its center). See Section \ref{section:tiling_flows}. 

Tiling flows can also be \textit{scaled} so that they are supported on $\frac{1}{n} \m Z^3$ instead of $\m Z^3$. We scale them so that the total flow per edge is proportional to $\frac{1}{n^3}$, so that in the fine-mesh limit with $n\to \infty$, the total flow in a compact region in $\m R^3$ converges to a finite value. The scaled tiling flows takes values $-\frac{1}{6 n^3}$ and $\frac{5}{6 n^3}$. A \textit{scale $n$ dimer tiling} is a dimer tiling of $\frac{1}{n} \m Z^3$ (or a subset of it).

Fix a compact region $R\subset \m R^3$ (which is sufficiently nice, e.g.\ the closure of a connected domain with piecewise smooth boundary). We define the \textit{free-boundary tilings of $R$ at scale $n$} to be tilings $\tau$ at scale $n$ such that every point in $R$ is contained in a tile in $\tau$, and every tile has some intersection with $R$. 
\begin{figure}[H]
    \centering
\includegraphics[scale=0.4]{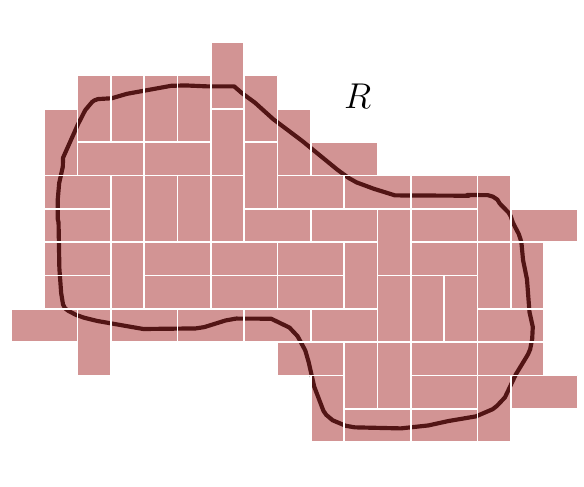}
\caption{An example of a free-boundary tiling of $R$}
    \label{fig:free_boundary_ex}
\end{figure}
We denote the corresponding \textit{free-boundary tiling flows of $R$ at scale $n$} by $TF_n(R)$. Note that $TF_n(R)$ is a finite set for all $n$. There is a signed {\em flux measure} on $\partial R$ (supported on the points where edges of $\frac{1}{n} \mathbb Z^3$ cross $\partial R$) that encodes the net amount of flow directed into $R$. Since $f_\tau$ is divergence-free, the {\em total} flux through $\partial R$ is zero. (See Definition \ref{sec: boundary values of tiling flows}.)

If $\tau$ is a {\em random} tiling of $\mathbb Z^3$ whose law is invariant under translations by even vectors, then one can define the {\em mean current per even vertex} to be $s=\mathbb E[\eta]$ where $\eta$ is the vertex of $\mathbb Z^3$ matched to the origin by $\tau$. Note that $s$ lies in the set
$$\mc O = \{ (s_1,s_2,s_3)\in \m R^3: |s_1| + |s_2| + |s_3|\leq 1\},$$
which is the convex hull of $\{ (\pm 1,0,0), (0,\pm 1, 0), (0,0,\pm 1)\}$, and which we call the {\em mean-current octahedron}. It indicates the expected total amount of current in $v_\tau$ (or equivalently $f_\tau$) per even vertex; see Section \ref{section:measures-currents}.  The vertices of $\mc O$ arise for a random $\tau$ that is a.s.\ equal to one of the six brickwork tilings in three dimensions.

We define $AF(R)$ to be the space of measurable, divergence-free vector fields supported in $R$ taking values in $\mc O$.
The notation $AF$ stands for {\em asymptotic flow} and is chosen because of the fact (formalized in Theorem \ref{thm:formal_fine_mesh}) that these are precisely the flows that can arise as $n \to \infty$ limits of tiling flows on $R$, under a suitable topology.

The topology we use is the {\em weak topology} on the space of flows obtained by interpreting each vector component of the flow as a signed measure, see Section~\ref{sec:asympflows}. This topology can also be generated by the \textit{Wasserstein distance} for flows. Loosely speaking, two flows are considered Wasserstein close if one can be transformed into the other by moving, adding, and deleting flow with low ``cost." The cost is the amount of flow moved times the distance moved, plus the amount of flow added or deleted. The large deviation principles we prove use the same topology (weak topology, which is generated by Wasserstein distance) for the fine-mesh limits of random free-boundary tiling flows of $R$. 

\begin{rem} The 2D large deviation analysis in \cite{cohn2001variational} is based on a topology that at first glance looks different from ours, namely the topology in which two flows are considered close if their corresponding height functions $h$ are close in $L^\infty$. However, it is not too hard to see that $1$-Lipschitz functions on $R$ (with fixed boundary values on $\partial R$) are $L^\infty$ close if and only if their gradients (or equivalently the dual flows of their gradients) are Wasserstein close. We will not use or prove this fact here. 
\end{rem}

In three dimensions, we will also derive an LDP for random perfect matchings on finite regions approximating a continuum domain $R$, with boundary conditions chosen so that the flux through the boundary approximates a continuum boundary flow $b$, in a sense we will explain below. As in two dimensions, the rate function $I_b(\cdot)$ is a function of an asymptotic flow $g\in AF(R)$ and is (up to an additive constant) equal to 
\begin{align} \label{eqn::3dent}
    -\Ent(g) = -\frac{1}{\text{Vol}(R)} \int_{R} \ent(g(x))\, \dd x,
\end{align}
where the additive constant is $C = \max_{f\in AF(R,b)} \Ent(f)$ ($I_b$ depends on $b$ because of this constant). We interpret \eqref{eqn::3dent} as the three-dimensional analog of \eqref{eqn::2dent}. The function $\ent(s)$ is the maximal \textit{specific entropy} of a measure with mean current $s\in \mc O$. It turns out that for every $s \in \mc O$ there exists an \textit{ergodic Gibbs measure} of mean current $s$ that achieves this maximal entropy $\ent(s)$. This essentially follows from the strict concavity of $\ent$ when $s$ is in the interior of $\mc O$, and can also be checked for $s \in \partial \mc O$. See Theorem~\ref{theorem: existence of gibbs} and Theorem~\ref{thm: extremal_entropy}. 

To state our theorems, we need a way of fixing for each $n$ a region $R_n$ that approximates a continuum region $R$, such that boundary flow corresponding to $R_n$ approximates a continuum boundary flow $b$ on $\partial R$. The precise analog of the statement given in \cite{cohn2001variational} is not exactly true in 3D, due to subtleties related to the fact that in 3D the function $\ent$ can be nonzero even on the boundary of $\mc O$ (see  Section~\ref{sec:extreme}). To briefly illustrate what can go wrong, consider a finite region $S$ tiled with only three types of tiles: north, east and up. Then every vertex $x = (x_1,x_2,x_3)$ with $x_1+x_2+x_3 = c$ (with $c$ even) must be connected to a vertex $y$ with $y_1+y_2+y_3 = c+1$, and vice versa. The vertices with coordinate sums in $\{c, c+1 \}$ thus form a ``slab'' of points that are only connected to each other, and one can use Hall's matching theorem to show that this must be the case for {\em any} tiling of $S$. Thus we can view $S$ as a stack of two-dimensional slabs, where the tilings within the different slabs are independent of each other.  These slabs could alternate between brickwork patterns (one slab has only east-going tiles, the next has only north-going tiles, the next has only up-going tiles, etc.) but they could also all be nonzero-entropy mixtures of the three tile types. Rescalings of both types of $S$ might approximate the same continuum $(R,b)$, but the corresponding uniform-random-tiling measures would have very different entropy and very different large deviation behavior (see Example \ref{ex:hb_false} and Section~\ref{sec:extreme}).

We will present two ways of formulating a theorem that {\em is} true in 3D. In the first approach we replace the {\em hard constraint} on the boundary behavior (where an $R_n$ to be tiled is explicitly specified for each $n$) with a {\em soft constraint} (where all scale $n$ tilings that cover $R$ are allowed, provided they induce boundary flows that are ``sufficiently close'' to the desired limiting flow $b$). This ``soft constraint'' LDP will apply to a fairly general set of pairs $(R,b)$. In the second approach we keep the hard constraints (i.e., the fixed $R_n$ regions) but require $(R,b)$ to be ``flexible'' in a sense that prevents the degenerate situation sketched above (where on the discrete level there might be slab boundaries that cannot be crossed by the tiles of {\em any} tiling of $R_n$). Precisely, we say $(R,b)$ is \textit{flexible} if for every $v \in R$ there exists an asymptotic flow $g$ for $(R,b)$ such that for some neighborhood $U \ni v$ we have $\overline{g(U)} \subset \text{Int}(\mathcal O)$. Informally, this means there is no interior point near which $g$ is {\em forced} to lie on $\partial \mathcal O$.
For later purposes, we say that $(R,b)$ is \textit{semi-flexible} if for every $v \in R$ there exists an asymptotic flow $g$ for $(R,b)$ such that for some neighborhood $U \ni v$, the set $\overline{g(U)}$ is disjoint from the edges of $\mc O$. In other words, there is no interior point near which $g$ is forced to lie on an edge of $\mathcal O$.

Using compactness of the space of flows, it is not hard to show that there exists $g$ that minimizes \eqref{eqn::3dent}. However it takes a bit more work to see whether such a $g$ is unique.  If $\ent$ were strictly concave everywhere, then $I_b(g_1)$ and $I_b(g_2)$ could not be minimal for distinct $g_1$ and $g_2$ (since the strict concavity would imply that $I_b\bigl(\frac{g_1+g_2}{2}\bigr)$ was even smaller). The trouble is that $\ent$ is not strictly concave on the edges of $\mathcal O$ (where it is constant) even though we will show that it is strictly concave everywhere else. In principle, there could still exist distinct minimizers $g_1$ and $g_2$ that (outside a set of measure zero) disagree {\em only} at points where they both take values on the same edge of $\mathcal O$. (See Problem~\ref{prob: region with non unique max}.) We have not ruled out this possibility in general, but we can prove that \eqref{eqn::3dent} has a unique minimizer if $(R,b)$ is semi-flexible. (See \cite[Proposition 7.10]{gorin2021lectures} for a 2D analog of this argument.) This in turn implies that the random flows in the corresponding LDP theorems {\em concentrate} near the unique minimizer $g$ of $I_b$. If $(R,b)$ is not semi-flexible then we call it \textit{rigid}. We briefly summarize the conditions needed for each result in the following table before giving more precise statements.
\begin{center} \begin{tabular}{|c|c|c|c|}
    \hline
     $(R,b)$ & SB LDP & $I_b$ has unique minimizer & HB LDP\\
\hline
rigid & {yes} & {not known} & {no} \\
\hline
semi-flexible & {yes} & {yes} & {no}\\
\hline
flexible & {yes} & {yes} & {yes}\\
\hline
\end{tabular}
\end{center}
The results marked ``no" in this table are provably not true in general. By taking limits of the ``stack of slabs" regions discussed above, one can produce a semi-flexible (or rigid) pair $(R,b)$ for which the hard boundary large deviation principle is false (for further discussion of this see Example \ref{ex:hb_false}). 

Now, to introduce the soft boundary large deviation principle, we define the probability measure $\rho_n$ to be the uniform measure on the space of flows in $TF_n(R)$ whose boundary values lie within Wasserstein distance $\theta_n$ of the desired limiting boundary flow, where $\theta_n\to 0$ as $n\to \infty$. We call $\theta_n$ the ``threshold sequence'' and it can be chosen arbitrarily provided that it does not tend to zero {\em too quickly} in a sense we explain later. A rough statement of our main theorem is the following.

\begin{thm}[See Theorem \ref{thm:sb-ldp}] \label{thm:roughmain}
Let $R\subset \m R^3$ be a compact region which is the closure of a connected domain, with piecewise smooth boundary $\partial R$. Let $b$ be a boundary value for an asymptotic flow and let $(\theta_n)_{n\geq 1}$ be a (good enough) sequence of thresholds. Let $\rho_n$ be uniform measure on $TF_n(R)$ conditioned on the boundary values being within $\theta_n$ of $b$. 

Then the measures $(\rho_n)_{n\geq 1}$ satisfy a large deviation principle in the Wasserstein topology on flows with good rate function $I_b(\cdot)$ and speed $v_n = n^3 \text{Vol}(R)$, namely for any Wasserstein-measurable set $A$,
    \begin{equation}\label{eq:ldp-prob-1}
    -\inf_{ g\in A^\circ} I_{ b}( g) \leq \liminf_{n\to \infty} v_n^{-1} \log \rho_n(A) \leq \limsup_{n\to \infty} v_n^{-1} \log \rho_n(A) \leq -\inf_{ g\in \overline{A}} I_{ b}( g).
\end{equation}
If $g$ is an asymptotic flow, the rate function $I_{ b}( g)$ is equal up to an additive constant to $-\Ent(g)$. (Otherwise it is $\infty$.)
\end{thm}

\begin{rem}
The requirements for the region $R$ are mild---for example, we do not require that $R$ is simply connected. In this sense, the theorem can be viewed as extending both \cite{cohn2001variational} (simply connected 2D) and \cite{kuchumov22} (multiply connected 2D) to three dimensions. The requirement that $\partial R$ be piecewise smooth is probably not necessary, but if the boundary of $R$ is allowed to be too rough, the theorem statements one can make will depend more sensitively on how the boundary conditions are handled. For example, if the boundary of $R$ has positive volume, then tilings that cover $R$ may have volume-order more tiles than the tilings that approximate $R$ ``from within'' and the extra tiles may contribute to the limiting entropy. If the boundary of $R$ has infinite area, then the flux through the boundary may be an infinite signed measure, which would have to be defined more carefully. (For example, we could say that two flows that vanish outside of $R$ have ``equivalent boundary values'' if their difference is a flow on all of $\mathbb R^3$ that is divergence-free in the distributional sense, and then let $b$ denote an equivalence class.) For simplicity, we will focus on the piecewise smooth setting in this paper.
\end{rem}

The distinction between soft and hard boundary conditions only substantially impacts one step of the proof: the argument that there exists a tiling (in the support of $\rho_n$) whose flow approximates a piecewise-constant flow that in turn approximates a given $g \in AF(R)$. Theorem~\ref{thm:roughmain} would still apply if the boundary conditions defining the $\rho_n$ were specified in another way (ensuring convergence to $(R,b)$ in the limit) as long as some version of this step could be implemented. We show using the \textit{generalized patching theorem} (Theorem \ref{thm:generalized_patching}) that under the condition that $(R,b)$ is flexible, this step can be implemented and a hard boundary LDP holds. 

Fixing a specific region $R_n\subset \frac{1}{n}\m Z^3$ and looking only at tilings of exactly this sequence of regions puts a stronger constraint on the boundary values. A rough statement of the \textit{hard boundary} large deviation principle, where we fix the approximating regions $R_n$ exactly, is as follows.

\begin{thm}[Theorem \ref{thm:hb-ldp}]\label{thm:roughmain-hb}
    Suppose that $(R,b)$ is flexible and that $R_n\subset \frac{1}{n}\m Z^3$ is a sequence of tileable regions with boundary values converging to $b$. Let $\overline \rho_n$ be uniform measure on dimer tilings of $R_n$. Then the measures $(\overline{\rho}_n)_{n\geq 1}$ satisfy a large deviation principle in the Wasserstein topology on flows with the same good rate function $I_b(\cdot)$ and speed $v_n = n^3 \text{Vol}(R)$ as the soft boundary measures $\rho_n$. 
\end{thm}

It is straightforward to show that the $(R,b)$ pairs obtained as fine-mesh limits of the ``Aztec regions'' above (Figures \ref{fig:large_pyramid}, \ref{fig:aztecoctahedron}, \ref{fig:prism}) are flexible, despite the fact that {\em typical} tilings appear (in simulations) to have frozen brickwork regions. (See Remark \ref{rem:aztec_is_flexible}.)

Under the condition that $(R,b)$ is \textit{semi-flexible} (this is weaker than flexible, see further discussion in Definition \ref{def: flexible}) we show that $\Ent$ has a unique maximizer (Theorem \ref{thm: unique maximizer}). This together with some basic topological results shows rigorously that \textit{concentration} around a deterministic limit shape occurs, as we see in the simulations. This concentration holds for either the soft boundary measures $\rho_n$ or the hard boundary measures $\overline \rho_n$. 
\begin{cor}[See Corollary \ref{cor:concentration} and Corollary \ref{cor:concentration-hb}]
Assume that $(R,b)$ is semi-flexible. For any $\epsilon>0$, the probability that a uniformly random tiling flow on $R$ at scale $n$ (either sampled from $\rho_n$, i.e.\ with boundary flow conditioned to be in a shrinking interval around $b$, or sampled from $\overline \rho_n$ if $(R,b)$ is flexible, i.e.\ tilings of a fixed region $R_n$ with boundary values converging to $b$) differs from the entropy maximizer with boundary value $b$ by more than $\epsilon$ goes to $0$ exponentially fast in $n^3$ as $n\to \infty$. 
\end{cor}
The methods in this paper are substantially different from the methods used to prove the large deviation principle for dimer tilings in two dimensions. The two-dimensional dimer model is \textit{integrable} or {\em exactly solvable} in the sense that one can derive a (beautiful) explicit formula for the specific entropy function analogous to our function $\ent$, and this explicit formula is used in the large deviations proof. The three-dimensional dimer model is not known to be integrable in this way, so we rely on ``softer" arguments. We comment on a few of these below.

One of the key ingredients which does have a 2D analog in \cite{cohn2001variational} is the \textit{patching argument} (Theorem \ref{patching}) which essentially states that if two tilings $\tau_1,\tau_2$ satisfy a requirement that they ``asymptotically have the same mean current $s$" for some $s\in \text{Int}(\mc O)$, then we can cut out a bounded portion of $\tau_2$ and patch it into an unbounded portion of $\tau_1$ by tiling a thin annular region. 

\begin{figure}[H]
    \centering
    \includegraphics[scale=0.65]{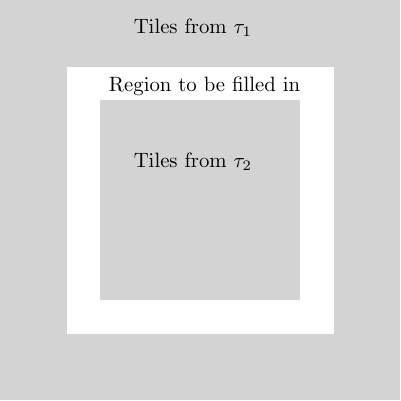}
    \caption{Schematic for the patching theorem.}
    \label{fig:patching_scheme}
\end{figure}

For the hard boundary large deviation principle, we also prove a \textit{generalized patching theorem} (Theorem \ref{thm:generalized_patching}), which says roughly that two tilings can be patched on a \textit{general} annular region $R\setminus R'$ if they have the same boundary value $b$ on $\partial R$ and the inner one approximates a \textit{flexible} flow $g\in AF(R,b)$. 

Proving the patching theorems will be one of the more challenging aspects of this paper. The main input is Hall's matching theorem (proved by Hall in 1935 \cite{hallmarriage}) which gives us a criterion to check if a region (e.g.\ the annular region between the two tilings) is tileable by dimers. It turns out that the criterion we need to check can be phrased in terms of the existence of a discretized minimal surface and leads to an interesting sequence of constructions described for boxes in Section \ref{sec:patching plus hall} and generalized in Section~\ref{sec:ldp}. These arguments are more involved than the 2D patching arguments in \cite{cohn2001variational}, which rely on height functions and Lipschitz-extension theory. It is hard to summarize the argument without giving the details, but the following is a very rough attempt (which can be skimmed on a first read).
\begin{enumerate}
    \item For each $n$, define an annular region $A$ that we want to tile (which is roughly a scale $n$ approximation of a fixed continuum annular region, with outer boundary conditions determined by $\tau_1$ and inner by $\tau_2$). Use Hall's matching theorem to show that if $A$ is not tileable there must exist a ``surface'' dividing the cubes in $A$ into two sub-regions such that 1) the cubes with faces on the surface are odd if they are in the first sub-region, even if in the second and 2) the first sub-region has more even than odd cubes overall.
    \item Reduce to the case that the surface is in some sense a ``minimal monochromatic surface'' given its boundary, which touches both the inside and outside boundaries of $A$. (Here monochromatic means that all cubes on one side of the surface, and adjacent to it, are odd, and minimal monochromatic means that it has minimal area given the monochromatic constraint.)
    \item Use an argument involving isoperimetric bounds to show that such a surface must have at least a constant times $n^2$ faces when $n$ is large.
    \item Show that the even-odd imbalance in the first sub-region {\em cannot} be as large as it would need to be to provide a non-tileability proof. Do this by covering the first region by dominoes (from a tiling sampled from an ergodic measure in Section~\ref{sec:patching plus hall}, or from a tiling that approximates a flow $g$ whose existence is guaranteed by the flexible condition in Section~\ref{sec:ldp}) and use geometric considerations to show that there must be a lot of dominoes with only an odd cube in the first sub-region (including order $n^2$ in the middle of $A$) and relatively fewer dominoes with only an even cube in the first sub-region (using the ergodic theorem and the fact that both tilings approximate the same constant flow $s$, or in the generalized case by using Wasserstein distance bounds that apply near the boundary of $A$). Conclude that the first sub-region has at least as many odd cubes as even cubes, and hence does not prove non-tileability. This argument shows that there exists no surface that proves non-tileability and hence (by Hall's matching theorem) the region is tileable.  
\end{enumerate}

Another key step in proving the main theorems is to derive properties of the entropy function $\ent$ despite not being able to compute it explicitly on all of $\mc O$. Instead, $\ent(s)$ is defined abstractly as the maximum specific entropy $h(\cdot)$ of a measure with mean current $s$ (see Section \ref{sec: entropy prelim}). From a straightforward adaptation of the classical variational principle of Lanford and Ruelle \cite{LanfordRuelle}, it follows that $\ent(s)$ is always realized by a \textit{Gibbs measure} of mean current $s$ (see Theorem~\ref{theorem: entropy maximizer gibbs}, see also Section~\ref{sec::prelim} for the definition of a Gibbs measure). 

To prove strict concavity of $\ent$ on $\text{Int}(\mc O)$ (Theorem \ref{theorem: entropy is strictly concave}), we note that a translation invariant measure $\mu$ with mean current $s$ and with $h(\mu) = \ent(s)$ must be a Gibbs measure, and we then use a variant of the cluster swapping technique used in \cite{AST_2005__304__R1_0} to compare measures of different mean currents. We call this variant \textit{chain swapping}. It is an operation on measures on \textit{pairs} of dimer tilings $(\tau_1,\tau_2)$. From a pair of tilings $(\tau_1,\tau_2)$ (sampled from $\mu$), chain swapping constructs a pair of tilings $(\tau_1',\tau_2')$ by ``swapping" the tiles of $\tau_1$ and $\tau_2$ along some of the infinite paths in $(\tau_1,\tau_2)$ with independent probability $1/2$ (or any probability $p\in (0,1)$). We say that $(\tau_1',\tau_2')$ is sampled from the \textit{swapped measure} $\mu'$. See Section \ref{Subsection: Dimer Swapping} for a more detailed discussion of chain swapping, including Figure \ref{fig:chain swap} for an example.

At a high level, chain swapping is an operation that allows us to take a coupling $\mu$ of measures $\mu_1,\mu_2$ on dimer tilings of mean currents $s_1,s_2$, and construct a coupling $\mu'$ of two new measures $\mu_1',\mu_2'$ on dimer tilings both of mean current $\frac{s_1+s_2}{2}$. We show that this operation preserves the total specific entropy (i.e. $h(\mu) = h(\mu')$) and ergodicity, but \textit{breaks} the Gibbs property. More precisely, if $\mu_1, \mu_2$ are ergodic Gibbs measures of mean currents $s_1,s_2$ and $\frac{s_1+s_2}{2}\in \text{Int}(\mc O)$, then $\mu_1',\mu_2'$ are \textit{not} Gibbs, and hence do \textit{not} have maximal entropy among measures of mean current $\frac{s_1+s_2}{2}$. The proof that the Gibbs property is broken under chain swapping requires very different techniques from those used in \cite{AST_2005__304__R1_0}. 

Under the assumption that there exists an \textit{ergodic} Gibbs measure $\mu_s$ of mean current $s$ for any $s\in \mc O$, and that $\ent(s)=h(\mu_s)$, strict concavity would follow easily: let $\mu_1 = \mu_{s_1}, \mu_2=\mu_{s_2}$ and apply chain swapping to get new measures $\mu_1',\mu_2'$ of mean current $\frac{s_1+s_2}{2}$. Since total entropy is preserved,
$$h(\mu_1') + h(\mu_2') = h(\mu_1)+h(\mu_2)=\ent(s_1)+\ent(s_2).$$
On the other hand, since $\mu_1',\mu_2'$ are not Gibbs, $$h(\mu_1')+h(\mu_2') < 2\ent(\frac{s_1+s_2}{2}),$$
which would complete the proof. A rigorous proof of the theorem is given in Section \ref{subsection: strict concavity}, and relies on casework based on ergodic decompositions as we do not know, a priori, that ergodic Gibbs measures of mean current $s$ exist for all $s\in \mc O$. However it will then follow \textit{from} strict concavity that this is true, and there exist ergodic Gibbs measures of all mean currents $s\in \mc O$ (Corollary~\ref{cor: EGMs exist!}).

The above is a discussion of $\ent$ on $\text{Int}(\mc O)$, where no explicit formula is known. We remark that $\ent$ is explicitly computable when restricted to $\partial \mc O$, since it reduces to a two-dimensional problem (see Section \ref{sec:extreme}).

\subsection{Three-dimensional history and pathology} \label{subsec:history}

The three-dimensional model is fundamentally different from the two-dimensional version in many respects. To give one example, we recall that if $\tau$ and $\sigma$ are distinct perfect matchings of $\mathbb Z^2$ that agree on all but finitely many edges, then one can construct a sequence $\tau = \tau_0, \tau_1, \tau_2, \ldots, \tau_n = \sigma$ of perfect matchings such that for each $k$, the matchings $\tau_k$ and $\tau_{k-1}$ agree on all edges except those contained in a single $2 \times 2$ square ---  and on that square one of $\{\tau,\sigma\}$ has two parallel vertical edges and the other has two parallel horizontal edges \cite{Thurston}. From the domino tiling point of view, we say that $\tau_{k-1}$ and $\tau_k$ differ by a {\em local move} that corresponds to rotating a pair of dominoes as shown below.
\begin{figure}[H]
    \centering
        \includegraphics[scale=0.25]{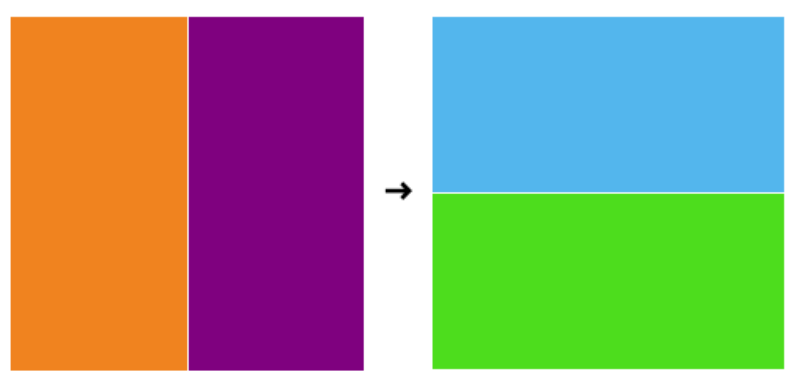}
    \caption{A local move or flip in 2D.}
    \label{fig:2dflip}
\end{figure}
It turns out that the analogous statement is false in 3D. In fact, as we will explain in Section~\ref{sec:localmoves}, there is {\em no} collection of local moves for which the analogous statement would be true in 3D. In 3D, one can construct (for any $K>0$) a tiling $\tau$ of $\m Z^3$ that is
\begin{enumerate}
\item {\em non-frozen} --- i.e., there exists a tiling $\tau' \neq \tau$ that disagrees with $\tau$ on finitely many edges.
\item {\em locally frozen} to level $K$ --- i.e., there exists no tiling $\tau' \neq \tau$ that disagrees with $\tau'$ on fewer than $K$ edges.
\end{enumerate}

To understand why this is the case, recall that in two dimensions, one can superimpose an arbitrary perfect matching with a brickwork matching to obtain a collection of non-intersecting left-to-right lattice paths as follows:
\begin{figure}[H]
    \centering
    \includegraphics[scale=0.5]{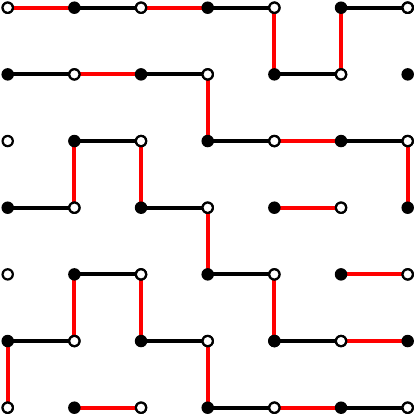}
    \caption{2D non-intersecting paths.}
    \label{fig:non-intersecting2d}
\end{figure}
There is an obvious bijection between non-intersecting path ensembles (as shown above) and dimer tilings (which is one way to deduce the integrability of the dimer model in two dimensions). Applying local moves corresponds to shifting these paths up and down locally. One can analogously superimpose a red three-dimensional matching with a black brickwork matching, to obtain an ensemble of left-to-right paths in three dimensions. But in this case the function that maps each ``left endpoint'' to the ``right endpoint'' on the same path may not be uniquely determined, as the following example shows. For clarity, the black and red edges that coincide with each other are not drawn---so both figures indicate a red perfect matching of $\mathbb Z^3$ that (when restricted to the box) consists only of right-going edges in the brickwork pattern (not shown) and a few non-right-going edges (shown together with the black right-going edges that share their endpoints).
\begin{figure}[H]
    \centering
    \includegraphics[scale=0.7]{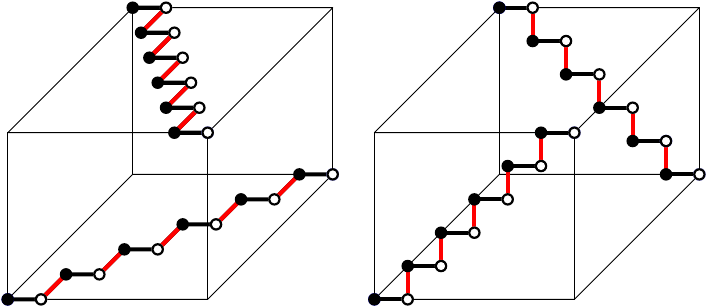}
    \caption{Two examples of 3D non-intersecting paths with the same endpoints.}
    \label{fig:non-intersecting3d}
\end{figure}
In general, there could be many different paths, and many ways to permute the wires from the left before plugging them in on the right. In the example above, the paths are ``taut'' in the sense that they have no freedom to ``move locally'' using local moves that change only, say, three or four edges at a time (and they can be extended to taut paths on all of $\mathbb Z^3$). In general, 3D path ensembles are not nicely ordered from top to bottom, and do not have the same lattice structure that 2D path ensembles enjoy. They can be braided in complicated ways.

Despite this complexity, various ``local move connectedness'' results for 3D tilings have been obtained. See, for example, the series of works by subsets of Freire, Klivans, Milet and Saldanha \cite{ milet2014domino, milet2014twists, milet2015domino, MR3341585, freire2022connectivity, saldanha2019domino, saldanha2020domino, klivans2020domino, MR4245261}, the recent work \cite{localdimer} by Hartarsky, Lichev, and Toninelli, and physics papers by Freedman, Hastings, Nayak, Qi, and separately Bednik about topological invariants and so-called Hopfions \cite{freedman2011weakly, bednik2019hopfions, bednik2019probing}.

One of the basic observations is that even on box-shaped regions in 3D, one cannot transform any tiling to any other tiling with a sequence of {\em flips} (which swap two edges of a lattice square with the other two). There is a quantity associated to a tiling, called the {\em twist} (related to the linking number from knot theory) that is preserved by flip moves but changed by so-called {\em trit} moves, which involve removing three edges contained in the same $2 \times 2 \times 2$ cube and replacing them with three others, see below:
\begin{figure}[H]
    \centering
        \includegraphics[scale=0.25]{flip3d.pdf}$\qquad\qquad$
\includegraphics[scale=0.2]{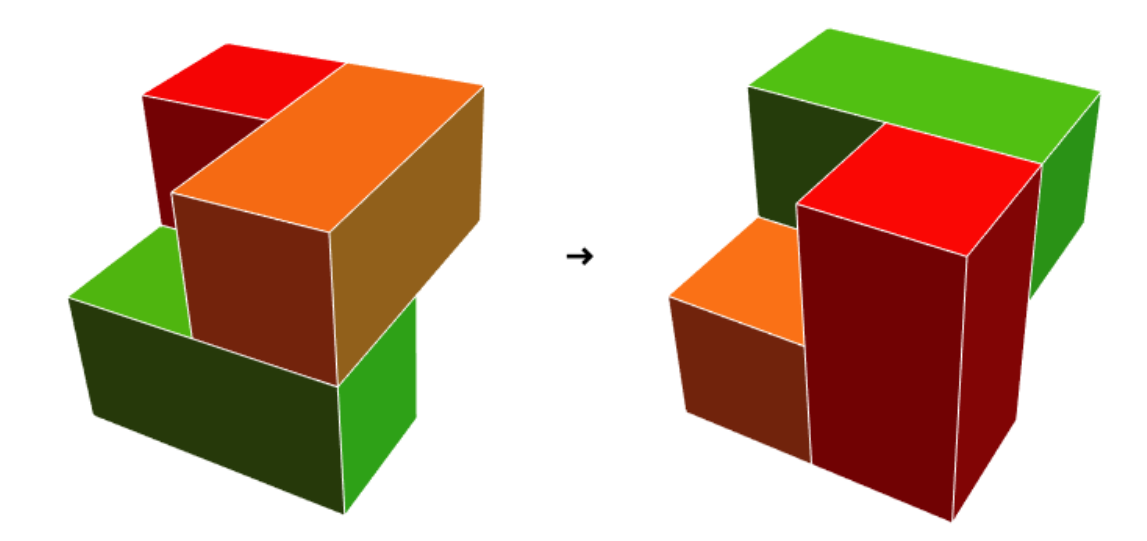}$\qquad\qquad$\includegraphics[scale=0.15]{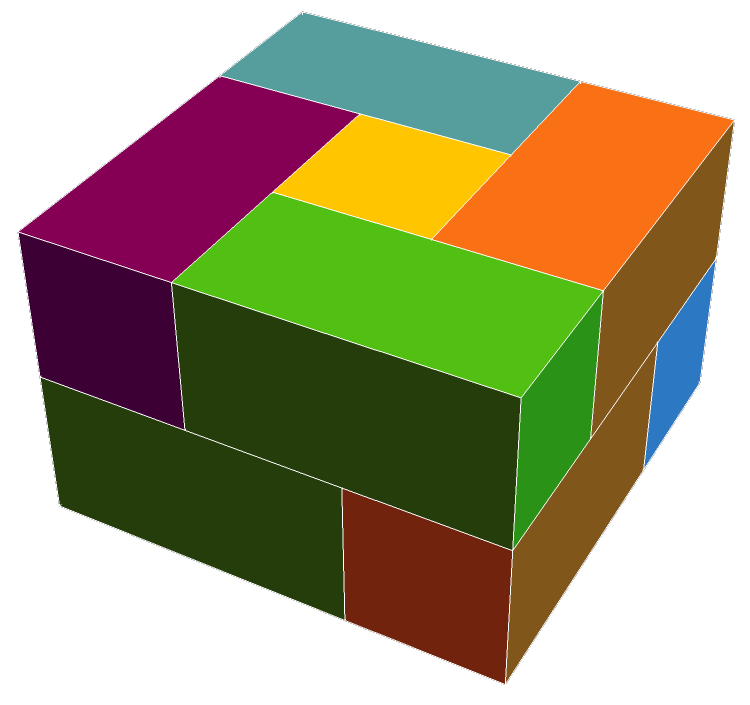}
    \caption{A \textit{flip}, a \textit{trit} and a flip-rigid configuration called a \textit{hopfion}.}
    \label{fig:hopfion_trit}
\end{figure}
It remains open whether it is possible to connect any tiling of a rectangular box to any other using {\em both} the flip and trit moves shown above. It is still possible in 3D to generate random tilings of finite regions using Glauber dynamics (using an update algorithm that allows for the tiling to be modified along cycles of arbitrary length, see Section \ref{sec:uniform sampling}) but little is known about the rate of mixing (though bounds were given for another 3D tiling model in \cite{MR1755523}). 

Quitmann and Taggi have some additional important work on the 3D dimer model, which studies the behavior of loops formed by an independently sampled pair of dimer configurations \cite{taggi2022uniformly,
quitmann2022macroscopicdimer,
quitmann2022macroscopicloops}. Among other things, they find that when one superimposes two independent random dimer tilings on an $n \times n \times n$ torus, the union of the tilings will typically contain cycles whose length has order $n^3$.

Throughout this paper our basic physical intuition is that the 3D dimer model describes a steady current through a non-isotropic medium, and we are studying how the current varies in space. But we stress that papers like the one by Freedman et al \cite{freedman2011weakly} have other field theoretic phenomena in mind ({\em topological excitations}, {\em Majorana fermions}, {\em Abelian anyons}, etc.)\ and we will not attempt to explain these interpretations here, though we will briefly mention a gauge theoretic interpretation of the dimer model in Problem~\ref{prob:gauge}.

Let us also remark that the literature on {\em related} topics is quite large, including (to give just a few examples) work on large deviations for graph homomorphisms $h: \m Z^d\to \m Z$ \cite{KriegerMenzTassy}, weakly non-planar dimer models \cite{Giuliani2, wealkynonplanar}, and a generalization of rhombus tilings to $n$ dimensions \cite{LindaMooreNordahl, Lammers, arctic_oct_MC}.

\subsection{Outline of paper}
We establish notation and a few basic preliminaries in Section~\ref{sec::prelim}. We then illustrate the complexity of the 3D model with a brief discussion of the local move problem and related topics in Section~\ref{sec:localmoves}. 

In Section~\ref{sec:extreme} we describe the ergodic Gibbs measures of {\em boundary mean current} (i.e., having mean current that lies on the boundary $\partial \mathcal{O}$, where $\mathcal O$ is the octahedron of possible mean currents). Not all  Gibbs measures with boundary mean current have zero entropy, but we can still compute the entropy function $\ent$ on $\partial \mathcal O$ by reducing it to a two-dimensional problem (see Theorem \ref{thm: extremal_entropy}).

In Section~\ref{sec:patching plus hall}, we deal with the fundamental problem of how one ``patches together" regions of different tilings to form one perfect matching of a large region (Theorem \ref{patching}). As noted above, the key tool is \textit{Hall's matching theorem}. We give an outline of the proof of Theorem \ref{patching} (in the ``square annulus setting'') in Section \ref{sec: ideas for patching} accompanied by a sequence of two dimensional pictures. In three dimensions, Hall's matching theorem relates non-tileability to the existence of a certain type of minimal surface. The other key classical input in the proof of Theorem \ref{patching} is the isoperimetric inequality.

Section~\ref{sec:entropy} concerns properties of the local entropy function $\ent$, such as continuity and strict concavity. Since no exact formula for $\ent(s)$ is known for mean currents $s$ in the interior of $\mc O$ this section involves interesting methods fairly different from dimension 2, in particular the chain swapping constructions in Section \ref{Subsection: Dimer Swapping}.

Section~\ref{sec:asympflows} is a technical section where we present some of the function-theoretic preliminaries about flows that we need for the large deviation principle. We define scaled tiling flows, the Wasserstein metric on flows for comparing them, and asymptotic flows (which we prove are the scaling limits of tiling flows in Theorem \ref{thm:formal_fine_mesh}). We also define boundary values for both types of flows using a \textit{trace operator} $T$, and show that $T$ is uniformly continuous as a function of the flow (Theorem \ref{thm: boundary_value_uniformly_continuous}). 

In Section~\ref{subsection: properties of Ent}, with the the topology on asymptotic flows established, we leverage the results for $\ent$ in Section~\ref{sec:entropy} to show that $\Ent$ is upper semicontinuous and has a unique maximizer (when the region and boundary conditions are semi-flexible). 

Section~\ref{sec:ldp} finally ties together the ingredients of the previous sections to produce the two large deviation principles (Theorem \ref{thm:sb-ldp} for soft boundary conditions and Theorem \ref{thm:hb-ldp} for hard boundary conditions) which are our main results. Both of these are broken down into proving a lower bound on probabilities (Theorem \ref{thm:lower} for soft boundary and Theorem \ref{thm:lower-hb} for hard boundary) and an upper bound (Theorem \ref{thm:upper} for soft boundary and Theorem \ref{thm:upper-hb} for hard boundary). One of the slightly difficult parts of the paper is the explicit construction of a tiling flow approximating an asymptotic flow.  This is a step in proving the lower bound which we call the ``shining light" argument (Theorem \ref{thm:shininglight}). For the hard boundary lower bound, on top of this we also need a \textit{generalized patching theorem} (Theorem \ref{thm:generalized_patching}) to show that any asymptotic flow can be approximated by a tiling of a fixed region. The proof of the generalized patching theorem is where we make use of the \textit{flexible} condition on $(R,b)$ in the hard boundary large deviation principle. 

Several open problems are given in Section~\ref{sec:open problems}. See the chart below for a graphical representation of some of the dependencies and results that we highlight.

\footnotesize
\tikzset{every picture/.style={line width=0.75pt}} 

\begin{tikzpicture}[x=0.7pt,y=0.75pt,yscale=-0.93,xscale=0.93]

\draw   (90,430) -- (170,430) -- (170,470) -- (90,470) -- cycle ;
\draw   (190,430) -- (310,430) -- (310,470) -- (190,470) -- cycle ;
\draw   (330,430) -- (440,430) -- (440,470) -- (330,470) -- cycle ;
\draw   (190,350) -- (310,350) -- (310,390) -- (190,390) -- cycle ;
\draw   (340,300) -- (460,300) -- (460,360) -- (340,360) -- cycle ;
\draw   (490,270) -- (650,270) -- (650,340) -- (490,340) -- cycle ;
\draw   (490,370) -- (600,370) -- (600,420) -- (490,420) -- cycle ;
\draw   (310,220) -- (410,220) -- (410,270) -- (310,270) -- cycle ;
\draw  [color={rgb, 255:red, 245; green, 166; blue, 35 }  ,draw opacity=1 ] (490,190) -- (600,190) -- (600,250) -- (490,250) -- cycle ;
\draw  [color={rgb, 255:red, 245; green, 166; blue, 35 }  ,draw opacity=1 ] (490,90) -- (600,90) -- (600,150) -- (490,150) -- cycle ;
\draw  [color={rgb, 255:red, 245; green, 166; blue, 35 }  ,draw opacity=1 ] (190,142) -- (293,142) -- (293,193) -- (190,193) -- cycle ;
\draw  [color={rgb, 255:red, 245; green, 166; blue, 35 }  ,draw opacity=1 ] (317,142) -- (420,142) -- (420,193) -- (317,193) -- cycle ;
\draw  [color={rgb, 255:red, 245; green, 166; blue, 35 }  ,draw opacity=1 ] (190,50) -- (395,50) -- (395,110) -- (190,110) -- cycle ;
\draw  [color={rgb, 255:red, 245; green, 166; blue, 35 }  ,draw opacity=1 ] (90,240) -- (230,240) -- (230,300) -- (90,300) -- cycle ;
\draw  [color={rgb, 255:red, 245; green, 166; blue, 35 }  ,draw opacity=1 ] (30,140) -- (125,140) -- (125,193) -- (30,193) -- cycle ;
\draw  [color={rgb, 255:red, 245; green, 166; blue, 35 }  ,draw opacity=1 ] (435,11) -- (615,11) -- (615,51) -- (435,51) -- cycle ;
\draw  [color={rgb, 255:red, 245; green, 166; blue, 35 }  ,draw opacity=1 ] (20,20) -- (130,20) -- (130,110) -- (20,110) -- cycle ;
\draw   (10,340) -- (130,340) -- (130,380) -- (10,380) -- cycle ;
\draw    (160,420) -- (179.11,381.79) ;
\draw [shift={(180,380)}, rotate = 116.57] [color={rgb, 255:red, 0; green, 0; blue, 0 }  ][line width=0.75]    (10.93,-3.29) .. controls (6.95,-1.4) and (3.31,-0.3) .. (0,0) .. controls (3.31,0.3) and (6.95,1.4) .. (10.93,3.29)   ;
\draw    (340,420) -- (320.89,381.79) ;
\draw [shift={(320,380)}, rotate = 63.43] [color={rgb, 255:red, 0; green, 0; blue, 0 }  ][line width=0.75]    (10.93,-3.29) .. controls (6.95,-1.4) and (3.31,-0.3) .. (0,0) .. controls (3.31,0.3) and (6.95,1.4) .. (10.93,3.29)   ;
\draw    (250,424) -- (250,396) ;
\draw [shift={(250,394)}, rotate = 90] [color={rgb, 255:red, 0; green, 0; blue, 0 }  ][line width=0.75]    (10.93,-3.29) .. controls (6.95,-1.4) and (3.31,-0.3) .. (0,0) .. controls (3.31,0.3) and (6.95,1.4) .. (10.93,3.29)   ;
\draw    (278,343) -- (329.69,318.85) ;
\draw [shift={(331.5,318)}, rotate = 154.95] [color={rgb, 255:red, 0; green, 0; blue, 0 }  ][line width=0.75]    (10.93,-3.29) .. controls (6.95,-1.4) and (3.31,-0.3) .. (0,0) .. controls (3.31,0.3) and (6.95,1.4) .. (10.93,3.29)   ;
\draw    (261,342) -- (249.66,200.99) ;
\draw [shift={(249.5,199)}, rotate = 85.4] [color={rgb, 255:red, 0; green, 0; blue, 0 }  ][line width=0.75]    (10.93,-3.29) .. controls (6.95,-1.4) and (3.31,-0.3) .. (0,0) .. controls (3.31,0.3) and (6.95,1.4) .. (10.93,3.29)   ;
\draw    (150.5,230) -- (183.17,193.49) ;
\draw [shift={(184.5,192)}, rotate = 131.82] [color={rgb, 255:red, 0; green, 0; blue, 0 }  ][line width=0.75]    (10.93,-3.29) .. controls (6.95,-1.4) and (3.31,-0.3) .. (0,0) .. controls (3.31,0.3) and (6.95,1.4) .. (10.93,3.29)   ;
\draw    (143,230) -- (125.75,208.56) ;
\draw [shift={(124.5,207)}, rotate = 51.19] [color={rgb, 255:red, 0; green, 0; blue, 0 }  ][line width=0.75]    (10.93,-3.29) .. controls (6.95,-1.4) and (3.31,-0.3) .. (0,0) .. controls (3.31,0.3) and (6.95,1.4) .. (10.93,3.29)   ;
\draw    (235,138) -- (273.85,117.14) ;
\draw [shift={(275.5,116)}, rotate = 145.34] [color={rgb, 255:red, 0; green, 0; blue, 0 }  ][line width=0.75]    (10.93,-3.29) .. controls (6.95,-1.4) and (3.31,-0.3) .. (0,0) .. controls (3.31,0.3) and (6.95,1.4) .. (10.93,3.29)   ;
\draw    (367,138) -- (325.2,117.05) ;
\draw [shift={(323.5,116)}, rotate = 31.83] [color={rgb, 255:red, 0; green, 0; blue, 0 }  ][line width=0.75]    (10.93,-3.29) .. controls (6.95,-1.4) and (3.31,-0.3) .. (0,0) .. controls (3.31,0.3) and (6.95,1.4) .. (10.93,3.29)   ;
\draw    (415.5,295) -- (395.06,275.26) ;
\draw [shift={(392.5,274)}, rotate = 38.93] [color={rgb, 255:red, 0; green, 0; blue, 0 }  ][line width=0.75]    (10.93,-3.29) .. controls (6.95,-1.4) and (3.31,-0.3) .. (0,0) .. controls (3.31,0.3) and (6.95,1.4) .. (10.93,3.29)   ;
\draw    (485,396) -- (461.82,369.51) ;
\draw [shift={(460.5,368)}, rotate = 48.81] [color={rgb, 255:red, 0; green, 0; blue, 0 }  ][line width=0.75]    (10.93,-3.29) .. controls (6.95,-1.4) and (3.31,-0.3) .. (0,0) .. controls (3.31,0.3) and (6.95,1.4) .. (10.93,3.29)   ;
\draw    (485,305) -- (465.86,325.54) ;
\draw [shift={(464.5,327)}, rotate = 312.98] [color={rgb, 255:red, 0; green, 0; blue, 0 }  ][line width=0.75]    (10.93,-3.29) .. controls (6.95,-1.4) and (3.31,-0.3) .. (0,0) .. controls (3.31,0.3) and (6.95,1.4) .. (10.93,3.29)   ;
\draw    (481.5,212) -- (429.19,179.07) ;
\draw [shift={(427.5,178)}, rotate = 32.2] [color={rgb, 255:red, 0; green, 0; blue, 0 }  ][line width=0.75]    (10.93,-3.29) .. controls (6.95,-1.4) and (3.31,-0.3) .. (0,0) .. controls (3.31,0.3) and (6.95,1.4) .. (10.93,3.29)   ;
\draw    (544.5,186) -- (544.5,157) ;
\draw [shift={(544.5,155)}, rotate = 90] [color={rgb, 255:red, 0; green, 0; blue, 0 }  ][line width=0.75]    (10.93,-3.29) .. controls (6.95,-1.4) and (3.31,-0.3) .. (0,0) .. controls (3.31,0.3) and (6.95,1.4) .. (10.93,3.29)   ;
\draw    (397.5,83) -- (428.94,32.7) ;
\draw [shift={(430,31)}, rotate = 122.01] [color={rgb, 255:red, 0; green, 0; blue, 0 }  ][line width=0.75]    (10.93,-3.29) .. controls (6.95,-1.4) and (3.31,-0.3) .. (0,0) .. controls (3.31,0.3) and (6.95,1.4) .. (10.93,3.29)   ;
\draw    (546,85) -- (546,56) ;
\draw [shift={(546,54)}, rotate = 90] [color={rgb, 255:red, 0; green, 0; blue, 0 }  ][line width=0.75]    (10.93,-3.29) .. controls (6.95,-1.4) and (3.31,-0.3) .. (0,0) .. controls (3.31,0.3) and (6.95,1.4) .. (10.93,3.29)   ;
\draw [color={rgb, 255:red, 208; green, 2; blue, 27 }  ,draw opacity=1 ] [dash pattern={on 4.5pt off 4.5pt}]  (126.5,170) -- (179.5,170) ;
\draw [shift={(181.5,170)}, rotate = 180] [color={rgb, 255:red, 208; green, 2; blue, 27 }  ,draw opacity=1 ][line width=0.75]    (10.93,-3.29) .. controls (6.95,-1.4) and (3.31,-0.3) .. (0,0) .. controls (3.31,0.3) and (6.95,1.4) .. (10.93,3.29)   ;
\draw    (303.5,230) -- (274.89,200.44) ;
\draw [shift={(273.5,199)}, rotate = 45.94] [color={rgb, 255:red, 0; green, 0; blue, 0 }  ][line width=0.75]    (10.93,-3.29) .. controls (6.95,-1.4) and (3.31,-0.3) .. (0,0) .. controls (3.31,0.3) and (6.95,1.4) .. (10.93,3.29)   ;
\draw    (424.5,293) -- (483.61,132.79) ;
\draw [shift={(484.5,131)}, rotate = 108.28] [color={rgb, 255:red, 0; green, 0; blue, 0 }  ][line width=0.75]    (10.93,-3.29) .. controls (6.95,-1.4) and (3.31,-0.3) .. (0,0) .. controls (3.31,0.3) and (6.95,1.4) .. (10.93,3.29)   ;


\draw    (483.5,297) -- (414.07,242.24) ;
\draw [shift={(412.5,241)}, rotate = 38.26] [color={rgb, 255:red, 0; green, 0; blue, 0 }  ][line width=0.75]    (10.93,-3.29) .. controls (6.95,-1.4) and (3.31,-0.3) .. (0,0) .. controls (3.31,0.3) and (6.95,1.4) .. (10.93,3.29)   ;

\draw (195,53) node [anchor=north west][inner sep=0.75pt]   [align=left] {large deviation principle \\(SB \, / \, \textcolor[rgb]{0.82,0.01,0.11}{HB})\\Theorem \ref{thm:sb-ldp} / Theorem \ref{thm:hb-ldp}};
\draw (437,14) node [anchor=north west][inner sep=0.75pt]   [align=left] {concentration/limit shape\\Corollary \ref{cor:concentration}};
\draw (492,93) node [anchor=north west][inner sep=0.75pt]   [align=left] {Unique Ent\\maximizer\\Theorem \ref{thm: unique maximizer}};
\draw (492,372) node [anchor=north west][inner sep=0.75pt]   [align=left] {chain swapping\\machinery\\Section \ref{Subsection: Dimer Swapping}};
\draw (492,273) node [anchor=north west][inner sep=0.75pt]   [align=left] {measures with\\boundary mean current\\in 3D\\Section \ref{sec:extreme}};
\draw (492,193) node [anchor=north west][inner sep=0.75pt]   [align=left] {Ent upper\\semicontinuous\\Prop. \ref{proposition: Upper semicontinuous}};
\draw (191,353) node [anchor=north west][inner sep=0.75pt]   [align=left] {patching theorem\\Theorem \ref{patching}};
\draw (101,432) node [anchor=north west][inner sep=0.75pt]   [align=left] {ergodic\\theorem};
\draw (196,432) node [anchor=north west][inner sep=0.75pt]   [align=left] {Hall's matching\\theorem};
\draw (341,432) node [anchor=north west][inner sep=0.75pt]   [align=left] {isoperimetric\\inequality};
\draw (342,303) node [anchor=north west][inner sep=0.75pt]   [align=left] {strict concavity of\\ent on $\mc O\setminus \mc E$\\Theorem \ref{theorem: entropy is strictly concave}};
\draw (22,23) node [anchor=north west][inner sep=0.75pt]  [color={rgb, 255:red, 245; green, 166; blue, 35 }  ,opacity=1 ] [align=left] {definitions, \\basic properties \\of Wasserstein\\distance\\Section \ref{sec:asympflows}};
\draw (312,223) node [anchor=north west][inner sep=0.75pt]   [align=left] {EGMs exist for\\all $s\in \mc O$\\Corollary \ref{cor: EGMs exist!}};
\draw (12,343) node [anchor=north west][inner sep=0.75pt]   [align=left] {local moves in 3D\\Section \ref{sec:localmoves}};
\draw (32,143) node [anchor=north west][inner sep=0.75pt]   [align=left] {generalized\\patching\\Theorem \ref{thm:generalized_patching}};
\draw (92,243) node [anchor=north west][inner sep=0.75pt]   [align=left] {existence of\\tiling approximations\\Theorem \ref{thm:shininglight}};
\draw (192,145) node [anchor=north west][inner sep=0.75pt]   [align=left] {lower bounds\\Theorem \ref{thm:lower}\\ Theorem \ref{thm:lower-hb}};
\draw (320,145) node [anchor=north west][inner sep=0.75pt]   [align=left] {upper bounds\\Theorem \ref{thm:upper} \\ Theorem \ref{thm:upper-hb}};
\draw (128,147) node [anchor=north west][inner sep=0.75pt]  [color={rgb, 255:red, 208; green, 2; blue, 27 }  ,opacity=1 ] [align=left] {for HB};
\end{tikzpicture}
\normalsize

The results in orange boxes are stated using the Wasserstein metric for flows, and rely on many of its properties described in Section \ref{sec:asympflows}.

\bigskip
\noindent\textbf{Acknowledgments.} The authors have enjoyed and benefited from conversations with many dimer theory experts, including (but not limited to) Nathana\"el Berestycki, Richard Kenyon, and Marianna Russkikh. The authors were partially supported by NSF grants DMS 1712862, DMS 2153742 and thank the Institute for Advanced Study, where part of this work was completed. N.C.\ was in addition supported by an SERB Startup grant and the Department of Atomic Energy, Government of India, under project indentification number RTI 4014, and would like to thank Tom Meyerovitch for initiating his interest in the dimer model and introducing him to the flows associated with the model, and Kedar Damle and Piyush Srivastava for many helpful conversations regarding our simulations. C.W.\ was also partially supported by NSF grant DMS-2401750.

\section{Preliminaries} \label{sec::prelim}

As we mentioned earlier, it is sometimes convenient to represent a vertex of $\mathbb Z^3$ by the unit cube centered at that vertex, and to represent an edge $e=(a,b)$ of $\mathbb Z^3$ by the union of the two cubes centered at $a$ and $b$ (a domino). Both perspectives are useful for visualization, and we will use the terms {\em perfect matching} and {\em dimer tiling} somewhat interchangeably. We denote the space of dimer tilings of $\m Z^3$ by $\Omega$. \symindex{Chapter 2!$\Omega$ - the space of dimer tilings}

Recall that $\m Z^3$ is a bipartite lattice, with bipartition into even points (where the coordinate sum is even) and odd points (where the coordinate sum is odd). In a dimer tiling of $\m Z^3$, there are six possible types of tiles corresponding to the six possible unit coordinate vectors. We denote the unit coordinate vectors by $\eta_1 = (1,0,0)$, $\eta_2 = (0,1,0)$, and $\eta_3=(0,0,1)$. We denote the edge in $\m Z^3$ connecting the origin to $\eta_i$ by $e_i$ and the edge connecting the origin to $-\eta_i$ by $-e_i$. \symindex{Chapter 2!$\eta_i$ and $e_i$ - unit vectors and edges connecting the origin}

\subsection{Tilings and discrete vector fields}\label{section:tiling_flows}

Given a perfect matching of $\mathbb{Z}^3$, there is a natural way to associate a discrete, divergence-free vector field valued on oriented edges. We will call the flow corresponding to a tiling $\tau$ a \textit{tiling flow}, denoted $f_\tau$. Like height functions in two dimensions, tiling flows have well-defined scaling limits called \textit{asymptotic flows} (which we describe in Section \ref{sec:asympflows}). Asymptotic flows capture the broad statistics of dimer tilings in a given region. Since our main results (e.g.\ our large deviation principle, analogous to \cite{cohn2001variational}) are related to the large scale statistics of dimer tilings, they are naturally formulated in terms of tiling flows.

 Let $E$ denote the set of edges in $\m Z^3$.\symindex{Chapter 2!$E$ - edges in $\m Z^3$} A \textit{discrete vector field} or \textit{discrete flow} \termindex{Chapter 2!discrete vector field/discrete flow} is a function from oriented edges of $\m Z^3$ to the real numbers. Unless stated otherwise, we assume all edges are oriented from even to odd (flipping the orientation of the edge $e$ reverses the sign of the discrete vector field on $e$). For a dimer tiling $\tau$ of $\m Z^3$, we associate a discrete vector field $v_\tau$ valued on the edges $e\in E$ defined by \termindex{Chapter 2!pretiling flow}
\begin{equation}\label{eq:v_tau}
	v_{\tau}({e}) = \begin{cases}   +1 &\quad \text{ if } e\in \tau,\text{ oriented even to odd} \\ 0 &\quad \text{ if } e\not\in \tau \end{cases}
\end{equation}\symindex{Chapter 2!$v_{\tau}$ - flow associated to a tiling $\tau$}

We call $v_\tau$ the \textit{pretiling flow}. Recall that by our definition of discrete vector fields, if $e$ is oriented odd to even, we say that $v_\tau(e) = -1$. The divergence of a discrete vector field $v$ is given by 
\begin{equation}\label{eq: div free discrete}
    \text{div}\,v(x) = \sum_{e\ni x} v(e)
\end{equation}\symindex{Chapter 2!$\text{div}\,v$ - divergence of a discrete vector field $v$}

where the sum is over edges $e$ oriented away from $x$ (e.g. if $x$ is even, the edges in the sum are oriented from even to odd, and the opposite if $x$ is odd). From this definition, we compute that 
\begin{align*}
	\text{div}\, v_{\tau}(x) = \begin{cases} +1 &\quad \text{ if } x \text{ is even } \\ -1 &\quad \text{ if } x \text{ is odd } \end{cases}
\end{align*}
Therefore $v_{\tau}$ itself is not divergence-free, but the divergences don't depend on $\tau$, so we can construct a divergence-free flow corresponding to a tiling $\tau$ by subtracting a fixed reference flow $r$. There are lots of reasonable choices for the reference flow. For simplicity and symmetry we choose: \symindex{Chapter 2!$r$ - reference flow }
\begin{align*}
	r({e}) = \frac{1}{6} \qquad \text{ for all edges } {e} \in E \text{ oriented from even to odd}
\end{align*}
We can now define the \textit{tiling flow} \termindex{Chapter 2!tiling flow} corresponding to a tiling $\tau$ of a region $R\subseteq \mathbb{Z}^3$.\symindex{Chapter 2!$f_{\tau}$ - (divergence free) tiling flow associated with the tiling $\tau$}
\begin{definition}\label{def: tiling flow}
	Let $\tau$ be a dimer tiling of $\m Z^3$. The divergence-free, discrete vector field corresponding to $\tau$ is $f_{\tau} := v_{\tau} - r$. We call $f_\tau$ a \textit{tiling flow}. 
	\begin{align*}
		f_{\tau}({e}) = \begin{cases}  +5/6 &\quad \text{ if } e\in \tau \\ -1/6 &\quad \text{ if } e\not\in \tau \end{cases}
	\end{align*}
	If $\tau$ is a dimer tiling of a subgraph $G\subset \m Z^3$, we define the tiling flow by restriction.
\end{definition}

\begin{rem}
In dimension 2, the analogous definition of a tiling flow $f_\tau$ also works (in this case the reference flow is $1/4$ on all edges oriented from even to odd). For every discrete flow defined on edges (whose endpoints are vertices of $\mathbb Z^2$) there is a dual flow on dual edges (whose endpoints are faces of $\mathbb Z^2$) obtained by rotating each edge $90$ degrees clockwise. If the original flow is divergence-free, then the dual flow is curl-free and is hence equal to the gradient of a function (this function is called the \textit{height function} or \textit{scalar potential}).  It is also worth noting that there is an analog of the height function in three dimensions. Namely, since $f_\tau$ is a divergence-free flow it can be written as the curl of another flow, that is, ${f}_\tau = \nabla \times A$, where $A$ is a so-called {\em vector potential} which is defined modulo the addition of a curl-free flow. However the set of vector potentials $A$ is more complicated than the set of height functions (it does not have a similar lattice structure, the potentials are not uniquely defined, etc.) and is not as useful for our purposes as height functions are in two dimensions. Because of that, we do not work with the vector potential in this paper, and instead just work with the tiling flow $f_\tau$ itself.
\end{rem}

A pair of dimer tilings $(\tau_1,\tau_2)\in \Omega\times \Omega$ is called a \textit{double dimer tiling} or \textit{double dimer configuration}\termindex{Chapter 2!double dimer tiling/configuration}. The double dimer model is a model of independent interest, but we mention it because it will be a tool for comparing dimer tilings. This will be used in Section \ref{sec:localmoves}, Section \ref{sec:ldp}, and substantially in Section \ref{sec:entropy}.

There is a natural way to associate a divergence-free discrete flow to a double dimer configuration, namely for $e$ an edge oriented from even to odd, \symindex{Chapter 2!$f_{(\tau_1,\tau_2)}$ - discrete flow associated to a double dimer configuration }
\begin{equation}
    f_{(\tau_1,\tau_2)}(e) = f_{\tau_1}(e) - f_{\tau_2}(e) = v_{\tau_1}(e) - v_{\tau_2}(e) = \begin{cases}1&\text{ if }e\in \tau_1\setminus \tau_2\\
 -1&\text{ if }e\in \tau_2\setminus \tau_1\\
 0&\text{ if }e\in \tau_1\cap \tau_2\text{ or if }e\not\in \tau_1\cup\tau_2.
 \end{cases}.
\end{equation}
Unlike the tiling flow for a single tiling, the flow associated with a double dimer configuration $(\tau_1,\tau_2)$ does not determine $(\tau_1,\tau_2)$, since it does not specify the tiles in $\tau_1\cap \tau_2$. However, the collection of loops formed by $\tau_1\cup \tau_2$ (including the double tiles) and the flow $f_{(\tau_1,\tau_2)}$ together determine $(\tau_1,\tau_2)$. See Section \ref{subsection: flows for double dimer} for more about double dimer flows.

\subsection{Measures on tilings and mean currents}\label{section:measures-currents}

Recall that $\Omega$ denotes the set of dimer tilings of $\m Z^3$. The group $\m Z^3$ acts naturally on $\Omega$ by translations, namely given $x\in \m Z^3$ and $\tau\in \Omega$, $\tau + x$ is the tiling where $(a,b)\in \tau$ if and only if $(a+x,b+x)\in \tau + x$. There is natural topology on $\Omega$ induced by viewing it as a subset of $\{0,1\}^{E}$ and giving the latter the product topology over the discrete set $\{0,1\}$ (recall that $E$ denotes the edges of $\m Z^3$). This makes $\Omega$ a compact metrizable space and the translation action on it continuous.

Let $\threeeven$ denote the set of even vertices in $\Z^3$. We define $\mc P(\Omega) = \mc P$ to be the space of Borel probability measures on $\Omega$ invariant under the action of $\threeeven$. \symindex{Chapter 2!$\threeeven$}\symindex{Chapter 2!$\mc P$ - the space of $\threeeven$ invariant measures on the space of tilings}

To explain why we look at $\threeeven$-invariant measures instead of $\m Z^3$-invariant measures, recall that $\m Z^3$ is a bipartite lattice, with bipartition consisting of even points and odd points. In the interpretation of a dimer tiling as a flow from in Section \ref{section:tiling_flows}, the sign of the flow on an edge oriented parallel to $(1,0,0)$ (for example) depends on whether the edge starts at an even point or an odd point. E.g. consider the tiling
$$\tau=\{( x,  x+  (1,0,0))~:~ x \in \Z^3\text{ is even}\}.$$
The flow associated to $\tau$ moves current (on average) in the direction $(1,0,0)$, while the flow for $\tau+(1,0,0)$ moves current (on average) in the direction $(-1,0,0)$. We want to our measures to be invariant under an action that preserves the asymptotic direction of the flow associated to a tiling, and this is why we consider $\threeeven$-invariant measures instead of $\m Z^3$-invariant measures. 

The \emph{ergodic measures} $\mc P_e$ are the extreme points of the convex set $\Prob$. A good reference for basic ergodic theory suitable for our purposes is \cite{keller1998equilibrium}. Any invariant measure $\mu\in \mc P$ can be decomposed in terms of ergodic measures, i.e.\ there exists a measure $w_\mu$ on $\mc P_e$ such that \termindex{Chapter 2!ergodic measures}
\begin{align*}
    \mu = \int_{\mc P_e} \nu\, \dd w_\mu(\nu).
\end{align*}
The measures $\nu$ in the support of $w_\mu$ are called the \textit{ergodic components} of $\mu$. Sampling from $\mu$ can be viewed as first sampling an ergodic component $\nu$ from $w_\mu$ and then sampling from $\nu$.\termindex{Chapter 2!ergodic components/decomposition}

We will also frequently make use of the so-called \textit{uniform Gibbs measures} \termindex{Chapter 2!Gibbs measures}on $\Omega$ defined as follows: a measure $\mu\in \mc P$ is a \textit{uniform Gibbs measure} if for any finite set $R\subset \m Z^3$, we can say that {\em given} that $\tau$ contains no edges that cross the boundary of $R$, and {\em given} the tiling $\tau$ induces on $\m Z^3 \setminus R$, the $\mu$ {\em conditional law} of the restriction of $\tau$ to $R$ is the uniform measure on dimer tilings of $R$. We will see in the next section that the measures that maximize specific entropy are uniform Gibbs measures. Throughout the rest of the paper, we refer to uniform Gibbs measures simply as \textit{Gibbs measures}.

A useful reference for Gibbs measures is \cite{Georgii}. We denote the set of $\threeeven$-invariant Gibbs measures by $\Prob_G$ and the set of ergodic Gibbs measures (EGMs) by $\Prob_{G,e}$. A useful fact throughout is that the ergodic components of $\threeeven$-invariant Gibbs measures are themselves $\threeeven$-invariant Gibbs measures. \symindex{Chapter 2!$\mc P_e, \mathcal P_G, \mathcal P_{G,e}$ - ergodic, Gibbs and ergodic Gibbs measures respectively on the space of tilings}\termindex{Chapter 2!EGM/ergodic Gibbs measures}

\begin{prop}\cite[Theorem 14.15]{Georgii}\label{prop: ergodic deocomposition gibbs}
The ergodic components of an invariant Gibbs measure are ergodic Gibbs measures almost surely.
\end{prop}

We remark that the analogous constructions work for \textit{weighted Gibbs measures}. For example, one may assign weights $a_1,a_2,a_3,a_4,a_5,a_6$ to the six possible tile orientations. A $\threeeven$-invariant Gibbs measure $\mu$ with these weights is a measure where for any finite set $R$, the conditional law of $\mu$ given a tiling of ${\m Z^3\setminus R}$ is the one in which each tiling of $R$ has probability proportional to $\prod_{i=1}^6 a_i^{N_i}$, where $N_i$ is the number of tiles of weight $a_i$. We expect that our main results could be extended to weighted dimer models (and perhaps also models with weights that vary by location in a periodic way as in \cite{AST_2005__304__R1_0, kenyon2006dimers}) but for simplicity we focus on the unweighted case here.

A key quantity associated to a $\threeeven$-invariant measure is the \textit{mean current} which (as mentioned in the introduction) represents the expected current flow {\em per even vertex}. This is a generalization of the notion of \textit{height function slope} from two dimensions. Recall that $ \eta_1, \eta_2, \eta_3$ denote the standard basis for $\Z^3$, the edge connecting the origin with $\eta_i$ is denoted by $e_i$, and the edge connecting the origin with $-\eta_i$ is denoted by $-e_i$. 
\begin{definition}\label{def: mean current}
    The \emph{mean current} of a measure $\mu\in \Prob$, denoted $s(\mu)$, is an element of $\mathbb R^3$ such that its $i^{th}$-coordinate is\termindex{Chapter 2!mean current}\symindex{Chapter 2!$s(\mu)$ - the mean current of the measure $\mu$}
$$( s(\mu))_i= \mu(e_i\in \tau)-{\mu}(-e_i\in \tau).$$
\end{definition}
Note that the mean current is an affine and continuous function of the measure. The mean current is invariant under the action of $\threeeven$ and takes values in
$$\mathcal O= \{s \in \m R^3~:~ |s_1| + |s_2| + |s_3|\leq 1\}$$
which we call \emph{the mean current octahedron}.\termindex{Chapter 2!mean current octahedron}\symindex{Chapter 2!$\mathcal O$ - mean current octahedron/possible values of the mean current}

There are a few other useful formulations of the mean current. We define the function $s_0:\Omega\to \m R^3$ to be the direction of the tile at the origin in $\tau$. Then the mean current can be computed as an expected value of $s_0$: \symindex{Chapter 2!$s_0(\tau)$ - the direction of the tile in $\tau$ at the origin}
\begin{equation}\label{eq:mean current as expection}
    s(\mu) = \int_{\Omega} s_0(\tau) \, \dd \mu(\tau).
\end{equation}
Similarly let $\Lambda_n=[-n,n]^3$, and let $\text{even}(\Lambda_n)$ denote the even points in $\Lambda_n$. We define the function \symindex{Chapter 2!$s_n(\tau)$ - the average direction of tiles in $\tau$ in the box $\Lambda_n$}
\begin{equation}
    s_n(\tau) = \frac{1}{\text{even}(\Lambda_n)} \sum_{x\in \text{even}(\Lambda_n)} s_0(\tau + x).
\end{equation}
The function $s_n(\tau)$ measures the average tile direction of $\tau$ in the box $\Lambda_n$. By $\threeeven$-invariance,
\begin{equation}
    s(\mu) = \int_{\Omega} s_n(\tau) \, \dd \mu(\tau).
\end{equation}

We let $\Prob^{s}$\symindex{Chapter 2!$\Prob^s$ - the space of probability measures with mean current $s$} denote the space of $\threeeven$-invariant probability with mean current $s$. Adding the subscripts $G$ and $e$ will denote whether the measure is a Gibbs measure and whether it is ergodic with respect to the $\threeeven$ action. 

\subsection{Entropy}\label{sec: entropy prelim}

As is common in statistical physics models, entropy plays an important role in the large deviation principle for dimer tilings in 3D. There are a few different functions that we refer to as ``entropy" (of a probability measure with finite or infinite support, of a mean current, of an asymptotic flow). Here we give some definitions and explain how these notions of entropy are related to each other. The primary reference for this section is also \cite{Georgii}.

For a probability measure $\nu$ with finite support $S$, its \textit{Shannon entropy}, denoted $H(\nu)$, is \termindex{Chapter 2!Shannon entropy}\symindex{Chapter 2!$H(\nu)$ - Shannon entropy of the measure $\nu$}
\begin{align*}
    H(\nu) = -\sum_{\sigma \in S} \nu(\sigma) \log \nu(\sigma). 
\end{align*}
For a $\threeeven$-invariant probability measure $\mu$ with infinite support, we can define the \textit{specific entropy} \termindex{Chapter 2!specific entropy}\symindex{Chapter 2!$h(\mu)$ - specific entropy of the measure $\mu$} of $\mu$ as a limit of Shannon entropy per site. Given a finite region $\Lambda\subset \m Z^3$, let $\Omega (\Lambda)$ \symindex{Chapter 2!$\Omega (\Lambda)$ - dimer tilings of $\m Z^3$ restricted to a finite set $\Lambda$} denote the dimer tilings of $\Lambda$ (i.e.\ tilings of $\m Z^3$ restricted to $\Lambda$, so tiles are allowed to have one cube outside $\Lambda$). For $\sigma\in \Omega(\Lambda)$, define \symindex{Chapter 2!$X(\sigma)$}
\begin{align*}
    X(\sigma) = \{\tilde{\sigma} \in \Omega : \tilde{\sigma}\mid_{\Lambda} = \sigma\}
\end{align*}
and then \symindex{Chapter 2!$H_{\Lambda}(\mu)$}
\begin{align*}
    H_{\Lambda}(\mu) := -\sum_{\sigma \in \Omega(\Lambda)} \mu(X(\sigma)) \log \mu(X(\sigma)). 
\end{align*}
Let $\Lambda_n = [-n,n]^3$ be a sequence of growing cubes. If $\mu$ is a $\threeeven$-invariant probability measure on $\Omega$, the \textit{specific entropy} of $\mu$, denoted $h(\mu)$, is 
\begin{align*}
    h(\mu) := \lim_{n\to \infty} |\Lambda_n|^{-1} H_{\Lambda_n}(\mu).
\end{align*}
This limit exists because the terms form a subadditive sequence. In fact, one can also show that
\begin{align*}
    h(\mu) = \inf_{\Lambda\in \mc S} |\Lambda|^{-1} H_{\Lambda}(\mu),
\end{align*}
where $\mc S$ is the set of all possible finite regions in $\m Z^3$. See \cite[Theorem 15.12]{Georgii}. As a function of  the measure, $h(\cdot)$ is affine and upper semicontinuous \cite[Proposition 15.14]{Georgii}. 

The reason that Gibbs measures (introduced in the previous section) play a special role in our study is the \textit{variational principle} which says that a measure $\mu\in \mc P$ maximizes $h(\cdot)$ if and only if $\mu$ is a Gibbs measure. This is a classical result going back to \cite{LanfordRuelle}, see \cite[Theorem 15.39]{Georgii} for exposition. 

The \textit{local} or \textit{mean-current entropy function} $\ent: \mc O\to \m R$ is defined \symindex{Chapter 2!$\ent(s)$ - entropy of the mean current $s$} \termindex{Chapter 2!local/mean-current entropy function}
\begin{align*}
    \ent(s) = \max_{\mu\in \mc P^s}h(\mu).
\end{align*}
This function is the main focus of Section \ref{sec:entropy}, where we show it has a number of useful properties (continuity, concavity) and show that the maximum is always realized by an ergodic Gibbs measure of mean current $s$. In Theorem \ref{thm: extremal_entropy} we compute its restriction to $\partial \mc O$ by relating it to the analogous local entropy function for lozenge tilings in two dimensions.

We conclude this section with one more use of the term entropy. In Section \ref{sec:asympflows}, we will show that the ``fine-mesh limits" of rescaled tiling flows are precisely the measurable vector fields we call \textit{asymptotic flows}. Asymptotic flows are valued in $\mc O$ and supported on some compact region $R$. The \textit{entropy of an asymptotic flow} $g$ can then be defined as \symindex{Chapter 2!$\Ent(g)$ - entropy of an asymptotic flow $g$} \termindex{Chapter 2!entropy of an asymptotic flow}
\begin{align*}
    \Ent(g) = \frac{1}{\text{Vol}(R)} \int_{R} \ent(g(x)) \, \dd x.
\end{align*}
Properties of this function are studied in Section~\ref{subsection: properties of Ent}. Up to a sign and an additive constant depending on the boundary conditions, it is the rate function for the large deviation principles in Section~\ref{sec:ldp}. 

\section{Local moves}\label{sec:localmoves}

A number of the papers about the 3D dimer model are about local moves. Here we present some simple examples, briefly review the literature, and explain why local move connectedness fails for the torus in dimensions $d>2$. Most of the ideas in this section are already known, but we include a few elementary observations we have not seen articulated elsewhere.

This section can be skipped on a first read, since the results are not essential for the rest of the paper. However, it is useful for understanding some of the ways that the $d=2$ problem differs from the $d=3$ problem (e.g., why the Kasteleyn determinant approach to computing entropy does not work in the same way) and also what makes $d=3$ different from $d>3$ (e.g., the integer-valued twist function is indexed by $\mathbb Z$ when $d=3$ and by $\mathbb Z/2 \mathbb Z$ when $d > 3$). This section will also explain how the figures in the introduction were generated.

\subsection{Local moves in two dimensions}

In two dimensions, a \textit{local move} or \textit{flip}\termindex{Chapter 3!local move/flip} is the operation of choosing a pair of parallel dimers in the tiling, and switching them out for the other pair. See Figures \ref{fig:flip} and \ref{fig:flip sequence}.
\begin{figure}
    \centering
    \includegraphics[scale=0.34]{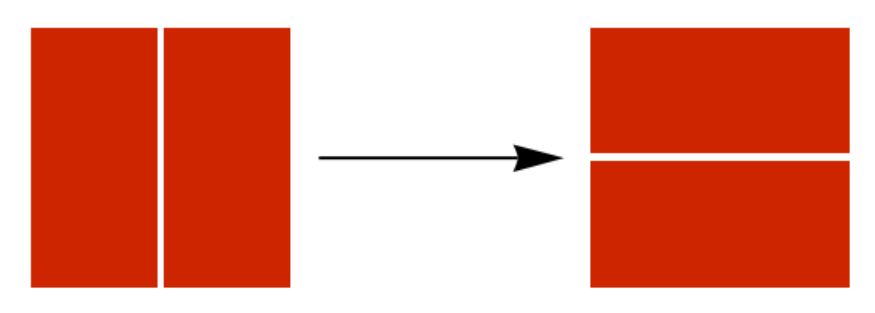}\\
    \caption{A local move or flip in two dimensions.}
    \label{fig:flip}
\end{figure}
Let $R$ be a subgraph of $\mathbb{Z}^2$ and let $\mathcal{T}(R)$ be a graph on the set of dimer tilings of $R$ where two tilings $\tau$ and $\tau'$ are connected by an edge if they differ by a single flip. It is shown using height functions in \cite{Thurston} that if $R\subset \mathbb{Z}^2$ is simply connected and finite, then any two dimer tilings of $R$ differ by a finite sequence of flips. In other words, $\mathcal{T}(R)$ is a connected graph. 
\begin{figure}
    \centering
    \hspace{1cm}\includegraphics[scale=0.55]{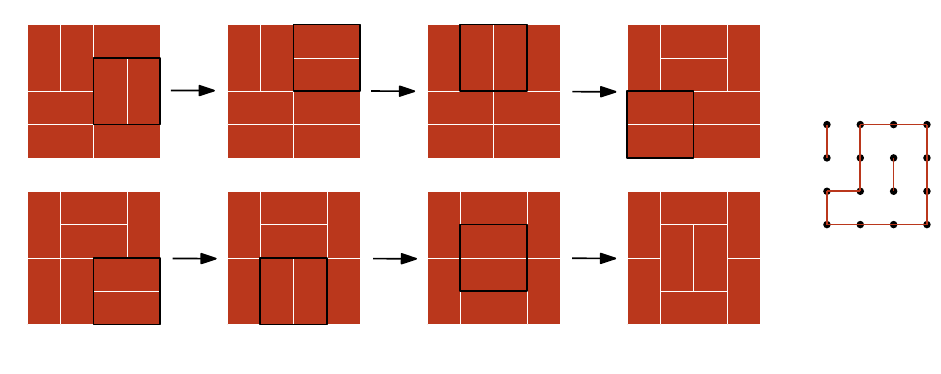}\\
    \caption{(1) an example of a sequence of local moves transforming one tiling into another and (2) the collection of cycles from overlaying the first and last tilings in this sequence.  }
    \label{fig:flip sequence}
\end{figure}

Local move connectedness in the 2D dimer model means that it is possible to probe all tilings of a region using simple local updates, and this is useful for both theoretical and practical purposes. It means that uniformly random dimer tilings in 2 dimensions can be simulated using Markov chain Monte Carlo methods called \textit{Glauber dynamics}. For the 2D dimer model, one can give an explicit polynomial bound on the mixing time of this algorithm \cite{randall2000mixing}.

\subsection{Local moves in three dimensions}

The same local moves (flips) make sense for the 3D dimer model, but local move connectedness with these manifestly fails, even for very small regions. There is a simple counterexample on the $3\times 3\times 2$ box which is called a \textit{hopfion} in the physics literature (see Figure \ref{fig: flip trit hopfion}). 
\begin{figure}
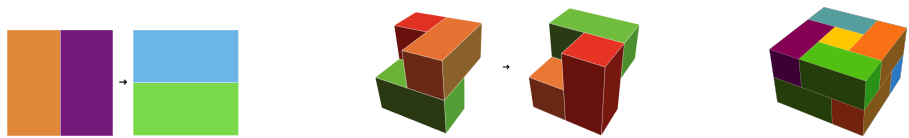

\centering         \includegraphics[scale=0.24]{flip3d.pdf}$\qquad\qquad$\includegraphics[scale=0.2]{trit.pdf}$\qquad\qquad$\includegraphics[scale=0.15]{brick_hopfion.pdf}\\
   \caption{A flip, a trit, and a flip-rigid configuration called a \textit{hopfion}\termindex{Chapter 3!hopfion}. The hopfion has no parallel pairs of tiles, so it is not connected under flips to any other tiling of the $3\times 3\times 2$ box.}
   \label{fig: flip trit hopfion}
\end{figure}

There is a series of papers by Fiere, Milet, Klivans, and Saldanha studying local move connectedness in dimension three under flips and \textit{trits}\termindex{Chapter 3!trit}, a new local move in three dimensions involving three tiles (see Figure \ref{fig: flip trit hopfion}). In \cite{MR3341585, milet2014twists} they show that any two tilings of a region of the form $D\times [0,1]$ where $D$ is simply connected and planar are connected under flips and trits. In subsequent works \cite{milet2014domino,saldanha2019domino,milet2015domino,freire2022connectivity,MR4245261} they introduce and study an invariant called the \textit{twist}\termindex{Chapter 3!twist invariant}, related to the linking number or writhing number. We will present below a brief and informal overview of the various ways the twist is defined and how it is related to a linking number. More detailed exposition is found e.g.\ in \cite{saldanha2020domino} or the references above.

Given two distinct smooth curves $\gamma_1,\gamma_2:S^1\to \m R^3$ embedded in $\m R^3$, one can compute their integer-valued \textit{linking number} $L(\gamma_1,\gamma_2)$ by projecting them to a generic plane and summing the signatures of the crossings. (Recall that the signature of a crossing of two oriented paths is $1$ or $-1$ depending on whether the upper curve crosses the lower curve from right to left or left to right, when the bottom curve is viewed as being oriented from down to up.) 
It is a standard result that this number is independent of the plane one projects onto, see e.g.\ \cite[pages 20-21]{adams1994knot}. (The idea is to show that any one projection can be transformed into another by a sequence of Reidemeister moves, and that these moves preserve the linking number.) The linking number can also be computed with an integral formula: if $ r_1, r_2$ are parametrizations of $\gamma_1,\gamma_2$, then
\begin{align*}
    L(\gamma_1,\gamma_2) = \frac{1}{4\pi} \oint_{\gamma_1}\oint_{\gamma_2} \frac{ r_1 -  r_2}{| r_1- r_2|^3} \dd  r_1 \times \dd  r_2.
\end{align*}
Informally, this represents the line integral along $\gamma_1$ of the magnetic field generated by a steady current through $\gamma_2$. One can analogously compute a ``linking number" of a pair of tilings in a box by summing crossings. Namely, imagine that each edge in the matching is extended $\epsilon>0$ units in either direction. Then the crossing number is obtained by flattening these extended edges to a horizontal plane and summing the signatures of the crossings.  To be more precise, we say two edges $(a,b)$ and $(c,d)$ constitute a crossing if their orientations are both orthogonal to the vertical (third-coordinate) direction and orthogonal to each other {\em and} one of the endpoints of $(a,b)$ differs from one of the endpoints of $(c,d)$ in the vertical coordinate and in no other coordinate. This is the same as an ordinary crossing if we assume each edge is extended $\epsilon$ units beyond its endpoints, and the sign of the crossing is defined in the usual way. We can define the linking of $\tau_1$ and $\tau_2$ to be the signed sum $L(\tau_1, \tau_2)$ of all crossings involving a tile in $\tau_1$ and a tile in $\tau_2$.  This is a quadratic form, and the twist of a tiling $\tau$ is defined by $T(\tau)=  \frac{1}{4}L(\tau, \tau)$. This decomposes as a sum over pairs of horizontal tiles in vertical columns. For reasonable regions (i.e., $D\times [1,N]$, where $D\subset \m Z^2$ is simply connected and $N$ is even), the twist is integer-valued despite the $\frac{1}{4}$ and is independent of the direction for the orthogonal projection \cite[Proposition 6.4]{milet2014domino}. Within a rectangular box, one can easily show that trits increment the twist and flips leave the twist unchanged (in fact this holds for any region of the form $D\times [1,N]$, \cite[Theorem 1]{milet2014domino}). There are also examples of tilings with twist $T(\tau)=0$ that are not connected under flips alone (\cite[Figure 7]{MR4245261}), meaning that $T(\tau)=T(\sigma)$ does not imply that $\tau,\sigma$ are connected under flips. 

Simple questions about local move connectedness under flips and trits still remain open, for example it is not known whether all tilings of an $M\times N\times L$ box are connected under flips and trits when $M,N,L>2$ (see Problem \ref{prob: local moves on box}). See \cite{444box} for an enumeration of all tilings of the $4\times 4\times 4$ box, which shows explicitly that all tilings of this region are connected under flips and trits.

In dimensions $d>3$, Klivans and Saldanha \cite{klivans2020domino} show that the twist is valued in $\m Z/2$. In dimension $d=4$, even tilings of the $2 \times 2 \times 2 \times 2$ box fail to be connected under flips (see \cite[Example 2.2]{klivans2020domino}). They also show that tilings within certain larger boxes are ``almost'' connected under flips, i.e. they can be connected if the boxes are extended in some way.

The works of Friere, Klivans, Milet and Saldanha rely mostly on geometric and algebraic constructions to study local move problems in dimensions $d\geq 3$, but the recent work \cite{localdimer} by Hartarsky, Lichev, and Toninelli (written concurrently with this paper) makes progress using purely combinatorial arguments. In particular it follows from their results that any tiling of a rectangular box in $\m Z^3$ which is tileable by dimers admits at least one flip or trit \cite[Theorem 3]{localdimer}, providing a partial answer to Problem \ref{prob: local moves on box} in Section~\ref{sec:open}. 

In fact, \cite[Theorem 3]{localdimer} is a statement that holds for any dimension $d\geq 2$. It states that any tiling $\tau$ of a rectangular box in $\m Z^d$ of dimensions $(n_1, \dots, n_d)$ which is tileable by dimers contains a copy of $[0,1]^d$ such that $\tau$ restricted to this copy of $[0,1]^d$ contains at least $2^{d-2}+1$ dimers. Specialized to the case $d=3$, this means that there is a copy of $[0,1]^3$ which completely contains at least three dimers from $\tau$, and the only way this can happen is if $[0,1]^3$ contains tiles which make up a flip or a trit in $\tau$. The main idea of the proof is a clever but simple counting argument. Following the ideas in \cite{localdimer}, we present a slight modification of their proof specialized to the $d=3$ case, with the aim of just showing the flip/trit result. 
\begin{prop}[\cite{localdimer}]
    Let $R=[1,n_1]\times[1,n_2]\times[1,n_3]\subset \m Z^3$ with $n_1,n_2,n_3\geq 2$ and $n_1n_2n_3$ even. Any tiling $\tau$ of $R$ admits at least one flip or trit. 
\end{prop}
\begin{proof}
    Fix a tiling $\tau$ of $R$. We view $\tau$ as a tiling of the torus with the same dimensions (i.e., $\tau$ is a tiling of the torus such that no dimers cross the identifications). On one hand, $\tau$ contains $n_1n_2n_3/2$ tiles, and each tile is contained in exactly four translates of $[0,1]^3$. On the other hand, there are $n_1n_2n_3$ possible choices of translates of the unit cube in the torus, so the average number of tiles per unit cube is $2$. 

    If a unit cube contains an above-average number of tiles from $\tau$, it contains at least three tiles. If this unit cube is in the interior of $R$, or is cut in half by only one face of $R$, then since the tiles in $\tau$ do not cross the identifications, this implies there is a flip or trit in $\tau$ as a tiling of $R$. The result then follows by showing that the unit cubes which are cut into four pieces along the edges (or eight pieces at the corner) by the identifications of the torus contain a below-average number of tiles from $\tau$.

    The number of such ``edge unit cubes" is $(n_1-1)+(n_2-1)+(n_3-1) + 1 = n_1 + n_2 + n_3 -2$. Any dimer contained in an edge unit cube must be contained along one of the edges around $R$. The number of vertices in the edges around $R$  is $4(n_1+n_2+n_3) - 16$ (there are $8$ corners, but each one is contained in three edges), hence the maximum number of dimers contained in this region is $2 n_1 + 2 n_2 + 2 n_3 - 8$. Given this, the average number of dimers in $\tau$ per edge unit cube is bounded by
    \begin{align*}
          \frac{2n_1+2n_2+2n_3-8}{n_1 + n_2 + n_3 -2} < 2.
    \end{align*}
    Therefore there must be a non-edge unit cube containing at least three tiles from $\tau$, which completes the proof.
\end{proof}

For the hypercube $[0,1]^d\subset \m Z^d$, Hartarsky, Lichev, and Toninelli also show that for $d\geq 3$, the connected components of the graph on dimer configurations of $[0,1]^d$ under local moves of length up to $d-1$ (here the trit is a move of length three and the flip is a move of length two) have size exponential in $d$ \cite[Theorem 5]{localdimer}, and that for $d\geq 2$, any two dimer tilings of $[0,1]^d$ are connected by a sequence of moves of length $\leq 2(d-1)$ \cite[Theorem 6]{localdimer}. For $[0,n]^d\subset \m Z^d$, $d\geq 2$, $n$ odd, they show that the diameter of the graph on dimer configurations of $[0,n]^d$ under local moves of length $\leq \ell$ is at least $n^{d-1}(n^2-1)/(6\ell^2)$ \cite[Theorem 7]{localdimer}.

Flip connectedness has also independently been studied in the physics literature, from the perspective of looking for ``topological invariants" preserved by flips. In \cite{freedman2011weakly} the authors define a ``Hopf number" for dimer tilings of $\m Z^d$ valued in $\pi_d(S^{d-1})$ which is invariant under flips. The hopfion (see Figure \ref{fig: flip trit hopfion}) has Hopf number $\pm 1$ (depending on its orientation). This construction works for any dimension $d\geq 2$. The fact that $\pi_2(S^1) = 0$ corresponds to no obstruction to connectedness under flips, and $\pi_3(S^2) = \m Z$ corresponds to there being at least countably many connected components under flips in dimension $3$. For all $d>3$, $\pi_d(S^{d-1}) = \m Z/2$, implying at least two connected components under flips. 

In \cite{bednik2019hopfions} it is shown in examples that the Hopf number from \cite{freedman2011weakly} can be computed using discrete versions of Cherns-Simon integral formulas for the Hopf number applied to a version of the tiling flow and its vector potential. See also \cite{bednik2019probing}.

\begin{rem}
The failure of local move connectedness in three dimensions is intimately related to the failure of (at least a straightforward generalization) of Kasteleyn theory. 

In two dimensions, the partition function for dimer tilings of a simply connected planar graph can be computed as the Pfaffian of an adjacency matrix of the directed graph with appropriate weights (this can also be done with a determinant when the graph is bipartite). Recall that if $M = (m_{ij})$ is an $2n\times 2n$ skew-symmetric matrix, 
\begin{align*}
    \text{Pf}(M) = \frac{1}{2^n n!}\sum_{\sigma\in S_{2n}} \text{sign}(\sigma) \prod_{i=1}^n m_{\sigma(2i-1),\sigma(2i)}.
\end{align*}
There are two key observations in two dimensions. First, the weights can be chosen so that the term is $\pm 1$ if and only if the pairing $\{\sigma(2i-1),\sigma(2i)\}_{1\leq i \leq n} $ corresponds to a dimer tiling and otherwise it is $0$. By this, it is clear that the partition function can be computed as a \textit{permanent} (i.e., like the above without the sign terms). The second key observation, which is why this reduces to a Pfaffian computation, is that the weights can be chosen so that applying a flip does not change the sign of the term. From here, flip connectedness in two dimensions shows that the Pfaffian is counting tilings. 

In three dimensions it is still possible to choose weights so that a term is $\pm 1$ if and only if it corresponds to a dimer tiling, and all other terms are $0$. Choosing certain weights such that flips do not change the sign, it is observed in \cite{freedman2011weakly} that the Hopf number invariant mod 2 is equal to the sign of the term in the Pfaffian (and one can check that the trit increments this number). From this they note that if $M$ is defined analogously to in two dimensions, then in 3D
\begin{align*}
    \text{Pf}(M) = A - B
\end{align*}
where $A + B$ would be the partition function. The term $A$ counts tilings with Hopf number $0\mod 2$ and $B$ counts tilings with Hopf number $1\mod 2$. 

In \cite{klivans2020domino}, the number $A-B$ is called the \textit{defect}. The definition of the \textit{twist invariant} discussed above is extended to dimensions $d>3$ as the sign of the appropriate Kasteleyn determinant \cite[Definition 3.1]{klivans2020domino}.

One can check by enumerating the equations for a single cube (i.e.\ 12 edges) that it is not possible to choose 12 nonzero weights so that the six flips (corresponding to $\sigma$ with sign $-1$) and four trits (corresponding to $\sigma$ with sign $+1$) contained in the cube all preserve the sign of the term in the Pfaffian. In fact the six flip equations plus one trit equation have no simultaneous solution with all positive weights. 

More generally, there is a complete characterization of which graphs admit \textit{Pfaffian weights} and thereby make it possible to compute the partition function (which is a priori a permanent) as a determinant or Pfaffian of a re-weighted matrix. It is shown that a bipartite graph $G$ admits Pfaffian weights if and only if it does not ``contain" $K_{3,3}$ \cite{LITTLE1975187}. Here ``contain" means $G$ can be modified (by replacing a collection of disjoint paths of edges containing an even number of vertices with a single edges) to a graph $H$ which has $K_{3,3}$ as a subgraph. One can see that in this sense $\m Z^3$ contains $K_{3,3}$, and hence does not have Pfaffian weights. The class of graphs that have Pfaffian weights can also be described in a way so that the Pfaffian is computable by a polynomial-time algorithm \cite{Pfaffians1999}.
\end{rem}

\subsection{Loop shift Markov chain for uniform sampling}\label{sec:uniform sampling}

In two dimensions, uniformly random dimer tilings of finite simply connected regions can be efficiently simulated by a Markov chain that generates random flips, see \cite{randall2000mixing}. As we have seen, dimer tilings of topologically trivial finite regions in dimensions $d>2$ are not connected under flips, and it is an open question even for very simple regions whether flips and trits are sufficient. Here we describe a different, non-local Markov chain method to generate uniform random dimer tilings. The algorithm works in any dimension and for regions that are not simply connected, and is how the simulations in the introduction are generated. The simple move executed at each step of our chain is to construct a ``random loop" in the given dimer tiling, and ``shift" the tiles along the loop. This is a well-known construction in computer science, see for instance \cite[Section 3]{marry_at_random}. In the physics literature, see also \cite{Huse2003} for Monte Carlo simulations of dimers in three dimensions based on algorithms from \cite{Krauth2003pocket,Dress1995}.

Given a dimer tiling $\tau$ of a finite region $R\subset \Z^3$, a \emph{loop} $\gamma$ in $\tau$ is a sequence of distinct edges $e_0, e_1, \ldots e_{k-1}\in \tau$ where the odd vertex of $e_i$ is adjacent to the even vertex of $e_{i+1}$ for all $i \in \Z/k\Z$ for some $k \geq 2$.\termindex{Chapter 3!loop}\termindex{Chapter 3!loop shift} A \emph{loop shift} of $\tau$ along $\gamma$ is a move which replaces edges along $\gamma$ by their complementary edges. Specifically the resulting tiling is
$$\tau'=(\tau\setminus \{e_0, e_1, \ldots e_{k-1}\})\cup \{f_0, f_1, \ldots, f_{k-1}\}$$
where  $ \{e_0, e_1, \ldots e_{k-1}\}\cup \{f_0, f_1, \ldots, f_{k-1}\}$ form a loop in $\Z^3$. Since $R$ is finite, given any two dimer tilings $\tau,\sigma$ of $R$ the double dimer tiling $(\tau,\sigma)$ is a finite collection of double edges and loops $\gamma_1,...\gamma_n$ of finite length. Loop shifting $\tau$ along $\gamma_i$ for each of these transforms $\tau$ into $\sigma$. In particular, we have shown that 

\begin{prop}\label{prop:loop_shift}
	Let $\tau$ and $\sigma$ be tilings of a finite set $R\subset \Z^3$. Then $\tau$ can be transformed into $\sigma$ by a finite sequence of loop shifts. 
\end{prop}

\textbf{Loop shift Markov chain $M$.} Given that any two tilings of a finite region $R\subset \mathbb{Z}^3$ differ by a finite sequence of loop shifts, we define a Markov chain $M$ where one step proceeds as follows: 
\begin{itemize}
    \item Start with some dimer tiling $\tau$ of the region $R$. 
    \item Sample an odd vertex in $R$ uniformly at random. Start a path by following the tile from $\tau$ at this point. 
    \item Uniformly at random choose a direction (other than the one we came from), and move in that direction for the next step. 
    \item Repeat this (following the tile from $\tau$, then following a uniform random choice, etc) until the path hits itself to form a loop. Call the loop $\gamma$.
    \item Drop any initial segment of the path which is not part of the loop $\gamma$. Then shift along $\gamma$, switching the tiles from $\tau$ for the random choices that we made along the path, and replace $\tau$ with $\sigma$ which differs from $\tau$ only along $\gamma$. 
\end{itemize}
By Proposition \ref{prop:loop_shift}, $M$ is an irreducible Markov chain and hence has a stationary distribution $\pi$. A bound on the mixing time of $M$ is not known, see Problem \ref{prob: markov chain mixing time}. 
\begin{thm}
The stationary distribution $\pi$ of $M$ is the uniform distribution on dimer tilings of $R$. 
\end{thm}
\begin{proof}
 Let $P$ be its transition matrix. It is sufficient to prove that $P$ is symmetric. If $\tau,\sigma$ are tilings such that $P(\tau,\sigma) \neq 0$, then they differ along a single loop $\gamma$. 

Suppose that $\lambda$ is a connected path alternating between tiles of $\tau,\sigma$ which consists of an initial segment $\alpha$ plus the loop $\gamma$. $P(\tau,\sigma)$ is a sum of the probability of paths $\lambda$ of this form. We will show that the probability of generating $\lambda$ in $\tau$ is the same as the probability of generating $\lambda'$ in $\sigma$, where $\lambda'$ has the same initial segment as $\lambda$, then traverses $\gamma$ with the reverse orientation. 

Let $v_1,...,v_{2n}$ be the vertices along $\lambda$. Note that the vertices with odd index are odd, and out of these we follow a tile from $\tau$. The vertices with even index are even, and out of these we follow a random choice. Thus the probability of generating the path $\lambda$ in $\tau$ is $\prod_{k=1}^n \frac{1}{\text{deg}(v_{2k})-1}$. 

The sequence of vertices along $\lambda'$ is the same, just in a different order. However the even vertices are still the sites where we make a random choice of direction to follow, so the probability of generating the path $\lambda'$ in $\sigma$ is also $\prod_{k=1}^n \frac{1}{\text{deg}(v_{2k})-1}$. 

Hence $P(\tau,\sigma) = P(\sigma, \tau)$. 
\end{proof}

\subsection{Local move connectedness on the torus and k-Gibbs measures}

Here we discuss local move connectedness for dimer tilings of the torus, which is not simply connected. For any tiling $\tau$ of the $d$-dimensional torus $\mathbb T^d$, there is a standard, natural way to associate a homology class $[ a(\tau)]\in H_1(\mathbb T^d)$. For each $i=1,...,d$, let $P_i$ be any plane with normal vector ${\eta}_i$, the $i^{th}$ unit coordinate vector. Let $\mathbb T_{{n}}^d$ denote the $ n = n_1\times n_2\times ... 
\times n_d$ torus in dimension $d$. Without loss of generality, $n_1$ is even. Let $\tau_0$ be the tiling of $\mathbb T_{{n}}^d$ where all tiles $t\in \tau_0$ are of the form $t = ((2x,y,z),(2x+1,y,z))$. With slight abuse of notation, we let $v_\tau(p) = v_\tau(e) e$ for the edge $e$ incident to $p$ containing a dimer (in particular $v_\tau(p)$ is one of the $2d$ unit coordinate vectors). For $i=1,\dots, d$, we define
\begin{align*}
    a_i(\tau) = \sum_{p\in P_i\cap \mathbb T_{ n}^d}  \langle v_{\tau}(p),\eta_i\rangle - \langle v_{\tau_0}(p) ,\eta_i\rangle =\sum_{p\in P_i\cap \mathbb T_{ n}^d} \langle v_{\tau}(p), \eta_i\rangle.
\end{align*}
Since $v_{\tau}-v_{\tau_0}$ is divergence-free, this is independent of the choice of plane $P_i$ normal to $\eta_i$. The second equality follow from the fact that $v_{\tau_0}$ contributes $0$ to the overall sum. The homology class of $\tau$ is
$$[{a}(\tau)]:=[a_1(\tau),...,a_d(\tau)]\in H_1(\mathbb T^d) \simeq \mathbb Z^d.$$
Note that a parallel pair of tiles contributes $0$ total flow across any coordinate plane intersecting it. In particular, in any dimension $d>1$, flips cannot change the homology class of a tiling of $\mathbb T^d$. However, when $d=2$ the homology class is the only obstruction: if $\tau,\tau'$ are tilings of an $n_1\times n_2$ torus $\mathbb T_{n_1,n_2}^2$ and $[{a}(\tau)]=[{a}(\tau')]$, then $\tau,\tau'$ are connected by a finite sequence of flips. 

In dimension $d=3$, the story is very different. In fact:
\begin{prop}\label{prop:three_torus_no_finite}
    There is no finite collection of local moves that can connect all homologically equivalent dimer tilings of $\m T^3$.
\end{prop}
\begin{rem} The authors of \cite{freire2022connectivity} exhibited a tiling of the $8\times 8\times 4$ torus with no flips or trits, obtained by stacking horizontal brickwork patterns of different orientations. We use similar stacked brickwork patterns (but with thicker layers) in our proof of Proposition~\ref{prop:three_torus_no_finite}. 
\end{rem}

\begin{proof}
    The fundamental example is the following. Let $\tau$ be a tiling of $\mathbb T_{n_1,n_2,4}^3$ where the first layer is an $\eta_1$ brickwork tiling, the second layer is an $\eta_2$ brickwork tiling, the third layer is a $-\eta_1$ brickwork tiling, and the fourth layer is a $-\eta_2$ brickwork tiling. By construction, $[ a(\tau)] = (0,0,0)$. On the other hand, $\tau_0$ also has $[{a}(\tau_0)] = (0,0,0)$, so $\tau$ and $\tau_0$ are homologically equivalent. On the other hand, the length of the shortest alternating-tile loop in $\tau$ is $\min\{n_1,n_2,4\}$. To see this, note that if the loop is homologically trivial, it must be long enough to visit at least three different horizontal layers. If it is homologically non-trivial, then its length must be at least $\min\{n_1,n_2,4\}$. 

More generally, for any $ n = (n_1,n_2,4 n_3)$, we can construct a tiling $\tau$ of $\mathbb T_{ n}^3$ which has $n_3$ layers of $\eta_1$ brickwork, followed by $n_3$ layers of $\eta_2$ brickwork, $n_3$ layers of $-\eta_1$ brickwork, and $n_3$ layers of $-\eta_2$ brickwork. Again $[ a(\tau)] = (0,0,0)$, however the shortest contractible loop in $\tau$ has length $4n_3$ (since, again, it has to visit at least three different brickwork patterns). Therefore to connect $\tau,\tau_0$ we need loops of length at least $\min\{n_1,n_2,4n_3\}$. These dimensions can be arbitrarily large, so this completes the proof.
    \end{proof}

By lifting this construction from $\m T^3$ to $\m R^3$, we get the following corollary: 
\begin{cor}
There is no finite collection of local moves which connects any two tilings of $\m Z^3$ which differ at only finitely many places. 
\end{cor}

\begin{proof}
    Fix an integer $n>0$. Tile all of $\m Z^3$ with alternating brickwork layers so that there are $n$ layers of $\eta_1$ brickwork, $n$ layers of $\eta_2$ brickwork, $n$ layers of $-\eta_1$ brickwork, and $n$ layers of $-\eta_2$ brickwork. We denote the resulting tiling of $\m Z^3$ by $\tau_n$. 
    
    The shortest length of a cycle in $\tau_n$ is $4n$. Since there are finite cycles in $\tau_n$, there exist tilings $\sigma$ which differ from $\tau_n$ at only finitely many places. On the other hand, we need a local move of length at least $4n$ to make any change to $\tau_n$. Since $n$ is arbitrary this completes the proof.
\end{proof}

Another interesting observation can be made from the example used in the proof of Proposition \ref{prop:three_torus_no_finite}. A measure $\mu$ is \textit{k-Gibbs} if for any box $B$ with side length $k$, it holds that conditioned on a tiling $\tau$ of $\m Z^3\setminus B$, $\mu$ is the uniform measure on tilings $\sigma$ of $B$ extending $\tau$. If a measure is $k$-Gibbs for all $k$, then it is Gibbs. 

In two dimensions, any two tilings of a $k\times k$ box (with the same boundary condition) are connected by some finite sequence of flips. Therefore if a measure on dimer tilings of $\m Z^2$ is $2$-Gibbs, then it is $k$-Gibbs for all $k$ and hence Gibbs. The analogous statement does not hold in three dimensions.

\begin{prop}\label{prop: k gibbs fails} For any integer $k \geq 2$ there exist $k$-Gibbs measures on $\Omega$ which are not Gibbs measures. 
\end{prop}
\begin{proof}
    Take $n = (n_1,n_2,n_3,n_4)$ and consider the tiling of $\m Z^3$ which alternates between $n_1$ layers of $\eta_1$ bricks, $n_2$ layers of $\eta_2$ bricks, $n_3$ layers of $-\eta_1$ bricks, and $n_4$ layers of $-\eta_2$ bricks. Define a measure $\mu_m$ by averaging over translations by $\threeeven$ in the $m\times m\times m$ box and let $\mu$ be a subsequential limit as $m\to \infty$. The measure $\mu$ is invariant under the action of $\threeeven$. For $k\leq \min\{n_1,n_2,n_3,n_4\}$, $\mu$ is $k$-Gibbs since within any size $k$ cube, a tiling sampled from $\mu$ is frozen for $k\leq \min\{n_1,n_2,n_3,n_4\}$. For $k > \min\{n_1,n_2,n_3,n_4\}$, $\mu$ still a.s.\ samples tilings which are brickwork patterns restricted to horizontal layers. However tilings of these larger boxes are not frozen, and are connected by shifting on finite loops to tilings which are not brickwork on every layer. Therefore $\mu$ is not $k$-Gibbs for $k > \min\{n_1,n_2,n_3,n_4\}$, hence $\mu$ is not Gibbs.
\end{proof}

The construction in the proof works to construct a $k$-Gibbs-but-not-Gibbs measure for any mean current $s=(s_1,s_2,0)$. A more complicated construction allows us to show that there exist $k$-Gibbs measures which are not Gibbs and correspond to an $s$ in the interior of $\mathcal O$ for which $s_1 s_2 s_3 \neq 0$. (Essentially one can arrange a periodic pattern of infinite non-intersecting taut paths like the ones shown in the introduction.) We have not found a construction that works for every $s \in \mathcal O$.

\section{Measures with boundary mean current}
\label{sec:extreme}

Recall from Section \ref{section:measures-currents} that $\threeeven$-invariant measures on dimer tilings of $\mathbb{Z}^3$ come with a parameter called the \textit{mean current}. This definition makes sense in any dimension $d$. When $d=2$, the mean current is a 90-degree rotation of the height function \textit{slope}, and in general it is valued in the convex polyhedron\symindex{Chapter 4!$\mathcal{O}_d$ the space of admissible mean currents in $d$ dimensions}
\begin{align*}
    \mathcal{O}_d = \{(s_1,...,s_d) : |s_1| + ... + |s_d| \leq 1\}.
\end{align*}
Recall that the mean current of a measure $\mu$ is defined in terms of tile densities (Definition \ref{def: mean current}). 
Given a standard basis $\eta_1, \eta_2, \ldots, \eta_d$ of $\Z^d$ denote by $e_i$ the edge connecting $ 0$ with $ \eta_i$ and $-e_i$ the edge connecting $ 0$ with $- \eta_i$. The \emph{mean current} of a measure $\mu\in \Prob(\Omega)$ is an element of $\mathbb R^d$ such that its $i^{th}$-coordinate is
$$( s(\mu))_i= \mu(e_i\in \tau)-{\mu}(-e_i\in \tau).$$

 If $s\in \partial \mathcal O_d$ we say that $s$ is a \textit{boundary mean current}. In terms of tiles, a measure $\mu$ has boundary mean current if and only if with probability $1$ it samples at most one of the two possible tile types in each coordinate direction. The purpose of this section is to describe ergodic Gibbs measures with boundary mean current in dimension three. Using this, we compute the entropy function $\ent(\cdot)$ in 3D restricted to $\partial \mc O = \partial \mc O_3$ (Theorem \ref{thm: extremal_entropy}).

 We will see that measures with boundary mean current in 2D and 3D are qualitatively very different. While the EGMs with boundary mean current in two dimensions all have zero entropy, EGMs with boundary mean current $s\in \partial \mc O$ in three dimensions can have positive entropy when $s$ is contained in the interior of a face of $\partial \mc O$. Further, in three dimensions for any value $a$ between $0$ and $\ent(s)$, there exists an EGM $\mu$ with specific entropy $h(\mu) = a$. 
 
Despite these differences, in 2D and 3D the general principle is that measures with boundary mean current in dimension $d$ correspond to sequences of measures on a $(d-1)$-dimensional lattice. This is easy to see in 2D, and we use it as a warm-up for the 3D version.

\subsection{Review: EGMs with boundary mean current in two dimensions}
 
Call the four possible tile directions in two dimensions (east, west) and (north, south). It is sufficient to describe measures with boundary mean current $(s_1,s_2)$ for which $s_1,s_2\geq 0$ and $s_1+s_2=1$, i.e.\ measures that sample only north and east tiles. The first step is to understand what tilings containing only north and east tiles look like. 

For an even point $(x_1,x_2)$, the north tile connects it to $(x_1,x_2+1)$ and the east tile connects it to $(x_1+1,x_2)$. In other words, north and east tiles always connect points along the line $x_1+x_2 = 2c$ to points along the line $x_1+x_2 = 2c+1$. Therefore a tiling consisting of only north and east tiles can be partitioned into an infinite sequence of complete dimer tilings of \textit{strips} $S_c = \{(x_1,x_2): x_1+x_2 = 2c \text{ or } 2c+1\}$. 

\begin{figure}
    \centering
        \includegraphics[scale=0.8]{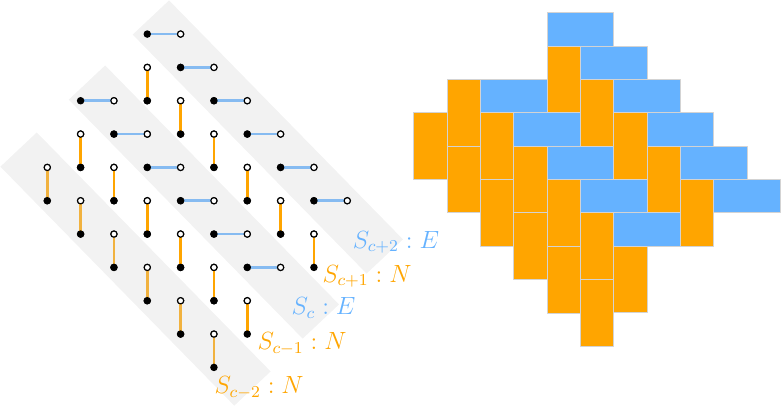}\\
    \caption{Parts of five strips drawn on the dual graph (left) and as a tiling (right)}
    \label{fig:extreme tilings 2d}
\end{figure}

Along each strip, there are only two complete dimer tilings: one where the tiles are all east, and one where the tiles are all north. As such, any tiling $\tau$ with only north and east tiles consists of a sequence of choices of north or east tiles along the strips. See Figure \ref{fig:extreme tilings 2d}. 

\termindex{Chapter 4!frozen tilings}All tilings $\tau$ of $\m Z^2$ containing only north and east tiles are \textit{frozen}, meaning they contain no finite cycles. To see this, note that if $\tau$ contains a finite cycle, then local move connectedness (see Section \ref{sec:localmoves}) implies it could be broken down into cycles of length 2. However a cycle of length 2 requires a north-south or east-west pair of tiles, which is not possible if the tiling contains only north and east tiles. Since tilings containing only north and east tiles are frozen, any measure $\mu$ which a.s.\ samples such tilings is automatically Gibbs. Three useful observations follow from this discussion: 
\begin{itemize}
    \item All ergodic Gibbs measures with boundary mean current in two dimensions have zero entropy. In other words, entropy is zero when restricted to $\partial \mc O_2$. 
 	\item There is a bijection between 1) Gibbs measures on dimer tilings of $\m Z^2$ that contain only $E$ (east) and $N$ (north) tiles and 2) measures on integer-indexed $\{N,E\}$ sequences. Any sample of a process taking value $E$ with proportion $s_1$ and $N$ with proportion $s_2$ corresponds to a sample of a Gibbs measure on dimer tilings (obtained by placing $N$ and $E$ tiles on consecutive strips) of $\m Z^2$ with mean current $(s_1,s_2)$ and vice versa.
 	\item There is also a bijection between 1) {ergodic} Gibbs measures on dimer tilings of $\m Z^2$ that contain only $E$ and $N$ titles and have mean current $(p, 1-p)$ and 2) {ergodic} measures on integer-indexed $\{N, E\}$ sequences where the origin has probability $p$ of being being assigned $E$.
\end{itemize}

\subsection{EGMs with boundary mean current in three dimensions}
Now we will consider the three dimensional case. Let the types of tiles be (east, west), (north, south), (up, down). Without loss of generality we consider measures with boundary mean current that almost surely sample only north, east, and up tiles, i.e.\ mean current $s = (s_1,s_2,s_3)$ with $s_1+s_2+s_3 = 1$, $s_1,s_2,s_3\geq 0$.

For an even point $(x_1,x_2,x_3)$, an east tile connects it to $(x_1+1,x_2,x_3)$, a north tile connects it to $(x_1,x_2+1,x_3)$, and an up tile connects it to $(x_1,x_2,x_3+1)$\termindex{Chapter 4!north, east and up tiles}. Therefore a tiling in 3D using only these three tile types corresponds to a sequence of tilings of two-dimensional \textit{slabs}, \termindex{Chapter 4!slabs}\symindex{Chapter 4!$L_c$ - a slab in $\Z^3$}
$$L_c = \{(x_1,x_2,x_3) : x_1+x_2+x_3 = 2c \text{ or } 2c+1\}.$$
These slabs turn out to be a familiar two-dimensional lattice, namely the hexagonal lattice (with dimers viewed as edges) or the dual triangular lattice (with each dimer is viewed as a ``lozenge'' obtained as the union of two adjacent triangles), see Figure \ref{figure: lozenge dimer}. 

\begin{figure}
	\includegraphics[scale=0.8]{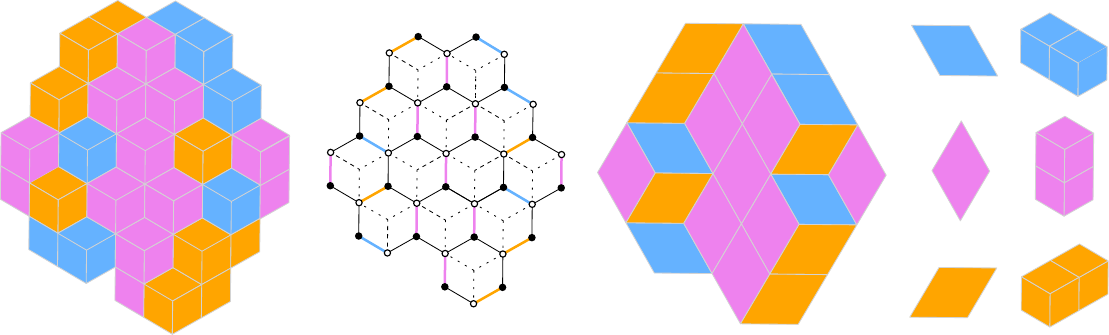}\\
	\caption{ (1) Cubes from a slab of $\mathbb{Z}^3$ visible from above the slab in a dimer tiling $\tau$ of $\mathbb{Z}^3$, (2) tiles from $\tau$ drawn on the hexagonal lattice as edges colored pink,blue and orange, (3) the same tiles drawn as lozenges obtained by taking the Voronoi cells containing these edges, and (4) a key giving the translation between lozenge tiles and 3D dimer bricks.}
	\label{figure: lozenge dimer}
\end{figure}

In the following, given a dimer tiling $\tau$ of a slab $L_c$, we will say that a particular tile type (north, east or up) has density $s_i$ if the proportion of tiles of that type in $\tau\cap[-n,n]^3$ converges to $s_i$ as $n\to \infty$. Similarly we can define the density for lozenge tilings. 

\begin{prop}\label{prop: Sc lattice description}
    For each $c\in \m Z$, the slab $L_c$ is a copy of the hexagonal lattice. There is a correspondence between tilings $\tau$ of $\m Z^3$ which use only north, east and up tiles restricted to $L_c$ and lozenge tilings. This correspondence takes a tiling of $L_c$ with density $(s_1,s_2,s_3)$ of the north, east and up tiles to a lozenge tiling where the density of the three lozenge tiles is also $(s_1,s_2,s_3)$.
\end{prop}

\begin{rem}
There is a completely analogous correspondence for $s\in \partial \mc O$ when some of the components of $s$ are negative. If the signs of $s$ are $(\epsilon_1,\epsilon_2,\epsilon_3)$ then a tiling with boundary mean current $s$ would restrict to a lozenge tiling on $\{(x_1,x_2,x_3) : \epsilon_1 x_1 + \epsilon_2 x_2 + \epsilon_3 x_3 = 2c \text { or }2c+1\}$. To simplify the presentation, some of the results in this section are stated for $s_1,s_2,s_3\geq 0$, but the analogous statements hold for all $s\in \partial \mc O$.
\end{rem}

\begin{proof}
Here we view the tiling as a collection of edges. Since $\tau$ uses only north, east and up tiles, the restriction $\tau_c=\tau\mid_{L_c}$ is a complete tiling of $L_c$. A single cube $C$ in the $\mathbb{Z}^3$ lattice intersects four layers of the form $x_1+x_2+x_3=a$. Let $\mc C_a$ be the collection of cubes in $\m Z^3$ which intersect the layers $x_1+x_2+x_3 = a-1, a, a+1, a+2$. By construction, $L_c \subset \mc C_{2c}$.
\begin{figure}[H]
    \centering
     \includegraphics[scale=0.5]{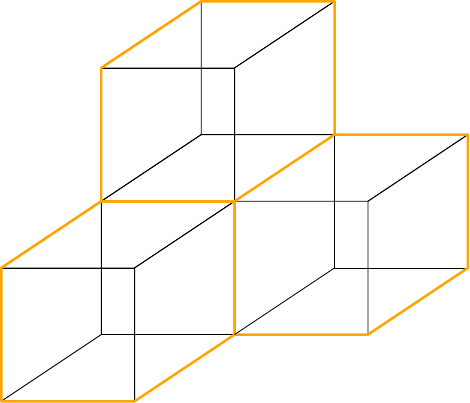}
    \caption{Three adjacent cubes in $\mc C_{2c}$, with intersection with $L_c$ in orange.}
    \label{fig:cube_hex}
\end{figure}

For each $C\in \mc C_{2c}$, $C\cap L_c$ is a hexagon, hence the faces of $L_c$ are hexagons. By observation we see that any two adjacent hexagons meet in an edge, any three adjacent hexagons meet at a vertex, and there are no collections of $>3$ adjacent hexagons. Hence $L_c$ is a copy of the hexagonal lattice. Finally Figure \ref{figure: lozenge dimer} gives the correspondence between the north, east and up tiles with the three kinds of lozenges which preserves their densities.
\end{proof}

Recall that $\Prob^{s}$ denotes the set of $\threeeven$-invariant probability measures on dimer tilings of $\Z^3$ of mean current $s=(s_1,s_2,s_3)$. We add subscripts $G$ and $e$ to denote Gibbs and ergodic measures respectively. Consider the group
$$\Z_{\text{loz}}=\threeeven\cap \{(x_1,x_2,x_3)~:~x_1+x_2+x_3=0\}.$$
Let $\mc P_{\text{loz}}$ denote the space of probability measures on dimer tilings of the slab $L_0$ (i.e.\ lozenge tilings) which are invariant under the $\m Z_{\text{loz}}$ action. The \textit{slope} of a measure $\rho$ on lozenge tilings is the triple $s=(s_1,s_2,s_3)$ of expected densities of the three types of lozenges with respect to $\rho$. A lozenge tiling slope satisfies $s_1,s_2,s_3\geq 0$ and $s_1+s_2+s_3=1$.\symindex{Chapter 4!$\Z_{\text{loz}},\mc P_{\text{loz}}$ - the space of invariant probability measures on the space of lozenge tilings/ dimer tilings of the hexagonal lattice} \symindex{Chapter 4!$s(\mu)$ - an abuse of notation specific to this section where it also represents the slope of lozenge tilings}\termindex{Chapter 4!slope for lozenge tilings}

We abuse notation slightly and write $s(\mu)$ to mean the mean current or slope depending on what space $\mu$ is a measure on. To reduce notation issues, for the rest of the subsection we denote measures on dimer tilings of $\m Z^3$ by $\mu$ or $\nu$ and measures on lozenge tilings by $\rho$ or $\lambda$. 

In this section, we use the notation $\tau_B$ to mean $\tau$ restricted to $B\subset \m Z^3$. \symindex{Chapter 4!$\tau_B$ - the restriction of a tiling $\tau$ to $B$}

\begin{prop}\label{prop: extreme mc gibbs slope}
    Suppose $\mu$ is an ergodic Gibbs measure on dimer tilings of $\m Z^3$ with mean current $s=(s_1,s_2,s_3)\in \partial \mc O$, $s_1,s_2,s_3\geq 0$. Let $\rho_0$ be the marginal measure of $\mu$ on the slab $L_0$. Then $\rho_0$ is a $\m Z_{\text{loz}}$-invariant Gibbs measure with lozenge tiling slope $s(\rho_0) = s$. \symindex{Chapter 4!$\rho_0$ - marginal of a measure $\mu$ to the slab $L_0$}
\end{prop}
\begin{rem}
    Since $\mu$ is $\threeeven$-invariant, $\mu$ is invariant under the $\m Z$-action of translating by $(0,0,2c)$, which takes $L_0$ to $L_c$. Therefore $\rho_c$ and $\rho_0$ are identically distributed for all $c\in \m Z$. 
\end{rem}
\begin{proof}
    Since $\m Z_{\text{loz}}\subset \threeeven$, $\rho_0$ is $\m Z_{\text{loz}}$-invariant. Consider a finite connected set $B\subset L_0$ with boundary $\partial B$ in $\m Z^3$. Suppose $\tau\in \Omega$ is a tiling in the support of $\mu$, implying that it is a tiling using only north, east, and up tiles. Since $\mu$ is a Gibbs measure, and since there is $\mu$-a.s. no tile in $\tau$ connecting $L_0$ and $L_c$ for $c\neq 0$, we have for any tiling $\sigma\in \Omega$,
\begin{align*}
    \mu(\sigma_{B} ~|~ \tau_{\m Z^3 \setminus B} ) = \mu(\sigma_{B} ~|~ \tau_{\partial B \cap L_0}) = \rho_0(\sigma_{B} ~|~ \tau_{\partial B \cap L_0}).
\end{align*}
In the above we use the notation that for a tiling $\sigma\in \Omega$ and a set $A\subset \m Z^3$, $\sigma_A$ means $\sigma$ restricted to $A$. Since $\mu$ is a Gibbs measure, the left hand side is uniform. Therefore $\rho_0$ is also a Gibbs measure. 

Relating $s(\rho_0)$ to $s(\mu)$ is straightforward. Recall from Section \ref{section:measures-currents} that $s_0(\tau)$ denotes the direction of the tile at the origin in $\tau$, and that $s(\mu) = \m E_{\mu}[s_0(\tau)]$. The same function $s_0$ can be used to compute the slope of a lozenge tiling measure, and $s(\rho_0)= \m E_{\rho}[s_0(\sigma)]$ where $\sigma$ is a full-plane lozenge tiling.

Let $\tau_0=\tau_{L_0}$ be $\tau$ restricted to $L_0$. Since $\mu$ has boundary mean current, $\tau_0$ is a full-plane lozenge tiling $\mu$ a.s.\ and $s_0(\tau) = s_0(\tau_0)$. Thus 
\begin{align*}
    s(\mu) = \m E_{\mu}[s_0(\tau)] = \m E_{\rho_0}[s_0(\tau_0)] = s(\rho_0). 
\end{align*}
\end{proof}

To show that $h(\mu) = h(\rho_0)$ (Proposition \ref{prop: extreme mc entropy}), we use the fact that any Gibbs measure can be uniquely decomposed into {\textit{extreme} Gibbs measures} \cite[Theorem 7.26]{Georgii}. Extreme Gibbs measures are the extreme points of the convex set of Gibbs measures (analogous to how ergodic measures are the extreme points of the convex set of invariant measures). A Gibbs measure is extreme if and only if it is tail trivial \cite[Theorem 7.7]{Georgii}.\termindex{Chapter 4!extreme Gibbs measures}

If $\rho$ is a Gibbs measure, there is a unique weight function $g_\rho$ on the extreme Gibbs measures which gives its \textit{extreme Gibbs decomposition}, 
\begin{align*}
    \rho = \int \lambda \,\dd g_\rho(\lambda). 
\end{align*}
This decomposition \symindex{Chapter 4!$g_{\rho}$ - extreme Gibbs decomposition} means that sampling from a Gibbs measure $\rho$ can be thought of as a two step process: 1) sample an extreme Gibbs component $\lambda$ from $\dd g_\rho$, 2) sample a tiling $\tau$ from $\lambda$.  Given a tiling $\tau$ sampled from a Gibbs measure $\rho$, we can a.s.\ recover the extreme Gibbs component $\lambda$ that $\tau$ was sampled from \cite[Theorem 7.12]{Georgii}. If $\lambda$ is the extreme Gibbs component that $\tau$ is sampled from, we say that $\tau$ is \textit{generic} for $\lambda$.

\begin{prop}\label{prop: extreme mc entropy}
    Suppose $\mu$ is an EGM on $\Omega$ with mean current $s=(s_1,s_2,s_3)\in \partial \mc O$, $s_1,s_2,s_3\geq 0$. Let $\rho_c$ be the marginal measure of $\mu$ on the slab $L_c$ for $c\in \m Z$. Sampling a tiling $\tau$ from $\mu$ induces a choice of extreme Gibbs component $\lambda_c$ of $\rho_c$ for all $c\in \m Z$. For each $c\in \m Z$, let $\tau_c = \tau_{L_c}$.\symindex{Chapter 4!$\tau_c$ - the tiling $\tau$ on the slab $L_c$}

    \begin{enumerate}
        \item Conditional on the choice of extreme Gibbs component $\lambda_c$ of $\rho_c$ for each $c\in \m Z$, the samples $(\tau_c)_{c\in \m Z}$ are independent. 
        \item For any $c\in \m Z$, $h(\mu) = h(\rho_c)$. 
    \end{enumerate}
\end{prop}
\begin{rem}
    Since $\rho_0$ and $\rho_c$ are identically distributed, $h(\rho_c) = h(\rho_0)$ for all $c\in \m Z$. Thus it suffices to prove (2) for $c=0$. We also note that we could have used the ergodic decomposition instead of the extreme Gibbs decomposition to prove this theorem. The upshot of using the extreme Gibbs decomposition is that conditional on a choice of extreme Gibbs component $\lambda_c$ on each slab $L_c$, the samples $\tau_c$ from $\lambda_c$ for all $c\in \m Z$ are independent. Conditioned on a choice of ergodic component $\eta_c$ on each slab $L_c$, only the samples from $\eta_c$ with $s(\eta_c)=(l_1^c,l_2^c,l_3^c)$ and $l_1^c,l_2^c,l_3^c>0$ are necessarily independent \cite[Theorem 9.1.1]{AST_2005__304__R1_0}.
\end{rem}
\begin{proof}[Proof of Proposition \ref{prop: extreme mc entropy}] 
Since $\mu$ is Gibbs, it has an extreme Gibbs decomposition 
\begin{align*}
    \mu = \int \nu \, \dd g_\mu(\nu). 
\end{align*}
Sampling a tiling $\tau$ from $\mu$ is equivalent to sampling an extreme Gibbs component $\nu$ of $\mu$ (from $g_\mu$), and then sampling a tiling $\tau$ from $\nu$. Since $\nu$ is tail-trivial, its marginal $\lambda_c$ on $L_c$ is also tail-trivial. For all $c\in \m Z$, the extreme Gibbs decomposition of the marginal $\rho_c$ can be written as
\begin{align*}
    \rho_c = \int \lambda_c \, \dd g_{\rho_c}(\lambda_c)
\end{align*}
where $g_{\rho_c}$ is the extreme Gibbs decomposition of $\rho_c$.

Let $B_n = [-n,n]^3$. Since $\lambda_0$ is extreme Gibbs it is tail trivial, so
    $$\lim_{m\to \infty}\mu(\tau_{B_n\cap L_0}~|~\tau_{\Z^3\setminus (B_m\cap L_0)},\lambda_0)=\lambda_0(\tau_{B_n\cap L_0}).$$
    Therefore conditioned on $\lambda_0$, $\tau_0 = \tau_{L_0}$ is independent of $\tau_{(\m Z^3\setminus L_0)}$. In particular, conditioned on the sequence of measures $(\lambda_c)_{c\in \m Z}$ (equivalently, conditioned on choosing an extreme Gibbs component of $\mu$), the tilings on each slab $(\tau_c)_{c\in \m Z}$ are independent.

Now we relate the specific entropies. Recall from Section \ref{sec: entropy prelim} that for a region $\Delta\subset \m Z^3$ and an invariant measure $\mu$ on tilings of $\m Z^3$,
\begin{align*}
    H_{\Delta}(\mu) = -\sum_{\sigma\in \Omega(\Delta)} \mu(X(\sigma))\log \mu(X(\sigma)),
\end{align*}
where $\Omega(\Delta)$ is the free-boundary tilings of $\Delta$, and $X(\sigma)$ is the collection of tilings of $\m Z^3$ which extend $\sigma$.
Taking $A_n(0) = B_n\cap L_0$, let $A_n(c) = A_n(0) + (0,0,2c)$, and finally let $A_{n,m} = \cup_{c=-m}^m A_n(c)$.\symindex{Chapter 4!$A_n(c),A_{n,n}$}
 It is well known that the specific entropy can be computed as
\begin{align*}
    h(\mu)= \lim_{n\to \infty} |A_{n,n}|^{-1} H_{A_{n,n}}(\mu). 
\end{align*}
Instead of free-boundary tilings, we can choose $\tau\in \Omega$ and let $\Omega_\tau(\Delta)=\{\sigma \in \Omega(\Delta): \sigma\mid_{\partial \Delta} =\tau\}$ be the tilings of $\Delta$ with boundary condition agreeing with $\tau$. Then we define the entropy of $\mu$ given a fixed boundary condition $\tau$:
\begin{align*}
    H_{\Delta}(\mu|\tau) = -\sum_{\sigma \in \Omega_\tau(\Delta)} \mu(\sigma~|~\tau_{\m Z^3\setminus \Delta}) \log \mu(\sigma~|~\tau_{\m Z^3\setminus \Delta}). 
\end{align*}
Again for $A\subset \m Z^3$, $\tau_A$ means $\tau$ restricted to $A$. We remark that this is not the usual definition of conditional entropy where we condition on a random variable or a sigma algebra. Instead we are fixing the value of the random variable, namely, the boundary condition of the tiling in $\Delta$. Indeed, if $\tau$ is generic for an extreme Gibbs measure component $\nu$ of $\mu$ then we have that $H_{\Delta}(\mu|\tau)=H_{\Delta}(\nu|\tau)$. We will restrict this non-standard usage to this proof. It is standard that the specific entropy of $\mu$ can also be computed using this conditional definition as 
\begin{equation}\label{eq: ent boundary fixed}
    h(\mu) = \lim_{n\to \infty} |A_{n,n}|^{-1} (\int_{\Omega} H_{A_{n,n}}(\mu|\tau) \, \dd \mu(\tau) +H_{\partial A_{n,n}}(\mu))= \lim_{n\to \infty} |A_{n,n}|^{-1} \int_{\Omega} H_{A_{n,n}}(\mu|\tau) \, \dd \mu(\tau). 
\end{equation}
In the second equality, we use that the entropy term for $\partial A_{n,n}$ is of order $n^2$ so it does not contribute in the limit. Now we rewrite the argument of the limit using the extreme Gibbs decomposition. 
\begin{align*}
     \int_{\Omega} H_{A_{n,n}}(\mu|\tau) \, \dd \mu(\tau)  = \int \int_{\Omega} H_{A_{n,n}}(\nu|\tau)\, \dd \nu(\tau) \, \dd g_\mu(\nu). 
\end{align*}
Recall that $\tau_c = \tau_{L_c}$. Conditional on sampling the process $(\lambda_c)_{c\in \m Z}$ (equivalently, conditional on sampling $\nu$), the samples $(\tau_c)_{c\in \m Z}$ are independent. Thus 
\begin{align*}
    H_{A_{n,n}}(\nu|\tau) = \sum_{c=-n}^n H_{A_{n}(c)}(\lambda_c|\tau_c). 
\end{align*}
Recall that $\Omega_{\text{loz}}$ is the set of full-plane lozenge tilings. Thus
\begin{align}\label{eq: ent independence split}
     \int \int_{\Omega} H_{A_{n,n}}(\nu|\tau)\, \dd \nu(\tau) \, \dd g_\mu(\nu)  &= \sum_{c=-n}^n \int \int_{\Omega_{\text{loz}}} H_{A_{n}(c)}(\lambda_c|\tau_c) \, \dd \lambda_c(\tau_c)\,\dd g_{\rho_c}(\lambda_c) \\ &= \sum_{c=-n}^n \int_{\Omega_{\text{loz}}} H_{A_{n}(c)}(\rho_c|\tau_c) \, \dd {\rho_c}(\tau_c).
\end{align}
Since $\rho_c$ is equal in distribution to $\rho_0$, for all $c\in \m Z$, 
\begin{align*}
    |A_n(c)|^{-1}\int_{\Omega_{\text{loz}}} H_{A_{n}(c)}(\rho_c|\tau_c)\, \dd {\rho_c}(\tau_c) = |A_n(0)|^{-1} \int H_{A_{n}(0)}(\rho_0|\tau_0)\, \dd {\rho_0}(\tau_0).
\end{align*}
At the same time, 
\begin{equation}\label{eq: ent rho}
    \lim_{n\to \infty} |A_n(0)|^{-1} \int_{\Omega_{\text{loz}}} H_{A_{n}(0)}(\rho_0|\tau_0)\, \dd {\rho_0}(\tau_0) =  h(\rho_0). 
\end{equation}
Therefore 
\begin{align*}
    h(\mu) &= \lim_{n\to \infty} |A_{n,n}|^{-1} \int_{\Omega} H_{A_{n,n}}(\mu|\tau)\,\dd \mu(\tau) \\ &= \lim_{n\to \infty} \frac{1}{2n+1} (2n+1) |A_n(0)|^{-1} \int_{\Omega_{\text{loz}}} H_{A_{n}(0)}(\rho_0|\tau_0)\, \dd {\rho_0}(\tau_0) = h(\rho_0). 
\end{align*}
\end{proof}

Recall from Section \ref{sec: entropy prelim} that the \textit{mean-current entropy function} $\ent: \mc O\to [0,\infty)$ is defined by 
\begin{align*}
    \ent(s) = \max_{\mu\in \mc P^s} h(\mu). 
\end{align*}
This entropy function plays a central role in our work and will be studied extensively in Section  \ref{sec:entropy}. 

Let $\mc T_2 = \{ s_1,s_2,s_3\geq 0 : s_1 + s_2+s_3 = 1\}$ be the space of possible lozenge tiling slopes. The \textit{slope entropy function} $\ent_{\text{loz}}:\mc T_2\to [0,\infty)$ for lozenge tilings is defined by 
\begin{align*}
    \ent_{\text{loz}}(s) = \max_{\rho\in \mc P^s_{\text{loz}}} h(\rho). 
\end{align*}
It was shown in \cite[Theorem 9.2]{cohn2001variational} that $\ent_{\text{loz}}$ \symindex{Chapter 4!$ \ent_{\text{loz}}$ - the entropy function for ergodic Gibbs measures of a given slope on the space of lozenge tilings}
has the explicit form
\begin{align*}
    \ent_{\text{loz}}(s_1,s_2,s_3) = \frac{1}{\pi}(L(\pi s_1)+ L(\pi s_2)+ L(\pi s_3))
\end{align*}
where $L:[0,\pi]\to \m R$ is the Lobachevsky function given by 
\begin{align*}
    L(\theta)= -\int_0^\theta \ln (2\sin(x))dx.
\end{align*}
Using this two dimensional result, we can explicitly compute $\ent$ on $\partial \mc O$. Let $\mc E\subset \partial \mc O$ denote the edges of $\mc O$. 
\begin{thm}\label{thm: extremal_entropy}
    For $s=(s_1,s_2,s_3)\in \partial \mc O$, 
    \begin{align*}
        \ent(s) = \ent_{\text{loz}}(|s_1|,|s_2|,|s_3|) = \frac{1}{\pi}(L(\pi |s_1|)+ L(\pi |s_2|)+ L(\pi |s_3|)).
    \end{align*}
    Further, if $s\not\in \mc E$, then any measure $\mu$ realizing $h(\mu) = \ent(s)$ is an ergodic Gibbs measure on $\Omega$ with respect to the $\threeeven$ action. If $s\in \mc E$, then $\ent(s) = 0$.
\end{thm}
It is well-known that $\ent_{\text{loz}}$ is strictly concave as a function of slope on the interior of allowed slopes \cite[Theorem 10.1]{cohn2001variational}. Thus as an immediate corollary, we get that 
\begin{cor}\label{cor: ent boundary}
	Let $\mc F$ be any face of $\partial \mc O$. The entropy function $\ent(\cdot)$ is strictly concave on the interior of $\mc F$.
\end{cor}
\begin{proof}[Proof of Theorem \ref{thm: extremal_entropy}]

By Theorem \ref{theorem: entropy maximizer gibbs}, if $\mu\in \mc P^s$ satisfies $h(\mu) = \ent(s)$, then $\mu$ is a Gibbs measure. While we include this result later in the paper for organizational reasons, it follows easily from the classical variational principle for Gibbs measures \cite{LanfordRuelle} (the only adaptation is that we are looking at the maximizer with a fixed mean current).

Without loss of generality assume that $s_1,s_2,s_3\geq 0$. First suppose that $\mu\in \mc P^s$ is an EGM, and as usual let $\rho_c$ denote its marginal on $L_c$. By Proposition \ref{prop: extreme mc gibbs slope} and Proposition \ref{prop: extreme mc entropy}, 
\begin{align*}
    s = s(\mu) = s(\rho_0) \qquad \text{ and }\qquad
    h(\mu) = h(\rho_0).
\end{align*}
Combining the results of \cite{cohn2001variational} and \cite{AST_2005__304__R1_0},
\begin{itemize}
    \item If $s$ has $s_1,s_2,s_3>0$, then $\rho_0 \in \mc P_{\text{loz}}^s$ satisfies $h(\rho_0) = \ent_{\text{loz}}(s)$ if and only if $\rho_0$ is the unique ergodic Gibbs measure of slope $s$, which we denote by $\lambda_s$. 
    \item If $s$ has $s_i=0$ for some $i=1,2,3$, then $h(\rho_0) = \ent_{\text{loz}}(s) = 0$ for all $\rho_0\in \mc P_{\text{loz}}^s$. 
\end{itemize}

Therefore if $s_1,s_2,s_3>0$ and $h(\rho_0) = \ent_{\text{loz}}(s)$, then by strict concavity of $\ent_{\text{loz}}$ \cite[Theorem 10.1]{cohn2001variational}, $\rho_c$ are  identically distributed and equal to $\lambda_s$ a.s. By \cite[Theorem 9.1.1]{AST_2005__304__R1_0}, when $s_1,s_2,s_3>0$, the unique ergodic Gibbs measure $\lambda_s$ is an extreme Gibbs measure, and thus $\{\rho_c\}_{c\in \m Z}$ is i.i.d.\ by Proposition \ref{prop: extreme mc entropy}. Alternatively if $s_i=0$ for some $i$, then $h(\rho_0) = 0$, and hence $h(\mu) = 0$. 

If $\mu$ is not ergodic with respect to the $\threeeven$ action, then it can be decomposed 
\begin{align*}
    \mu = \int_{\mc P_e} \nu \, \dd w_\mu(\nu),
\end{align*}
where 
\begin{align*}
    s(\mu) = \int_{\mc P_e} s(\nu) \, \dd w_\mu(\nu).
\end{align*}
Note that $w_\mu$ almost surely, $s(\nu)$ is contained in the same face of $\partial \mc O$ as $s=s(\mu)$. By the analysis above for an ergodic measure, if $s(\nu)\not\in \mc E$ then $h(\nu) = \ent(s(\nu)) = \ent_{\text{loz}}(s(\nu))$ if and only if $\nu$ is an EGM of mean current $s(\nu)$ with marginals on each slab i.i.d.\ and equal to the unique lozenge tiling EGM $\lambda_{s(\nu)}$ of slope $s(\nu)$. If $s(\nu)\in \mc E$, then $h(\nu) =0$. Since $\ent_{\text{loz}}$ is strictly concave on the interior of allowed slopes, if $s$ is contained in the interior of a face of $\partial \mc O$, then $h(\mu) = \ent(s)$ if and only if $\mu$ is an ergodic Gibbs measure of mean current $s$. 
\end{proof}

As seen in the proof of Theorem \ref{thm: extremal_entropy}, we get an explicit description of the entropy maximizers for $s\in \partial \mc O$. In contrast to two dimensions, the maximum entropy is positive for mean currents in the interior of the faces of $\partial \mc O$.
\begin{cor} \label{Corollary: entropy maximiser}
    Suppose $s=(s_1,s_2,s_3)\in \partial \mc O$. 
    \begin{itemize}
        \item If $s\in \mc E$ (i.e.\ $s_i=0$ for some $i=1,2,3$), then $h(\mu) = 0$ for any $\mu \in \mc P^s$.
        \item If $s_1s_2s_3\neq 0$, then the entropy maximizer in $\mc P^s$ is an ergodic Gibbs measure such that for all $c\in \m Z$, $\rho_c=\lambda_s$ a.s., where $\lambda_s$ is the unique ergodic Gibbs measure on lozenge tilings with slope $(|s_1|,|s_2|,|s_3|)$. Here $\rho_c$ is the marginal on the slab $\{(x_1,x_2,x_3 : \epsilon_1 x_1 + \epsilon_2 x_2 + \epsilon_3 x_3 = 2c \text{ or }2c+1\}$, where $\epsilon_i$ is the sign of $s_i$.
    \end{itemize}
\end{cor}

It is also straightforward to show that there exist EGMs of a fixed boundary mean current with a range of different entropies. 
\begin{prop}\label{prop: same current different entropy}
	Suppose $s=(s_1,s_2,s_3)\in \partial \mc O$, $s_1,s_2,s_3>0$. Then for all $0\leq \theta\leq 1$ there is an ergodic Gibbs measure $\mu$ such that $h(\mu)=\theta\;\ent(s)$.
\end{prop}

\begin{proof}
Let $\rho_{\max}, \rho_1, \rho_2, \rho_3\in \Prob_{\text{loz}}$ be EGMs on lozenge tilings of slopes $(s_1,s_2,s_3)$, $(1,0,0)$, $(0,1,0)$, $(0,0,1)$ respectively. Now consider an i.i.d.\ process $(\eta_c)_{c\in \Z}$ with state space 
$$\{\rho_{\max}, \rho_1, \rho_2, \rho_3\}$$
such that the probability of $\rho_{\max}$ is $\theta$, and the probability of $\rho_i$ is $(1-\theta) s_i$ for $i=1,2,3$. 

Let $\mu$ be a measure on $\Omega$ given by taking a sample from $(\eta_c)_{c\in \m Z}$, this gives a tiling of $\m Z^3$ such that the restriction to each slab $L_c$ is an independent sample from $\eta_c$. Clearly $\mu$ is a Gibbs measure on $\Omega$. Since $(\eta_c)_{c\in \m Z}$ is an i.i.d. process it is ergodic so $\mu$ is ergodic with respect to $\threeeven$. By Proposition \ref{prop: extreme mc gibbs slope} $s(\mu) = s$ and by Proposition \ref{prop: extreme mc entropy} $h(\mu) = \theta \, \ent(s)$. 
\end{proof}

We now summarize the results from this section to illustrate the similarities and differences with the two dimensional case.

\begin{itemize}
    \item In three dimensions, EGMs of the same boundary mean current $s$ can have different specific entropy values (Proposition \ref{prop: same current different entropy}).
 	\item Every EGM $\mu$ on dimer tilings that contains only east, north, and up  tiles gives rise to a Gibbs measure on integer-indexed stationary sequences of extreme Gibbs measures $(\lambda_c)_{c\in \Z}$ on lozenge tilings (Proposition \ref{prop: extreme mc entropy}). 
 	\item If $s=(s_1, s_2, s_3)\in \partial \mc O$ is such that $s_1, s_2, s_3\neq 0$ then the entropy-maximizing measure with mean current $s$ is an EGM such that $(\lambda_c)_{c\in \m Z}$ is an i.i.d. sequence of copies of the unique EGM on lozenge tilings with slope $(|s_1|,|s_2|,|s_3|)$ (Corollary \ref{Corollary: entropy maximiser}).
\end{itemize}

\section{Patching}\label{sec:patching plus hall}

The main goal of this section is to prove a \textit{patching theorem} (Theorem \ref{patching}) which will be an essential tool throughout this paper. We show that if the flows associated to tilings $\tau_1,\tau_2$ of $\m Z^3$ are \textit{nearly-constant} (Definition \ref{def:nearly constant} below) with value $s\in \text{Int}(\mc O)$ (which loosely speaking means the flows associated with $\tau_1,\tau_2$ both approximate the constant flow equal to $s$), then we can remove a bounded piece from $\tau_1$, and patch it to $\tau_2$ by tiling a thin (cubic) annulus. 
\begin{figure}[H]
    \centering
   \includegraphics[scale=0.6]{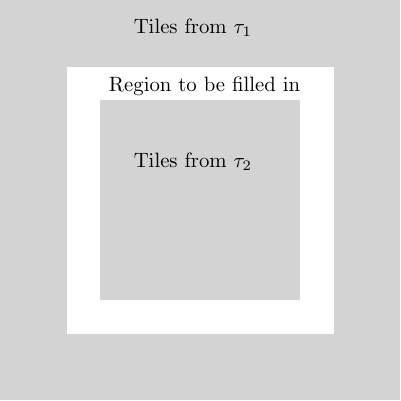}
    \caption{Two dimensional schematic for patching.}
    \label{fig:scheme}
\end{figure}
Equivalently, we want to show that this annular region can be tiled by dimers exactly so that it agrees with $\tau_1$ on one boundary and with $\tau_2$ on the other boundary. To do this, we need a condition to show that a region is tileable. 

A general condition for tileability, which works in any dimension, is given by the classical \textit{Hall's matching theorem} (\cite{hallmarriage}, stated here as Theorem \ref{thm:halls matching}), which says a region $R$ is tileable if and only if every $U \subset R$ of a particular form has certain properties; a $U$ that does not have such properties is a ``counterexample'' to the condition that implies tileability, as we explain in Section \ref{sec: hall}. In two dimensions, Hall's matching theorem implies that a simple condition on height function differences along the boundary of the region is equivalent to tileability \cite{Fournier}. 
In three dimensions, we show in Section \ref{sec: minimal} that it is sufficient to show that the region $R$ has no counterexample set $U$ whose boundary is a certain type of \textit{minimal surface}, built out of squares from the $\m Z^3$ lattice (Corollary \ref{cor:minimal_tileable}). We call surfaces built out of lattice squares \textit{discrete surfaces}. 

In Section \ref{sec: ideas for patching} we give the statement of the patching theorem, and explain the main ideas of the proof accompanied by a series of two dimensional figures. 

The main new difficulty in higher dimensions is that the counterexample sets $U$ can have more complicated geometry. In two dimensions, the boundary of the counterexample region is a union of curves. In three dimensions it is a union of surfaces. However the fact that these surfaces can be assumed to be in some sense \textit{minimal} gives us some control on their geometry. In Section \ref{sec: isoperimetric}, we prove some straightforward adaptations of the isoperimetric inequality for discrete surfaces. In Section \ref{sec: area growth}, we apply these to get useful bounds on the area growth for minimal discrete surface (Proposition \ref{prop:quadratic growth}), and show that they ``spread out" (Lemma \ref{lem: indenting}).

In Section \ref{sec: ergodic}, we prove an ergodic theorem for the flow of a tiling through a coordinate plane (Theorem \ref{thm: subaction ergodic}), and note that tilings sampled from ergodic measures satisfy the conditions of the patching theorem with probability going to $1$ as $n\to \infty$ (Corollary \ref{cor: ergodic implies nearly constant}). We show that ergodic measures of any mean current $s\in \mc O$ exist (Lemma \ref{lem: ergodic measures exist}), and prove some bounds for their expected flow through discrete surfaces (Lemma \ref{lem: flow bound}). One of the ideas in the proof of the patching theorem is to use a tiling sampled from an ergodic measure as a ``measuring stick" that we compare with the tilings we want to patch. 

Equipped with the lemmas from the previous sections, in Section \ref{sec: patching proof} we give the proof of the patching theorem (Theorem \ref{patching}). Finally in Section \ref{sec: patching ergodic} we give some immediate corollaries of patching for ergodic Gibbs measures (EGMs). 

We use the same tools and ideas developed in this section again in Section \ref{sec:generalized_patching} to prove a \textit{generalized patching theorem} (Theorem \ref{thm:generalized_patching}) where the flow the tilings approximate is not required to be constant, and where the annular region is allowed to have a more general shape. The main results of this paper are two versions of a large deviation principle: one with \textit{soft} boundary conditions (Theorem \ref{thm:sb-ldp}) and one with \textit{hard} boundary conditions (Theorem \ref{thm:hb-ldp}). The regular patching theorem proved here (Theorem \ref{patching}) is sufficient to prove the LDP with soft boundary conditions, but the generalized version (Theorem \ref{thm:generalized_patching}) is needed in the final steps to prove the version with hard boundary conditions. 

\subsection{Hall's matching theorem and non-tileability}\label{sec: hall}

When can a finite region $R\subset \mathbb{Z}^3$ be exactly tiled by dimers, i.e.\ without any tiles crossing the boundary, and with all cubes covered? This is equivalent to asking: when does a finite subgraph $G\subset \mathbb{Z}^3$ have a perfect matching? A straightforward observation is that for any bipartite graph $G$ with bipartition $(\mc A,\mc B)$, a necessary condition for $G$ to have a perfect matching is that it is \textit{balanced}\termindex{Chapter 6!balanced}, i.e.\ that $|\mc A| = |\mc B|$. The balanced condition is not sufficient though, see Figure \ref{fig:region_must_be_balanced}. 
\begin{figure}
    \centering
    \includegraphics[scale=0.4]{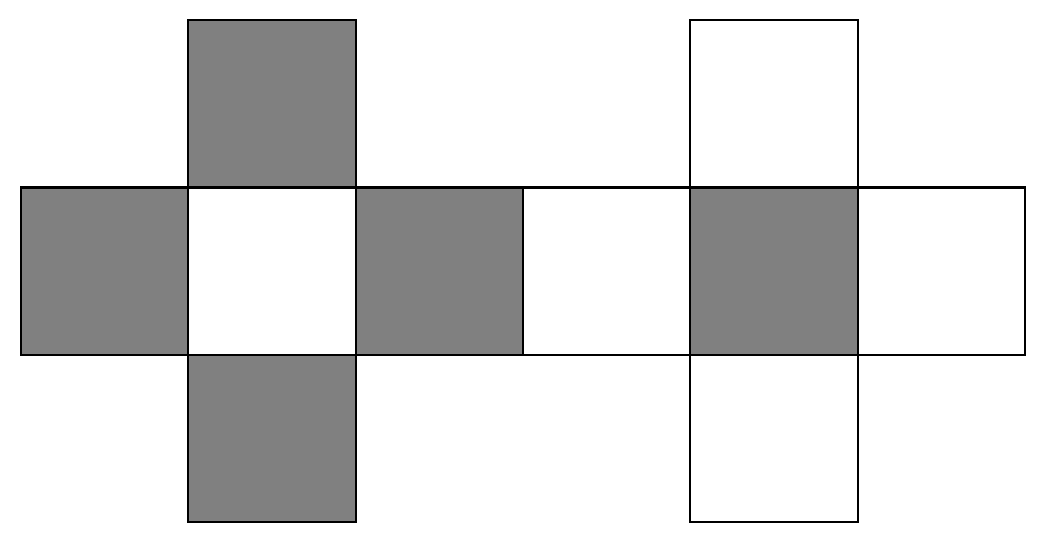}
    \caption{A region that is balanced (i.e.\ the number of black squares is equal to the number of white squares) but has no dimer tilings.} 
    \label{fig:region_must_be_balanced}
\end{figure}
Nonetheless, it turns out there is a very general necessary and sufficient condition which characterizes whether or not a finite bipartite graph $G$ has a perfect matching. In this section we explain these results from the graph point of view where dimers are edges and dimer tilings are perfect matchings. There are two perspectives, both of which can be useful: 
\begin{itemize}
    \item the min-cut, max-flow principle 
    \item Hall's matching theorem
\end{itemize}
We first describe the classical min-cut, max-flow principle. Let $G = (\mc A,\mc B)$ be the bipartition of the graph ($\mc A$ is ``even" and $\mc B$ is ``odd"). If $G$ has a perfect matching $\tau$, then there is a flow $v_\tau$ (the ``pretiling flow") which sends a unit of current from each even vertex $a\in \mc A$ to the odd vertex $b\in \mc B$ it is paired to. Note that $v_\tau$ has a source of $+1$ at each $a \in \mc A$ and a sink of $+1$ (or source of $-1$) at each $b \in \mc B$.  The existence of a perfect matching is equivalent to the existence of a flow $v_\tau$ with the desired source/sink values and a flow of $0$ or $1$ on each even-to-odd edge.

\termindex{Chapter 6!a cut}A \textit{cut} is a collection of edges in $G$ which, if deleted, separates $G$ into two pieces $G_1$ and $G_2$. Let $F_1$ be the net total flow sourced in $G_1$ (i.e.\ the number of even vertices in $G_1$ minus the number of odd vertices), let $F_2$ be the net total flow sourced in $G_2$, and let $c$ be the number of cut edges. If $G$ has a perfect matching, we must have $F_1 + F_2 = 0$. 

The value $F_1$ measures the amount of flow that would have to travel across the cut if $G$ has a perfect matching, so we must have $F_1 \leq c$. In other words, if $G$ has a perfect matching, then it must be the case that for \textit{every} cut, the excess flow on either side must be less than the size of the cut. It turns out that this is a sufficient condition too, so if $G$ does \textit{not} have a perfect matching then there is a cut of $c$ edges partitioning $G$ into $G_1$ and $G_2$ such that the excess flow $F_1$ that needs to cross the cut is more than $c$. In summary:

\begin{thm}[Min-cut, max-flow principle \cite{ford_fulkerson_1956}]\label{mincut}
A finite bipartite graph $G$ has a perfect matching (a.k.a.\ dimer tiling) if and only if there is no cut consisting of $c$ edges partitioning $G$ into two sets $G_1$ and $G_2$ such that $F_1 > c$. 
\end{thm}

In Hall's matching theorem, we shift our perspective from the cut to the sets in the partition. Instead of looking at sets that are a mixture of even and odd vertices, we consider a set $C$ of only even (resp.\ only odd) vertices, plus their neighbors 
$$N(C) = \{ b\in \mc B :  a\in C, (a,b) \in E\}$$
which are only odd (resp.\ only even). The set $C \cup N(C)$ is analogous to either $G_1$ or $G_2$ (without loss of generality $G_1$) plus some of the endpoints of the edges in the cut. The excess flow is now the number of even vertices (i.e.\ $|C|$) minus the number of odd vertices (i.e.\ $|N(C)|$).\symindex{Chapter 6!set $C$, its neighbors $N(C)$}  Hall's matching theorem is an analog of Theorem \ref{mincut} formulated in these terms:
\begin{thm}[Hall's matching theorem \cite{hallmarriage}]\label{thm:halls matching}
Suppose that $G = (V,E)$ is a finite bipartite graph with bipartition $G = \mc A\cup \mc B$. The graph $G$ admits a perfect matching consisting of $|\mc A|$ edges if and only if for all $C \subset \mc A$, 
\begin{align*}
    N(C) = \{ b\in \mc B : a\in C, (a,b) \in E\}
\end{align*}
satisfies $|N(C)| \geq |C|$.
\end{thm}
An analogous result holds with $\mc A$ and $\mc B$ switched. If $G$ is balanced (i.e.\ $|\mc A| = |\mc B|$), then the existence of a perfect matching with $|\mc A|$ edges is equivalent to the existence of a perfect matching of the whole graph $G$. If $G$ is not balanced (i.e.\ $|\mc A|\neq |\mc B|$) then $G$ does not have a perfect matching of the whole graph.

Note that if $C \subset \mc A$ satisfies $|N(C)| < |C|$, then the set $U:= C \cup N(C)$ has more even than odd vertices, despite having only odd vertices on its boundary within $G$. Therefore when $G$ is balanced, Theorem \ref{thm:halls matching} is equivalent to the following:
\begin{thm}\label{thm:halls matching 2}
Suppose that $G = (V,E)$ is a finite bipartite graph with bipartition $G = \mc A\cup \mc B$ and $|\mc A| = |\mc B|$. Then $G$ {\em fails to have} a perfect matching if and only if there exists a connected set $U \subset V$ such that $|U\cap \mc A| > |U \cap \mc B|$ but all boundary elements of $U$ (i.e., elements of $U$ that are adjacent to some point in $V \setminus U$) belong to $\mc B$.
\end{thm}

\termindex{Chapter 6!imbalance}\symindex{Chapter 6!imbalance}\termindex{Chapter 6!counterexample to tileability}We call the $U$ from Theorem~\ref{thm:halls matching 2} a \textit{counterexample to (the condition equivalent to) tileability} or just a \textit{counterexample}. In our context, $\mc A$ and $\mc B$ will always be sets of even and odd vertices in $\mathbb Z^3$, so for us a {\em counterexample} to tileability for $R \subset \m Z^3$ is any set $U$ that has more even than odd vertices, despite having only odd vertices on its {\em interior boundary}, which we define to be the set of $x \in U$ that are incident to some $y \in R \setminus U$. 

To show that a graph $R$ has a dimer tiling (a.k.a.\ a perfect matching), we check that it is balanced, and if it is, we have to show that there are no counterexamples. We call the excess flow of a counterexample its \textit{imbalance}, given by 
\begin{equation}\label{eq:imbalance}
	\text{imbalance}(U) = \text{even}(U) -\text{odd}(U).
\end{equation}
Note that if $U \subset R\subset \m Z^3$ has only odd vertices on its interior boundary (within $R$) then $\text{imbalance}(U)>0$ if and only if $U$ is a counterexample for $R$.

\subsection{Discrete surfaces and minimal counterexamples}\label{sec: minimal}

As mentioned earlier, it is often intuitively useful to think about perfect matchings as tilings by $2 \times 1 \times 1$ blocks. In this picture, a counterexample set $U$ is a collection of unit cubes, each centered at a point in $\mathbb Z^3$, and the edges out of it, i.e.\ its boundary $\partial U$, is a collection of squares in the translated lattice $(\frac12, \frac12, \frac12) + \mathbb{Z}^3$. In other words the boundary region is a surface built out of squares from the lattice. 

\begin{definition}\termindex{Chapter 6!discrete surface}
	A \textit{discrete surface} in $\mathbb{Z}^3$ is a collection of squares from the $(\frac12,\frac12,\frac12) + \m Z^3$ lattice. 
\end{definition}
A discrete surface in $\m Z^3$ is \textit{orientable} if there is a well-defined outward pointing normal vector to the surface. An orientable discrete surface $S$ with a choice of outward pointing normal vector is called \textit{oriented}. For a square $s\subset S$, we call the side of $s$ that the outward normal vector points toward the \textit{outside}. If the outward pointing normal vector to a square in an oriented surface is from even to odd, we color the outside of the square white. Otherwise we color it black.\termindex{Chapter 6!black and white surface}
\begin{definition}
	An oriented discrete surface $S$ in $\mathbb{Z}^3$ is \textit{monochromatic} if all the outsides of all the squares in $S$ are black (resp.\ all are white). \termindex{Chapter 6!monochromatic surface}
\end{definition}
We can rewrite Equation \eqref{eq:imbalance} for the imbalance of a counterexample $U$ in terms of the black and white surface area of $\partial U$. Let $(\mc A,\mc B)$ be the bipartition of $\m Z^3$ into even and odd vertices respectively. 
\begin{prop}\label{prop: white black imbalance}
	Suppose that $R$ is balanced but not tileable, and that $U\subset R\subset \m Z^3$ is a counterexample to tileability. Then 
	\begin{equation}\label{eq:imbalance_area}
		0 < \text{imbalance}(U) = \text{even}(U) - \text{odd}(U) = \frac{1}{6} \bigg( \text{white}(\partial U) - \text{black}(\partial U) \bigg).
	\end{equation}
\end{prop}
\begin{proof}
	 Define a flow $r$ on $\mathbb{Z}^3$ such that $r({e}) = \frac{1}{6}$ for every dual edge (a.k.a.\ face) ${e}$ oriented from even to odd. Then
\begin{align*}
    \text{div}\, r(v) = \begin{cases} -1 \quad &v \text{ is a odd cube} \\ +1 \quad &v \text{ is a even cube}\end{cases} 
\end{align*}
By the divergence theorem, with all edges $e\in \partial U$ oriented out of $U$, 
\begin{align*}
    \text{imbalance}(U) &= \text{even}(U) - \text{odd}(U) = \sum_{v\in U} \text{div}\, r(v) = \sum_{e\in \partial U} r({e}) \\
    &= \frac{1}{6} \bigg(\text{white}(\partial U) - \text{black}(\partial U)\bigg).
\end{align*}
\end{proof}
By Proposition \ref{prop: white black imbalance}, if $U$ is a counterexample then it must have more white surface area than black surface area. By the definitions in Section \ref{sec: hall}, $U$ must have only odd cubes along its \textit{interior boundary}, i.e.\ cubes $x\in U$ which are adjacent to $y\in R\setminus U$. However $U$ also has an \textit{exterior boundary} consisting of cubes which are adjacent to $y\in \m Z^3\setminus R$. Exterior boundary cubes can be even or odd. 

Correspondingly, the boundary $\partial U$ can be split into two pieces: the \textit{exterior boundary surface} $T = \partial R\cap \partial U$ and the \textit{interior boundary surface} $S = \partial U\setminus T$. The interior boundary surface $S$ must be built out of only black squares, while $T$ can be built out of a mixture of white and black squares. Given this, only squares in $T$ contribute positively to the imbalance of $U$. Intuitively to increase the imbalance of $U$, one should \textit{minimize} the area of $S$. 

A surface $P$ embedded in $\m R^3$ is said to \textit{locally minimize area} if given any point $p\in P$, there is a neighborhood $V\subset \m R^3$ containing $p$ such that $P\cap V$ has the minimal area of any surface with boundary $\partial (P\cap V)$. Surfaces that locally minimize area are called \textit{minimal surfaces}. We will be interested in certain discrete analogs of minimal surfaces.

\begin{definition}
A \textit{minimal discrete (monochromatic) surface with boundary $X$} is a surface $S$ that minimizes area subject to the constraint that it is discrete, (monochromatic), and has $\partial S = X$. In particular, there is no way to ``tighten the surface locally'' by changing a few faces in a way that maintains the overall boundary conditions and reduces the overall area.\termindex{Chapter 6!minimal (monochromatic) surface}
\end{definition}

\begin{prop}\label{prop:minimalsurface}
Let $R\subset \mathbb{Z}^3$ be a finite balanced region which is not tileable, and suppose that $U$ is a counterexample to tileability in $R$. Let $T = \partial U \cap \partial R$ and let $S = \partial U\setminus T$ be the interior boundary surface. Then there exists another counterexample $U'$ in $R$ such that $\partial U' \cap \partial R = T$, and $S':= \partial U' \setminus T$ is a minimal monochromatic black discrete surface.
\end{prop}
\begin{proof}
	The new set $U'$ is defined so that $\partial U' = T \cup S'$, where $S'$ is a minimal monochromatic discrete surface. Since $S$ is all black, by Proposition \ref{prop: white black imbalance},
	\begin{equation}\label{eq:imbalance_area_2}
		6\cdot \text{imbalance}(U) =  \text{white}(\partial U) - \text{black}(\partial U) = \text{white}(T) - \text{black}(T) - \text{area}(S)
	\end{equation}
	replacing $S$ by $S'$ only makes the imbalance larger, so $U'$ is still a counterexample.
\end{proof}

We call counterexamples where the internal boundary surface is a minimal surface \textit{minimal counterexamples}. We immediately get the following corollary.
\begin{cor}\label{cor:minimal_tileable}
    A finite balanced region $R\subset \m Z^3$ is tileable if and only if it has no minimal counterexamples. 
\end{cor}

\subsection{Statement of patching theorem and outline of proof}\label{sec: ideas for patching}

We now give the statement of the patching theorem mentioned at the beginning of the section. We also provide illustrations of the analogous constructions in 2D (because they are easier to draw) in order to explain the 3D concepts. Let $B_n=[-n,n]^3$, and for any $\delta>0$ define \symindex{Chapter 6!$B_n$ - the box $[-n,n]^3$ and the annulus $A_n=B_n \setminus B_{(1-\delta)n}$}
\begin{align*}
    A_n = B_n \setminus B_{(1-\delta)n}.
\end{align*}
Given two tilings $\tau_1,\tau_2\in \Omega$, we look at the region between $\tau_1 \mid_{\m Z^3\setminus B_n}$ and $\tau_2\mid_{B_{(1-\delta)n}}$. This will be the annulus $A_n$, with some cubes removed along its boundary (see Figure \ref{fig:squareannulus}). We call this \textit{$A_n$ with boundary conditions $\tau_1$ and $\tau_2$.}

\begin{figure}[H]
    \centering
        \includegraphics[scale=0.39]{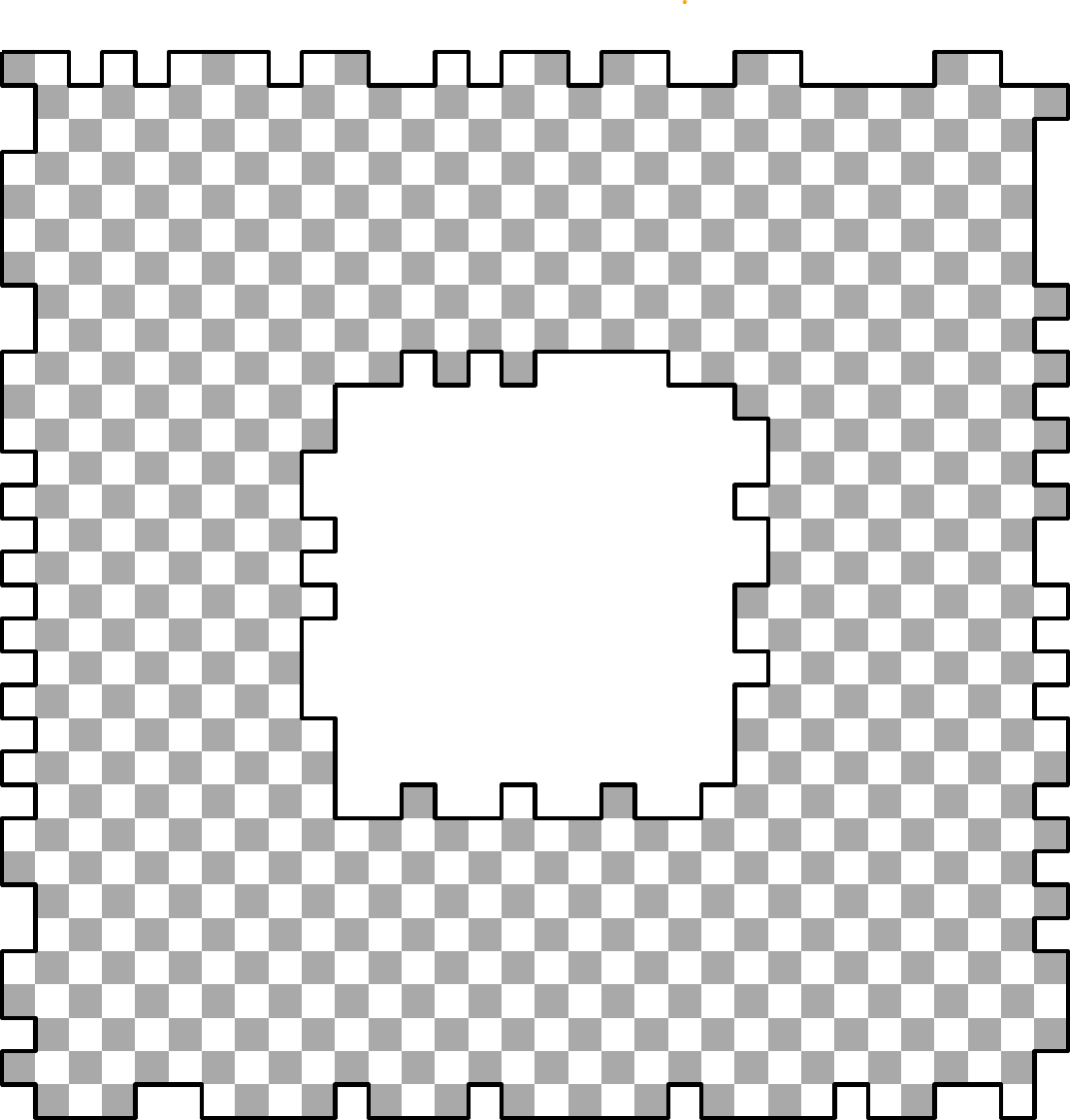}\\
    \caption{Here is an example of a region of the form ``$A_n$ with boundary conditions $\tau_1$ and $\tau_2$" in two dimensions, which we call $A$. The region shown is a subset of the ``square annulus'' bounded between the boundary of a $12 \times 12$ box and the boundary of a $32 \times 32$ box. It is obtained from the square annulus by removing some of the squares along the outer boundary and some of the squares along the inner boundary. Given any dimer tiling $\tau$ of a region containing the $32 \times 32$ box, the union of the dimers of $\tau$ that are {\em strictly contained} in the square annulus would also be a region of this form.}
    \label{fig:squareannulus}
\end{figure}

The main question is: given a tiling $\tau_1$ restricted to $\m Z^3\setminus B_n$ and a tiling $\tau_2$ restricted to $B_{(1-\delta)n}$, under what conditions can we patch them together, i.e.\ find a tiling of $A_n$ with inner boundary condition $\tau_2$ and outer boundary condition $\tau_1$? We are interested in showing that this is possible when $n$ is large enough, when $\tau_1,\tau_2$ satisfy a consistency condition that they are \textit{nearly constant} for the same $s\in \text{Int}(\mc O)$. 

To specify the nearly constant condition, we give a few definitions. 

\begin{definition}
    An $\epsilon$ \textit{patch} $\alpha$ on $\partial B_n$ is an $\epsilon n\times \epsilon n$ square contained in a face of $\partial B_n$. \termindex{Chapter 6!$\epsilon$ patch}
\end{definition}

We can then measure the \textit{flux} of a discrete flow through a patch.

\begin{definition}\label{def: flux}
Let $S$ be an oriented discrete surface with outward normal vector $\xi$. We define the \textit{flux} of a discrete vector field $v$ through $S$ by \termindex{Chapter 6!flux}\symindex{Chapter 6!$\text{flux}$ - flux of a discrete vector field}
\begin{align*}
    \text{flux}(v,S) = \sum_{e\in E(\m Z^3), e\cap S \neq \emptyset} \text{sign}\langle \xi(e\cap S), e\rangle v(e).
\end{align*}
Here $E(\m Z^3)$ denotes the edges of $\m Z^3$ oriented from even to odd.
\end{definition}

We now use the definition of \textit{nearly-constant} in terms of the flux of the pre-tiling flow $v_\tau$ through patches. 

\begin{definition}\label{def:nearly constant}
    Fix $s\in \mc O$, let $B_n=[-n,n]^3$. A tiling $\tau\in \Omega$ is \textit{$\epsilon$-nearly-constant with value $s$} if there exists $M=M(\epsilon)$ such that for all $n>M$ and all $\epsilon$ patches $\alpha$ on $\partial B_n$, 
    \begin{align*}
        \text{flux}(v_\tau, \alpha) = \frac{1}{2}\langle \xi_\alpha, s\rangle \text{area}(\alpha) + o(\text{area}(\alpha)) = \frac{1}{2}\langle \xi_\alpha, s\rangle \epsilon^2 n^2 + o(\epsilon^2 n^2),
    \end{align*}
    where $\xi_\alpha$ is the outward pointing unit normal vector to $\alpha$ (where outward means away from $B_n$). \termindex{Chapter 6!$\epsilon$-nearly-constant}
\end{definition}
\begin{rem}
    Any patch $\alpha$ is contained in a flat coordinate plane, so its area is simply $\epsilon^2n^2$. The $\frac{1}{2}$ comes from the fact that the mean current is actually measuring the average flow per \textit{even} vertex. 

    Tilings satisfying the $\epsilon$-nearly-constant condition with value $s$ mimic the behavior of tilings sampled from ergodic measures of mean current $s$. (This is made precise in Corollary \ref{cor: ergodic implies nearly constant} after Theorem \ref{thm: subaction ergodic}.)
\end{rem}
With the conditions defined, we can now state the patching theorem. 
\begin{thm}[patching theorem]\label{patching}
    Fix $\delta>0$ and a mean current $s\in \text{Int}(\mc O)$. Let $B_n=[-n,n]^3$ be the cube of radius $n$, and let $A_n = B_n\setminus B_{(1-\delta) n}$ be the cube annulus of width $\delta n$. For $\epsilon>0$ small enough, if $\tau_1,\tau_2\in \Omega$ are $\epsilon$-nearly-constant with value $s$, then for $n$ large enough $A_n$ can be tiled with outer boundary condition $\tau_1$ and inner boundary condition $\tau_2$. 
\end{thm}
The main tool in the proof is Hall's matching theorem (Theorem \ref{thm:halls matching}). In this section we explain the main steps of the proof, guided by a series of two dimensional figures, and comment on differences between the two and three dimensional versions of this story. After this, in the remaining subsections we prove a series of lemmas (needed to control the more complicated geometric situations that can occur in three dimensions) before giving the formal proof of Theorem \ref{patching} in Section \ref{sec: patching proof}. 

\textbf{Steps of proof}

Given two tilings $\tau_1,\tau_2$ of $\m Z^2$, we want to know whether they can be patched together. In other words, we want to know whether a region $A$, which is a square annulus with some squares removed along the outer boundary if they are connected by $\tau_1$ to the outside of the annulus and some squares removed along the inner boundary if they are connected by $\tau_2$ to the outside of the annulus, is tileable. See Figure \ref{fig:squareannulus} for an example of a region $A$ of this form. If $A$ is not tileable, then there exists a \textit{counterexample set} $U$ as in Figure \ref{fig:squareannuluswithred}.
\begin{figure}[H]
    \centering
        \includegraphics[scale=0.39]{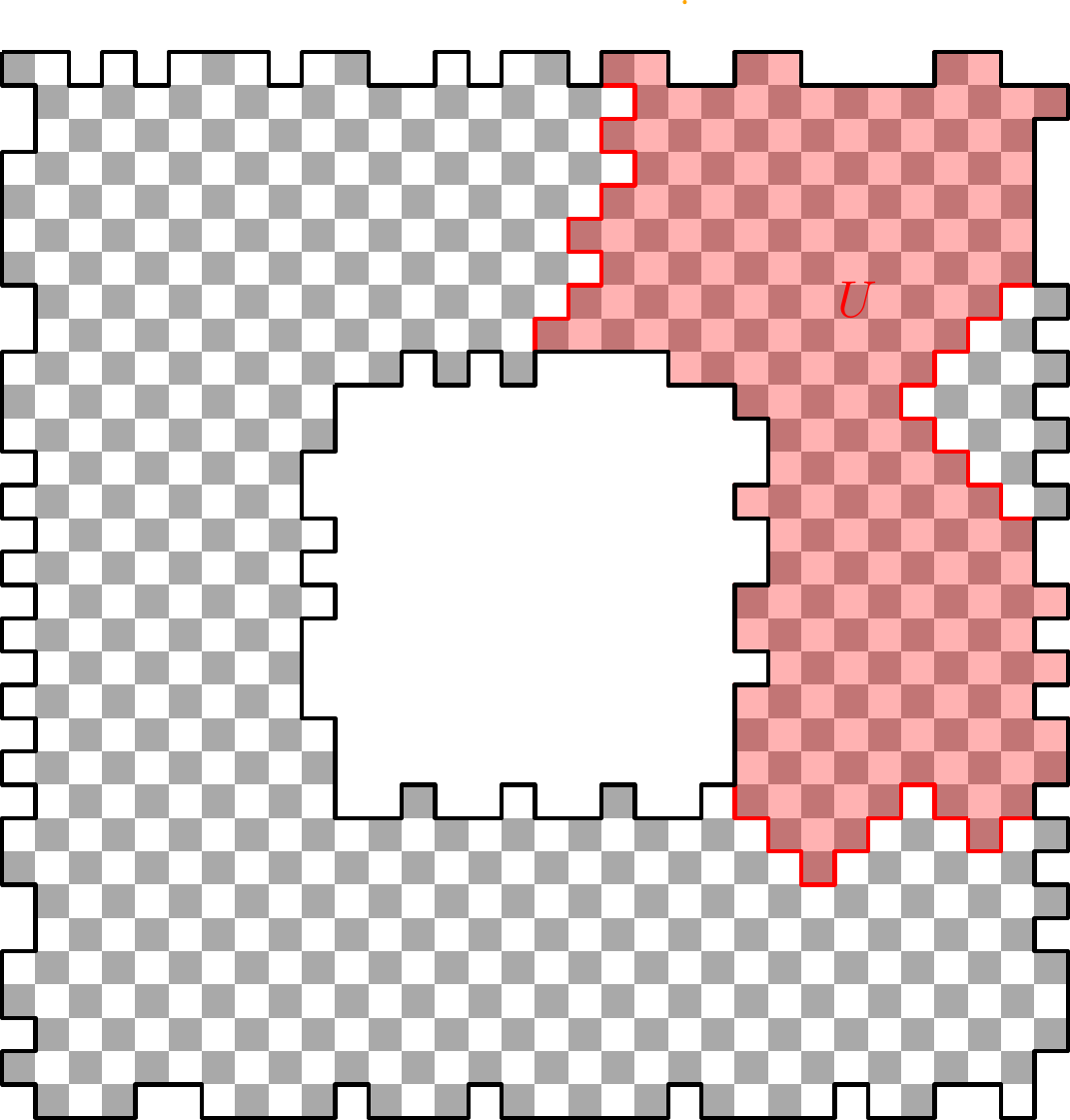}\\
    \caption{A potential counterexample set $U\subset A$.}
    \label{fig:squareannuluswithred}
\end{figure}
The red set $U \subset A$ in Figure \ref{fig:squareannuluswithred} has the property that every square on its {\em inner boundary} (i.e., every square of $U$ that is incident to a square in $A \setminus U$) is black. 

By Hall's matching theorem (Theorem \ref{thm:halls matching 2}), there exists a dimer tiling of $A$ if and only if every $U$ of this form is \textit{not} a counterexample. In other words, every set $U$ of this form has $\text{imbalance}(U) = \text{white}(U) - \text{black}(U) \leq 0$. We remark that the colors white and black used here stand for even and odd vertices and not for the colors we give to surfaces in 3 dimensions in the previous section. We do this because it becomes easier to illustrate the main ideas using the figures.

Given this, our strategy to show that $A$ is tileable for $n$ large enough proceeds by contradiction. We suppose that for all $n$ there is a set $U$ of the form above that is a counterexample, but then show that for $n$ large enough, it must have $\text{white}(U)\leq \text{black}(U)$ and hence not be a counterexample. To do this, we cut up $U$ into various smaller pieces, and bound the white minus black in each piece. 

First we divide the annular $A$ into regions as depicted in Figure \ref{fig:squareannuluswithmiddle}. We call these the ``middle region'' (blue), the ``thin region'' (orange) and the ``corner region'' (green). The middle region is a centered square annulus whose size will have to be appropriately tuned. The thin region is the union of the ``columns'' obtained as straight-line paths of squares that go from the middle square annulus boundary to the boundary of $A$. The corner region is the part leftover.

We then define $U'$ to be $U$ intersected with the middle layer, depicted in Figure \ref{fig:squareannuluswithblue}.
\begin{figure}[H]
    \centering
        \includegraphics[scale=0.39]{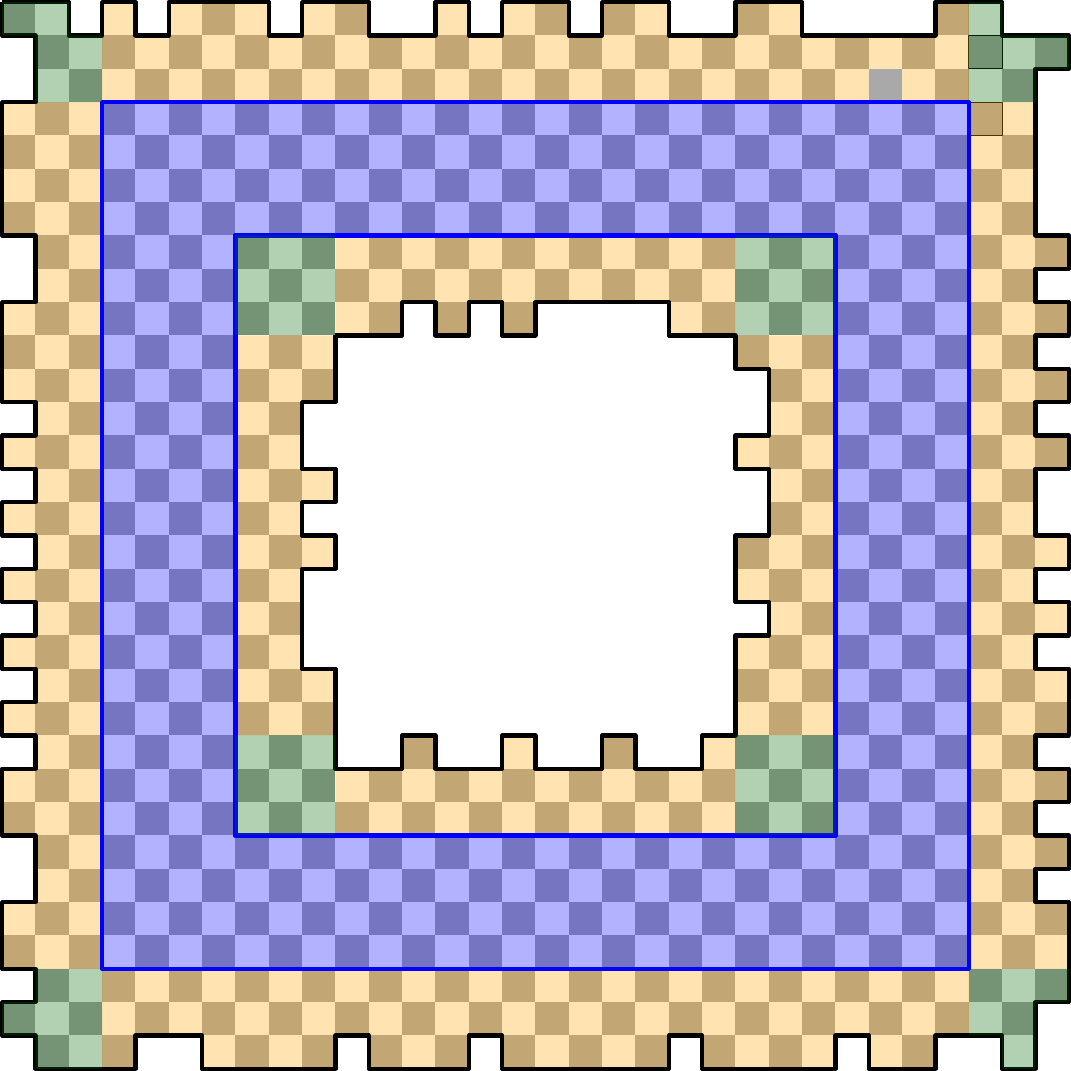}\\
    \caption{The middle (blue), thin (orange), and corner (green) regions of $A$. 
    }
    \label{fig:squareannuluswithmiddle}
\end{figure}
\begin{figure}[H]
    \centering
        \includegraphics[scale=0.39]{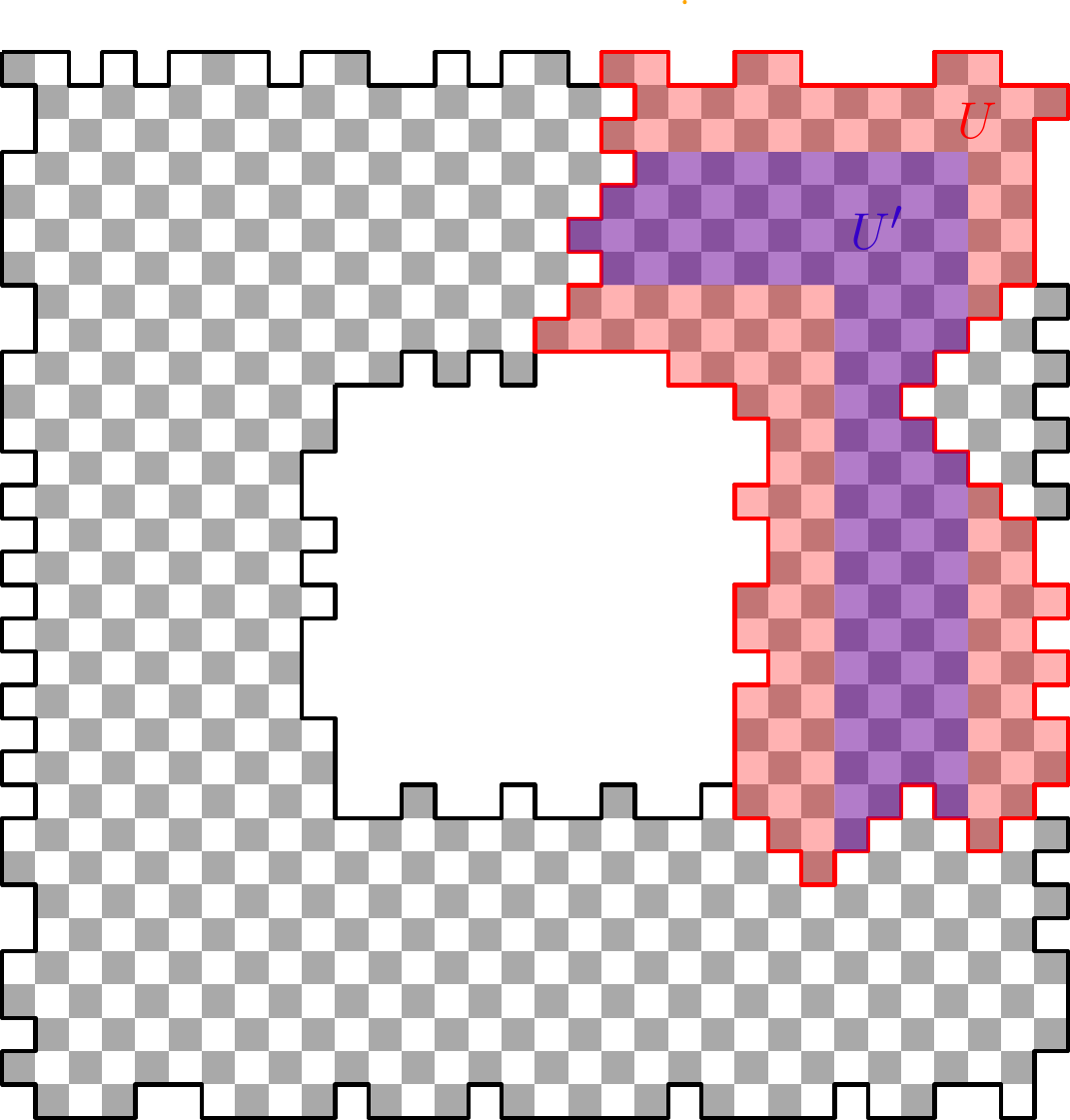}\\
    \caption{We can define $U'$ to be the intersection of $U$ with the middle region from Figure~\ref{fig:squareannuluswithmiddle}.}
    \label{fig:squareannuluswithblue}
\end{figure}
Given a tiling $\tau$ of a region containing $A$, we can define $U_\tau$ to be the region covered by tiles from $\tau$ intersecting $U'$.  The set $U''$ shown in Figure \ref{fig:squareannuluswithtiledblue} is the subset of $U_\tau$ that consists of the union of $U'$ together with all of the squares covered by tiles from $\tau$ that are contained in $U$ but in the complement of the middle region. Note that $U_\tau \setminus U''$ consists of only white squares and that $U_\tau$ is by construction evenly balanced between black and white squares. If we can show that $|U_\tau \setminus U''|$ is large then we know that $U''$ has many more black squares than white.

Indeed, we show that we can choose this \textit{test tiling} $\tau$ so that $|U_\tau \setminus U''|$ is large and hence $U''$ has many more black than white squares. Since $U''\subset U$, it remains to show that there cannot be enough white squares in $U\setminus U''$ for $U$ to be a counterexample. 

\begin{figure}[H]

    \centering
        \includegraphics[scale=0.39]{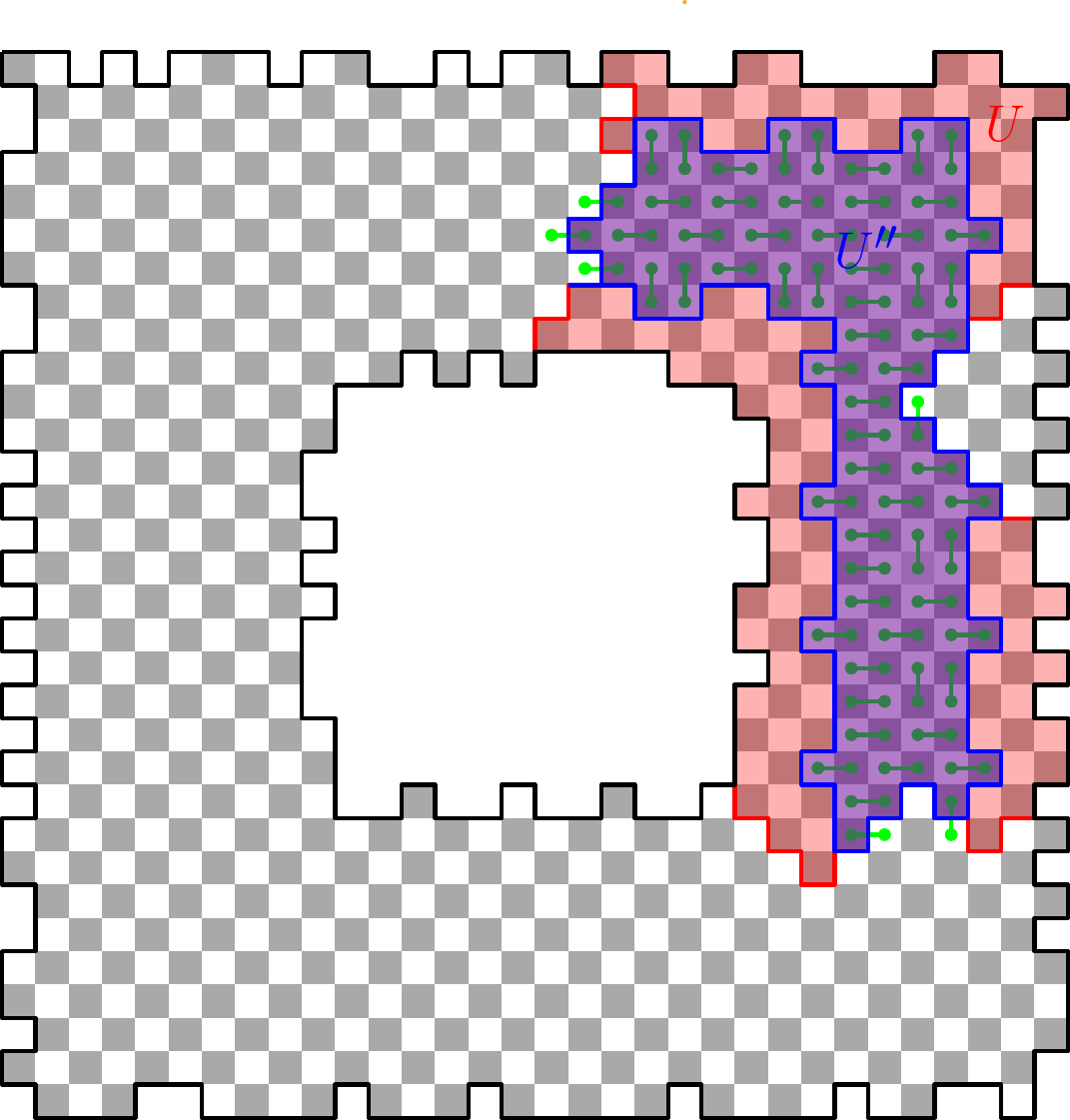}\\
    \caption{The region $U''$ (blue) and the tiles from $\tau$ that intersect $U'$ (green).}
    \label{fig:squareannuluswithtiledblue}
\end{figure}

In order to prove that $U$ itself has more black than white squares, we will divide the rest of $U$ into multiple pieces to treat separately, depicted in Figure \ref{fig:squareannuluswithshadow}. Here $U''$ is as given in Figure~\ref{fig:squareannuluswithtiledblue}. The ``shadow region'' is $U\setminus U''$ restriced to the ``thin region'' from Figure~\ref{fig:squareannuluswithmiddle}. It consists of the union of the columns that can be extended all the way from $\partial A$ to $\partial U''$. The corner region here is the intersection of $U$ with the corner region from Figure~\ref{fig:squareannuluswithmiddle}. The ``leftover pieces'' are the parts of $U$ that do not belong to one of the other three regions. Roughly speaking, we aim to show that $U$ has more black than white by showing that (i) $U''$ has a lot more black than white, (ii) the shadow region can only have a little more white than black (because of the nearly-constant condition), (iii) the corner region can only have a little excess white (since it has small volume), and (iv) the leftover pieces have at least as much black as white.
\begin{figure}[H]
    \centering
        \includegraphics[scale=0.39]{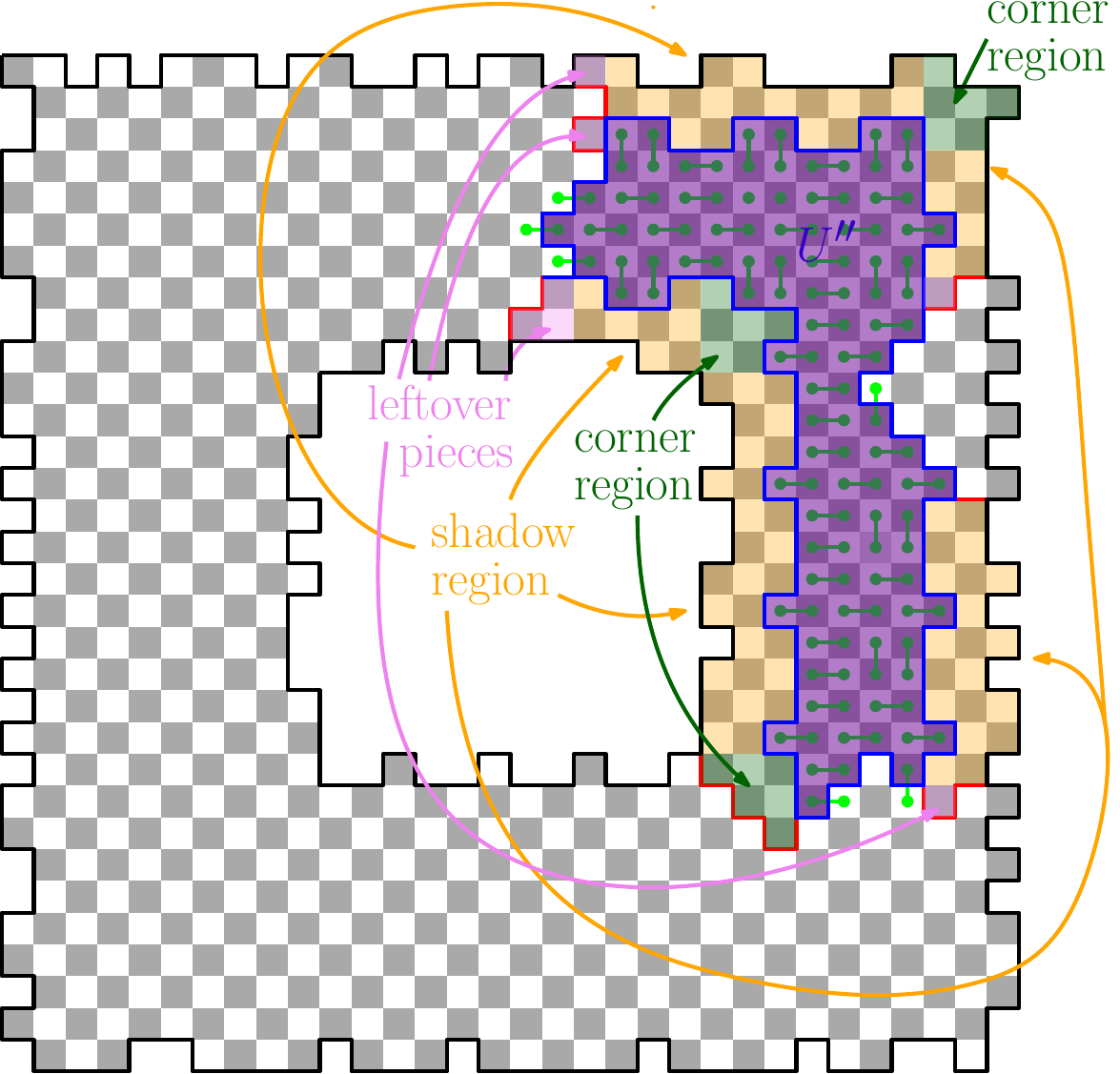}\\
    \caption{Depiction of all the regions that we divide $U$ into: (i) $U''$, (ii) shadow region, (iii) corner region, and (iv) leftover pieces.}
    \label{fig:squareannuluswithshadow}
\end{figure}

Ultimately to achieve (i), (ii), (iii), and (iv), we need to understand more about the possible shapes of counterexample sets $U$. The main difference between two dimensions and three dimensions is that the geometry of $U$ can be more complicated. 

In two dimensions, since $A$ is balanced, if $U$ is a counterexample then so is $U^c$ (with even/odd switched). As a consequence, without loss of generality $U$ and $U^c$ are both connected, and (since $U$ and $U^c$ must both cross the annulus) the interior boundary curve $S$ between them consists of two minimal monochromatic paths $S_1,S_2$ from the inner boundary to the outer boundary. There are then four points $a_1,b_1$ and $a_2,b_2$ which are the endpoints of $S_1,S_2$ on the inner and outer boundaries respectively. From here, one can argue as in \cite{Fournier} that $A$ is tileable as long as the boundary height differences (given by the boundary condition tilings $\tau_1,\tau_2$) at $a_1$ and $b_1$ and at $a_2$ and $b_2$ are in some sense compatible. 

In three dimensions, since $A$ is balanced, it again follows that both $U$ and $U^c$ are counterexamples, and that they are without loss of generality connected and have interior boundary surface $S$ between them which is a minimal monochromatic discrete surface. However the geometry (and topology) of 2D surfaces and connected sets $S$ in 3D can be much more complicated than that of 1D curves in 2D. For example, the set $U$ in 3D need not be simply connected, and instead of $S\cap \partial A$ consisting of four points, it consists of a collection of curves. 

To control the more complex \textit{geometric possibilities} for counterexamples in three dimensions, we prove that the interior boundary surface $S$ has area of order $n^2$ (Lemma \ref{annulusbound}), and that there is a choice for the inward blue layer where $S$ restricted to the layer (which is some union of curves) has good properties (Lemma \ref{lem: indenting}). This is the content of Section \ref{sec: isoperimetric} and Section \ref{sec: area growth}. 

The other key tool is the notion of a \textit{test tiling}. We show that we can use a tiling $\tau$ sampled from an ergodic measure of mean current $s\in \text{Int}(\mc O)$ to define $U''$ so that $U''$ will have order $n^2$ more black cubes than white cubes by showing that $\tau$ on expectation has an order $n^2$ number of tiles crossing $S$ (Lemma \ref{lem: flow bound}). This quantity is somewhat analogous to the height difference in two dimensions. Further, a tiling sampled from ergodic measure of mean current $s\in \text{Int}(\mc O)$ is nearly constant with high probability (Corollary \ref{cor: ergodic implies nearly constant}). This is the content of Section \ref{sec: ergodic}.

\subsection{Discrete isoperimetric inequalities}\label{sec: isoperimetric}

The classical isoperimetric inequality says that the minimal area of a region $D$ bounded by a smooth closed curve $\gamma$ in $\mathbb R^2$ of length $l$ is $\frac{1}{4\pi}(l^2)$. The equality case is achieved when $\gamma$ is a circle and $D$ is a disk. A lesser known fact is that this bound also holds for curves in $\mathbb R^3$ \cite{almgren1986optimal}. In this section we prove a discrete version of the isoperimetric inequality.

\begin{prop}\label{iso1}
Given any simple closed curve $\gamma$ in $\mathbb Z^3$, there is a surface $S$ with $\partial S = \gamma$ and 
$$\text{area}(S)\leq \frac{1}{8} (\text{length}(\gamma))^2.$$
\end{prop}
\begin{rem}The constant $\frac{1}{8}$ that we get from the proof is not optimal but our argument is simple and the result is sufficient for our purposes. Also note that the boundary curve $\gamma$ can be replaced by a multicurve $\Gamma = \gamma_1 \cup\dots \cup \gamma_k$ and the same result holds. 
\end{rem}
\begin{proof}
We proceed by induction. Suppose that for all simple closed curves $\gamma$ in $\m Z^3$ with $\text{length}(\gamma)=2m\leq 2n$, there exists a surface $S$ with $\partial S = \gamma$ and
$$\text{area}(S)\leq \frac{(m-1)m}{2}.$$
This is sufficient because it implies that
$$\text{area}(S)\leq \frac{1}{8}(\text{length}(\gamma))^2.$$
Now suppose that $\beta$ is a simple closed curve of length $2n+2$, and equip $\beta$ with an orientation. Choose two parallel edges $e_1=(a_1,b_1),e_2=(a_2,b_2)$ in $\beta$ with opposite orientations. Removing $e_1,e_2$ and identifying $a_i$ with $b_i$ for $i=1,2$ results in a new curve $\beta'$ of length $2n$. The identification means that $\beta'$ is a union simple curves $\gamma_1,...,\gamma_k$ of length $\leq 2n$ and double edges (note that each double edge contributes $2$ to the length of $\beta'$). By the inductive hypothesis, there exist surfaces $S_i$ with boundary $\gamma_i$ satisfying the bound for each $i=1,...,k$. To get a surface with boundary $\beta$, we find a path in $\m Z^3$ from $a_1$ to $a_2$ along $\cup_{i=1}^k S_i\cup D$, where $D$ is the double edges in $\beta'$. Since $\text{length}(\beta')\leq 2n$, we can find a path of length $\leq n$. We add back the edges $e_1, e_2$, splitting the path into two parallel paths. We then add squares along the path from $e_1$ to $e_2$ to construct a surface $S$ with boundary $\beta$ satisfying 
\begin{align*}
    \text{area}(S) \leq \text{area}(\cup_{i=0}^k S_i) + n \leq \frac{n(n+1)}{2}.
\end{align*}
\end{proof}

This can easily be extended to the monochromatic case; the only thing that changes is the constant.

\begin{cor}\label{iso2}
    Given any collection of simple closed curves $\gamma_1,...,\gamma_k$ in $\mathbb{Z}^3$, there is a monochromatic surface $S$ with $\Gamma=\gamma_1\cup \dots\cup\gamma_k$ as its boundary and 
    \begin{align*}
        \text{area}(S) \leq \frac{5}{8} \text{length}(\Gamma)^2
    \end{align*}
\end{cor}

\begin{proof}
    Using Proposition \ref{iso1}, we can find a surface $T$ with boundary $\gamma_1\cup\dots\cup\gamma_k$ satisfying $\text{area}(T) \leq \frac{1}{8}\text{length}(\Gamma)^2$. Replacing every white square in $T$ by at most 5 black squares, we get a new surface $S$ which is monochromatic (it is all black) satisfying the same bound with the constant $\frac{5}{8}$. 
\end{proof}

\subsection{Area growth of minimal monochromatic discrete surfaces}\label{sec: area growth}

We now use the discrete isoperimetric inequalities from the previous section to show that minimal monochromatic discrete surfaces have quadratic area growth (Proposition \ref{prop:quadratic growth}). We then apply this to the cube annulus $A_n = B_n\setminus B_{(1-\delta)n}$ to get two results (Lemma \ref{annulusbound} and  Lemma \ref{lem: indenting}) which will serve as lemmas for the patching theorem (Theorem \ref{patching}). 

\begin{prop}\label{prop:quadratic growth}
	Let $S$ be a minimal monochromatic discrete surface. Let $p\in \m Z^3 + (\frac{1}{2},\frac{1}{2},\frac{1}{2})$ be a vertex on $S$ and let $B_n(p) = p + [-n,n]^3$ be such that $S$ is not contained in $B_n(p)$. Then there is a universal constant $\kappa>0$ (i.e., independent of $S$ and $n$) such that 
	\begin{align*}
	\text{area}(S) \geq \kappa n^2.
	\end{align*}
\end{prop}
\begin{proof}
    To do this, we show that $\text{area}(S\cap B_n(p))$ grows quadratically in $n$. Let $m = \lfloor{\frac{n}{2}}\rfloor$. For $k\leq m$, define annular regions
    $$A_k = B_{m+k}(p)\setminus B_{m-k}(p)$$
    Let $S_k = S\cap A_k$ be $S$ restricted to $A_k$ and let $\Gamma_k=\partial S_k$. By Corollary \ref{iso2}, 
    \begin{align*}
        \text{area}(S_k) \leq \frac{5}{8} \text{length}(\Gamma_k)^2.
    \end{align*}
    Note that $S_k\cap \partial A_k$ might be larger than $\Gamma_k$, since there might be squares from $S_k$ contained in $\partial A_k$. 

    Any face in $S_{k+1}\setminus S_k$ corresponds to at most $4$ edges along $\Gamma_k$. Therefore 
    \begin{align*}
        \text{area}(S_{k+1})-\text{area}(S_k) \geq \frac{1}{4} \text{length}(\Gamma_k) \geq \frac{1}{4}\sqrt{\frac{5}{8}} \sqrt{\text{area}(S_k)}.
    \end{align*}
    Therefore the function $f(k):=\text{area}(S_k)$ satisfies the inequality $f(k+1)-f(k)\apprge \sqrt{f(k)}$. Extending $f$ linearly to a continuous function that is differentiable from the left, this becomes $f'(k) \apprge \sqrt{f(k)}$, which implies that $f$ grows at least quadratically in $k$. Applying this to $S$ itself we get that $\text{area}(S) \geq \kappa n^2$, where the universal constant $\kappa$ comes from the isoperimetric inequalities. 
\end{proof}

\begin{lemma}\label{annulusbound}
	Let $B_n=[-n,n]^3$ and suppose that $A = B_n\setminus B_{(1-\delta)n}$ for some $\delta\in (0,1)$. Suppose that $S$ is a minimal monochromatic discrete surface in $A$ which connects the inner and outer boundaries of $A$. Then there exist constants $c_1,c_2$ independent of $S$ and $n$ such that 
\begin{align*}
	c_1 n^2 \geq \text{area}(S) \geq c_2 n^2
\end{align*}
where $c_2 \sim \delta^2$. 
\end{lemma}
\begin{proof}
	For the upper bound, notice that $\partial S \subset \partial A$. From $\partial A$, we construct a monochromatic surface $T$ by capping every white square on $\partial A$ with a odd cube. There is a surface $S'\subset T$ with the same boundary as $S$. Since $\text{area}(T) \leq 5\, \text{area}(\partial A)$, there exists a constant $c_1$ such that $\text{area}(T)<c_1 n^2$. Since $S$ is a minimal monochromatic surface, 
	\begin{align*}
		c_1 n^2 \geq \text{area}(S') \geq \text{area}(S).
	\end{align*}
	Now we prove the lower bound. Since $S$ connects the inner and outer boundaries, there is a point $p\in S$ where we can apply Proposition \ref{prop:quadratic growth} to $B_{(\delta/3) n}(p)$. Hence
	\begin{align*}
		\text{area}(S) \geq \kappa ((\delta/3) n)^2 =: c_2 n^2
	\end{align*}
	and $c_2$ is of order $\delta^2$. 
\end{proof}

The next application of the area growth results is loosely that minimal surfaces ``spread out." If $X\subset \partial B_n$ for some $n$ is a surface, we define the \textit{$\epsilon$ covering area} of $X$ to be the total area in disjoint $\epsilon n \times \epsilon n$ size squares needed to contain $X$. We denote this by $\text{Cov}_{\epsilon}(X)$.\termindex{Chapter 6!$\epsilon$ covering area}\symindex{Chapter 6!$\text{Cov}_{\epsilon}(X)$ - $\epsilon$ covering area}

\begin{lemma}[Indenting lemma]\label{lem: indenting}
Fix $\delta>0$. Let $A= B_n\setminus B_{(1-\delta)n}$, and suppose that $S$ is a minimal monochromatic discrete surface connecting the inner and outer boundaries of $A$. Let $\Gamma_l=\partial (S\cap \partial B_l)$. There exist constants $c,c'>0$ independent of $S$ and $n$ such that for any $\epsilon>0$ small enough, there exists $l\in ((1-\epsilon^{1/2})n , n)$ such that 
\begin{align*}
    \text{Cov}_{\epsilon}(S \cap \partial B_l) \leq c \epsilon^{1/2} n^2 
\end{align*}
and
\begin{align*}
    \text{length}(\Gamma_l) \leq c' \epsilon^{-1/2} n.
\end{align*}

The analogous statements also hold for $l\in ((1-\delta)n, (1-\delta + \epsilon^{1/2})n)$.
\end{lemma}
\begin{proof}
    Cut the region between $\partial B_n$ and $\partial B_{(1-\epsilon^{1/2})n}$ into $M = \lfloor \frac{1}{3\epsilon^{1/2}}\rfloor$ layers $L_1,...,L_M$ of width $\frac{\epsilon^{1/2} n}{M}\geq 3\epsilon n$. By Lemma \ref{annulusbound} plus the pigeonhole principle, there exists $j$ such that 
    \begin{equation}\label{eq: area S cap L upper bound}
        \text{area}(S\cap L_j) \leq \frac{c_1 n^2}{M} \leq 3 c_1 \epsilon^{1/2} n^2.
    \end{equation}
    We further subdivide $L_j$ into three layers $L_j^{(1)},L_j^{(2)}, L_j^{(3)}$, where $L_j^{(2)}$ is the middle one. These each have width at least $\epsilon n$. The $l$ we find satisfying the conditions will have $\partial B_l \subset L_j^{(2)}$. 

    Let $\mc J$ be the collection of $(\epsilon n)^3$ sized boxes needed to cover $S\cap L_j^{(2)}$ and let $J = |\mc J|$. Then there are at least $J/9$ disjoint cubes of size $(3\epsilon n)^3$ in $L_j$, centered on an $(\epsilon n)^3$ cube in $L_j^{(2)}$, such that the central cube is in $\mc J$ (i.e., $S$ intersects the $(\epsilon n)^3$ central cube of this $(3\epsilon n)^3$ cube). Given this, we can apply Proposition \ref{prop:quadratic growth} to a point in each of the $J/9$ boxes with radius $\epsilon n$ to get the bound 
     \begin{equation}\label{eq: area S cap L lower bound}
        \text{area}(S\cap L_j) \geq \kappa \frac{J}{9} (\epsilon n)^2.
    \end{equation}
    Combining Equations \eqref{eq: area S cap L upper bound} and \eqref{eq: area S cap L lower bound} and solving for $J$ gives 
    \begin{align*}
        J \leq \frac{27 c_1}{\kappa} \epsilon^{-3/2}.
    \end{align*}
    Therefore for any $l$ such that $\partial B_l \subset L_j^{(2)}$, 
    \begin{align*}
        \text{Cov}_\epsilon(S\cap \partial B_l) \leq J (\epsilon n)^2 \leq \frac{27 c_1}{\kappa} \epsilon^{1/2} n^2.
    \end{align*}
    so the first part holds with $c = 27 c_1/\kappa$. 
        
    For the second part, we note that $L_j^{(2)}$ has width $\epsilon n$. Any square in $S\cap L_j^{(2)}$ contributes length at most four to the curves $\Gamma_l$ for $l$ such that $\partial B_l\subset L_j^{(2)}$, thus $\sum_{l: \partial B_l\subset L_j^{(2)}}\text{length}(\Gamma_l)\leq 4\text{area}(S\cap L_j^{(2)})$. Therefore by Equation \eqref{eq: area S cap L upper bound} and the pigeonhole principle again, we can find $l$ with $\partial B_l\subset L_j^{(2)}$ such that 
    \begin{align*}
        \text{length}(\Gamma_l) \leq \frac{12 c_1 \epsilon^{1/2} n^2}{\epsilon n} =  12 c_1 \epsilon^{-1/2} n,
    \end{align*}
    so the second part holds with $c' = 12 c_1$. 
\end{proof}

\subsection{Tilings sampled from ergodic measures}\label{sec: ergodic}

Recall that $\mc P_e$ denotes the measures on $\Omega$ which are ergodic with respect to the action of $\threeeven$, and that $\{\eta_i\}_{i=1}^3$ denote the standard unit basic vectors. In this section we prove a few results for tilings sampled from ergodic measure of mean current $s\in \mc O$. In the proof of the patching theorem, we use a ``test tiling" sampled from an ergodic measure  which we compare with the two other tilings we want to patch. First we note that there exist ergodic measures of mean current $s$ for all $s\in \mc O$.

\begin{lemma}\label{lem: ergodic measures exist}
    For every $s\in \mc O$, there exists an ergodic measure on dimer tilings of $\m Z^3$ of mean current $s$. 
\end{lemma}
\begin{rem}
   We use methods called \textit{chain swapping} described in Section \ref{Subsection: Dimer Swapping} to construct ergodic measures of any edge, then face, then interior mean current from the ones for $s\in \mc V$. The only results that we use about chain swapping here (Propositions \ref{prop:preserve_ergodicity} and \ref{prop: coupling swap mean current}) are essentially computations, and do not rely on any results presented in this section. 

   Note that by Theorem \ref{thm: extremal_entropy}, there exist ergodic measures for all mean currents $s\in \partial \mc O$, so we only need to use chain swapping to show existence for $s\in \text{Int}(\mc O)$. However we choose not to rely on this here, since the chain swapping techniques allow us to show existence easily just from existence of ergodic measures at the vertices of $\partial \mc O$. 
\end{rem}
\begin{proof}[Proof of Lemma \ref{lem: ergodic measures exist}]
   Let $\mc V\subset \partial \mc O$ denote its vertices. For each $s\in \mc V$, the atomic measure which samples the corresponding brickwork pattern is an ergodic measure of mean current $s$. 
        
   Given any $s\in \text{Int}(\mc O)$ (resp.\ $s$ contained in a face of $\partial \mc O$, resp.\ an edge of $\partial \mc O$), there exists $p\in(0,1)$ such that 
    \begin{align*}
        s = (1-p) s_1 + p s_2
    \end{align*}
    for $s_1, s_2\in \partial \mc O$ (resp.\ contained in the edges of $\partial \mc O$, resp.\ contained in $\mc V\subset \partial \mc O$). Let $\mu_1$ and $\mu_2$ be ergodic measures of mean current $s_1,s_2$ respectively. Let $\mu$ be a measure on $\Omega\times \Omega$ which is an ergodic coupling of $\mu_1$ and $\mu_2$, let $\mu'$ be obtained from $\mu$ by chain swapping with swap probability $p$, and let $\mu_1'$ and $\mu_2'$ denote its marginals. By Proposition \ref{prop: coupling swap mean current}, the mean current of $\mu_1'$ is
    \begin{align*}
        s(\mu_1') = (1-p) s_1 + p s_2 = s.
    \end{align*}
    By Proposition \ref{prop:preserve_ergodicity}, $\mu'$ is an ergodic measure on $\Omega\times \Omega$. Therefore $\mu_1'$ is an ergodic measure of mean current $s$. 
\end{proof}

Recall that the \textit{pretiling flow} $v_\tau$ is defined for $e$ oriented from even to odd by
\begin{align*}
    v_\tau(e) = \begin{cases} +1 &\quad e\in \tau \\ 0 &\quad e\not\in \tau.\end{cases}
\end{align*}
Let $S$ be an oriented discrete surface with outward normal vector $\xi$. Applying Definition \ref{def: flux} to $v_\tau$, the \textit{flux} of $v_\tau$ through $S$ is \termindex{Chapter 6!flux}\symindex{Chapter 6!$\text{flux}(v_\tau,S)$ - flux of the flow $v_{\tau}$ through the surface $S$}
\begin{align*}
    \text{flux}(v_\tau,S) = \sum_{e\in E(\m Z^3), e\cap S \neq \emptyset} \text{sign}\langle \xi(e\cap S), e\rangle v_\tau(e).
\end{align*}
As in the definition, $E(\m Z^3)$ denotes the edges of $\m Z^3$ oriented from even to odd. Flux of $v_\tau$ has a simple combinatorial interpretation. It counts the number of tiles in $\tau$ which cross $S$, with sign corresponding to the parity of the tile. If $S$ is monochromatic black, then $\text{flux}(v_\tau,S)$ is minus the number of tiles in $\tau$ which cross $S$. 

\begin{thm}\label{thm: subaction ergodic}
    Let $P$ be a coordinate plane with normal vector $\eta_i$ for some $i=1,2,3$, and let $P_n = P\cap B_n$ (recall $B_n=[-n,n]^3$). If $\mu\in \mc P_e$ has mean current $s\in \mc O$, then 
    \begin{align*}
        \lim_{n\to \infty} \frac{1}{|P_n|} \text{flux}(v_\tau,P_n) = \frac{1}{2}\langle s,\eta_i\rangle 
    \end{align*}
    where this limit converges almost surely and in probability.
\end{thm}
\begin{rem}
    The completely equivalent statement holds with $v_\tau$ replaced by $f_\tau$. In fact since $P_n$ is contained in a coordinate plane 
$\text{flux}(f_\tau,P_n) = \text{flux}(v_\tau,P_n)+o(1)$. The reason for the $\frac{1}{2}$ factor is that the mean current is the average current per \textit{even} vertex. The number of even vertices in $P_n$ is $|P_n|/2$.
\end{rem}
\begin{proof}
Without loss of generality let $P$ denote the $(x,y)$ coordinate plane, so the normal vector is $\eta_3 =(0,0,1)$. Recall that discrete surfaces consist of squares in $\Z^3+(\frac{1}{2},\frac{1}{2},\frac{1}{2})$, so $P=\{(x,y,1/2)~:~x,y\in \R\}$. Let $\Z^2_{\text{even}}=(\Z^2\times\{0\})\cap \threeeven$. Since $\mu$ is invariant under the $\Z^2_{\text{even}}$ action, we can apply the ergodic theorem for this subaction. Let $$T=P\cap([-1/2,3/2]\times[-1/2,1/2]\times[0,1])$$ and $\text{even}(B_n)=\Z^3_{\text{even}}\cap B_n$. The set $T$ is defined so that it contains the intersection points of two adjacent edges, namely $(0,0,0)$ to $(0,0,1)$ and $(1,0,0)$ to $(1,0,1)$. Consider the function $F:\Omega\to \R$ given by $F(\tau)=\frac{1}{2}\text{flux}(v_{\tau},T)$. We have that
$$\left\vert\frac{1}{|\text{even}(B_n)|}\sum_{\eta\in B_n}F(\tau+\eta)-\frac{1}{|P_n|} \text{flux}(v_\tau,P_n)\right\vert=o(n^2)$$
and by the ergodic theorem 
$$\lim_{n\to \infty} \frac{1}{|P_n|} \text{flux}(v_\tau,P_n)$$
exists in probability and almost surely. Temporarily we call the limit $\text{flux}^\star(v_{\tau},P)$. We know that this is $\Z^2_{\text{even}}$-invariant. By integrating the flux across $T$ we get that 
$$\int_{\Omega}\text{flux}^\star(v_{\tau},P)\;\dd\mu(\tau)=\int_{\Omega}F(\tau)\:\dd\mu(\tau)=\frac{1}{2}\langle s,\eta_3\rangle.$$

If we can now prove that the average $\text{flux}^\star(v_{\tau},P)$ is not just invariant under the two-dimensional action of $\Z^2_{\text{even}}$-invariant, but invariant under the full $\threeeven$ action, it will follow that it is constant almost surely and equal to $\frac{1}{2}\langle s,\eta_3\rangle$. To show this, we use  the fact that $v_\tau$ is essentially ``divergence free''. Indeed the tiling flow $f_{\tau}=v_{\tau}-r$ where $r$ is a reference flow and the flux of $r$ across $T$ is zero. It follows that 
$$|\text{flux}(f_\tau,P_n) - \text{flux}(v_\tau,P_n)|=o(n^2)$$
and therefore we have that
$$\lim_{n\to \infty} \frac{1}{|P_n|} \text{flux}(v_\tau,P_n)=\lim_{n\to \infty} \frac{1}{|P_n|} \text{flux}(f_\tau,P_n).$$
Since $f_{\tau}$ is divergence free it follows that for all $\gamma\in \threeeven\setminus \Z^2_{\text{even}}$,
$$|\text{flux}(f_\tau,P_n)-\text{flux}(f_{\tau+\gamma},P_n)|$$
is given by the flux through the sides of parallelopiped formed by $P_n$ and $P_n+\gamma$. Thus for a fixed $\gamma\in \threeeven$
$$|\text{flux}(f_\tau,P_n)-\text{flux}(f_{\tau+\gamma},P_n)|=o(n^2)$$
and hence the same holds for $v_{\tau}$ in place of $f_{\tau}$. Thus $\text{flux}^\star(v_{\tau},P)$ is $\threeeven$-invariant, which completes the proof.
\end{proof}

As a straightforward corollary of Theorem \ref{thm: subaction ergodic}, we see that tilings sampled from ergodic measures satisfy the $\epsilon$-nearly-constant condition with high probability. 
\begin{cor}\label{cor: ergodic implies nearly constant}
    Fix $\epsilon>1$. If $\mu$ is an ergodic measure of mean current $s$, then a tiling $\tau$ sampled from $\mu$ is $\epsilon$-nearly-constant on $B_n$ with value $s$ with probability arbitrarily close to $1$ for $n$ large enough. 
\end{cor}

The final goal of this section is to get an estimate on the expected flux of a pretiling flow across a monochromatic discrete surface (e.g.\ the boundary of a counterexample). To do this, we use the following combinatorial result. Here a \textit{closed surface} means one where every edge on the surface is contained in two or four squares of the surface.

\begin{lemma}\label{lem: counting squares}
    Suppose that $S$ is a monochromatic discrete surface. Let $X_1,X_2,X_3\subset S$ be the sets of squares with normal vectors $\eta_1,\eta_2,\eta_3$ respectively. 

    If $S$ is a closed surface, then $|X_1|=|X_2|=|X_3|$. If $S$ has boundary $\partial S$, then for all pairs $i\neq j$, $$|X_i|=|X_j|+O(\text{length}(\partial S)).$$
\end{lemma}
\begin{proof}
    If $S$ is a closed discrete surface, then every edge of $S$ is contained in either two or four squares from $S$ (four can happen if the edge is an edge of non-manifold points). For an edge contained in two squares from the surface, we say those squares are \textit{neighbors}. For an edge contained in four squares, we arbitrarily split the four into pairs, and say that the paired squares are neighbors. With this definition, every square $f\in S$ has exactly four neighbors. Since $S$ is monochromatic, if $f\in X_1$, then two of its neighbors are in $X_2$ and two of its neighbors are in $X_3$ (and similarly for any permutation of $1,2,3$). 

    View the set of all squares as a graph, where each square corresponds to a vertex, and two vertices are connected by an edge if the corresponding squares are neighbors. The number of edges connecting $X_1$ to $X_2$ must be equal to $2|X_1|$ (since every $f\in X_1$ has two neighbors in $X_2$), and analogously must be equal to $2|X_2|$ (since every $f\in X_2$ has two neighbors in $X_1$). Therefore $|X_1| = |X_2|$. An analogous argument shows that $|X_3| = |X_1|$ and completes the proof in the closed surface case. 

    If $S$ is not closed, then squares $f\in S$ which contain an edge along $\partial S$ do not have exactly four neighbors. Therefore the result holds up to an error of $\text{length}(\partial S)$. 
\end{proof}

\begin{lemma}\label{lem: flow bound}
    Let $S$ be a monochromatic black surface with boundary $\partial S$, and let $\Theta$ be the collection of odd cubes adjacent to $S$. For any tiling $\tau$, 
    \begin{align*}
        |\text{flux}(v_\tau,S)| \leq |\Theta|.
    \end{align*}
    If $\mu$ is an ergodic measure of mean current $s\in\text{Int}(\mc O)$, then there is a constant $K_\mu\in (0,1)$ independent of $S$ such that
    \begin{align*}
        \m E_{\mu}[|\text{flux}(v_\tau,S)|] \geq K_\mu |\Theta| + O(\text{length}(\partial S)).
    \end{align*}
\end{lemma}
\begin{rem}
    If $\mu$ has mean current $s\in \partial \mc O$, then $K_\mu=0$. 
\end{rem}
\begin{rem}
    Since $S$ is monochromatic black, $\text{flux}(v_\tau,S)$ is minus the number of tiles from $\tau$ crossing $S$. This is why we add the absolute value.
\end{rem}
\begin{proof}
    Any tile $e\in \tau$ crossing $S$ contains a cube from $\Theta$. From this it follows immediately that $|\text{flux}(v_\tau,S)| \leq |\Theta|$.
    
    Let $p_1,\dots,p_6$ be the probabilities under $\mu$ for the six types of tiles, ordered so that $s_1 = p_1 - p_2$, $s_2 = p_3 - p_4$ and $s_3 = p_5 - p_6$. Similarly let $N_1,\dots,N_6$ be the six types of squares on $S$, where the tile type parallel to the outward pointing normal vector at a square $f\in N_i$ has probability $p_i$. The random variable $|\text{flux}(v_\tau,S)|$ can also be written as $\sum_{f\in S} \mathbbm{1}_f(\tau)$, where $\mathbbm{1}_f(\tau)$ is the indicator variable which is $1$ if there is a tile in $\tau$ crossing $f$ and $0$ otherwise. From this, we see that
    \begin{align*}
        \m E_{\mu}[|\text{flux}(v_\tau,S)|] = \sum_{i=1}^6 p_i N_i. 
    \end{align*}
    We minimize the right hand side to get a positive lower bound (clearly $0$ is a lower bound). Let $\text{area}(S)=N= \sum_{i=1}^6 N_i$. By Lemma \ref{lem: counting squares}, $N_1+N_2$, $N_3+N_4$ and $N_5+N_6$ are equal to $N/3$ up to an error of $O(\text{length}(\partial S))$. Thus
    \begin{align*}
        \sum_{i=1}^6 p_i N_i &\geq \max\{p_1 N_1 + p_2 N_2,p_3 N_3 + p_4 N_4, p_5 N_5+p_6 N_6\} \\
        &\geq \max\{\min\{p_1,p_2\},\min\{p_3,p_4\},\min\{p_5,p_6\}\} (N/3 + O(\text{length}(\partial S)). 
    \end{align*}
    Since $s\in \text{Int}(\mc O)$, at least four of $\{p_i\}_{i=1}^6$ are nonzero, including one from each pair. Noting that $N = \text{area}(S)$, and that
    \begin{align*}
    |\Theta| \leq \text{area}(S) \leq 6|\Theta|,
    \end{align*}
    this proves the result with a constant of the form $K_\mu = p_i/3$ for some $i$ such that $p_i\neq 0$.
\end{proof}

\subsection{Proof of the patching theorem}\label{sec: patching proof}

We now give the proof of the {patching theorem} (Theorem \ref{patching}), as described in Section \ref{sec: ideas for patching}. We refer throughout to the corresponding figures from the outline there.

\begin{proof}[Proof of Theorem \ref{patching}]
    Since $\tau_1,\tau_2\in \Omega$, $A_n$ is balanced for all $n$. Thus by Corollary \ref{cor:minimal_tileable}, if $A_n$ is not tileable, it has a minimal counterexample $U$, i.e. a set with 
    \begin{align*}
        \text{imbalance}(U)=\text{even}(U)-\text{odd}(U)>0,
    \end{align*}
    despite $U$ having only odd cubes on its interior boundary. Let $S$ denote the interior boundary surface of $U$. Since $A_n$ is tileable with just the $\tau_1$ boundary condition (resp.\ with just the $\tau_2$ booundary condition), $S$ must connect the inner and outer boundaries of $A_n$. Thus by Lemma \ref{lem: indenting}, we can find $l_+\in ((1-\epsilon^{1/2})n,n)$ and $l_-\in ((1-\delta)n, (1-\delta+\epsilon^{1/2})n)$ such that for $l=l_+$ or $l=l_-$,
    \begin{equation}\label{eq: covering area}
        \text{Cov}_{\epsilon}(S \cap \partial B_{l}) \leq c \epsilon^{1/2} n^2 
    \end{equation}
    \begin{equation}\label{eq: length bound}
    \text{length}(\partial(S\cap \partial B_l))\leq c' \epsilon^{-1/2} n
    \end{equation}
    where $c,c'$ are constants. We define the ``middle region" $A_{\text{mid}} = (B_{l_+}\setminus B_{l_-})$, see Figure \ref{fig:squareannuluswithmiddle}. Then we let 
    \begin{align*}
        U' = U \cap A_{\text{mid}}.
    \end{align*}
    See Figure \ref{fig:squareannuluswithblue}. Let $\mu$ be an ergodic measure of mean current $s$ (this exists by Lemma \ref{lem: ergodic measures exist}). Let $\Theta$ be the collection of odd cubes adjacent to $S'=S \cap U'$. Note that 
    \begin{align*}
        \text{area}(S')/6 \leq |\Theta| \leq \text{area}(S').
    \end{align*}
    By Lemma \ref{annulusbound}, $|\Theta|\geq \text{area}(S')/6 \geq c_2' n^2$, where $c_2'\sim (\delta-2\epsilon^{1/2})^2$. Thus by Lemma \ref{lem: flow bound} and Equation \eqref{eq: length bound}, there is a constant $K_\mu\in(0,1)$ such that
    \begin{align*}
        \m E_{\mu}[|\text{flux}(v_\tau,S')|]\geq K_{\mu}c_2' n^2 + O(\epsilon^{-1/2} n). 
    \end{align*}
    By Theorem \ref{thm: subaction ergodic}, for any coordinate plane $P$, a tiling $\tau$ sampled from $\mu$ is $\epsilon$-nearly-constant on $P$ with value $s$ with probability approaching $1$ as $n$ goes to $\infty$. Therefore we can sample a tiling $\tau$ from $\mu$ which is $\epsilon$-nearly-constant with value $s$ on $A_{\text{mid}}$ for $n$ large enough and satisfies 
    \begin{equation}\label{eq: flux bound in proof}
        |\text{flux}(v_\tau,S')| \geq K_{\mu}c_2'n^2 + O(\epsilon^{-1/2} n). 
    \end{equation}
    We fix this choice of $\tau$ for the rest of the proof. Define $U_\tau$ to be the region covered by the tiles from $\tau$ which intersect $U'$, see Figure \ref{fig:squareannuluswithtiledblue}. Since $U_\tau$ is tileable, 
    \begin{align*}
        \text{imbalance}(U_\tau) = \text{even}(U_\tau) - \text{odd}(U_\tau) = 0. 
    \end{align*}
    We next define a new region $U''$, which is $U_\tau$ minus the even cubes adjacent to $S'$ (see Figure \ref{fig:squareannuluswithtiledblue} again). Note that the region $U_\tau\setminus U''$ consists of only even cubes. By Equation \eqref{eq: flux bound in proof}, 
    \begin{align*}
        |U_\tau\setminus U''| \geq K_{\mu} c_2' n^2 + O(\epsilon^{-1/2}n).
    \end{align*}
    Therefore 
    \begin{align*}
        \text{imbalance}(U'') \leq -K_{\mu} c_2' n^2 + O(\epsilon^{-1/2}n).
    \end{align*}
    It remains to show that $\text{imbalance}(U)$ is very close to $\text{imbalance}(U'')$, and this is where we use the $\epsilon$-nearly-constant condition. We split $A_n\setminus A_{\text{mid}}$ into two regions. First we define the ``thin region" $A_{\text{thin}}$ to be the union of columns parallel to one of the coordinate directions which connect $\partial A_n$ and $\partial A_{\text{mid}}$. See Figures  \ref{fig:squareannuluswithmiddle} and \ref{fig:squareannuluswithshadow}. The complement of $A_{\text{mid}}\cup A_{\text{thin}}$ we call the ``corner region" $A_{\text{corner}}$, and consists of a neighborhood of the edges of the outer boundary cube of $\partial A_n$, and the inner boundary cube of $\partial A_{\text{mid}}$. We note that 
    \begin{align*}
        \text{area}(A_{\text{corner}}\cap \partial A_n) \leq 24 \epsilon^{1/2} n^2.
    \end{align*}
    The constant factor comes from the fact that the cube has 12 edges. Therefore
    \begin{align*}
        \text{imbalance}(U\cap A_{\text{corner}}) \leq 24 \epsilon^{1/2} n^2.
    \end{align*}
    We define $U_{\text{shadow}} = (U\setminus U'')\cap A_{\text{thin}}$. We bound the imbalance of $U_{\text{shadow}}$ column-by-column, where each column $C$ consists of a straight line path of single cubes from $\partial A_{\text{mid}}$ to $\partial A_n$. Recall $U_{\text{shadow}}$ is defined with boundary condition $\tau$ on $\partial A_{\text{mid}}$ and boundary condition $\tau_1$ or $\tau_2$ on $\partial A_n$. We also note that for any column $C$, since $S$ is monochromatic black, $\text{imbalance}(U\cap C)\leq +1$.

    We cut $\partial A_{\text{mid}}$ into $(\epsilon n)\times (\epsilon n)$ patches $\alpha$. For each $\alpha$, there is a corresponding patch $\beta$ on $\partial A_n$ matched to $\alpha$ by columns. For any patch $\alpha\subset \partial A_{\text{mid}}$, since $\tau$ is sampled from an ergodic measure, it is $\epsilon$-nearly-constant with value $s$ on $\partial A_{\text{mid}}$ (Corollary \ref{cor: ergodic implies nearly constant}). Thus we have that 
\begin{equation}\label{eq: tau nc}
    \text{flux}(v_\tau,\alpha) = \frac{1}{2} \langle s, \xi_\alpha \rangle (\epsilon n)^2 + o(\epsilon^2 n^2).
\end{equation}
Let $v_\ast$ be equal to $v_{\tau_1}$ on the outer boundary of $\partial A_n$ and $v_{\tau_2}$ on the inner boundary. For $\beta\subset \partial A_n$ the patch connected by a column to $\alpha$, since $\tau_1,\tau_2$ are $\epsilon$-nearly-constant with value $s$, 
\begin{equation}\label{eq: tau_i nc}
    \text{flux}(v_{\ast},\beta) = \frac{1}{2} \langle s, \xi_\alpha \rangle (\epsilon n)^2 + o(\epsilon^2 n^2).
\end{equation}
(Note that $\xi_\alpha = \xi_\beta$.) For a patch $\alpha\subset \partial A_{\text{mid}}$, let $C(\alpha)$ be the union of columns incident to $\alpha$. The set $U_{\text{shadow}}$ is covered by these column sets, so it remains to control imbalance of $U_{\text{shadow}}\cap C(\alpha)$ for each patch $\alpha$.

By Equation \eqref{eq: covering area}, at most $c \epsilon^{1/2} n^2$ of the area of $\partial A_{\text{mid}}$ is in patches $\alpha$ which intersect $U$ and $U^c$. The total imbalance contribution from these is bounded by $c \epsilon^{1/2} n^2$. If $\alpha \subset U^c$ but $C(\alpha)$ still intersects $U$, then all the columns in $C(\alpha)$ have at least one end on $S$, and the imbalance in $C(\alpha)\cap U_{\text{shadow}}$ is at most $0$.

Now we look at the cases where $\alpha\subset U$. If $C(\alpha)\subset U$, then the total imbalance in $C(\alpha)\cap U$ is $o(\epsilon^2 n^2)$ by Equations \eqref{eq: tau nc} and \eqref{eq: tau_i nc}. If $\alpha\subset U$ but $C(\alpha)\not\subset U$, then some of the columns starting from $\alpha$ hit $S$ before hitting $\beta$. This means they end on an odd cube. Extending the column all the way to $\beta$ would only make the imbalance larger, however then the imbalance in $C(\alpha) \cap U$ is bounded above by $o(\epsilon^2 n^2)$ by Equations \eqref{eq: tau nc} and \eqref{eq: tau_i nc}. The number of columns is a bounded by a constant independent of $n$ times $\epsilon^{-2}$, hence in total the imbalance in $U_{\text{shadow}}$ is bounded by $c \epsilon^{1/2} n^2 + o(n^2)$.

Putting everything together, we have that 
    \begin{align*}
        \text{imbalance}(U) &\leq \text{imbalance}(U'') + \text{imbalance}(U_{\text{shadow}}) + \text{imbalance}(U\cap A_{\text{corner}})\\
        &\leq- K_\mu c_2' n^2 + 24\epsilon^{1/2} n^2 + c \epsilon^{1/2} n^2 + O(\epsilon^{-1/2} n) + o(n^2). 
    \end{align*}
    Recall that $c_2'\sim (\delta-2\epsilon^{1/2})^2$, so it gets larger as $\epsilon$ gets smaller. Fixing $\epsilon$ small enough as a function of the constants, for $n$ large enough $\text{imbalance}(U)\leq 0$ and hence $U$ is not a counterexample. Therefore by Proposition \ref{cor:minimal_tileable}, for $n$ large enough, $A_n$ is tileable with boundary conditions $\tau_1,\tau_2$.
\end{proof}

\subsection{Corollaries for ergodic Gibbs measures}\label{sec: patching ergodic}

\begin{cor}\label{cor:entsame}
If $\mu_1,\mu_2$ are EGMs of the same mean current $s\in \text{Int}(\mc O)$, then $h(\mu_1) = h(\mu_2)$.
\end{cor}
\begin{rem}
The proof of the patching theorem (Theorem \ref{patching}) uses that $ {s}\not \in \partial \mc O$ since this is a condition of Lemma \ref{lem: flow bound}. This corollary shows that $s\not\in \partial \mc O$ is a necessary condition and not just an artifact of the proof, since if $s\in \mc \partial \mc O$ then not all ergodic Gibbs measures of mean current $s$ have the same specific entropy (see Proposition \ref{prop: same current different entropy}). 
\end{rem}
\begin{proof}
Fix $\delta>0$, and let $B_n = [-n,n]^3$ and $A_n = B_n\setminus B_{(1-\delta)n}$. By the patching theorem (Theorem \ref{patching}) with outer boundary condition sampled from $\mu_1$ and inner boundary condition sampled from $\mu_2$, we get that 
\begin{align*}
    h(\mu_2) \leq (1+O(\delta)) h(\mu_1).
\end{align*}
Switching them, we find that 
\begin{align*}
    h(\mu_1)\leq (1+O(\delta)) h(\mu_2). 
\end{align*}
Therefore $h(\mu_1) = h(\mu_2)$. 
\end{proof}

Another useful result comes from applying patching to a sequence of $\epsilon$-nearly-constant tilings and a sample from an EGM. This relates the number of tilings of a region with fixed $\epsilon$-nearly-constant boundary conditions to the specific entropy of an EGM. This serves as a lemma in the proof of the lower bound in the large deviation principle (Theorem \ref{thm:lower}). Recall that $\Omega$ denotes the set of dimer tilings of $\m Z^3$.

\begin{prop}\label{thm:epsilon_nearly_constant_vs_EGM}
Fix $\delta>0$, $s\in \text{Int}(\mc O)$,  $B_n = [-n,n]^3$, and let $A_n = B_n\setminus B_{(1-\delta)n}$. For $\epsilon>0$ small enough given $\delta,s$, suppose that $(\tau_n)_{n\geq 1}\subset \Omega$ is such that $\tau_n$ is $\epsilon$-nearly-constant on $\partial B_n$ with value $s$ for $n$ large enough. Let $\pi_n$ be the uniform probability measure on tilings $\sigma$ of $B_n$ such that $\sigma\mid_{\partial B_n} = \tau_n$. Then for any EGM $\mu$ of mean current $s$ and $n$ large enough, 
\begin{align*}
    |B_n|^{-1} H(\pi_n) \geq h(\mu) (1+ O(\delta)). 
\end{align*}
\end{prop}
\begin{rem}
Since $\pi_n$ is a uniform measure, $H(\pi_n) = \log Z_n$, where $Z_n$ is the partition function of $\pi_n$. 
\end{rem}
\begin{proof}
We apply the patching theorem (Theorem \ref{patching}) to $A_n = B_n \setminus B_{(1-\delta)n}$ with $\tau_n$ on the outer boundary and a sample from $\mu$ on the inner boundary. For $n$ large enough, patching is possible with $\mu$-probability $(1-\epsilon)$ on an annulus of width $\delta$. 

For $\Lambda\subset \m Z^3$, let $\Omega(\Lambda)$ denote the free-boundary dimer tilings of $\Lambda$. For $\sigma\in \Omega(\Lambda)$, recall that $X(\sigma)$ is the set of \textit{extensions} of $\sigma$, i.e.\
$$X(\sigma) = \{\tilde{\sigma}\in \Omega : \tilde \sigma\mid_{\Lambda} = \sigma\}.$$ 
Then we compute
\begin{align*}
    H_{B_{(1-\delta)n}}(\mu) &= -\sum_{\sigma\in \Omega(B_{(1-\delta)n})} \mu(X(\sigma)) \log \mu(X(\sigma))\\
    &= -\sum_{\substack{\sigma\in \Omega(B_{(1-\delta)n})\\\sigma, \tau_n \text{ patchable}}} \mu(X(\sigma)) \log \mu(X(\sigma)) + O(\epsilon \log \epsilon)
\end{align*}
Let $U_{n}$ denote the uniform probability measure on 
\begin{align*}
    \{\sigma \in \Omega(B_{(1-\delta)n}) : \sigma, \tau_n \text{ patchable}\}. 
\end{align*}
Since uniform measure maximizes entropy, 
\begin{align*}
    H_{B_{(1-\delta)n}}(\mu) \leq H(U_{n}) + O(\epsilon \log \epsilon) \leq H(\pi_n) + O(\epsilon \log \epsilon). 
\end{align*}
Thus for $n$ large enough such that the patching theorem applies for $\tau_n$ and a sample from $\mu$ on $A_n = B_n\setminus B_{(1-\delta)n}$, and for $\epsilon$ small enough given $\delta$ and $s$, we have that
\begin{align*}
    |B_n|^{-1} H(\pi_n) \geq (1+ O(\delta)) |B_{(1-\delta)n}|^{-1} H_{B_{(1-\delta)n}}(\mu) \geq (1+O(\delta)) h(\mu). 
\end{align*}
\end{proof}

\section{Properties of entropy }\label{sec:entropy}

In this section we prove results about the local entropy function $\ent$ introduced in Section \ref{sec: entropy prelim}. Recall that $\Omega$ is the set of dimer tilings of $\m Z^3$, and that $\mc P$ denotes the space of $\m Z^3_{\text{even}}$-invariant probability measures on $\Omega$. Further recall that for any $s\in \mc O$, we define $\mc P^{s}\subset \mc P$ to be the set of measures which also have mean current $s$, $\mc P_e\subset \mc P$ to be the set of ergodic measures, and $\mc P_e^s$ to be the set of ergodic measures with mean current $s$. The mean-current entropy function $\ent:\mc O\to [0,\infty)$ is defined by 
\begin{align*}
    \ent(s) = \sup_{\mu \in \mc P^s} h(\mu),
\end{align*}
where $h(\cdot)$ denotes specific entropy (see Section \ref{sec: entropy prelim}). 

We saw in Section \ref{sec:extreme} that $\ent$ is equal to $\ent_{\text{loz}}$ when restricted to any face of $\partial \mc O$ (Theorem \ref{thm: extremal_entropy}). In particular this implies that $\ent$ is strictly concave when restricted to any face of $\partial \mc O$ (Corollary \ref{cor: ent boundary}). The main result of this section is that $\ent$ is {strictly concave} on all of $\mc O\setminus \mc E$, where $\mc E$ denotes the edges of $\partial \mc O$ (Theorem \ref{theorem: entropy is strictly concave}). \symindex{Chapter 7!$\mc E$ the edges of the mean-current octahedron $\mathcal O$}

In Section \ref{subsection: entropy maximizers are Gibbs measures} we show that the supremum of $\{h(\mu):\mu\in \mc P^s\}$ is realized by a Gibbs measure for all $s\in \mc O$ (Theorem \ref{theorem: entropy maximizer gibbs}). It is a classical result going back to Lanford and Ruelle \cite{LanfordRuelle}  that entropy maximizers in $\mc P$ are Gibbs measures. We extend this to show that the entropy maximizer in $\mc P^s$, where the mean current is fixed, is also a Gibbs measure. In Section \ref{subsection: easy observations} we show using elementary methods that $\ent$ is concave and continuous on $\mc O$. 

The proof that $\ent$ is strictly concave on $\mc O\setminus \mc E$ (Theorem \ref{theorem: entropy is strictly concave}) requires some new tools and ideas. We use a version of a technique called \textit{cluster swapping} used in \cite[Chapter 8]{AST_2005__304__R1_0}, which we call \textit{chain swapping}. As set up for the proof of strict concavity in Section \ref{subsection: strict concavity} we prove some preliminary results about flows in the double dimer model (Section \ref{subsection: flows for double dimer}) and introduce the chain swapping technique (Section \ref{Subsection: Dimer Swapping}). Combining chain swapping with the results for $\ent$ on $\partial \mc O$ from Section \ref{sec:extreme} we show that $\ent$ is strictly concave on $\mc O\setminus \mc E$. 

Strict concavity has a number of important consequences. We saw in Corollary \ref{cor:entsame} that if $\mu_1,\mu_2$ are EGMs with mean current $s\in \text{Int}(\mc O)$, then $h(\mu_1)= h(\mu_2)$. Combining this with strict concavity, we show that if $s\in \text{Int}(\mc O)$ and $\mu\in \mc P^s$, then $h(\mu) = \ent(s)$ if and only if $\mu$ is an EGM or weighted average of EGMs all of mean current $s$ (Theorem \ref{theorem: existence of gibbs}), and therefore that there exists an EGM of every mean-current $s\in \mc O$ (Corollary \ref{cor: EGMs exist!}). Heuristically, the main goal of this section is to show that the mean current $s$ captures broad statistics of dimer tilings sampled from $\mu \in \mc P^s_e$ when $s\in \text{Int}(\mc O)$. 

We use these results again in Section~\ref{subsection: properties of Ent}, where after introducing the topology, we use the properties of $\ent$ derived here to study $\Ent(g) = \frac{1}{\text{Vol}(R)} \int_R g(x)\, \dd x$.

\subsection{Entropy maximizers of a given mean current are Gibbs measures}\label{subsection: entropy maximizers are Gibbs measures}

We first study the maximizers of specific entropy $h(\cdot)$ in $\mc P^{s}$ for $s\in \mc O$ fixed. It is straightforward to show that there exists a measure $\mu \in \mc P^s$ which achieves $\sup\{h(\mu):\mu \in \mc P^s\}$ for any $s\in \mc O$. 

\begin{lemma}\label{lemma: entropy maximiser}
	Let $s\in \mathcal O$. There exists $\mu \in\Prob^{s}$ such that $h(\mu)=\ent(s)$.
\end{lemma}

\begin{proof}
The space $\Prob^{s}$ is compact with respect to the weak star topology. Since $h$ is an upper semicontinuous function of the measure \cite[Theorem 4.2.4]{keller1998equilibrium} it must achieve its maxima in $\Prob^{s}$. 
\end{proof}

\begin{thm}\label{theorem: entropy maximizer gibbs}
	Fix $s\in \mc O$. If $\mu\in \mc P^s$ has $h(\mu) = \ent(s)$ then $\mu$ is a Gibbs measure.
\end{thm}

If the mean current is not fixed, then this is a standard result originally shown by Landford and Ruelle in \cite{LanfordRuelle}. The main idea of the proof is a variational argument which says that if a measure $\mu$ is not  Gibbs, then there is a ``perturbation'' of $\mu$ which has more entropy. We need to show that this perturbation does not change the mean current, and this is the purpose of Lemma \ref{lem:no_infinite_paths}. After that, our proof of Theorem \ref{theorem: entropy maximizer gibbs} is inspired by the exposition in \cite{burton1994non}. 

To show the mean current does not change, we use double dimers to compare the mean currents of the two measures. Double dimers will be a tool throughout Section \ref{sec:entropy}. There is a natural action of the group $\threeeven$ on the product $\Omega\times \Omega$ acting coordinate wise. Superimposing two dimer tilings $\tau_1$ and $\tau_2$ gives us a \emph{double dimer configuration} $(\tau_1,\tau_2)$. The union $\tau_1 \cup \tau_2$ consists of finite cycles, infinite paths and isolated double edges. Each cycle or infinite path in $(\tau_1,\tau_2)$ is {\em oriented} in a way that agrees with the direction of the $\tau_1$ flow (for edges in $\tau_1$) or opposite the direction of the $\tau_2$ flow (for edges in $\tau_2$).

\begin{lemma}\label{lem:no_infinite_paths}
Let $\mathfrak m$ be a $\threeeven$-invariant measure on $\Omega\times \Omega$ such that for $\mathfrak m$ almost every $(\tau_1, \tau_2)$, the union $\tau_1\cup \tau_2$ does not contain infinite paths. Then
$$s(\pi_1(\mathfrak m))=s(\pi_2(\mathfrak m))$$
where $\pi_i:\Omega\times \Omega\to \Omega$ is projection onto the $i^{th}$ coordinate for $i=1,2$. 
\end{lemma}

\begin{rem}
    We remark that non-existence of infinite paths in a sample from a coupling like this has been used in other related but different ways in statistical mechanics, e.g.\ to show that two Gibbs measures are the same if there are no infinite paths in a sample from the coupling, or to compute covariances. See for example \cite{van_den_Berg}, \cite{van_den_Berg-Steif}.
\end{rem}

\begin{proof}
Since the mean current is an affine function of measure, by the ergodic decomposition theorem it is sufficient to prove this lemma for $\mathfrak m$ ergodic. For the rest of the proof we assume that $\mathfrak m$ is an ergodic measure. 

Recall that $v_\tau$ is the \textit{pretiling flow} of $\tau$ (for the definition see Equation \eqref{eq:v_tau} in Section \ref{section:tiling_flows}), and let $(\tau_1,\tau_2)$ be a sample from $\mathfrak m$. By assumption all paths $\gamma\subset (\tau_1,\tau_2)$ are finite loops or double edges (which are loops with just one edge from each tiling). Let $E(\gamma)$ denote the edges along $\gamma$ oriented from even to odd. (If $e\in E(\gamma)$, $-e$ is $e$ with orientation reversed.) With a slight abuse of notation, we also view $e$ as a vector oriented from even to odd. Since $\gamma$ is a loop,
\begin{equation}\label{eq: shift equal''}
     \sum_{e\in E(\gamma)} v_{\tau_1}(e) { e} = \sum_{e\in E(\gamma)} v_{\tau_2}(e) {e}.
\end{equation}
Given $x\in \Z^3$ let $\gamma_x$ denote the loop in $(\tau_1,\tau_2)$ containing $x$. We denote the number of edges in a loop $\gamma$ by $\text{length}(\gamma)$.  For any $\epsilon>0$, there exists $k$ such that 
\begin{align*}
    \mathfrak m(\text{length}(\gamma_0) >k ) < \epsilon.
\end{align*}
Let $B_n:=[1,n]^3$. By the mean ergodic theorem, there is $n$ large enough such that with $\mathfrak m$-probability $1-\epsilon$, the double dimer configuration $(\tau_1,\tau_2)$ satisfies the following.
\begin{enumerate}
    \item We have
    \begin{equation}\label{eq: points on long paths''}
        |\{ x\in B_n~:~\text{length}(\gamma_x) > k \}| < 2\epsilon n^3.
    \end{equation}
    \item Let $E(B_n)$ denote the edges in $B_n$ oriented from even to odd. For $i=1,2$ and any unit coordinate vector $\eta$, 
    \begin{equation}\label{eq: mean current approx i12''}
        \bigg\lvert \langle s(\pi_i(\mathfrak m)),\eta\rangle - \frac{2}{n^3} \sum_{e\in E(B_n)} \langle v_{\tau_i}(e) e, \eta\rangle \bigg\rvert < \epsilon.
    \end{equation}
\end{enumerate}
Let $C_n:=\{x\in B_n~:~ \gamma_x\subset B_n\}$ denote the $\mathfrak m$-random set of points on loops in $(\tau_1,\tau_2)$ contained in $B_n$. By Equation \eqref{eq: points on long paths''}, with $\mathfrak m$-probability $1-\epsilon$,
\begin{equation}\label{eq: C_n bound''}
    |C_n| \geq (n-k)^3 - 2\epsilon n^3.
\end{equation}
Clearly $\{\gamma_x: x\in C_n\}$ is a union of loops. Let $E(C_n)$ denote the edges of loops in this collection oriented from even to odd. By Equation \eqref{eq: shift equal''},
\begin{equation}\label{eq: zero on C_n''}
    \sum_{e\in E(C_n)} \bigg(v_{\tau_1}(e) e- v_{\tau_2}(e) e\bigg) = 0.
\end{equation}
There are at most $\frac{n^3 - |C_n|}{2}$ tiles in $\tau_i$ in $E(B_n)\setminus E(C_n)$ for $i=1,2$. Therefore by Equation \eqref{eq: C_n bound''}, with $\mathfrak m$ probability $1-2\epsilon$ 
\begin{equation}\label{eq: C_n vs B_n''}
    \bigg\lvert \frac{2}{n^3}\sum_{e\in E(B_n)} v_{\tau_i}(e)e - \frac{2}{n^3}\sum_{e\in E(C_n)} v_{\tau_i}(e)e \bigg\rvert \leq \frac{n^3 - (n-k)^3+2\epsilon n^3}{n^3} = 1 + 2\epsilon - \frac{(n-k)^3}{n^3}.
\end{equation}
Combining Equations \eqref{eq: mean current approx i12''}, \eqref{eq: zero on C_n''}, \eqref{eq: C_n vs B_n''} gives that for any unit coordinate vectors $\eta$,
\begin{equation}
    \bigg\lvert \langle s(\pi_1(\mathfrak m)), \eta\rangle - \langle s(\pi_2(\mathfrak m)), \eta\rangle \bigg\rvert < 6 \epsilon + \frac{2n^3 - 2(n-k)^3}{n^3}.
\end{equation}
Taking $n\to \infty$ and then $\epsilon\to 0$ completes the proof.

\end{proof}

To prove Theorem \ref{theorem: entropy maximizer gibbs}, we mimic the perturbative argument of \cite[Proposition 1.19]{burton1994non}, applying Lemma \ref{lem:no_infinite_paths} to show that this does not change the mean current. 

\begin{proof}[Proof of Theorem \ref{theorem: entropy maximizer gibbs}]
It suffices to show that if $\mu\in \mc P^s$ is not a  Gibbs measure, then there exists $\nu\in \mc P^s$ such that $h(\nu) > h(\mu)$. Under the assumption that $\mu$ is not a Gibbs measure, there exists a finite set $R\subset \Z^3$ and a positive measure set $\Omega'\subset \Omega$ such that for all $\tau\in \Omega'$, the conditional distribution of possible extensions of $\tau|_{\Z^3\setminus R}$ to a tiling of $\Z^3$ is not uniform. We can assume using stationarity that $R$ is contained in the positive quadrant. Let $n\in \mathbb N$ be such that $R\subset[1,n-1]^3$.  Since the number of tilings of $R$ depends only on the tiling restricted to $S:=[0,n]^3\setminus R$ there exists a tiling $\tau_0$ in the support of $\mu$ such that the conditional distribution on the possible extensions of $\tau_0|_{S}$ to $R$ is not uniform. Since entropy is maximized by the uniform measure we have that there is $\delta>0$ such that
$$H(\text{uniform distribution on extensions of $\tau_0|_{S}$ to $R$})-H(\mu|_{R}\text{ conditioned on }\tau_0|_S)>\delta.$$

We now construct a modification of $\mu$ and show that it has the same mean current but more entropy. For this take a sample $\tau$ from $\mu$. Let $n$ be an odd integer and divide $\Z^3$ into translates of $B=[0,n]^3$ by $(n+1)\Z^3$. For each such translated box $B$, we resample the tiling in $B$ from the uniform measure on tilings of $B$ with boundary condition $\tau\mid_{\partial^\circ B}$, where $\partial^\circ B = B\setminus [1,n-1]^3\subset B$ is the inner boundary of $B$.
This gives us a new measure $\nu$ on $\Omega$, which is invariant with respect to the $(n+1)\Z^3$ subaction. By averaging $\nu$ with respect to translations by elements of $[0,n]^3\cap \threeeven$ we get a $\threeeven$-invariant measure which we denote $\nu'$. 

Let $\mathfrak m$ be a measure on $\Omega\times \Omega$ which is a coupling of $\mu$ and $\nu'$, where the sample from $\nu'$ is derived by the construction above from the $\mu$ sample. 
If $(\tau_1,\tau_2)$ is sampled from $\mathfrak m$, $\tau_1$ and $\tau_2$ differ only on the interiors of copies of $B$. Therefore $(\tau_1,\tau_2)$ has no infinite paths $\mathfrak m$-a.s., so by Lemma~\ref{lem:no_infinite_paths}, 
\begin{align*}
    s(\nu') = s(\mu). 
\end{align*} 
On the other hand by the ergodic theorem, there exists $\epsilon>0$ for which there is a $(n+1)\Z^3$-invariant set $A \subset \Omega$ with the following properties:
\begin{enumerate}
	\item $\mu(A)>\epsilon$.
	\item For all $\tau\in A$, $\tau_0|_S$ appears in translated boxes $B$ with density greater than $\epsilon$. 
\end{enumerate}
Therefore
$$h(\nu')-h(\mu)> \frac{1}{(n+1)^3}\epsilon^2\delta.$$
\end{proof}

The proof of Theorem \ref{theorem: entropy maximizer gibbs} also has a useful consequence for the double dimer model. Recall the maps $\pi_1, \pi_2: \Omega\times\Omega\to \Omega$ given by $\pi_i(\tau_1, \tau_2)=\tau_i$. Let $\Prob^{s_1, s_2}$ be the space of invariant probability measures $\mu$ on $\Omega\times \Omega$ such that $\pi_i(\mu)\in \Prob^{s_i}$ for $i=1,2$. \symindex{Chapter 7!$\Prob^{s_1, s_2}$, - the space of probability measures on the double dimer model with mean-currents $s_1$ and $s_2$ and the two projections $\pi_1, \pi_2$}

\begin{cor}\label{Corollary: entropy maximizer double dimer}
	Let $s_1, s_2\in \mathcal O$. Then
	$$\sup_{\mu\in \Prob^{s_1, s_2}}h({\mu})= \ent(s_1)+\ent(s_2).$$ Further, the measures $\mu\in \Prob^{s_1, s_2}$ which maximize specific entropy on $\mc P^{s_1, s_2}$ are Gibbs measures on $\Omega\times \Omega$ and satisfy $h({\pi_i(\mu)})=\ent(s_i)$ for $i=1,2$. 
	\end{cor}

\begin{proof}
For any measure $\mu\in \mc P^{s_1,s_2}$ we have that 
\begin{equation}\label{eq:h_product}
    h(\mu)\leq h({\pi_1(\mu)})+h({\pi_2(\mu)})\leq \ent(s_1)+\ent(s_2).
\end{equation}
For $i=1,2$, let $\nu_{i}\in \Prob^{s_i}$ be such that $h({\nu_i})=\ent(s_i)$. The product measure $\nu=\nu_1\times \nu_2$ has $h(\nu) = \ent(s_1) + \ent(s_2)$, so by Equation \eqref{eq:h_product} $\nu$ maximizes specific entropy among measures in $\Prob^{s_1, s_2}$. Therefore if $\mu$ is a maximizer it must be in the equality case in Equation \eqref{eq:h_product}, which implies that $h(\pi_1(\mu)) = \ent(s_1)$ and $h(\pi_2(\mu)) = \ent(s_2)$. 

Finally the proof that the entropy maximizer must be a Gibbs measure is exactly the same as the proof of Theorem \ref{theorem: entropy maximizer gibbs}; if a measure in $\Prob^{s_1, s_2}$ is not a Gibbs measure then we can increase its entropy by locally modifying it.
\end{proof}

\subsection{Basic properties of ent}\label{subsection: easy observations}

In this section we prove some straightforward properties of the mean current entropy function $\ent$. The only tools here are basic real analysis and properties of $h(\cdot)$. 

\begin{lemma}\label{lemma: entropy_concave}
	$\text{ent}$ is a concave function on $\mc O$. 
\end{lemma}
\begin{proof}
	Fix $u, v\in \mathcal O$ and $\alpha\in (0,1)$. By Lemma \ref{lemma: entropy maximiser} we know that the entropy function $h(\cdot)$ on $\mathcal P^{u}$ (resp. on $\mathcal P^{v}$) achieves a maximum say at $\mu$ (resp. at $\nu$). Given this we have that $\alpha \mu +(1-\alpha) \nu \in \mathcal P^{\alpha u+ (1-\alpha)v}$ and 
	$$h({\alpha \mu +(1-\alpha) \nu}) =\alpha h({\mu}) + (1-\alpha)h({\nu})=\alpha \; \text{ent}(u) + (1-\alpha) \text{ent}(v) .$$
	Thus 
	$$\text{ent}(\alpha u +(1-\alpha) v)\geq \alpha\; \text{ent}(u) + (1-\alpha)\; \text{ent}(v)$$
	which shows that $\text{ent}$ is concave.
\end{proof}

\begin{lemma}\label{lem:ent_upper_semi_cont}
	$\text{ent}$ is an upper semi-continuous function on $\mc O$.
\end{lemma}

\begin{proof}
	Let $u_n\in \mathcal O$ be a sequence such that $u_n\to  u$. By Lemma \ref{lemma: entropy maximiser} there exist a measure $\mu_n$ maximizing the entropy function $h(\cdot)$ on $\Prob^{u_n}$. Since the mean current is a continuous function of the measure (Definition \ref{def: mean current}), we have that any subsequential limit $\mu$ of $\mu_n$ must lie in $\Prob^{u}$. Since $h(\cdot)$ is upper semicontinuous as a function on $\mc P$,
	$$\text{ent}(u)\geq h(\mu)\geq \limsup_{n \to \infty} h({\mu_n})= \limsup_{n\to \infty} \text{ent}(u_n),$$
which completes the proof.
\end{proof}

We put these together to show that $\ent$ is continuous.

\begin{lemma}\label{lemma: entropy_continuous}
	$\text{ent}$ is a continuous function on $\mc O$.
\end{lemma}

\begin{proof}
	Fix $u\in \mathcal O$ and $\epsilon>0$. Let $M:=\sup_{v \in \mathcal O} \text{ent}(v)$; $M$ is finite since $\mc O$ is compact and $\ent$ is upper semicontinuous on $\mc O$ (Lemma \ref{lem:ent_upper_semi_cont}). Let $\norm{\cdot}_1$ denote the $L^1$ norm. Again using Lemma \ref{lem:ent_upper_semi_cont}, there exists $\delta_1>0$ such that if $v\in \mathcal O$ is such that $\|v-u\|_1<\delta_1$, then
	\begin{equation*}
	\text{ent}(u)-\text{ent}(v)> -\epsilon.    
	\end{equation*}
	Choose $L>1+ (2M/\epsilon)$ and $\delta_2>0$ small enough such that if $\|v- u\|_1<\delta_2$ then 
	$u+ L(v -u)\in \mathcal O$ (note that this is possible even when $u\in \partial \mc O$ since $\mc O$ is convex).
	By Lemma \ref{lemma: entropy_concave}, 
	$$\frac{L-1}{L}\text{ent}(u)+\frac{1}{L}\text{ent}(u+ L(v -u))\leq \text{ent}(v).$$
	Rearranging the equation we get that 
	\begin{equation*}
	\text{ent}(u)-\text{ent}(v)\leq \frac{1}{L}\text{ent}(u)-\frac{1}{L}\text{ent}(u+ L(v -u))\leq 2M/L\leq \epsilon.
	\end{equation*}
	Taking $\delta<\delta_1\delta_2$, if $\|v -u\|_1<\delta$ then $|\text{ent}(u)-\text{ent}(v)|<\epsilon$, which proves that $\text{ent}$ is continuous.
\end{proof}

\subsection{Flows for the double dimer model}\label{subsection: flows for double dimer}

To prove that $\ent$ is strictly concave we use a double dimer model construction called \textit{chain swapping}, which is an operation on infinite paths in a double dimer configuration related to the cluster swapping technique used in \cite{AST_2005__304__R1_0}. In this section we give some of the necessary background results for the double dimer model, and in the next section we explain what chain swapping is.

Here we look at $\threeeven$-invariant couplings of $\threeeven$-invariant measures $\mu_1, \mu_2\in \Prob_e$ and study properties of the sample $(\tau_1, \tau_2)$. We distinguish between the pair of dimer tilings $(\tau_1,\tau_2)$ and the union of tilings $\tau_1\cup \tau_2$, where we forget the information of which tiling each edge $e\in \tau_1\cup \tau_2$ belongs to. As we observed earlier, $\tau_1\cup \tau_2$ is a set of isolated double edges, finite cycles, and infinite paths. We saw in Lemma \ref{lem:no_infinite_paths} that if $\tau_1\cup \tau_2$ consists of only finite cycles and double edges, then the marginals have the same mean current. This suggests that the infinite paths in a double dimer configuration carry a lot of information about the difference between the  mean currents of the measures involved. The main results of this section are Proposition \ref{prop: line slope integrates to mean current difference} and Corollary \ref{cor:new 7.15} which make this precise.

Recall from Section \ref{section:tiling_flows} that the \textit{flow associated with a double dimer configuration} $(\tau_1,\tau_2)$ is
$$f_{(\tau_1, \tau_2)}=v_{\tau_1}-v_{\tau_2},$$
where $v_\tau$ is the pretiling flow defined in Section \ref{section:tiling_flows}, Equation \eqref{eq:v_tau}. (Equivalently, $f_{(\tau_1, \tau_2)} = f_{\tau_1} - f_{\tau_2}$ where $f_\tau$ is the tiling flow, since the reference flows cancel.) Explicitly, for each edge $e$ oriented from even to odd,
 $$f_{(\tau_1, \tau_2)}(e)=\begin{cases}1&\text{ if }e\in \tau_1\setminus \tau_2\\
 -1&\text{ if }e\in \tau_2\setminus \tau_1\\
 0&\text{ if }e\in \tau_1\cap \tau_2\text{ or if }e\not\in \tau_1\cup\tau_2.
 \end{cases}$$
The vector field $f_{(\tau_1,\tau_2)}$ is divergence free, and its flow lines are the cycles and paths of the double dimer configuration $\tau_1\cup \tau_2$. In particular, each $x\in \m Z^3$ is in one of two cases:
\begin{enumerate}
    \item $f_{(\tau_1,\tau_2)}$ is equal to $1$ on exactly two edges $e_1,e_2$ incident to $x$, with one of the edges oriented into $x$ and the other oriented out of $x$.  
    \item $f_{(\tau_1,\tau_2)}$ is zero on all edges $e$ incident to $x$. 
\end{enumerate}
The set of vertices $x\in \m Z^3$ in Case 2 is the collection of vertices covered by $\tau_1\cap \tau_2$. In particular, we note that it is tileable by dimers. 

Conversely, if a discrete vector field $g$ satisfies these properties (i.e.\ all vertices are in Case 1 or Case 2, and the set of Case 2 vertices is tileable by dimers), then there exist tilings $\tau_1,\tau_2$ such that $g = f_{(\tau_1,\tau_2)}$. In fact we can explicitly construct a pair of tilings with double dimer flow $g$. Given any tiling $\tau$ of the Case 2 vertices $\{v\in \m Z^3 : g(e) = 0 \text{ for all } e \text{ incident to }v\}$, we define
\begin{eqnarray*}
\tau_1&=& \tau\cup\{e~:~g(e)=1\text{ where }e\text{ is an edge directed from an even to an odd vertex}\}\\
\tau_2&=& \tau\cup\{e~:~g(e)=-1\text{ where }e\text{ is an edge directed from an even to an odd vertex}\}.
\end{eqnarray*}
From this we see that that the flow $f_{(\tau_1,\tau_2)}$ determines the double dimer configuration $(\tau_1,\tau_2)$ up to the choice of tiling $\tau$ on the Case 2 vertices. On other other hand, the union of tilings $\tau_1\cup \tau_2$ determines $(\tau_1,\tau_2)$ on the set where $\tau_1 = \tau_2$, meaning it determines the tiling $\tau$ of the Case 2 vertices. Therefore together these are enough to recover $(\tau_1,\tau_2)$. In summary, we have shown:
\begin{prop}\label{prop: injection_tiling_to_flow}
	The pair $(\tau_1 \cup \tau_2, f_{(\tau_1,\tau_2)})$ uniquely determines the double dimer configuration $(\tau_1,\tau_2)$ and vice versa.
\end{prop}

\textbf{Shifting along flow lines.} We define a $\m Z$-action on $\Omega\times \Omega$ by translating in the direction of the double dimer flow $f_{(\tau_1,\tau_2)}$. Given $(\tau_1,\tau_2)\in \Omega\times \Omega$, let $b_1\in \tau_1$ be the edge incident to the origin, and suppose that $b_1 = (0,a_1)$, $a_1\in \m Z^3$ a neighbor of the origin. Following that, let $b_2\in \tau_2$ be the edge incident to $a_1$, and suppose that $b_2 = (a_1,a_1+a_2)$, where $a_2\in \m Z^3$ a neighbor of the origin. These are the first two edges of a path in $(\tau_1, \tau_2)$. We define $\alpha_1(\tau_1,\tau_2)$ to be the directed vector $a_1$ and $\alpha_2(\tau_1,\tau_2)$ to be the directed vector $a_2$. When the pair of tilings $(\tau_1,\tau_2)$ is understood, we drop them from the notation. We then define the function $F: \Omega\times \Omega \to [-2,2]^3$ by \symindex{Chapter 7!$F: \Omega\times \Omega \to [-2,2]^3$}
\begin{align*}
    F((\tau_1,\tau_2)) = \alpha_1 + \alpha_2.
\end{align*}
We define $F$ to be translation by two edges so that the parity of the even/odd vertices is preserved. This can be viewed as tracking the slope and speed of the flow $f_{(\tau_1,\tau_2)}$ (when there is a double edge at the origin in $(\tau_1,\tau_2)$, $F$ is $0$). Finally we define a transformation $T:\Omega\times \Omega\to \Omega\times\Omega$ given by translating the double dimer tiling by $\alpha_1+\alpha_2$. If $(\tau_1,\tau_2)$ has a double edge at the origin, then $T((\tau_1,\tau_2)) = (\tau_1,\tau_2)$. Otherwise, $T$ shifts $(\tau_1,\tau_2)$ along the path through the origin. The corresponding flow $T(f_{(\tau_1,\tau_2}))(e) = f_{(\tau_1,\tau_2)}(e+\alpha_1 + \alpha_2)$. 
If $\mu$ is a $\threeeven$-invariant measure on $\Omega\times \Omega$, then it is also $T$-invariant. Thus by the ergodic theorem we have that for $\mu$ almost every $(\tau_1,\tau_2)\in \Omega\times \Omega$,
$$\lim_{N\to \infty} \frac{1}{N}\sum_{i=1}^NF(T^i((\tau_1,\tau_2)))=: F^{\star}((\tau_1,\tau_2))$$
exists. Further, $F^{\star}$ is invariant under $T$ and 
$$\int_{\Omega\times\Omega}F^{\star}((\tau_1,\tau_2))\,\dd\mu((\tau_1,\tau_2))= \int_{\Omega\times\Omega}F((\tau_1,\tau_2))\,\dd\mu((\tau_1,\tau_2)).$$
By construction, $F^\star$ measures the slope or asymptotic direction of the 
path $\gamma_0$ through the origin in $(\tau_1,\tau_2)$. We call $F^{\star}$ the \textit{slope function}. If $\gamma_0$ is a double edge or finite cycle, then $F^{\star}$ is $0$. If $\gamma_0$ is an infinite path, then it can have nonzero slope. \symindex{Chapter 7!$F^{\star}$ - ``slope function''}

For any infinite path $\ell \subset (\tau_1,\tau_2)$ we can compute its slope by translating so that $\ell$ goes through the origin. We say that $\ell$ has \textit{nonzero slope} if $F^\star((\widetilde\tau_1,\widetilde\tau_2))\neq 0$, where $(\widetilde\tau_1,\widetilde\tau_2)$ is a translation of $(\tau_1,\tau_2)$ so that $\ell$ contains the origin (this is well-defined since $F^\star$ is $T$-invariant). With this we can prove the main result of this section. 

\begin{prop}\label{prop: line slope integrates to mean current difference}
Let $\mu$ be a measure on $\Omega\times \Omega$ which is a $\threeeven$-invariant coupling of $\threeeven$-invariant measures $\mu_1$ and $\mu_2$ on $\Omega$. Then  
 \begin{align*}
     \int_{\Omega\times \Omega} F^{\star}((\tau_1, \tau_2))\,\dd\mu((\tau_1, \tau_2))= s(\mu_1)-s(\mu_2).
 \end{align*}
\end{prop}
\begin{proof} 
Since $\mu$ is $T$-invariant, 
$$\int_{\Omega \times\Omega}F^{\star}((\tau_1, \tau_2))\,\dd\mu((\tau_1, \tau_2))= \int_{\Omega\times\Omega}(\alpha_1(\tau_1, \tau_2)+\alpha_2(\tau_1, \tau_2))\,\dd\mu((\tau_1, \tau_2)).$$
Since $\alpha_1(\tau_1, \tau_2)$ is the vector along the edge adjacent to $0$ in $\tau_1$ pointing away from it, it depends only on $\mu_1$. Hence
$$\int_{\Omega\times\Omega}\alpha_1(\tau_1, \tau_2) \dd\mu((\tau_1, \tau_2))= \int_{\Omega}\alpha_1(\tau_1, \tau_2) \dd\mu_1(\tau_1)=s(\mu_1).$$
The vector $\alpha_2(\tau_1,\tau_2)$ is defined similarly, but first we have to sum over the possible values of $\alpha_1$. 
$$\int_{\Omega\times\Omega}\alpha_2(\tau_1, \tau_2) \,\dd\mu ((\tau_1, \tau_2))= \sum_{a_1,a_2\in \star}a_2 \,\mu((0,a_1)\in \tau_1,(a_1,a_1+a_2)\in \tau_2)$$
where $\star$ is the six neighbors of the origin. By the $\threeeven$-invariance of $\mu$ we get that
\begin{eqnarray*}
	\int_{\Omega\times\Omega}\alpha_2(\tau_1, \tau_2) \,\dd\mu((\tau_1, \tau_2))&=& \sum_{a_1,a_2\in \star} a_2\,\mu((-a_1-a_2,-a_2)\in \tau_1, (-a_2,0)\in \tau_2)\\
	&=&\sum_{a_2\in \star}a_2 \mu_2((-a_2, 0)\in \tau_2)=-s(\mu_2).
	\end{eqnarray*}
This completes the proof.
	\end{proof}
 
As a corollary, we show that the mean current difference of a pair of measures $(\mu_1,\mu_2)$ can be computed by looking only at the tiles on infinite paths of nonzero slope. As a consequence, note also that if $s(\mu_1)\neq s(\mu_2)$ then an invariant coupling must have order $n^3$ tiles along infinite paths of nonzero slope. Here recall that $s_0(\tau)$ denotes the tile direction at the origin in $\tau$ and that for a measure $\mu_1$ on $\Omega$, $s(\mu_1) = \m E_{\mu}[s_0(\tau)]$.
 \begin{cor}\label{cor:new 7.15}
     Let $\mu$ be a measure on $\Omega\times \Omega$ which is a $\threeeven$-invariant coupling of $\threeeven$-invariant measures $\mu_1$ and $\mu_2$ on $\Omega$. Let $I_0$ be the event that the origin is contained in an infinite path of nonzero slope in $(\tau_1,\tau_2)$, and let $$s(\mu_1,I_0)-s(\mu_2,I_0) = \m E_{\mu}[(s_0(\tau_1)-s_0(\tau_2))\mathbbm{1}_{I_0}((\tau_1,\tau_2))].$$ Then 
     \begin{align*}
         s(\mu_1)-s(\mu_2) = s(\mu_1,I_0)-s(\mu_2,I_0).
     \end{align*}
 \end{cor}
 \begin{proof}
     Note that 
     \begin{align*}
            \int_{\Omega\times \Omega} F^{\star}((\tau_1, \tau_2))\,\dd\mu((\tau_1, \tau_2)) = \int_{\Omega\times \Omega} F^{\star}((\tau_1, \tau_2))\mathbbm{1}_{\{F^\star\neq 0\}}((\tau_1,\tau_2))\,\dd\mu((\tau_1, \tau_2))
     \end{align*}
     By Proposition \ref{prop: line slope integrates to mean current difference}, the left hand side is equal to $s(\mu_1)-s(\mu_2)$. Since $I_0=\{F^{\star}\neq 0\}$, the right hand side is equal to $s(\mu_1,I_0)-s(\mu_2,I_0)$. 
 \end{proof}

Finally we observe that the number of infinite paths of nonzero slope in $(\tau_1,\tau_2)$ that intersect two far away boxes is $0$ with probability $1$ as the distance between the boxes goes to $\infty$. This serves as a lemma for Proposition \ref{prop:preserve_ergodicity}. 

\begin{lemma}\label{prop: infinite lines zero measure}
Let $\mu$ be a $\threeeven$-invariant probability measure on $\Omega\times \Omega$, and fix $m\in \mathbb N$. Given a sample $(\tau_1, \tau_2)$ from $\mu$ and $x,y\in \threeeven$, let $L_{x,y}$ denote the number of infinite paths of nonzero slope in $(\tau_1, \tau_2)$ which intersect both $x+ [1,m]^3$ and $y+[1,m]^3$. Then 
\begin{align*}
    \lim_{n\to \infty} \frac{1}{n^3} \sum_{x\in [1,n^3]} \mu( L_{0,x} = 0 ) = 1.
\end{align*}
\end{lemma}

\begin{proof} 
There are at most $m^3$ infinite paths of nonzero slope in $(\tau_1,\tau_2)$ passing through $[1,m]^3$. On the other hand, for any infinite path $\ell$ with nonzero slope, 
$$|\ell \cap [1,n]^3| = O(n).$$ 
Therefore each infinite path $\ell$ of nonzero slope intersecting $[1,m]^3$ intersects $x + [1,m]^3$ for at most $O(n)$ points $x\in [1,n]^3$. Since $m^3$ is a constant, this implies that the number of $x\in [1,n]^3$ such that $L_{0,x} \neq 0$ is also $O(n)$. We can rewrite 
$$\frac{1}{n^3}\sum_{x\in [1,n^3]}  \mu (L_{0, x}\neq 0)=\m E_{\mu}\left(\frac{1}{n^3}\sum_{x\in [1,n^3]}  \mathbbm{1}_{(L_{0, x}\neq 0)}(\tau_1,\tau_2)\right).$$
By the dominated convergence theorem, the right hand side tends to $0$ as $n\to \infty$. This completes the proof.
\end{proof}

\subsection{Chain swapping}\label{Subsection: Dimer Swapping}

We can now introduce the main tool of this section, namely \textit{chain swapping}, which is an operation on double dimer configurations similar to the \textit{cluster swapping} technique used in \cite[Chapter 8]{AST_2005__304__R1_0}. 

Let $(\tau_1,\tau_2)\in \Omega\times \Omega$ be a pair of dimer tilings. Corresponding to this are the collection of (unoriented) loops $\tau_1\cup \tau_2$ and the double dimer flow $f_{(\tau_1,\tau_2)}$. The flow $f_{(\tau_1,\tau_2)}$ determines the orientation of each loop or infinite path in $\tau_1\cup \tau_2$. 

For a fixed $p\in(0,1)$, from a random configuration $(\tau_1,\tau_2)$ we define a new random pair $(\tau_1',\tau_2')$ by ``shifting" the tiles along each infinite path of nonzero slope $\ell\subset (\tau_1,\tau_2)$ with independent probability $p$. In terms of the flow $f_{(\tau_1,\tau_2)}$, shifting on the infinite path $\ell$ corresponds to flipping the sign of $f_{(\tau_1,\tau_2)}$ along $\ell$. The new tilings $\tau_1',\tau_2'$ have the following properties:
\begin{enumerate}
    \item $\tau_1\cup \tau_2 = \tau_1'\cup \tau_2'$, i.e.\ they correspond to the same collection of double edges, finite loops, and infinite paths. 
    \item Let $\ell_1,\ell_2,...$ be the infinite paths of nonzero slope in $(\tau_1, \tau_2)$. Independently for each $i$, either with probability $1-p$ the tiles on $\ell$ were not swapped, in which case
    $$\tau_1'\cap \ell_i = \tau_1 \cap \ell_i, \qquad \tau_2' \cap \ell_i = \tau_2 \cap \ell_i$$
    or with probability $p$ the tiles were swapped, in which case 
    $$\tau_2'\cap \ell_i = \tau_1 \cap \ell_i, \qquad \tau_1' \cap \ell_i = \tau_2 \cap \ell_i.$$
    \item On the complement of the infinite paths with nonzero slope in $(\tau_1,\tau_2)$, $\tau_1'$ is equal to $\tau_1$ and $\tau_2'$ is equal to $\tau_2$. 
\end{enumerate}
\begin{figure}
    \centering
    \includegraphics[scale=0.5]{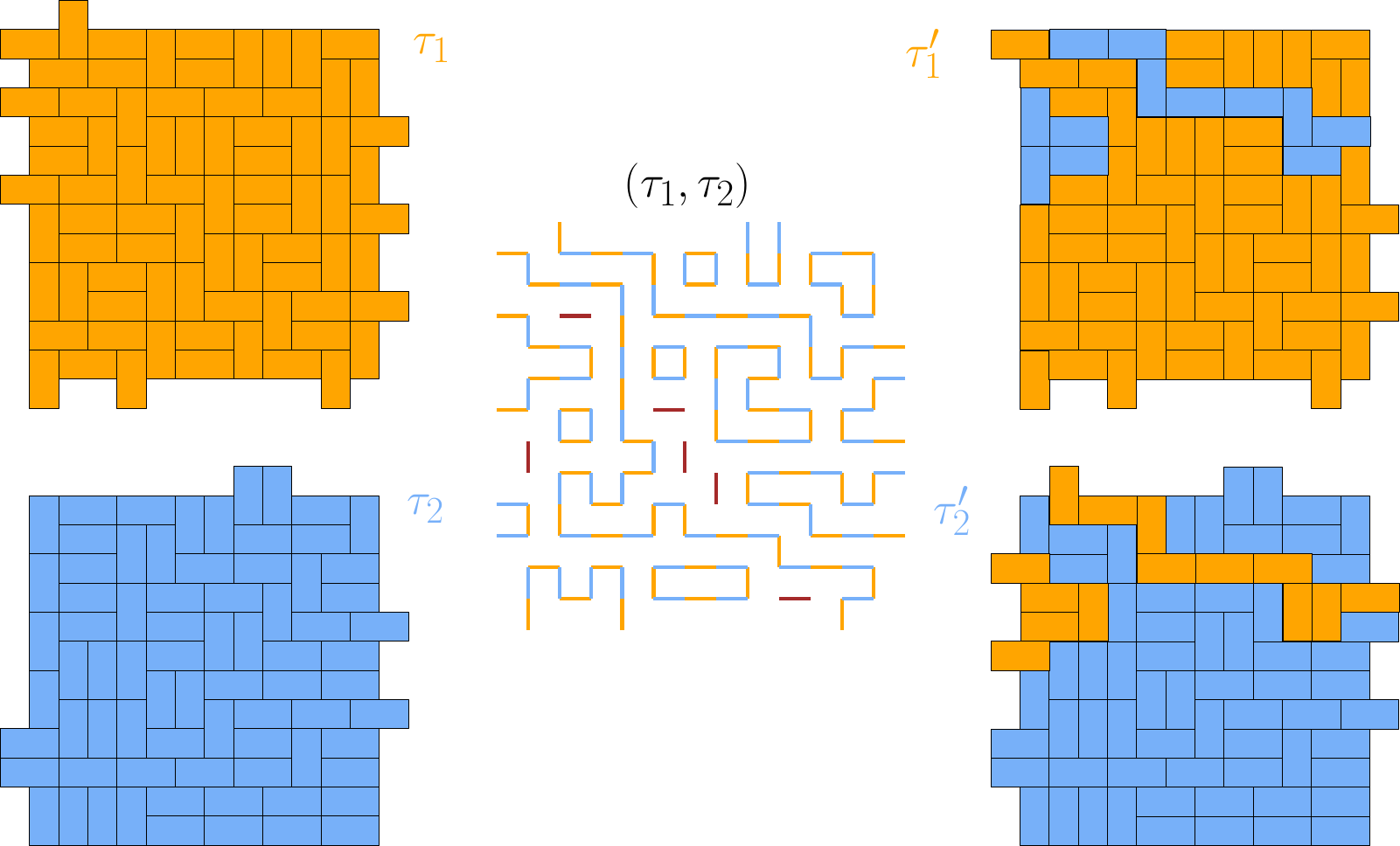}
    \caption{An example of tilings $\tau_1,\tau_2$, the loops in $(\tau_1,\tau_2)$, and chain swapped tilings $\tau_1',\tau_2'$.}
    \label{fig:chain swap}
\end{figure}\symindex{Chapter 7! given a pair of tilings $(\tau_1, \tau_2)$, the swapped tiles are usually represented by $(\tau_1', \tau_2')$}\symindex{Chapter 7! given a measure $\mu$ with marginals $\mu_1, \mu_2$ on double dimer configurations, the swapped measure is usually denoted by  $\mu'$ with marginals $\mu_1'$, $\mu_2'$}

We call this procedure \emph{chain swapping with probability $p$}. See Figure \ref{fig:chain swap}. Chain swapping transforms a measure $\mu$ on $\Omega\times \Omega$ into a new measure $\mu'$ on $\Omega\times \Omega$ which we call the \textit{swapped measure}.\termindex{Chapter 7!chain swapping, swapped measure}

\begin{rem}
    Note that we only swap on the infinite paths of \textit{nonzero slope}. This is a technical point. We do not know if the ``asymptotic independence" result of Lemma \ref{prop: infinite lines zero measure} holds for infinite paths of zero slope, but we need asymptotic independence to show that the swapped measure is still ergodic (Proposition \ref{prop:preserve_ergodicity}). 
\end{rem}

For the rest of this section, we study whether or not certain properties (ergodicity, the Gibbs property) are preserved under chain swapping, and how certain quantities (entropy, mean current) transform under chain swapping. 

The first result is that chain swapping preserves ergodicity. 

\begin{prop}\label{prop:preserve_ergodicity}
If $\mu$ is a ergodic measure on $\Omega\times \Omega$ with respect to the $\threeeven$ action and $\mu'$ is obtained from $\mu$ by chain swapping with probability $p$, then $\mu'$ is also ergodic.
\end{prop}
\begin{proof}
    In this proof, we have two different parameters $n$ (parameterizing possible translations of boxes) and $m$ (the size of the boxes). Let $B_n = [1,n]^3$ and $\text{even}(B_n) = B_n\cap \threeeven$. Let $B_m+x$ denote $B_m$ translated by $x\in \threeeven$. It is enough to show that for any two double dimer patterns restricted to $B_m$, denoted $\Pi_1,\Pi_2$,
    \begin{align*}
        &\lim_{n\to \infty} \frac{1}{\text{even}(B_n)} \sum_{x\in \text{even}(B_n)} \mu'\bigg((\tau_1',\tau_2')\mid_{B_m} = \Pi_1, (\tau_1',\tau_2')\mid_{B_m+x} = \Pi_2\bigg) \\
        &= \mu'\bigg((\tau_1',\tau_2')\mid_{B_m} = \Pi_1\bigg) \mu'\bigg((\tau_1',\tau_2')\mid_{B_m+x} = \Pi_2\bigg).
    \end{align*}
    Define the random variable $L_x$ to be the number of infinite paths of nonzero slope in $(\tau_1, \tau_2)$ sampled from $\mu$ which intersect $B_m + x$. Similarly define $L_{0,x}$ to be the number of infinite paths of nonzero slope which intersect $B_m$ and $B_m+x$. Since the collection of tiles on infinite paths of nonzero slope is the same for $\mu$ and $\mu'$, the quantities $L_0$ and $L_{0,x}$ are preserved by chain swapping. Let $(\tau_1,\tau_2)$ have law $\mu$ and $(\tau_1',\tau_2')$ have law $\mu'$. Finally let $\mathfrak m$ be the coupling of $\mu,\mu'$ given by chain swapping. Then
    $$
        \frac{1}{\text{even}(B_n)} \sum_{x\in \text{even}(B_n)} \mu'((\tau_1',\tau_2')\mid_{B_m} = \Pi_1, (\tau_1',\tau_2')\mid_{B_m+x} = \Pi_2)\\
        $$ 
        \begin{align*}
            = \frac{1}{\text{even}(B_n)} \sum_{\substack{k_1,k_2,k_3 \geq 0 \\ \Sigma_1,\Sigma_2 \text{ double dimer} \\ \text{tilings of }B_m}}  \sum_{x\in \text{even}(B_n)} \mathfrak m\bigg(&(\tau_1,\tau_2)\mid_{B_m} = \Sigma_1, (\tau_1,\tau_2)\mid_{B_m+x} = \Sigma_2, \\
        &(\tau_1',\tau_2')\mid_{B_m} = \Pi_1, (\tau_1',\tau_2')\mid_{B_m+x} =  \Pi_2, \\
        &L_0=k_1, L_x=k_2,
        L_{0,x} = k_3\bigg).
        \end{align*}
       
    For each infinite path of nonzero slope in $(\tau_1, \tau_2)$ we have an independent probability $p$ of reversing its direction. For any triple $l =(l_1,l_2,l_3)$ with $l_i\leq k_i$ for each $i$, we define the notation
    \begin{align*}
        q_{k,l} = p^{l_1+l_2-l_3}(1-p)^{k_1+k_2-k_3-l_1-l_2+l_3}.
    \end{align*}
    This is the probability of switching $(l_1,l_2,l_3)$ of the $(k_1,k_2,k_3)$ paths. With this notation, for each $x\in B_n$, the $x$ term in the sum above is equal to
    \begin{equation}\label{eq:x_term}
      \sum_{\substack{k_1,k_2,k_3 \geq 0 \\ \Sigma_1,\Sigma_2 \text{ double dimer tilings} \\ \text{ of } B_m
        \text{which can swap to} \\ \Pi_1,\Pi_2 \text{with }(l_1,l_2,l_3) \text{ swaps}}}  \!\!\!\!\!\!  \mu\bigg((\tau_1,\tau_2)\mid_{B_m} = \Sigma_1, (\tau_1,\tau_2)\mid_{B_m+x} = \Sigma_2, L_0=k_1, L_x = k_2, L_{0,x} = k_3\bigg) q_{k,l}.
    \end{equation}
    For any $K> 0$, 
    \begin{align*}
        \frac{1}{\text{even}(B_n)} \sum_{x\in \text{even}(B_n)} \bigg( \text{$k_3 = K$ term in Equation \eqref{eq:x_term}} \bigg) \leq \frac{1}{\text{even}(B_n)}\sum_{x\in \text{even}(B_n)} \mu(L_{0,x} = K). 
    \end{align*}
    By Proposition \ref{prop: infinite lines zero measure}, the right hand side goes to $0$ as $n\to \infty$. Therefore in the limit as $n\to \infty$, it suffices to consider the terms where $k_3 = 0$ (corresponding to the set of lines hitting $B_n$ and the set of infinite paths hitting $B_n + x$ being disjoint). 
    Therefore 
$$
        \frac{1}{\text{even}(B_n)} \sum_{x\in \text{even}(B_n)} \mu'\bigg((\tau_1',\tau_2')\mid_{B_m} = \Pi_1, (\tau_1',\tau_2')\mid_{B_m+x} = \Pi_2\bigg) $$ $$
        =\!\!\!\! \sum_{\substack{k_1,k_2 \geq 0 \\ (\Sigma_1,\Sigma_2) \text{ double dimer on }B_m 
        \text{which can swap to $\Pi_1,\Pi_2$}\\\text{with }(l_1, l_2)\text{ swaps}}}   \\ $$ $$\frac{1}{\text{even}(B_n)} \sum_{x\in \text{even}(B_n)} \mu\bigg((\tau_1,\tau_2)\mid_{B_m} = \Sigma_1, (\tau_1,\tau_2)\mid_{B_m+x} = \Sigma_2, L_0 = k_1, L_x = k_2\bigg) r_{k,l} + o(1),
$$
    where $r_{k,l} = p^{l_1+l_2} (1-p)^{k_1+k_2-l_1-l_2}$ (i.e.\ $q_{k,l}$ when $k_3 =0$). Since $\mu$ is ergodic, for each $\Sigma_1,\Sigma_2,k_1,k_2$, 
    \begin{align*}
        &\lim_{n\to \infty}\,\frac{1}{\text{even}(B_n)} \sum_{x\in \text{even}(B_n)} \mu\bigg((\tau_1,\tau_2)\mid_{B_m} = \Sigma_1, (\tau_1,\tau_2)\mid_{B_m+x} = \Sigma_2, L_0 = k_1, L_x = k_2\bigg) r_{k,l} \\
        &= \mu\bigg((\tau_1,\tau_2)\mid_{B_m} = \Sigma_1, L_0 = k_1\bigg) \mu \bigg( (\tau_1,\tau_2)\mid_{B_m+x} = \Sigma_2, L_x = k_2\bigg) r_{k,l}
    \end{align*}
    Therefore 
    \begin{align*}
        &\lim_{n\to \infty}\, \frac{1}{\text{even}(B_n)} \sum_{x\in \text{even}(B_n)} \mu'\bigg((\tau_1',\tau_2')\mid_{B_m} = \Pi_1, (\tau_1',\tau_2')\mid_{B_m+x} = \Pi_2\bigg) \\
        &= \!\!\! \sum_{\substack{k_1,k_2 \geq 0 \\ \Sigma_1,\Sigma_2 \text{ double dimer tilings of }B_m\\
        \text{which can swap to $\Pi_1,\Pi_2$}\\\text{with }(l_1,l_2) \text{ swaps}}}  \!\! \mu\bigg((\tau_1,\tau_2)\mid_{B_m} = \Sigma_1, L_0 = k_1\bigg) \mu \bigg( (\tau_1,\tau_2)\mid_{B_m+x} = \Sigma_2, L_x = k_2\bigg) r_{k,l} \\
        &=\mu'\bigg((\tau_1',\tau_2')\mid_{B_m} = \Pi_1\bigg) \mu'\bigg((\tau_1',\tau_2')\mid_{B_m+x} = \Pi_2\bigg).
    \end{align*}
\end{proof}

We now see how chain swapping affects the entropy and mean current of the marginal distributions. 

\begin{prop}\label{prop: coupling swap entropy}
	Let $\mu$ be a measure on $\Omega\times \Omega$ which is an ergodic coupling of ergodic measures $\mu_1\in \Prob_e^{s_1}$ and $\mu_2\in \Prob_e^{s_2}$. If $\mu'$ is the measure obtained from $\mu$ by chain swapping with probability $p\in (0,1)$, then $h(\mu')= h(\mu)$.
	\end{prop}
\begin{proof}
In this proof, for a stationary random field $X$ we let $h(X)$ denote the specific entropy of the law of $X$. Let $(\tau_1,\tau_2)$ be a sample from $\mu$ and $(\tau_1',\tau_2')$ be obtained by chain swapping. By Proposition \ref{prop: injection_tiling_to_flow},
    \begin{align*}
        h(\mu') = h((\tau_1',\tau_2')) = h(\tau_1'\cup \tau_2') + h(f_{(\tau_1',\tau_2')}\mid \tau_1'\cup \tau_2').
    \end{align*}
  Since chain swapping preserves the set of tiles, $\tau_1'\cup \tau_2' = \tau_1 \cup \tau_2$, $h(\tau_1'\cup \tau_2') = h(\tau_1\cup\tau_2)$. On the other hand note that
\begin{align*}
h(f_{(\tau_1',\tau_2')}, f_{(\tau_1, \tau_2)}\mid \tau_1'\cup \tau_2') &= h(f_{(\tau_1,\tau_2)} \mid \tau_1\cup \tau_2) + h(f_{(\tau_1',\tau_2')} \mid \tau_1\cup \tau_2,
\,f_{(\tau_1,\tau_2)})\\&=
h(f_{(\tau_1',\tau_2')} \mid \tau_1\cup \tau_2) + h(f_{(\tau_1,\tau_2)} \mid \tau_1\cup \tau_2,
\,f_{(\tau_1',\tau_2')})
\end{align*}    
Conditioned on $\tau_1\cup \tau_2$ and $f_{(\tau_1,\tau_2)}$, the distribution of the flow $f_{(\tau_1',\tau_2')}$ is determined by independent random choices for the orientation of each infinite path of nonzero slope in $\tau_1\cup \tau_2$. Let $B_n = [1,n]^3$, and let $\ell\subset \tau_1\cup \tau_2$ be an infinite path of nonzero slope. If $\ell \cap B_n$ is nonempty, then the orientation of $\ell$ is determined by its direction when it intersects $\partial B_n$. Therefore there exists a constant $c>0$ such that 
    \begin{align*}
        h(f_{(\tau_1',\tau_2')} \mid \tau_1\cup \tau_2,
    \,f_{(\tau_1,\tau_2)}) \leq \lim_{n\to \infty}\frac{|\partial B_n|}{|B_n|}\leq \lim_{n\to \infty} \frac{c n^2}{n^3} = 0.
    \end{align*}
 We can analogously show that $ h(f_{(\tau_1,\tau_2)} \mid \tau_1\cup \tau_2,
\,f_{(\tau_1',\tau_2')}) = 0$. Therefore 
$$h(f_{(\tau_1',\tau_2')}\mid \tau_1'\cup \tau_2')= h(f_{(\tau_1,\tau_2)}\mid \tau_1\cup \tau_2)$$
so
    \begin{align*}
        h(\mu') = h(\tau_1\cup \tau_2) + h(f_{(\tau_1,\tau_2)} \mid \tau_1\cup \tau_2) = h(\mu).
    \end{align*}
 \end{proof}
 
\begin{prop}\label{prop: coupling swap mean current}
	Let $\mu$ be a measure on $\Omega\times \Omega$ which is an ergodic coupling of ergodic measures $\mu_1,\mu_2$ with mean currents $s(\mu_1),s(\mu_2)$. If $\mu'$ is the measure obtained from $\mu$ by chain swapping with probability $p\in(0,1)$, then the marginal measures $\mu_1'=\pi_1(\mu')$ and $\mu_2'=\pi_2(\mu')$ have mean currents
	\begin{align*}
	    s(\mu_1') &= (1-p)s(\mu_1) + p s(\mu_2)\\
     s(\mu_2') &= p s(\mu_1) + (1-p)s(\mu_2).
	\end{align*}
	\end{prop}
\begin{proof}
As in the previous section let $I_0$ be the event that the origin is contained on an infinite path of nonzero slope in $(\tau_1,\tau_2)$. By Corollary \ref{cor:new 7.15},
    \begin{align*}
        s(\mu_1) - s(\mu_2) = s(\mu_1,I_0) - s(\mu_2,I_0),
    \end{align*}
    where $s(\mu,I_0)$ is shorthand for the mean current computed as an average over only tilings where the origin is along an infinite path of nonzero slope. Since chain swapping only changes tiles that are contained on infinite paths of nonzero slope,
   \begin{align*}
       s(\mu_1') - s(\mu_1) &= s(\mu_1', I_0) - s(\mu_1, I_0)\\
       s(\mu_2') - s(\mu_2) &= s(\mu_2', I_0) - s(\mu_2, I_0).
   \end{align*}
    On the other hand, since each infinite path of nonzero slope is swapped with independent probability $p$, 
    \begin{align*}
        s(\mu_1', I_0) &= (1-p) s(\mu_1,I_0) + p s(\mu_2,I_0)\\
        s(\mu_2',I_0) &= p s(\mu_1,I_0) + (1-p) s(\mu_2,I_0).
    \end{align*}
    Combining gives
    \begin{align*}
        s(\mu_1') - s(\mu_1) = -p s(\mu_1,I_0) + p s(\mu_2,I_0) = -p s(\mu_1) + ps(\mu_2). 
    \end{align*}
    Therefore $s(\mu_1') = (1-p) s(\mu_1) +p s(\mu_2)$. An analogous calculation gives the result for $s(\mu_2')$.
\end{proof}

Finally we will show that chain swapping does {\textbf{not}} preserve the Gibbs property. To do this, we need two technical lemmas about double dimer configurations. This result is more involved than the other chain swapping results, so for simplicity we only prove this in the $p=1/2$ case. 

Let $\mu$ be a measure on $\Omega\times \Omega$ which is an ergodic coupling of ergodic measures $\mu_1,\mu_2$ on $\Omega$ such that $s(\mu_1) \neq s(\mu_2)$. Let $P$ be a plane with normal vector $\xi$ such that $\langle s(\mu_1)-s(\mu_2),\xi\rangle \neq 0$. Given a sample $(\tau_1,\tau_2)$ from $\mu$, we define the random set of ``last cross points" $C_P$ by
\begin{align*}
    C_P =\{ x\in P : \text{ there is an infinite path of slope $s$, $\langle s, \xi\rangle\neq 0$, in $(\tau_1,\tau_2)$ } \\ \text{ which hits $P$ for the \textit{last} time at $x$}\}.
\end{align*}
We analogously define the random set of ``first cross points" $A_P$ by \symindex{Chapter 7!$C_P, A_P$}
\begin{align*}
    A_P = \{ x\in P : \text{ there is an infinite path of slope $s$, $\langle s, \xi\rangle\neq 0$, in $(\tau_1,\tau_2)$} \\ \text{ which hits $P$ for the \textit{first} time at $x$}\}.
\end{align*}
\begin{lemma}\label{lemma: lines in coupling} 
    With the set up above, for $\mu$-almost every $(\tau_1,\tau_2)$, both
    \begin{align*}
        \lim_{n\to \infty} \frac{|C_P \cap [1,n]^3|}{n^2} \qquad \text{ and } \qquad  \lim_{n\to \infty} \frac{|A_P \cap [1,n]^3|}{n^2}
    \end{align*}
    exist and are greater than $0$.
\end{lemma}
\begin{proof}
    The fact that the limits exist follows from the $\m Z^2$ ergodic theorem applied along $P$. 
    
    The proofs are analogous, so we just present the proof for $C_P$. By Proposition \ref{prop: line slope integrates to mean current difference}, the $\mu$-expected value of the slope along the component $\gamma$ containing the origin in $(\tau_1,\tau_2)$ is $s(\mu_1) - s(\mu_2)$. If $\gamma$ is a double edge or finite cycle then the slope along $\gamma$ is $0$, so the set $S$ of pairs of tilings $(\tau_1,\tau_2)$ such that there is an infinite path with slope in the set $\{ s : \langle s, \xi\rangle \neq 0\}$ through the origin has $\mu(S) > p$ for some $p>0$. 
    
    Since $\mu$ is ergodic with respect to the $\threeeven$ action, it follows that along any $\threeeven$-orbit, the proportion of the orbit in $S$ is $>p$. 
    On the other hand, an infinite path with slope in $\{s :\langle s, \xi \rangle \neq 0\}$ only crosses $P$ finitely many times almost surely. In particular, for any $\delta>0$, there exists $M$ such that
    \begin{align*}
        \mu ( \ell \text{ is an infinite path passing through the}&\text{ origin with slope $\langle s(\ell), \xi\rangle \neq 0$} \\
         &\text{and hits $P$ more than $M$ times} ) < \delta.
    \end{align*}
    Therefore
    \begin{align*}
        \mu\bigg( \lim_{n\to \infty} \frac{|C_P\cap [1,n]^3|}{n^2} > \frac{p}{M} \bigg) \geq 1-\delta,
    \end{align*}
    which completes the proof. 
\end{proof}

The next technical lemma is about the distribution of the distance between \textit{hit points}. Given a plane $P$, let $\alpha\subset (\tau_1,\tau_2)$ be an arc of a path (finite or infinite) between two points in $P$, such that $\alpha$ is disjoint from $P$ except its endpoints $x_\alpha,y_\alpha\in P$. We define the \textit{distance between hits} by 
\begin{align*}
    d_P(\alpha) = \text{dist}(x_\alpha,y_\alpha)
\end{align*}
where $\text{dist}$ denotes $L^1$ distance on $P$. \symindex{Chapter 7!$d_P(\alpha)$}
\begin{lemma}\label{lem: path segment length}
    Let $\mu$ be an ergodic coupling of ergodic measures $\mu_1,\mu_2$ on $\Omega$. Let $B_n=[1,n]^3$. For any $\beta>0$, there exists $M$ such that for all $\theta>0$, there exists $N$ such that if $n\geq N$, then 
    \begin{align*}
        \mu\bigg(\#\{ \alpha \text{ arc of path hitting }P\cap B_n ~:~ d_P(\alpha)>M\} \leq \beta n^2\bigg) > 1 - \theta. 
    \end{align*}
\end{lemma}
\begin{proof}
As there is some probability distribution on the distance between hit points, by $\threeeven$-invariance given $\epsilon>0$ there exists $M$ large enough such that for all $v\in P$,
\begin{equation}\label{eq:length bound through v}
    \mu(x_\alpha = v, d_P(\alpha) > M) < \epsilon.
\end{equation}
For a set of points $A\subset \m Z^3$ let $\text{even}(A), \text{odd}(A)$ denote the subset of even, odd points respectively, and define 
\begin{align*}
    S_n^{\text{even}} &= \frac{2}{n^2}\bigg[\# \{\alpha\subset (\tau_1,\tau_2) ~:~ x_\alpha\in \text{even}(B_n\cap P),\, d_P(\alpha)>M\}\bigg],\\
    S_n^{\text{odd}} &= \frac{2}{n^2}\bigg[\# \{\alpha\subset (\tau_1,\tau_2) ~:~ x_\alpha\in \text{odd}(B_n\cap P),\, d_P(\alpha)>M\}\bigg].
\end{align*}
By the $\m Z_{\text{even}}^2$ ergodic theorem applied along $P$, $\mu$-almost everywhere $S_n^{\text{even}}$ converges to a limit $S^{\text{even}}$ as $n\to \infty$ (and similarly for $S_n^{\text{odd}}$). Further, we get that 
\begin{equation}\label{eq:length through 0 as integral}
        \mu(x_\alpha=0,\, d_P(\alpha)>M) =  \int_{\Omega\times \Omega} S^{\text{even}}(\tau_1,\tau_2)\, \dd \mu(\tau_1,\tau_2)
\end{equation}
and if $v$ is an odd point,
\begin{equation}\label{eq: length through v as integral}
     \mu(x_\alpha=v,\, d_P(\alpha)>M)=\int_{\Omega\times \Omega} S^{\text{odd}}(\tau_1,\tau_2)\, \dd \mu(\tau_1,\tau_2).
\end{equation}
Since $S_n^{\text{even}},S^{\text{even}}\geq 0$, Equation \eqref{eq:length bound through v} and Equation \eqref{eq:length through 0 as integral} (and analogously Equations \eqref{eq:length bound through v}, \eqref{eq: length through v as integral} for the odd case) combine to show that for $n$ large enough,
\begin{align*}
    \mu(S_n^{\text{even}} \leq 3 \epsilon) \geq 1-2\epsilon\qquad \text{ and } \qquad \mu(S_n^{\text{odd}} \leq 3 \epsilon) \geq 1-2\epsilon.
\end{align*}
Putting together the even and odd cases, for $n$ large enough,
\begin{align*}
      \mu\bigg(\#\{ \alpha \text{ arc of path hitting }P\cap B_n ~:~ d_P(\alpha)>M\} \leq 3\epsilon n^2/2\bigg) \geq 1-2\epsilon.
\end{align*}
Choosing $\epsilon$ appropriately given $\beta,\theta$ completes the proof.
\end{proof}

We can now state and prove the theorem about the effect of chain swapping (with probability $p=1/2$) on the Gibbs property.

\begin{thm}\label{theorem: Preservation of the Gibbs prop}
	Let $\nu$ be a Gibbs measure on $\Omega\times \Omega$ which is an ergodic coupling of ergodic measures $\nu_1\in \Prob^{s_1}_e$ and $\nu_2\in \Prob^{s_2}_e$ with $s_1\neq s_2$ and $(s_1+s_2)/2\in \text{Int}(\mc O)$. The measure $\nu'$ obtained from $\nu$ by chain swapping with probability $p=1/2$ is not a Gibbs measure on $\Omega\times \Omega$. 
\end{thm}
\begin{rem}
    The condition $(s_1 + s_2)/2\in \text{Int}(\mc O)$ is necessary in the proof so that we can use the patching theorem (Theorem \ref{patching}).
\end{rem}

\begin{proof}

    Let $B_n = [1,n]^3$. Let $P$ be a coordinate plane with normal vector denoted $\xi$ such that $P\cap B_n$ is a face of $\partial B_n$ (denoted $F$) and such that $\langle s(\nu_1) - s(\nu_2), \xi \rangle \neq 0$. By Lemma \ref{lemma: lines in coupling}, 
    \begin{align*}
        \nu \bigg( \lim_{n\to \infty} \frac{|A_P\cap B_n|}{n^2} > 0 \bigg) = 1.
    \end{align*}
    Recall that $A_P$ is the collection of points $x\in P$ on infinite paths $\ell \subset (\tau_1,\tau_2)$ with $\langle s(\ell),\xi\rangle \neq 0$ such that $x$ is the \textit{first} time that $\ell$ intersects $P$. Given a sample $(\tau_1,\tau_2)$ from $\nu$, we look at the collection of infinite paths $\ell$ satisfying $\langle s(\ell),\xi\rangle \neq 0$. 
    
    The part of $\ell$ outside $B_n$, $\ell\setminus B_n$, always has exactly two infinite components, a \textit{left ray} (half-infinite path entering $B_n$) and a \textit{right ray} (half-infinite path exiting $B_n$). We define \textit{first entrance points} of $B_n$ by
     \begin{align*}
        S_{\text{first}}(B_n)=\{x\in \partial B_n : \text{ there is an infinite path }\ell \subset (\tau_1,\tau_2) \text{ with }\langle s(\ell),{\xi}\rangle \neq 0, \\ \,\ell \text{ enters $B_n$ for the \textit{first} time at $x$}\}.
    \end{align*}
   Note that left rays hit $\partial B_n$ at first entrance points. Similarly define \textit{last exit points} of $B_n$ by 
        \begin{align*}
        X_{\text{last}}(B_n)=\{x\in \partial B_n : \text{ there is an infinite path }\ell \subset (\tau_1,\tau_2) \text{ with }\langle s(\ell),\xi\rangle \neq 0,\\ \,\ell \text{ exits $B_n$ for the \textit{last} time at $x$}\}.
    \end{align*}
    Right rays hit $\partial B_n$ at last exit points. \symindex{Chapter 7!$S_{\text{first}},X_{\text{last}}$} See Figure \ref{fig:entrance_exit} for an illustration. 
    \begin{figure}
       \centering
      \includegraphics[scale=0.7]{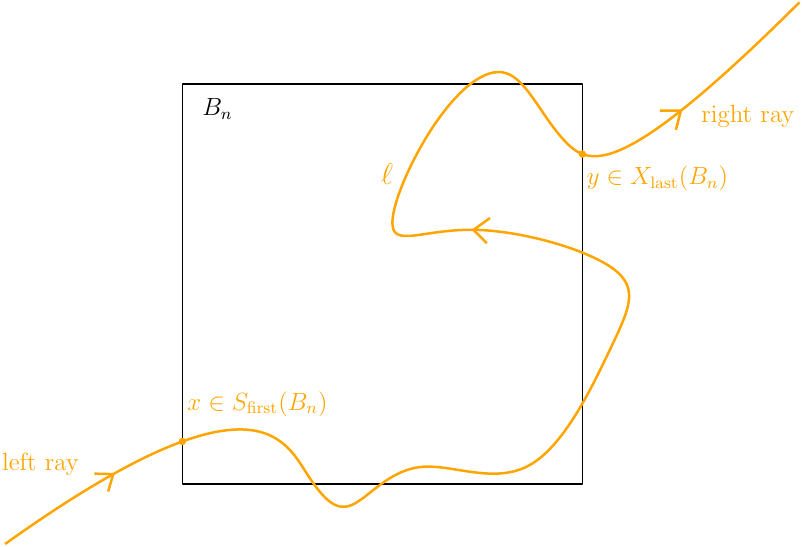}
        \caption{Example of an infinite path $\ell\subset (\tau_1,\tau_2)$ hitting $B_n$, with first entrance, last exit, and left and right rays labeled.}
        \label{fig:entrance_exit}
    \end{figure}
    We show that without loss of generality (i.e.\ up to translating $P$) there are many left rays incident to the face $F = P\cap \partial B_n$, in particular that it contains many points in $S_{\text{first}}(B_n)$. To do this, let $\tilde{B}_n$ be $B_n$ reflected over $P$ and notice that
    \begin{align*}
        A_P\cap B_n\subset S_{\text{first}}(B_n) \cup S_{\text{first}}(\tilde B_n).
    \end{align*}
    Therefore at least one of $A_P \cap S_{\text{first}}(B_n)$ and $A_P\cap S_{\text{first}}(\tilde{B}_n)$ has size of order $n^2$. (It is possible for only one to have order $n^2$ points, for example if all paths in $(\tau_1,\tau_2)$ are in the same direction.) Without loss of generality (by translating and possibly changing the choice of face $F$), there exists $c\in(0,1)$ such that given $\delta>0$, for $n$ large enough
    \begin{equation}\label{eq:AP_lower}
        |A_P\cap S_{\text{first}}(B_n)|>cn^2
    \end{equation}
    with $\nu$-probability $1-\delta$. 

    Given $x\in A_P\cap S_{\text{first}}(B_n)$, there exists a unique infinite path $\ell\subset (\tau_1,\tau_2)$ with slope denoted $s(\ell)$ containing $x$. Since $x\in A_P$ this path will have $\langle s(\ell),\xi\rangle \neq 0$, so $\ell$ hits $P$ finitely many times almost surely. We define the function $D_P(\ell)$ to be the distance along $P$ from $\ell \cap A_P$ to $\ell \cap C_P$ (note that this is different from $d_P(\cdot)$ defined in Lemma \ref{lem: path segment length}). \symindex{Chapter 7!$D_P$}
    
    Without loss of generality assume that the origin is contained in $P$. Let $\ell_0$ be the path through the origin in $(\tau_1,\tau_2)$. Then for any $\theta>0$ there exists $M$ such that
    \begin{equation}\label{eq: distance bigger than M at 0}
       \nu ( D_P(\ell_0) > M ~|~ \langle s(\ell), \xi\rangle \neq 0 ) < \theta.
    \end{equation}
   By $\threeeven$-invariance, this holds for any $\ell$ through an even point on $P$ with $\langle s(\ell),\xi\rangle \neq 0$. An analogous statement to Equation \eqref{eq: distance bigger than M at 0} holds if we look at an odd point $v\in P$, and $\threeeven$-invariance again implies that it holds for any $\ell$ through an odd point on $P$ with $\langle s(\ell),\xi\rangle \neq 0$. Putting these together, we have that for $n$ large
   \begin{align*}
        \nu(\#\{\ell ~:~ \ell\cap C_P\not\in B_n, \ell \cap A_P\in B_n\} \leq 4 M n + \theta n^2) > 1 - \theta.
   \end{align*}
   Taking $\theta$ sufficiently small and $n$ large, we can set $M = \epsilon n$, with $\epsilon>0$ small to be specified below. Then for $n$ large this becomes 
   \begin{align*}
        \nu(\#\{\ell ~:~ \ell\cap C_P\not\in B_n, \ell \cap A_P\in B_n\} \leq (4 \epsilon +\theta)n^2) > 1 - \theta.
   \end{align*}
    As any infinite path $\ell\subset (\tau_1,\tau_2)$ has well-defined slope, if $\langle s(\ell),\xi\rangle\neq 0$ then $\ell$ must be on opposite sides of $P$ before $A_P$ and after $C_P$. 
    Hence by the above and Equation \eqref{eq:AP_lower}, for $n$ large enough,
    \begin{equation}\label{eq: exit not in P}
        \nu(\#\{\ell~:~ \ell \cap S_{\text{first}}(B_n)\in P, \ell\cap X_{\text{last}}(B_n)\not\in P\} > (c-4\epsilon - \theta)n^2 ) > 1-\theta-\delta. 
    \end{equation}
    Therefore with $\nu$-probability $1-\theta-\delta$, at least $c'n^2=(c-4\epsilon-\theta)n^2$ infinite paths $\ell$ entering at $x\in A_P\cap B_n\subset F$ exit $B_n$ at $y\not\in F$.
    
    On the other hand, since $s_1\neq s_2$, we can apply chain swapping with $p=1/2$ to get a new measure $\nu'$ distinct from $\nu$. By Proposition \ref{prop:preserve_ergodicity}, the marginals $\nu_1',\nu_2'$ of $\nu$ are ergodic. By Proposition \ref{prop: coupling swap mean current}, they satisfy
    \begin{align*}
        s(\nu_1') = s(\nu_2') = \frac{s_1+s_2}{2}.
    \end{align*}
    Together this means that $\nu_1',\nu_2'$ satisfy the conditions of the patching theorem (Theorem \ref{patching}). Fixing $\epsilon\in(0,1)$, let $A_n$ be the cubic annulus between $B_n$ and $(1-\epsilon)B_n$. By Theorem \ref{patching} applied to $\nu_1',\nu_2'$ on $A_n$, for $(\tau_1',\tau_2')$ sampled from $\nu'$, for $n$ large enough we can with $\nu'$-probability $1-\epsilon$ find a tiling $\tau$ such that 
    \begin{itemize}
        \item $\tau\mid_{(1-\epsilon)B_n} = \tau_2'$
        \item $\tau\mid_{\m Z^3\setminus B_n} = \tau_1'$.
    \end{itemize}
    Let $Z_n\subset A_P \cap S_{\text{first}}(B_n)$ be the subset of points $x$ such that the infinite path $\ell$ through $x$ in $(\tau_1,\tau_2)$ satisfies:
    \begin{itemize}
        \item $\ell$ has $C_P\cap \ell\in B_n$ (so that $\ell$ exits $B_n$ through $\partial B_n\setminus P$);
        \item  $\ell$ did not have its orientation reversed by the chain swapping (in other words, $\ell \subset (\tau_1',\tau_2')\cap (\tau_1,\tau_2)$.)
    \end{itemize}
    By Equation \eqref{eq: exit not in P} and since chain swapping reverses the orientation of each infinite path with independent probability $1/2$, given $\delta'>0$, for $n$ large enough, setting $c'=c-4\epsilon -\theta$ we have
    \begin{equation}\label{eq:Zn size bound}
        \nu' ({|Z_n|} > c'n^2/2 ) > \frac{1}{2}-\delta' > 0.
    \end{equation}
    Conditional on the double dimer configuration $(\tau_1',\tau_2')$ on $\m Z^3\setminus B_n$, if $\nu'$ is a Gibbs measure then it must assign the same probability to $(\tau_1',\tau_2')$ and $(\tau,\tau_2')$. However since $\tau,\tau_2'$ agree on $(1-\epsilon)B_n$, there are no infinite paths in $(\tau,\tau_2')$ through $(1-\epsilon)B_n$. 
        
    Let $S'_{\text{first}}(B_n)$ and $X'_{\text{last}}(B_n)$ denote the first entrance and last exit points in $(\tau_1',\tau_2')$. We note that 
    \begin{align*}
        S'_{\text{first}}(B_n)\cup X'_{\text{last}}(B_n) = S_{\text{first}}(B_n)\cup X_{\text{last}}(B_n)
    \end{align*}
    because on infinite paths where the orientation was swapped, the first entrance and last exit points are swapped.
    
    On the other hand, since $(\tau_1',\tau_2')$ and $(\tau,\tau_2')$ agree on $\m Z^3\setminus B_n$, they have the same first entrance and last exit points and the same left and right rays. If $x\in S'_{\text{first}}(B_n)$ and $y\in X'_{\text{last}}(B_n)$, we denote the left and right rays incident to them by $\ell_-(x)$ and $\ell_+(y)$ respectively. The tiling $(\tau,\tau_2')\mid_{B_n}$ pairs up all the left rays with right rays in a new way to make full infinite paths.
    
    However recall that 
    an infinite path in a double dimer configuration sampled from $\nu'$ has well-defined slope almost surely. We show that $\nu'$ is \textit{not} Gibbs by showing that it is \textit{not} possible to pair order $n^2$ of the left rays entering at $x\in Z_n$ with right rays of the same slope. For $x\in Z_n$, let $\gamma(x)\subset (\tau,\tau_2')$ denote the path that connects $\ell_-(x)$ to an exit point $y$. Then the infinite path in $(\tau,\tau_2')$ through $x$ is 
    \begin{align*}
        \ell_-(x)\cup \gamma(x) \cup \ell_+(y). 
    \end{align*}
    The remainder of the proof is casework to show there only a small number of these infinite paths can have well-defined slope. Recall that $F = B_n \cap P$ and let $F^\circ$ denote the points in $F$ which are distance $\geq \epsilon n$ from $\partial F$.
    \begin{enumerate}
       \begin{figure}
        \centering
        \includegraphics[scale=0.5]{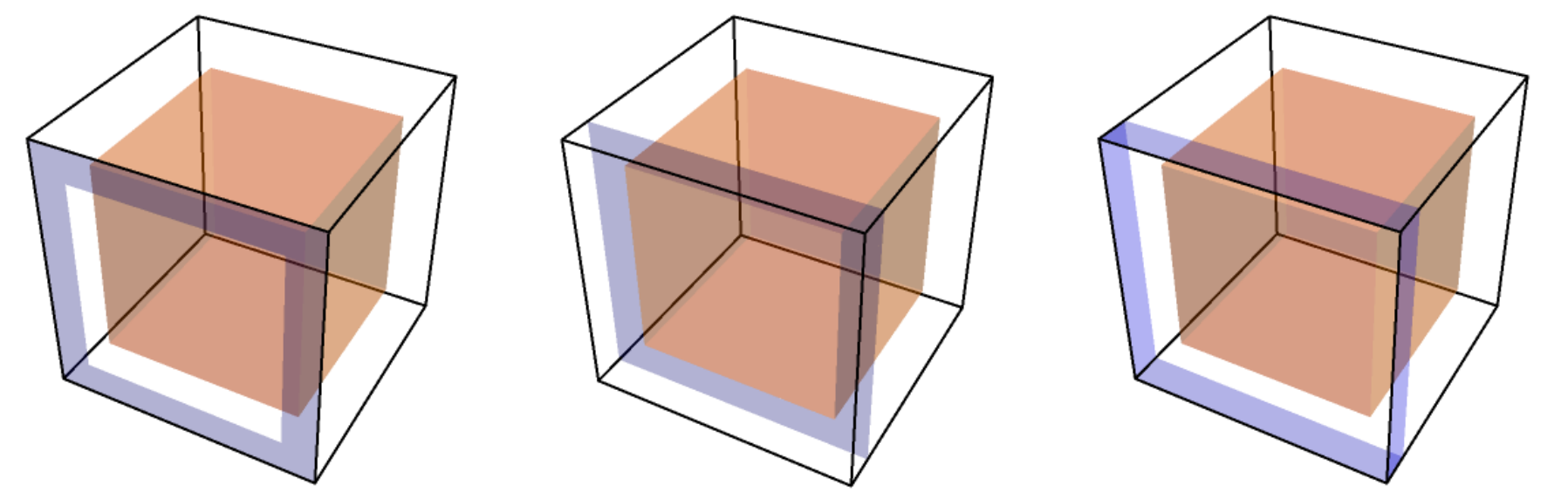}
        \caption{In all three pictures, the transparent cube is $B_n$ and smaller orange cube inside it is $(1-\epsilon)B_n$. The front left face is $F$. In $(\tau,\tau_2')$, all tiles in $(1-\epsilon)B_n$ are double edges, so infinite paths can't enter the orange cube. The region $T$, corresponding to Case \ref{it:bounded by area}, is the union of the three blue regions.}
        \label{fig:pieces of T}
    \end{figure}
    \item\label{it:bounded by area} \textbf{Bounded by area:} We define the \textit{thin region} $T$, which is a union of three things: i) $F\setminus F^\circ$, ii) $F\setminus F^\circ$ translated $\epsilon n$ inward, iii) the part of $\partial B_n$ between i) and ii). See Figure \ref{fig:pieces of T} for an illustration. Since
    \begin{align*}
        \text{area}(T) \leq 12 \epsilon n^2,
    \end{align*}
    the number of infinite paths in $(\tau,\tau_2')$ which intersect $T$ is bounded by $12\epsilon n^2$. 
    \item \label{it: connecting paths}\textbf{Bounded by number of possible connecting paths:} choose $x\in Z_n$, and suppose that $\ell_-(x)\cup \gamma(x)\cup \ell_+(y)$ does not intersect $T$. To have well-defined slope, $\ell$ must still cross $P$ some time after $x$, and to avoid $T$ it must at some point cross $P$ in $P\setminus F$. See Figure \ref{fig:Gibbs_cases}.
    
    Therefore the rest of the path $\gamma(x)\cup \ell_+(y)$ must use part of at least one finite cycle or infinite path in $$(\tau,\tau_2')\mid_{\m Z^3 \setminus B_n} =(\tau_1',\tau_2')\mid_{\m Z^3 \setminus B_n}$$
    to connect a point in $F^\circ$ to $P\setminus F$. This path will be an arc on $P$ in the $P$-half-space on the opposite side of $P$ from $B_n$. Chain swapping only changes the directions of paths, so the collection of arcs and their lengths are the same in $(\tau_1,\tau_2)$ and $(\tau_1',\tau_2')$. Since the arcs are outside $B_n$, they are also the same in $(\tau,\tau_2')$. Thus by Lemma \ref{lem: path segment length} applied to $M = \epsilon n$, we can find $\beta,\theta>0$ small enough so that for $n$ large enough, 
    \begin{align*}
        \nu'\bigg(\# \{\alpha \text{ arc of path hitting } F ~:~ d_P(\alpha) > \epsilon n\} < \beta n^2  \bigg) > 1-\theta,
    \end{align*}
    where recall that $d_P(\alpha)$ is the distance along $P$ between the two intersection points of the arc $\alpha$ with $P$. Therefore with $\nu'$-probability $1-\theta$, the number of $x\in Z_n$ such that the path
    \begin{align*}
        \ell_-(x)\cup \gamma(x) \cup \ell_+(y)
    \end{align*}
    is disjoint from $T$ but crosses $P\setminus F^\circ$ is at most $\beta n^2$. 
    \begin{figure}
        \centering
        \includegraphics[scale=0.4]{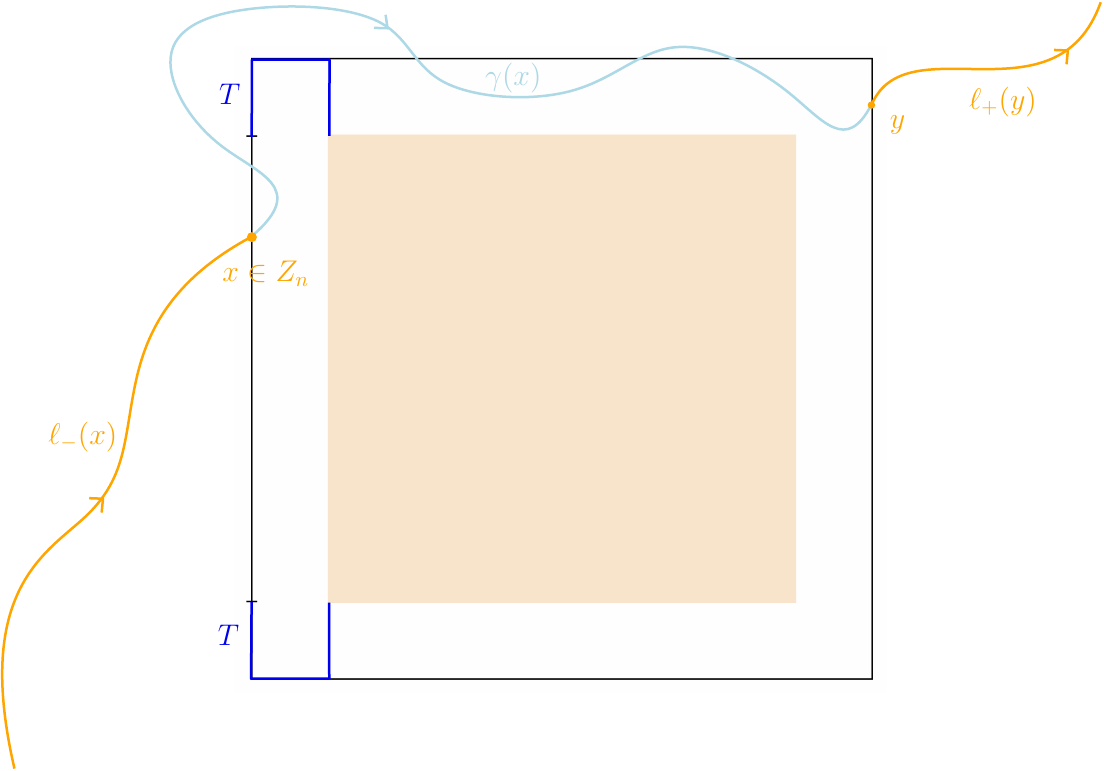}\includegraphics[scale=0.4]{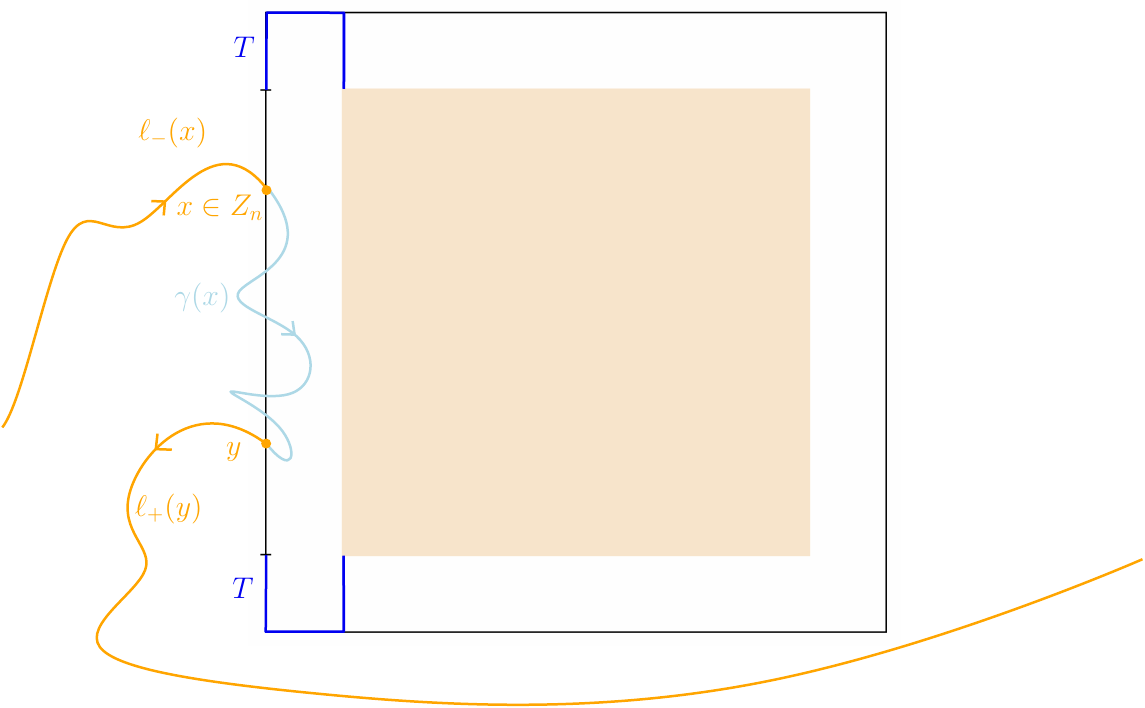}
        \caption{Two examples corresponding to Case \ref{it: connecting paths}. In this case the infinite path does not intersect $T$, so this can happen either if the final exit point $y\in \partial B_n\setminus F$ (left) or if the final exit point $y\in F$, but the right ray $\ell_+(y)$ crosses $P$ again outside $B_n$ (right).}
        \label{fig:Gibbs_cases}
    \end{figure}
    \item \label{it: bad case}\textbf{Remaining paths forced to have no well-defined slope:} if $x\in Z_n$ is not in Case \ref{it:bounded by area} or Case \ref{it: connecting paths}, then the path $\ell:=\ell_-(x) \cup \gamma(x) \cup \ell_+(y)$ does not intersect $T$ and does not cross $P\setminus F^\circ$ at any time after going through $x$. This implies that $\ell_-(x)$ and $\ell_+(y)$ are contained in the same $P$ half-space, in which case $\ell$ cannot have well-defined slope. See Figure \ref{fig:final_Gibbs_case}.
    \begin{figure}
        \centering
    \hspace{-2.5cm}\includegraphics[scale=0.4]{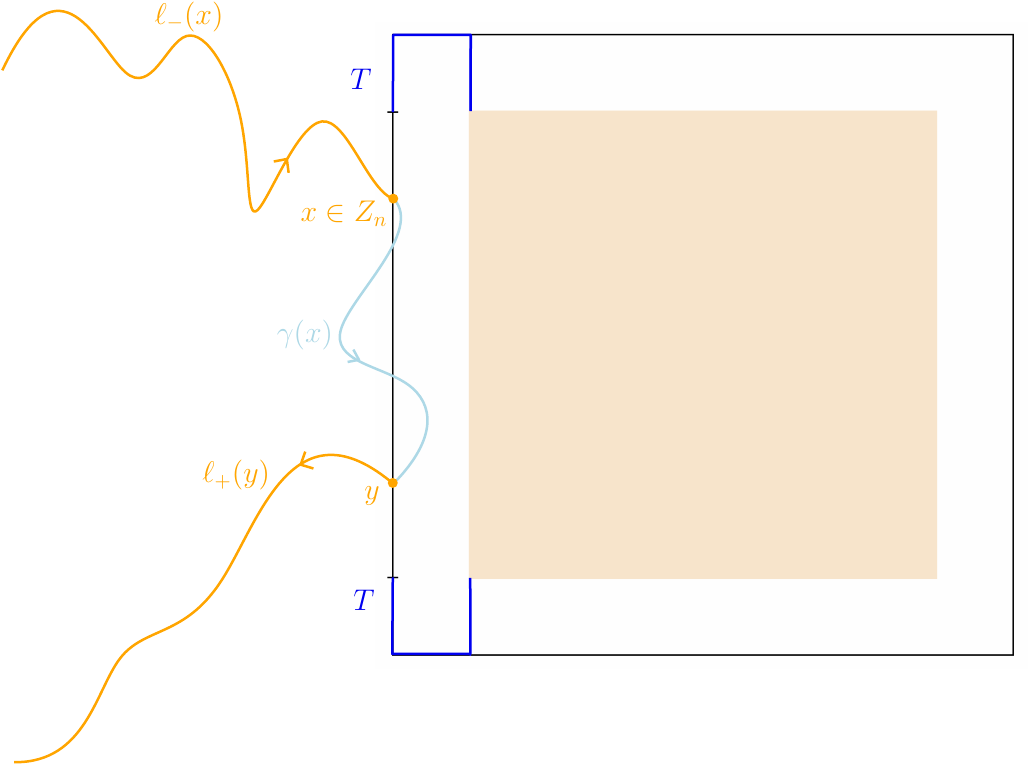}    \caption{Corresponding to Case \ref{it: bad case}, if $\gamma(x)\cup \ell_+(y)$ never crosses $P\setminus F$, then the resulting infinite path cannot have well-defined slope.}
        \label{fig:final_Gibbs_case}
    \end{figure}
    \end{enumerate}
    
    In summary, with probability $1-\theta$, there are at most 
    \begin{align*}
        (12 \epsilon + \beta)n^2
    \end{align*}
    points $x\in Z_n$ such that we can connect $\ell_-(x)$ to have well-defined slope. However by Equation \eqref{eq:Zn size bound}, with positive $\nu$-probability $|Z_n| > c'n^2/2$. We can take $\epsilon,\beta,\theta$ to be arbitrarily small compared to $c$, and thus since $\nu'$ is $\threeeven$-invariant (meaning infinite paths must have well-defined slope a.s.), $\nu'(\cdot~|~(\tau_1',\tau_2')\mid_{\m Z^3\setminus B_n})$ cannot assign the same probability to $(\tau_1',\tau_2')$ and $(\tau,\tau_2')$. Therefore $\nu'$ is not a Gibbs measure. 
\end{proof}

Theorem \ref{theorem: Preservation of the Gibbs prop} and Corollary \ref{Corollary: entropy maximizer double dimer} combine to give the following corollary. 
\begin{cor}
If $\mu$ is a Gibbs measure on $\Omega\times \Omega$ which is an ergodic coupling of $\mu_1 \in \mc P_e^{s_1}$ and $\mu_2\in \mc P_e^{s_2}$ for $s_1 \neq s_2$ and $\frac{s_1 + s_2}{2}\in \text{Int}(\mc O)$, then the measure $\mu'$ obtained by chain swapping with probability $p=1/2$ does not maximize entropy in $\mc P^{\frac{s_1+s_2}{2},\frac{s_1+s_2}{2}}$.
\end{cor}

\subsection{Strict concavity of ent and existence of EGMs of every mean current}\label{subsection: strict concavity}

With the chain swapping machinery developed in Section \ref{Subsection: Dimer Swapping}, we can now prove one of the main results of this section, namely that that $\ent$ is \textit{strictly} concave on $\mc O\setminus \mc E$ (recall that $\mc E$ denotes the edges of $\mc O$). We already showed that $\ent$ is concave on $\mc O$ in Lemma \ref{lemma: entropy_concave}. We also already showed in Section \ref{sec:extreme} that $\ent$ restricted to the interior of any face of $\partial \mc O$ is strictly concave (Corollary \ref{cor: ent boundary}) by relating $\ent$ restricted to a face of $\partial \mc O$ to $\ent_{\text{loz}}$, the slope entropy function for two-dimensional lozenge tilings (Theorem \ref{thm: extremal_entropy}). 

\begin{thm}\label{theorem: entropy is strictly concave}
The entropy function $\ent$ is strictly concave on $\mc O \setminus \mc E$.
\end{thm}
\begin{proof}
By Lemmas \ref{lemma: entropy_concave} and \ref{lemma: entropy_continuous}, $\ent$ is concave and continuous on $\mathcal O$. To show strict concavity on $\mc O\setminus \mc E$, it suffices to show that if $s_1, s_2\in \mathcal O$ and $(s_1+s_2)/2\in \mc O\setminus \mc E$, then
$$\ent((s_1+s_2)/2)>(\ent(s_1)+\ent(s_2))/2.$$
If $(s_1+s_2)/2$ is contained in the interior of a face $\mc F \subset \partial \mc O$, then we are done by Corollary \ref{cor: ent boundary}. The remaining case is that $(s_1+s_2)/2\in \text{Int}(\mc O)$. In this case let $\mu_1$ and $\mu_2$ be entropy maximizers in $\mathcal P^{s_1}$ and $\mathcal P^{s_2}$ respectively (these exist by Lemma \ref{lemma: entropy maximiser}) and let $\mu$ be the independent coupling of $\mu_1$ and $\mu_2$. Then 
\begin{align*}
    \ent(s_1) + \ent(s_2) = h(\mu). 
\end{align*}
Consider the ergodic decomposition 
$$\mu=\int \nu \,\dd w_\mu(\nu)$$
where $w_\mu$ is a probability measure on the space of ergodic couplings of ergodic Gibbs measures (see Proposition \ref{prop: ergodic deocomposition gibbs}, which says that the ergodic components of a Gibbs measure are Gibbs a.s.). Let $\mu'$ be the measure obtained by applying chain swapping with probability $p=1/2$ to $\mu$. By Proposition \ref{prop:preserve_ergodicity},
$$\mu'=\int \nu' \,\dd w_\mu(\nu)$$
where $\nu'$ is obtained from $\nu$ by chain swapping also with $p=1/2$. By Proposition \ref{prop: coupling swap entropy}, $h(\nu) = h(\nu')$. Since $h(\cdot)$ is an affine function, we get that
$$h(\mu)=\int h(\nu) \,\dd w_\mu(\nu)=\int h(\nu') \,\dd w_\mu (\nu)=h({\mu'}).$$
By Proposition \ref{prop: coupling swap mean current}, $s(\pi_1(\nu')) = s(\pi_2(\nu')) = (s(\nu_1)+s(\nu_2))/2$. Since $s(\cdot)$ is an affine function,
\begin{align*}
    s(\pi_1(\mu')) = s(\pi_2(\mu')) = (s_1+s_2)/2.
\end{align*}
Let $(\mathfrak s_1, \mathfrak s_2)$ denote the random pair of mean currents for a double dimer configuration sampled from $\mu$. To complete the proof, we proceed with cases based on $w_\mu$. Consider the sets
\begin{align*}
      A = \{ \nu : \frac{s(\pi_1(\nu)) + s(\pi_2(\nu))}{2} \in \text{Int}(\mc O)\}, \qquad B = \{ \nu : s(\pi_1(\nu)) \neq s(\pi_2(\nu)) \}.
\end{align*}
If $\nu\in A\cap B$ is an ergodic coupling of ergodic measures, then $\nu$ satisfies the conditions of Theorem \ref{theorem: Preservation of the Gibbs prop}. Since $\m E[\mathfrak s_1 - \mathfrak s_2] = s_1 - s_2 \neq 0$, $w_\mu(B)>0$. If $w_\mu(A)>0$, then since $\mu$ is an independent coupling, we can argue in a few cases that $w_\mu(A\cap B)>0$: 
  \begin{itemize}
        \item If $\mathfrak s_1, \mathfrak s_2$ are both atomic, then $w_\mu(A\cap B)>0$. For the next cases we assume without loss of generality that $\mathfrak s_2$ is not atomic. 
        \item If $\{\mathfrak s_1\in \text{Int}(\mc O)\}$ has positive probability, then given any value of $\mathfrak{s}_1$ in $\text{Int}(\mc O)$, $\mathfrak{s}_2$ has positive probability to be different from it. Since $\mathfrak s_1 \in \text{Int}(\mc O)$, the average is in $\text{Int}(\mc O)$.
        \item If $\mathfrak s_1$ has positive probability to be contained in $\partial \mc O$, then given any value of $\mathfrak{s}_1$ in $\partial \mc O$, $w_\mu(A)>0$ implies that $\mathfrak{s}_2$ has positive probability to not be contained in the same face as $\mathfrak s_1$ (since on $A$, their average must be in $\text{Int}(\mc O)$). On the other hand if $\mathfrak{s}_2$ is not contained in the same face of $\partial \mc O$ as $\mathfrak s_1$, then it must be different from $\mathfrak s_1$.
    \end{itemize}
Applying Theorem \ref{theorem: Preservation of the Gibbs prop} shows that $\nu'$ is not a Gibbs measure for $\nu\in A\cap B$. Since the ergodic components of Gibbs measures are Gibbs a.s., if $w_\mu(A\cap B)> 0$ then $\mu'$ is not a Gibbs measure. By Corollary \ref{Corollary: entropy maximizer double dimer}, $\mu'$ is not an entropy maximizer in $\mathcal P^{\frac{s_1+s_2}{2}, \frac{s_1+s_2}{2}}$, and hence
\begin{align*}
 \ent(s_1) + \ent(s_2) = h(\mu)= h(\mu') < 2 \ent((s_1+s_2)/2).
\end{align*}
This completes the proof if $w_\mu(A) >0$.

However it can happen that $w_\mu(A) = 0$ (for example, if $\mathfrak s_1$ is supported at a corner vertex $v\in \partial \mc O$, and $\mathfrak s_2$ is supported on a square on $\partial \mc O$ around $v$). Since $\mu$ is an independent coupling, $w_\mu(A)=0$ implies that $\mathfrak s_1, \mathfrak s_2$ are supported in $\partial \mc O$. There are two remaining cases. 

First suppose there is a face $\mc F\subset \partial \mc O$ such that 
\begin{align*}
    C = \{ \nu : \frac{s(\pi_1(\nu)) + s(\pi_2(\nu))}{2} \in \text{Int}(\mc F)\}
\end{align*}
has $w_\mu(C)>0$. Since $\mu$ is an independent coupling, $w_\mu(C\cap B)>0$ by arguments analogous to those above for $A,B$. Let $\mu'$ be obtained from $\mu$ by chain swapping. By Proposition \ref{prop:preserve_ergodicity}, 
\begin{align*}
    \mu' = \int \nu'\, \dd w_\mu(\nu)
\end{align*}
where $\nu'$ is obtained from $\nu$ by chain swapping. Let $\nu_1,\nu_2$ denote the marginals of $\nu$ and let $\nu_1',\nu_2'$ denote the marginals of $\nu'$. If $\nu \in B$, then $s(\nu_1)\neq s(\nu_2)$ and $\nu'$ is distinct from $\nu$. By Proposition \ref{prop: coupling swap entropy} and Proposition \ref{prop: coupling swap mean current},
\begin{align*}
    h(\nu') = h(\nu), \qquad  s(\nu_1') = s(\nu_2') = \frac{s(\nu_1)+s(\nu_2)}{2}.
\end{align*}
By Theorem \ref{thm: extremal_entropy}, $\ent\mid_{\mc F} = \ent_{\text{loz}}$. Since $\ent_{\text{loz}}$ is strictly concave on $\text{Int}(\mc F)$, we have that for each $\nu\in C\cap B$, 
\begin{align*}
    h(\nu') = h(\nu_1')+h(\nu_2') < 2\, \ent(({s(\nu_1)+s(\nu_2))}/{2})
\end{align*}
Since $w_\mu(C\cap B)>0$, Lemma \ref{lemma: entropy_concave} and the affine property of $h$ implies that 
\begin{align*}
    \ent(s_1) + \ent(s_2) = h(\mu) = h(\mu') = \int h(\nu) \, \dd w_\mu(\nu) < 2\,\ent( (s_1+s_2)/2).
\end{align*}
This completes the proof in the case that $w_\mu(C)>0$. 

Finally if $w_\mu(A) =0$ and $w_\mu(C)= 0$ for all faces of $\partial \mc O$, then $\mathfrak s_1, \mathfrak s_2$ must be supported in $\mc E$ (for example, $\mathfrak s_1$ could be supported at one vertex $v\in \partial \mc O$, and $\mathfrak s_2$ could be supported on the four edges of $\partial \mc O$ incident to $v$). Since $\ent\mid_{\mc E}\equiv 0$, this implies that $h(\mu_1) + h(\mu_2) = h(\mu) = 0$ and hence that $h(\mu_1)= h(\mu_2) = 0$. 

However by Lemma \ref{lemma: entropy_concave} and Theorem \ref{thm: extremal_entropy}, $\ent\mid_{\mc O\setminus \mc E}>0$. Therefore if $s_1,s_2\in \mc O\setminus \mc E$, then $\mu_1,\mu_2$ cannot be entropy maximizers in $\mc P^{s_1}$, $\mc P^{s_2}$. This completes the proof.
\end{proof}

With this we can strengthen Theorem \ref{theorem: entropy maximizer gibbs}. 

\begin{thm}\label{theorem: existence of gibbs}
	For every $s\in \text{Int}(\mathcal O)$, a measure $\mu\in \mathcal P^{s}$ satisfies $h(\mu) = \ent(s)$ if and only if $\mu$ is a convex combination of ergodic Gibbs measures of mean current $s$. In particular, if $\nu\in \mc P^s$ is an ergodic Gibbs measure, then $h(\nu) = \ent(s)$.
\end{thm}
\begin{rem}
    In contrast, an EGM of mean current $s\in \partial \mc O$ can have any specific entropy between $0$ and $\ent(s)$, see Proposition \ref{prop: same current different entropy}. (Note however that for $s\in \mc E$, $\ent(s) = 0$.) All EGMs of mean current $s\in \text{Int}(\mc O)$ have the same specific entropy by Corollary \ref{cor:entsame}.
\end{rem}
\begin{proof}
Suppose $\mu\in \mathcal P^{s}$ maximizes entropy ($\mu$ exists by Lemma \ref{lemma: entropy maximiser}). By Theorem \ref{theorem: entropy maximizer gibbs}, $\mu$ is a Gibbs measure. Consider its ergodic decomposition
$$\mu=\int \nu \, \dd w_\mu(\nu).$$
Since ergodic components of Gibbs measures are Gibbs a.s., $w_\mu$ is a probability measure on ergodic Gibbs measures. Since $h(\cdot)$ is an affine function, it follows that
$$\ent(s) = h({\mu})=\int h({\nu})\,\dd w_\mu(\nu)\leq \int \ent({s(\nu)})\,\dd w_\mu(\nu).$$
Since $s\in \text{Int}(\mc O)$, by Theorem \ref{theorem: entropy is strictly concave} if $s(\nu)$ is not constant then the middle inequality below is strict:
$$ \ent( s) \leq \int \ent(s (\nu)) \, \dd w_\mu(\nu) < \ent \bigg( \int s (\nu) \, \dd w_\mu(\nu) \bigg) = \ent(s).$$ 
Therefore all ergodic components $\nu$ of $\mu$ must have $s(\nu) = s$ a.s., i.e., the support of $w_\mu$ is contained in the set of ergodic Gibbs measures of mean current $s$. 

By Corollary \ref{cor:entsame}, if $\nu_1,\nu_2$ are EGMs of the same mean current $s\in \text{Int}(\mc O)$, then $h(\nu_1)= h(\nu_2)$. Therefore if $\nu\in \mc P^s$ is an ergodic Gibbs measure, $\ent(s) = h(\nu)$. 
\end{proof}
 
From Theorem \ref{theorem: existence of gibbs} and Lemma \ref{lemma: entropy maximiser} for interior mean currents and the results of Section \ref{sec:extreme} for boundary ones, there exist ergodic Gibbs measures of all mean currents. 

\begin{cor}\label{cor: EGMs exist!}
For all $s\in\mathcal O$, there exists an ergodic Gibbs measure of mean current $s$.
\end{cor}

\begin{rem}
    In two dimensions, there exists a \textit{unique} ergodic Gibbs measure of every interior slope. Uniqueness of EGMs for interior mean currents is open problem, see Problem \ref{prob: unique EGM} and the related Problem \ref{prob: infinite paths}. 
\end{rem}

\section{Free-boundary tilings, asymptotic flows, and Wasserstein distance}\label{sec:asympflows}

This section sets up some of the key function-theoretic preliminaries for the large deviation principle in Section \ref{sec:ldp}. Some parts of this section are technical and can be skipped on a first reading, e.g.\ the sections about boundary values of flows. Section~\ref{subsection: properties of Ent}, where we study the entropy function $\Ent$, primarily requires just the definition of asymptotic flows, and one could skip to Section~\ref{subsection: properties of Ent} after reading the definition of asymptotic flows.

\termindex{Chapter 5!domain}\termindex{Chapter 5!domain}A domain is an open subset of $\R^3$. Let $R\subset \m R^3$ be a compact region which is the closure of a connected domain and has piecewise smooth boundary $\partial R$\symindex{Chapter 5!$R$ - domain}\symindex{Chapter 5!$\partial R$ - boundary of the domain}. We say that a grid region $G$ is \textit{scale $n$}\termindex{Chapter 5!grid region of scale $n$} if $G\subseteq \frac{1}{n} \m Z^3$. If $R_n$ is a scale $n$ grid region, then with a slight abuse of notation we say that $R_n\supseteq R$ if the collection of $\frac{1}{n}$-width cubes centered at points in $R_n$ contains $R$. If $\tau$ is a tiling of $R_n$\termindex{Chapter 5!scale $n$ grid regions contained in domains: $R\subset R_n$}, we define the restriction of $\tau$ to $R$, denoted $\tau_R$, to be the collection of tiles from $\tau$ which intersect $R$.\symindex{Chapter 5!$\tau_R$ - collection of tiles intersecting $R$}
\begin{definition}\label{def:free_boundary_tilings}
    The \textit{free-boundary tilings of $R$ at scale $n$}\termindex{Chapter 5!free boundary tilings of $R$ at scale $n$}\symindex{Chapter 5!$T_n(R)$ - free boundary tilings at scale $n$} are
    \begin{align*}
        T_n(R) := \bigcup_{R_n \supseteq R} \{\tau_R : \tau \text{ is a tiling of }R_n \}.
    \end{align*}
    The \textit{free-boundary tiling flows on $R$ at scale $n$} are\termindex{Chapter 5!free boundary tiling flows on $R$ at scale $n$}\symindex{Chapter 5!$TF_n(R)$ - free-boundary tiling flows on $R$ at scale $n$}
    \begin{align*}
        TF_n(R) := \{f_\tau : \tau\in T_n(R)\}.
    \end{align*}
    Finally, we define \termindex{Chapter 5!free boundary tilings of $R$} the space of all free-boundary tiling flows on $R$ to be $TF(R) := \bigcup_{n\geq 1} TF_n(R)$. 
    
    The edges in $\frac{1}{n} \m Z^3$ have length $\frac1n$. To ensure that the total flow of a tiling flow is roughly constant in $n$, we need the flow per edge of $f_\tau\in TF_n(R)$ to be of order $\frac{1}{n^3}$. We can achieve that by rescaling the flow by a factor of $n^3$ so that it has magnitude $\frac{5}{6n^3}$ on each matched edge and $\frac{1}{6n^3}$ on each unmatched edge.
\end{definition}
\begin{rem}
Note that $TF_n(R)$ may contain elements that do not arise as restrictions of tilings of all of $\frac{1}{n} \m Z^3$ to $R$. That is, there may be free-boundary tilings of $R$ that cover $R$ but do not extend to tilings of all of $\m Z^3$. (These might exist, for example, if $R$ is a concave region.)
\end{rem}

We define a metric on flows (Section \ref{section: wasserstein for tiling flows}), denoted $d_W$, which is an adaptation of \textit{generalized Wasserstein distance} from signed measures to flows. Intuitively we want to consider two flows $f,g$ ``close" if we don't have to change the flow of $f$ too much---either by moving flow over, or by adding or deleting it---to transform it into $g$. This is what $d_W(f, g)$ will measure. In terms of this metric, the main question of this section is: if $f_n \in TF_n(R)$ for all $n\in \m N$, what are the possible limits of the form $\lim_{n\to \infty} f_n$? 

We show (Theorem \ref{thm:formal_fine_mesh}) that any fine-mesh limit of free-boundary tiling flows on $R$ is an \textit{asymptotic flow on $R$}, defined by:
\begin{definition}\termindex{Chapter 5!asymptotic flow}
	Let $R$ be a compact region which is the closure of a connected domain and has piecewise smooth boundary. We say that $f$ is an \textit{asymptotic flow on $R$} if it satisfies the following properties:
\begin{itemize}
    \item $f$ is a Borel-measureable vector field with support contained in $R$; 
    \item $f$ is valued in $\mathcal{O}$ (since $f$ is measurable, this means that $f$ is valued in $\mathcal{O}$ Lebesgue-a.e.);
    \item $f$ is divergence-free in the interior of $R$, i.e.\ $\text{div}\, f = 0$ as a distribution (so for any smooth function $\phi$ compactly supported in the interior of $R$, $\int_R \phi\, \text{div}\,f := \int_R \nabla \phi\, \cdot\, f = 0$.) \termindex{Chapter 5!measurable divergence-free vector field}
\end{itemize}
We denote the set of asymptotic flows on $R$ \symindex{Chapter 5!$AF(R)$ - the space of asymptotic flows} by $AF(R)$.
\end{definition}
In Theorem \ref{thm:AF_compact} we will show that $(AF(R),d_W)$ is a compact metric space. In Sections \ref{sec: boundary values of asymptotic flows} and \ref{sec: boundary values of tiling flows}, we define a boundary value operator $T$ (\textit{trace operator}) which takes a flow to its boundary value on $\partial R$. In fact we do something more general, and define the trace of a flow for any compact, piecewise smooth surface contained in $R$. After defining $T$ for asymptotic flows, we define the space of asymptotic flows with boundary value $b$, denoted $AF(R,b)$, and show that it is compact with respect to $d_W$ (Corollary \ref{cor: AF(R,b) compact}).

The boundary value operator is defined in different but analogous ways for asymptotic flows (Section \ref{sec: boundary values of asymptotic flows}) and tiling flows (Section \ref{sec: boundary values of tiling flows}). The main essential result about $T$ is that these definitions are compatible and that $T$ is uniformly continuous (Theorem \ref{thm: boundary_value_uniformly_continuous}). 

We remark that the main important property of Wasserstein distance in our analysis is that it metrizes weak convergence, and that it therefore formalizes the intuitive notion that the scaling limits of tiling flows are asymptotic flows, and that boundary values depend continuously on the flow. While the Wasserstein metric and other transportation metrics have a number of additional special properties, we do not use this theory here. All the properties of the Wasserstein metric that we use are described in Section \ref{sec:wass_background}.

Finally in Section \ref{subsection: properties of Ent}, we leverage properties of $\ent$ from the previous section to study the entropy functional $\Ent: AF(R) \to [0,\infty)$ on asymptotic flows given by 
\begin{align*}
    \Ent(f) = \frac{1}{\Vol(R)} \int_{R} \ent(f(x)) \, \dd x.
\end{align*}
The rate function for the large deviation principle in Section \ref{sec:ldp} will be $-\Ent$ (up to an additive constant). Using the properties of $\ent$ from the previous section, we show that $\Ent$ is upper semicontinuous in the Wasserstein topology (Proposition \ref{proposition: Upper semicontinuous}) and strictly concave on the subspace of flows which never take values in the edges $\mc E\subset \partial \mc O$ (Corollary \ref{cor: Ent_concave}). After that, we adapt an argument of V.~Gorin \cite{gorin2021lectures} to show that there is a unique $\Ent$ maximizer in $AF(R,b)$ for any boundary asymptotic flow $b$ under the mild condition that $(R,b)$ is \textit{semi-flexible} (Definition \ref{def: flexible}, Theorem \ref{thm: unique maximizer}).

\subsection{Background on (generalized) Wasserstein distance}\label{sec:wass_background}

The original \textit{Wasserstein distance} or \textit{earth-movers distance} is a metric on probability measures on a fixed metric space. Suppose that $(X,d)$ is a compact, separable metric space. The $L^1$ Wasserstein distance $W_1$ is a metric on $\mathcal{P}(X)$, the space of probability measures on $X$ and is given by
\begin{align*}
    W_1(\mu,\nu) := \inf_{\gamma\in \Gamma(\mu,\nu)} \int_{X\times X} d(x,y) d\gamma(x,y) 
\end{align*}
where $\Gamma(\mu,\nu)$ is the collection of all couplings of $\mu$ and $\nu$. Intuitively $W_1$ measures the cost---i.e.\ how much mass and how far it has to be moved---required to transform $\mu$ into $\nu$ by redistributing the mass of $\mu$. This metric was developed in the theory of optimal transport and has been applied in many different contexts including probability, Riemannian geometry, and image processing. For more see \cite{Villani2009}. 

We will define and use versions of Wasserstein distance for flows and their boundary values. In the next section, we define a mapping between flows and measures, where flow (with direction) corresponds to mass (with sign). The measures corresponding to flows do not necessarily have the same total mass and can be signed. Given this, our Wasserstein distance on flows will be based on a version of \textit{generalized Wasserstein distance}.\termindex{Chapter 5!generalized Wasserstein distance}

Let $\mc M(\m R^d)$ denote the space of Radon measures on $\m R^d$ with finite total mass. In \cite{Piccoli2014-vy} and \cite{Piccoli2016-fe}, they define a generalized Wasserstein distance on $\mc M(\m R^d)$ by introducing an $L^1$ cost for adding and deleting mass.  It is denoted $W_1^{1,1}$ and defined as 
 \begin{align*}
        W_1^{1,1}(\mu,\nu) = \inf_{\tilde{\mu},\tilde{\nu}} |\mu - \tilde{\mu}| + |\nu-\tilde{\nu}| + W_1(\tilde{\mu},\tilde{\nu})
\end{align*}
where the infimum is taken over $\mc M(\m R^d)$. 

In \cite{Ambrosio} the $L^1$ Wasserstein distance was generalized to signed probability measures. This metric is denoted $\m W_1$. If $\mu,\nu$ are signed measures with Jordan decompositions $\mu = \mu_+-\mu_-$ and $\nu = \nu_+-\nu_-$, then \symindex{Chapter 5!$\m W_1^{1,1}$ - generalized Wasserstein distance between signed measures}
\begin{align*}
    \m W_1(\mu,\nu) = W_1(\mu_+ + \nu_-,  \nu_+ + \mu_- ).
\end{align*}
In fact, note that this definition does not depend on the decomposition of the measures $\mu,\nu$. 

In \cite{piccoli_signed}, they combine these to give a definition of Wasserstein distance for signed measures of different masses. Let $\mc M^s(\m R^d)$ denote the space of signed Radon measures on $\m R^d$, i.e.\ measures $\mu$ that can be written $\mu_+-\mu_-$ for $\mu\pm\in \mc M(\m R^d)$. Denoted $\m W_1^{1,1}$, the generalized Wasserstein distance on $\mc M^s(\m R^d)$ is defined
\begin{align*}
    \m W_1^{1,1}(\mu,\nu) = W_1^{1,1}(\mu_+ + \nu_-,  \nu_+ + \mu_- ).
\end{align*}
This is the definition of Wasserstein distance that we will use in this paper. We note a few important facts about $\m W_1^{1,1}$ that we will use. 

\begin{lemma}[See {\cite[Lemma 18]{piccoli_signed}}]\label{lemma: wasserstein mass shift}
    If $\mu,\nu,\rho\in \mc M^s(\m R^d)$, then 
    \begin{align*}
        \m W_1^{1,1}(\mu,\nu) = \m W_1^{1,1}(\mu+\rho,\nu+\rho). 
    \end{align*}
\end{lemma}

\begin{prop}[See {\cite[Proposition 23]{piccoli_signed}}]\label{prop: dual definition wasserstein} 
Let 
    \begin{align*}
        \mc C_b^{0,Lip} = \{f : \m R^d\to \m R ~:~ f\text{ continuous, bounded, Lipschitz}\}.
    \end{align*} 
    Then
        \begin{align*}
            \m W_1^{1,1}(\mu,\nu) = \sup \bigg\{ \int_{\m R^d} \varphi\, \dd (\mu-\nu) ~:~ \varphi\in \mc C_b^{0,Lip}, \norm{\varphi}_{\infty} \leq 1, \norm{\varphi}_{Lip}\leq 1\bigg\}.
        \end{align*}
\end{prop}
From this it clearly follows that
\begin{cor}\label{cor: weak convergence wasserstein}
If $\lim_{n\to \infty} \m W_1^{1,1}(\mu_n,\mu) = 0$, then $\mu_n$ converges weakly to $\mu$. 
\end{cor}
The non-signed generalized Wasserstein distance $W^{1,1}_1$ metrizes weak convergence for tight sequences of measures in $\mc M(\m R^d)$ \cite[Theorem 13]{Piccoli2014-vy}. The original $L^1$ Wasserstein distance for probability measures metrizes weak convergence under an additional moment condition \cite[Theorem 6.9]{Villani2009}. With signed measures, slightly stranger behavior can occur in general, see e.g.\ \cite[Remark 26]{piccoli_signed}. However we show that the Wasserstein distance for flows defined below does metrize weak convergence, see Remark  \ref{rem:weak convergence for signed +}.

For $R\subset \m R^d$, we define $\mc M (R)$\symindex{Chapter 5!$\mc M (R),\mc M^s(R),M_{\text{ac}}^s(R,a,b)$,$\mc M^s_{\text{ac}}(R)$ - Various spaces of measures on $R$} to be the Radon measures supported in $R$, and $\mc M^s(R)$ to be signed Radon measures supported in $R$. We let $\mc M_{\text{ac}}(R)$ (resp.\ $\mc M^s_{\text{ac}}(R)$) denote the Radon measures (resp.\ signed Radon measures) supported in $R$ and absolutely continuous with respect to Lebesgue measure on $\m R^d$. We say that $\mc M_{\text{ac}}^s(R,a,b)$ denotes absolutely continuous signed measures with densities valued between $a$ and $b$. By Lemma \ref{lemma: wasserstein mass shift}, for any $a<0$ and $b>0$,
\begin{equation}\label{eq: mass shift isomorphism}
    (\mc M_{\text{ac}}^s(R,a,b), \m W_1^{1,1}) \cong     (\mc M_{\text{ac}}(R,0,b-a), W_1^{1,1})
\end{equation}
as metric spaces. This identification has some useful consequences. From Equation \eqref{eq: mass shift isomorphism} and {\cite[Proposition 15]{Piccoli2014-vy}} it follows that 
\begin{prop}\label{prop: wass compactness bounded ac measures}
    If $R$ is compact, then $ (\mc M_{\text{ac}}^s(R,a,b), \m W_1^{1,1})$ is a compact metric space for $a,b\in \m R$.
\end{prop}

\subsection{Wasserstein distance for flows}\label{section: wasserstein for tiling flows}

Let $R$ be a compact region which is the closure of a connected domain and has piecewise smooth boundary. If $f_\tau\in TF_n(R)$, then $f_\tau$ is supported in $B_{2/n}(R) = \{x : d(x,R) \leq 2/n\}$. We will define a correspondence between 
\begin{enumerate}
	\item vector fields $f$ on $R\subset \mathbb{R}^3$ or $f_\tau\in TF_n(R)$, and
	\item triples of signed measures $(\mu_1, \mu_2, \mu_3)$ supported in $B_{2/n}(R)$.
\end{enumerate}
The idea is that the flow of the vector field in coordinate direction $i$ corresponds to mass of the $i^{th}$ measure, with sign coming from the direction of the flow. We define Wasserstein distance on vector fields through this correspondence: 
\begin{definition}\symindex{Chapter 5!$d_W$ -Wasserstein distance on flows}\termindex{Chapter 5!Wasserstein distance on flows}
	The \textit{Wasserstein distance on flows}, denoted $d_W$, is the sum of the generalized Wasserstein distances between the component measures. For any two vector fields $f,g$ with corresponding triples of measures $(\mu_1,\mu_2,\mu_3)$ and $(\nu_1,\nu_2,\nu_3)$, we define 
	\begin{align*}
		d_W(f, g) := \m W_1^{1,1}(\mu_1,\nu_1) +  \m W_1^{1,1}(\mu_2,\nu_2) +  \m W_1^{1,1}(\mu_3,\nu_3).
	\end{align*}
\end{definition}

To complete the definition of the metric, we need to define the correspondences between vector fields and triples of measures. There will be two definitions, one for a measurable vector field on $R$ and one for a discrete vector field $f_\tau\in TF_n(R)$. Let $x = (x_1,x_2,x_3)$ denote a point in $\mathbb{R}^3$. 

\begin{definition}\termindex{Chapter 5!measures corresponding to a measurable vector field}
	(Measures corresponding to a measurable vector field.) Let $f$ be a measurable vector field supported in $R$. The components of $f(x) = (f_1(x), f_2(x), f_3(x))$ are measurable functions, and we define the corresponding triple of measures $(\mu_1,\mu_2,\mu_3)$ by 
	\begin{align*}
		\dd\mu_i(x) = f_i(x) \,\dd x_1 \dd x_2 \dd x_3 \qquad \qquad i = 1,2,3
	\end{align*}
	where $\dd x_1 \dd x_2 \dd x_3$ denotes Lebesgue measure on $\mathbb{R}^3$. 
\end{definition}

\begin{definition}\termindex{Chapter 5!measures corresponding to a free-boundary tiling flow}
	(Measures corresponding to a free-boundary tiling flow on $R$.) Suppose that $f= f_\tau\in TF_n(R)$ for some $n$. Let $\eta_1,\eta_2,\eta_3$ be the positively-oriented unit basis vectors. Orient all the edges $e$ of $\frac{1}{n}\mathbb{Z}^3$ to be parallel to $\eta_i$, which we denote by $e\parallel \eta_i$. (Recall that changing the orientation of $e$ changes the sign of $f(e)$.) The triple of measures $(\mu_1,\mu_2,\mu_3)$ corresponding to $f$ is given by
	\begin{align*}
		\dd \mu_i(x) = \sum_{{e} \parallel \eta_i} \frac{2}{n^2} f({e}) \mathbbm{1}_{e}(x) \dd x_i  \qquad \qquad i=1,2,3
	\end{align*}
	where $\mathbbm{1}_{e}$ denotes the indicator of the edge ${e}\in \frac{1}{n} \m Z^3$, and $\dd x_i$ is 1-dimensional Lebesgue measure in the direction of $\eta_i$. Note that $\mu_i$ is supported in $B_{2/n}(R)$.
\end{definition}
\begin{rem}
    The scaling factor $\frac{2}{n^2}$ ensures that each edge $e$ such that $e\parallel \eta_i$ contributes $\frac{2f(e)}{n^3}$ total mass to $\mu_i$. The normalization is justified by looking at the extreme examples corresponding to the brickwork tilings (i.e.\ tilings where all tiles are the same type). Each cube in the $\frac{1}{n} \m Z^3$ mesh can be viewed as corresponding to its lower left edge. In the brickwork pattern, exactly half of these cubes will have a dimer in the lower left edge. 
\end{rem}

\begin{rem}\label{rem:weak convergence for signed +}
    Now that Wasserstein distance on flows is defined, we can explain why it metrizes weak convergence of the component measures. We do this by explaining how we could ``shift" everything to have positive mass and use Equation \eqref{eq: mass shift isomorphism}. For asymptotic flows, we can just add a copy of the 3-dimensional Lebesgue measure $\dd x_1 \dd x_2 \dd x_3$. For tiling flows, we note that we could have defined the corresponding measures to be positive by translating the mean-current octahedron $\mc O$ by $\eta_1 + \eta_2 + \eta_3$. After the translation, a scale $n$ tiling flow measure would take values in $\{1/(3n^2),5/(3n^2),7/(3n^2),11/(3n^2)\}$ instead of $\{-5/(3n^2),-1/(3n^2),1/(3n^2),5/(3n^2)\}$. 
    
    In terms of adding measures, translating $\mc O$ corresponds to adding a copy of 1-dimensional Lebesgue $\frac{2}{n^2} \dd x_i$ along each edge $e\parallel \eta_i$ in $\frac{1}{n} \m Z^3$ to the scale $n$ tiling flow measure $\dd \mu_i$. In the scaling limit as $n\to\infty$, this sum of 1-dimensional Lebesgue measures converges in $\m W_1^{1,1}$ to $\dd x_1 \dd x_2 \dd x_3$. By Equation \eqref{eq: mass shift isomorphism}, this implies the ``translated" tiling flows measures converge (i.e.\ ones defined on the translated $\mc O$) to the ``translated" asymptotic flow measures (i.e.\ ones shifted by adding $\dd x_1 \dd x_2 \dd x_3)$ in $W_1^{1,1}$ if and only if the tiling flow measures converge to the asymptotic flow measures in $\m W_1^{1,1}$. Since $W_1^{1,1}$ metrizes weak convergence for tight sequences \cite[Theorem 13]{Piccoli2014-vy}, $d_W$ metrizes weak convergence of the component measures corresponding to tiling and asymptotic flows.
\end{rem}

\begin{prop}\label{prop:div_free}
    The measures corresponding to tiling flows are divergence-free on the interior of $R$ in the sense of distributions, i.e.\ if $f$ is a tiling flow with corresponding measures $(\mu_1,\mu_2,\mu_3)$, then for any $\phi$ smooth and supported in a compact set $C$ contained in the interior of $R$,
    \begin{align*}
        \int_{R} \frac{\partial \phi}{\partial x_1} \, \dd \mu_1 + \int_{R} \frac{\partial \phi}{\partial x_2} \, \dd \mu_2 + \int_{R} \frac{\partial \phi}{\partial x_3} \, \dd \mu_3 = 0.
    \end{align*}
\end{prop}
\begin{proof}
For $i=1,2,3$, let $e_1^i,...,e_{k_i}^i$ be the edges from $\frac{1}{n} \m Z^3$ such that $e_j^i\parallel \eta_i$, is oriented parallel to $\eta_i$, and which intersect $R$. Let $(a_j^i,b_j^i)$ be the endpoints of $e_j^i$ such that $b_j^i-a_j^i=\eta_i$. By the fundamental theorem of calculus, 
\begin{align*}
    \sum_{i=1}^3 \int_{R} \frac{\partial \phi}{\partial x_i} \, \dd \mu_i = \frac{2}{n^2} \sum_{i=1}^3 \sum_{j=1}^{k_i} (\phi(b_j^i)-\phi(a_j^i)) f(e_j^i) 
\end{align*}
If $v= a_j^i$ or $b_j^i$ is not contained in the interior of $R$, then $\phi(v) = 0$. Therefore we can rewrite the above as a sum over vertices $v\in \frac{1}{n}\m Z^3$ contained in the interior of $R$:
\begin{align*}
    \sum_{i=1}^3 \int_{R} \frac{\partial \phi}{\partial x_i} \, \dd \mu_i = \frac{2}{n^2} \sum_{v} \phi(v) F(v),
\end{align*}
where $F(v)$ is a sum (with appropriate signs) of the six $f(e)$ terms for $e$ incident to $v$. We show $F(v) = 0$. 

Let $e_i^-,e_i^+$ denote the edges incident to $v$ and oriented parallel to $\eta_i$. Let $e_i^+$ be the one for which the orientation parallel to $\eta_i$ coincides with the orientation even to odd. Then 
\begin{align*}
    F(v) = \sum_{i=1}^3 f(e_i^+) - f(e_i^-).
\end{align*}
But this is equal to $\sum_{\tilde{e}\ni v} f(\tilde{e})$, where the edges $\tilde{e}$ incident to $v$ are all oriented even to odd. Therefore $F(v) = 0$ since $f$ is divergence-free as a discrete vector field, see Equation \eqref{eq: div free discrete}. 
\end{proof}

Next we prove a lemma about generalized Wasserstein distance for signed measures, in the case that both signed measures correspond to either tiling or asymptotic flows. This is an elementary result that we will use repeatedly.
\begin{lemma}\label{lem:covered}
	Suppose that $\mu$ and $\nu$ are measures supported on a common compact set $K$ corresponding to components of tiling or asymptotic flows. Suppose there is a partition of $K$ into sets $\mathcal{B} = \{B_1,...,B_M\}$ of diameter at most $\epsilon$ such that $\bigg|\mu(B)-\nu(B) \bigg| < \delta$ for all $B\in \mathcal{B}$. If one of the measures corresponds to a scale $n$ tiling flow, then we require that $\frac{1}{n}\leq \epsilon$. Then 
	\begin{align*}
		\m W^{1,1}_1(\mu,\nu) \leq M(10\epsilon^4 +\delta).
	\end{align*}
\end{lemma}
\begin{proof}
Let $\mu = \mu_+ - \mu_-$ and $\nu = \nu_+ - \nu_-$ be decompositions into positive measures and recall that
\begin{align*}
    \m W_1^{1,1}(\mu,\nu) = W_1^{1,1}(\mu_+ + \nu_-, \mu_- + \nu_+).
\end{align*}
Let $\tilde{\mu} = \mu_+ + \nu_-$ and $\tilde{\nu} = \mu_- + \nu_+$. To get an upper bound for the distance, it suffices to give a method for redistributing and deleting mass to transform $\tilde{\mu}$ into $\tilde{\nu}$.  

We proceed as follows: transform $\tilde{\mu}\mid_{B_1}$ into $\tilde{\nu}\mid_{B_1}$, the cost of this is at most $W^{1,1}_1(\tilde{\mu}\mid_{B_1}, \tilde{\nu}\mid_{B_1})$. Denote the new version of $\tilde{\mu}$ by $\tilde{\mu}'$. $\tilde{\mu}'$ will agree with $\tilde{\mu}$ on $R\setminus B_1$ and with $\tilde{\nu}$ on $B_1$. Next transform $\tilde{\mu}'$ into $\tilde{\nu}$ on $B_2$. This will cost at most $W^{1,1}_1(\tilde{\mu}'\mid_{B_2},\tilde{\nu}\mid_{B_2}) \leq W^{1,1}_1(\tilde{\mu}\mid_{B_2}, \tilde{\nu}\mid_{B_2})$ with equality if $B_2$ is disjoint from $B_1$. Iterating this we get that
	\begin{align*}
		\m W^{1,1}_1({\mu},{\nu}) \leq \sum_{j=1}^{k} W^{1,1}_1(\tilde{\mu}\mid_{B_j}, \tilde{\nu}\mid_{B_j})
	\end{align*}
	Now we just have to compute the distance for a single $B_j$. First spend $\delta>0$ to delete the difference in mass on $B_j$. The total mass of $\mu, \nu$ on any $B\in \mc B$ is bounded by $10 \epsilon^3$, and the furthest it would need to move is $\epsilon$. Therefore
		\begin{align*}
		W^{1,1}_1(\tilde{\mu}\mid_{B_j}, \tilde{\nu}\mid_{B_j}) \leq 10 \epsilon^4 + \delta.
	\end{align*}
	Summing over $j$ gives the result. 
\end{proof}

An inverse version of the bound in Lemma \ref{lem:covered} also holds, but with a constant depending on the small region $B$.
\begin{lemma}\label{lem:constant_order_box_Wass_bound}
Suppose $B\subset R$ is a connected region with piecewise smooth boundary. If $\mu, \nu$ are component measures of tiling or asymptotic flows and $\m W_1^{1,1}(\mu,\nu)<\delta$, then there is a constant $C(B)$ depending only on $B$ such that
\begin{align*}
    \m W_1^{1,1}(\mu\mid_B, \nu\mid_B) < \delta + (C(B)+1) \delta^{1/2}.
\end{align*}
\end{lemma}
\begin{rem}\label{rem:constant_explained}
    The constant $C(B)$ is not hard to understand and control. The term $C(B)\delta^{1/2}$ is bounded by $2$ times the volume of the $\delta^{1/2}$ annulus with inner boundary $\partial B$. 
\end{rem}
\begin{proof}
    The redistribution, addition, and deletion of mass $\mu \to \nu$ gives a redistribution $\mu\mid_B\to \nu\mid_B$, except any mass moved into or out of $B$ now becomes an $L^1$ cost rather than a cost proportional to distance moved. Let $f(r)$ be the amount of flow moved distance $r$ into or out of $B$ by the $\mu\to \nu$ redistribution. Then 
    \begin{align*}
        \m W_1^{1,1}(\mu\mid_B,\nu\mid_B) \leq \delta + \int_0^\infty f(r) dr. 
    \end{align*}
    On the other hand, 
    \begin{align*}
        \int_{0}^\infty r f(r) dr < \delta. 
    \end{align*}
    We split the integral we want to bound into two pieces: 
    \begin{align*}
        \int_{0}^\infty f(r) dr = \int_{0}^{\delta^{1/2}} f(r) dr + \int_{\delta^{1/2}}^\infty f(r) dr. 
    \end{align*}
    Since $\mu,\nu$ are measures corresponding to components of tiling or asymptotic flows, we have $$\int_{0}^{\delta^{1/2}} f(r) dr \leq C(B) \delta^{1/2}$$ (this quantity is proportional to the volume of the $\delta^{1/2}$ annulus around $B$) and
    \begin{align*}
       \delta^{1/2} \int_{\delta^{1/2}}^\infty f(r) dr \leq \int_{\delta^{1/2}}^\infty r f(r) dr < \delta. 
    \end{align*}
    Combining these gives the desired bound. 
\end{proof}

\begin{lemma}\label{lem:boundedlimits}
	Let $\nu_n$ be a sequence of signed measures supported in $R$ which converges in $\m W_1^{1,1}$ to another measure $\nu$. Further suppose the $\nu_n$ are absolutely continuous with respect to 3-dimensional Lebesgue measure, and their densities $g_n(x)$ take values in $[-m,M]$. Then $\nu$ is also absolutely continuous with respect to 3-dimensional Lebesgue measure. 
\end{lemma}
\begin{proof}
    By Corollary \ref{cor: weak convergence wasserstein}, if $\m W_1^{1,1}(\nu_n,\nu)\to 0$ as $n\to \infty$, then $\nu_n$ converges to $\nu$ in the weak topology. 
    
    Define an operator $Q: C^\infty(R)\to \mathbb{R}$ by integrating against $\nu$: 
    \begin{align*}
        Q(h) := \int h \, \dd \nu. 
    \end{align*}
    Using Cauchy-Schwarz and the fact that $d\nu_n = g_n(x)\, \dd x$ we have that 
    \begin{align*}
        Q(h) = \lim_{n\to \infty} \int h \, g_n \, \dd x \leq \limsup_{n\to \infty} \norm{h}_{L^2(R)} \norm{g_n}_{L^2(R)} \leq \text{Vol}(R)^{1/2} (M+m)\norm{h}_{L^2}
    \end{align*}
    Therefore $Q$ extends to an operator on $L^2(R)$. By the Riesz representation theorem this means there exists an $L^2$ function $g$ such that $Q(h) = \langle h,g\rangle = \int h\,g \,\dd x$. Therefore $\dd\nu(x) = g(x) \dd x$, and $\nu$ is absolutely continuous with respect to Lebesgue measure on $\mathbb{R}^3$.
\end{proof}

\begin{prop}\label{prop: wass vs sup}
	Suppose that $f,g \in AF(R)$ are continuous and satisfy $|f(x)-g(x)| < \epsilon$ for all $x\in R$. Then $d_W(f,g) < \epsilon \,\text{vol}(R)$.
\end{prop}
\begin{proof}
	The $d_W$-distance from $f$ to $g$ is bounded by adding and subtracting mass for each of the component functions. Since pointwise they differ by at most $\epsilon$,  $d_W(f,g)\leq \epsilon\, \text{vol}(R)$. 
\end{proof}

\subsection{Main theorems}

Here we prove two of the theorems mentioned at the beginning of the section. First we show that fine-mesh limits of tiling flows are asymptotic flows. 

\begin{thm}\label{thm:formal_fine_mesh}
Let $R\subset \m R^3$ be a compact region which is the closure of a connected domain and has piecewise-smooth boundary. Let $f_n\in TF_{m_n}(R)$ be a free-boundary tiling flow on $R$ at scale $m_n$ with $m_n$ going to infinity with $n$. Any $d_W$-subsequential limit of tiling flows $f_* = \lim_{k\to \infty} f_{{n_k}}$ is in $AF(R)$. 
\end{thm}

Let $\mu_k = (\mu_k^1,\mu_k^2, \mu_k^3)$ be the measures corresponding to $f_{{n_k}}$ and let $\mu_* = (\mu_*^1,\mu_*^2,\mu_*^3)$ be the measures corresponding to $f_*$. The main idea of the proof is to smoothen the measures $(\mu_k^1,\mu_k^2, \mu_k^3)$ corresponding to the tiling flow in an especially nice way, then apply Lemma \ref{lem:boundedlimits} to say that their limits $(\mu_*^1,\mu_*^2,\mu_*^3)$ are absolutely continuous with respect to 3-dimensional Lebesgue measure. This shows that $f_*$ is a measurable flow on $R$. Using our well-chosen smoothings, we will show that $f_*$ has the other properties that an asymptotic flow must have. To make the argument easier to digest, we break down the construction of the smoothing into two lemmas.

\begin{lemma}\label{lem:niceboxes}
Let $S_N = [0,N-1]^3$ (note that there are $N$ vertices from $\m Z^3$ on each edge of $S_N$). Let $\tau$ be a tiling of $\mathbb{Z}^3$ with tiling flow $f_\tau$, and let $(\mu_1,\mu_2,\mu_3)$ be the corresponding measures. Then
\begin{align*}
    \frac{1}{\text{Vol}(S_N)} \bigg(\int_{S_N} \dd\mu_1, \int_{S_N} \dd\mu_2, \int_{S_N} \dd\mu_3 \bigg)
\end{align*}
is valued in $(1+ O(1/N)) \mathcal{O}$.
\end{lemma}
\begin{proof}
    Let $E(S_N)$ denote the edges intersecting in $S_N$. All edges intersecting $S_N$ in more than one point are contained in it and have length $1$. 
    
    The measure $\mu_1$ is supported on the edges parallel to $\eta_1$, and thus 
    $$\int_{S_N}\dd\mu_1 = \sum_{E(S_N)\ni e\parallel \eta_1} 2f(e),$$
    for edges $e$ oriented parallel to $\eta_1$. The results for $\mu_2,\mu_3$ are analogous. 
    
    For each $i=1,2,3$, there are $N^2(N-1)$ edges from $\m Z^3$ contained in $S_N$ parallel to $\eta_i$. This number is always even, so half of these edges from ($N^2(N-1)/2$) have even-to-odd orientation parallel to $\eta_i$ and half have the opposite orientation. Let $\alpha_+$ be the number of even-to-odd oriented tiles in $\tau\cap E(S_N)$ and $\alpha_-$ be the number of odd-to-even oriented tiles. Then
    \begin{align*}
        \int_{S_N} \dd\mu_1 = \sum_{E(S_N)\ni e\parallel \eta_i} 2 f(e) = 2(\alpha_+ - \alpha_-).
    \end{align*}
    Let $s_i$ denote the fraction of tiles in the $+\eta_i$ direction in $S_N$ minus the fraction of tiles in the $-\eta_i$ direction in $S_N$. Note that irrespective of the tiling we have that number of tiles in $S_N$ is $N^2(N-1)/2 +O(N^2)$. Thus it follows that
    \begin{align*}
        s_1 = \frac{2(\alpha_+-\alpha_-)}{N^2(N-1)+O(N^2)} = \frac{1}{N^2(N-1)+O(N^2)} \int_{S_N} \dd\mu_1.
    \end{align*}
    A similar equation holds for $s_2,s_3$. We have that $(s_1,s_2,s_3)\in \mathcal{O}$. Thus 
    \begin{align*}
        \frac{1}{\text{Vol}(S_N)} \bigg(\int_{S_N} \dd\mu_1, \int_{S_N} \dd\mu_2, \int_{S_N} \dd\mu_3 \bigg) &= \frac{N^2(N-1)+O(N^2)}{(N-1)^3}(s_1, s_2, s_3) \\
        &= (1 + \frac{O(N^2)}{(N-1)^3})(s_1,s_2,s_3),
    \end{align*}
    for some constant $C$ which is in $(1+ O(1/N)) \mathcal{O}$.
\end{proof}
If we smoothen the measures corresponding to the tiling flow over a partition consisting of boxes of the form in Lemma \ref{lem:niceboxes}, we can construct smoothings that satisfy a very nice list of properties. 
\begin{lemma}\label{lem:spreadflows}
	Let $(\mu_1,\mu_2,\mu_3)$ be measures corresponding to a tiling flow $f\in TF_n(R)$. For any $\epsilon>0$, there exists $\nu= (\nu_1,\nu_2,\nu_3)$ satisfying the following properties: 
	\begin{enumerate}
	    \item $\nu_i$ is supported in $R$ for all $i=1,2,3$;
	    \item $\nu_i$ has a density $g_i(x)$ with respect to Lebesgue measure on $\mathbb{R}^3$ for $i=1,2,3$;
	    \item $g = (g_1,g_2,g_3)$ is valued in $(1+O(\epsilon)) \mathcal{O}$ as a distribution;
	    \item $d_W(\mu,\nu) < C(R) (\epsilon n)^{-1}$ where $C(R)$ is a constant depending only on $R$. 
	\end{enumerate}
\end{lemma}
\begin{proof}
Choose $N$ such $N=\lfloor\frac1\epsilon\rfloor$ and a partition of cubes $\mc B = \{B_1,...,B_M\}$ that cover $R$, where each $B_i$ is an $N\times N\times N$ cube in $\frac{1}{n} \m Z^3$ (we define the flow $f$ to be $0$ outside $R$). This can be done so that $M\sim n^3/N^3$. For $i=1,2,3$ and all $B\in \mathcal{B}$, define 
\begin{align*}
    C_B^i := \frac{1}{\text{Vol}(B)} \int_{B} \dd \mu_i = \frac{(N-1)^3}{n^3}\int_{B} \dd \mu_i.
\end{align*}
Define the densities of $\nu_i$ by
\begin{align*}
    g_i(x) =  C_B^{i} \qquad \forall \, x\in B\cap R.
\end{align*}
This satisfies properties 1 and 2. By Lemma \ref{lem:niceboxes}, $(g_1,g_2,g_3)$ is valued in $(1+O(1/N)) \mc O= (1+O(\epsilon))\mc O$ as a distribution which completes property 3. Finally by Lemma \ref{lem:covered} applied to the partition $\mc B$, we have that 
\begin{align*}
    d_W(\mu,\nu) \leq M (N/n)^4\leq C(R)(n/N)^3 (N/n)^4\leq C(R)\frac{1}{\epsilon n}
\end{align*}
where $C(R)$ is a constant depending only on $R$.
\end{proof}

We now return to the proof of the theorem.

\begin{proof}[Proof of Theorem \ref{thm:formal_fine_mesh}] Fix $\epsilon>0$. Recall that $\mu_k = (\mu_k^1,\mu_k^2, \mu_k^3)$ is the triple of measures corresponding to $f_{{n_k}}$. Choose $\epsilon_k =  n_k^{-1/2}$ and let $\nu_k = (\nu_k^1,\nu_k^2, \nu_k^3)$ be the measures constructed in Lemma \ref{lem:spreadflows} for $\epsilon_k>0$. For $k$ large enough so that $d_W(\mu_*,\mu_k)<\epsilon$, by the triangle inequality
\begin{align*}
    d_W(\mu_*, \nu_k) \leq d_W(\mu_*,\mu_k) + d_W(\mu_k, \nu_k) \leq \epsilon + C(R) n_k^{-1/2}.
\end{align*}
Therefore the triple of measures $\nu_k$ also converges to $\mu_*$ in $d_W$. By Lemma \ref{lem:boundedlimits}, there are functions $f_*^i$ such that  $\mu_*^i = f_*^i(x) \dd x$ for $i=1,2,3$, so $f_*$ is a measurable vector field. It remains for us to check the additional properties to show that $f_*$ is an asymptotic flow. 

Since the $\nu_k$ are supported in $R$ for all $k$, so is $f_*$. By Lemma \ref{lem:spreadflows}, the densities of $\nu_k$ are valued in $(1+ O(\epsilon_k)) \mathcal{O}$, so the densities of $\mu_*$ are valued in $\mathcal{O}$ (any open neighborhood is a continuity set, so we get that the averages of $(f_*^1,f_*^2,f_*^3)$ are valued in $(1+ O(\epsilon_k)) \mathcal{O}$ for all $\epsilon_k$. This plus the Lebesgue differentiation theorem imply that $f_*$ is valued in $\mathcal{O}$). On the other hand, convergence in $d_W$ implies weak convergence of the component measures (Corollary \ref{cor: weak convergence wasserstein}). Since $\mu_k$ is divergence-free in the sense of distributions on the interior of $R$ (Proposition \ref{prop:div_free}) for all $k$, $\mu_*$ is also divergence-free in the sense of distributions on the interior of $R$. Thus $f_*\in AF(R)$.
\end{proof}

\begin{thm}\label{thm:AF_compact}
    The metric space $(AF(R), d_W)$ is compact.
\end{thm}
\begin{proof}

By Proposition \ref{prop: wass compactness bounded ac measures}, the space of triples of measures absolutely continuous with respect to Lebesgue measure, supported in $R$, and valued in $\mc O$ is compact. Since $(AF(R),d_W)$ is a subspace of this, it suffices to show that it is closed. Suppose that $\mu_n = (\mu_n^1, \mu_n^2,\mu_n^3)$ is a sequence in $AF(R)$ with $d_W$-limit $\mu_* = (\mu_*^1,\mu_*^2,\mu_*^3)$. By Lemma \ref{lem:boundedlimits}, $\mu_*^i$ has a density $g_*^i(x)$ for each $i=1,2,3$. Since convergence in $\m W_1^{1,1}$ implies weak convergence (Corollary \ref{cor: weak convergence wasserstein}), $g_* = (g_*^1,g_*^2, g_*^3)$ is divergence-free. To show that $g_*$ is valued in $\mathcal{O}$, note that any open ball $B\subset R$ is a continuity set, so since $\frac{1}{\text{Vol}(B)}(\int_B \dd \mu_n^1, \int_B \dd \mu_n^2, \int_B \dd\mu_n^3)\in \mc O$, the average of $g_*$ over $B$ is also valued in $\mathcal{O}$. Thus $g_*$ is valued in $\mathcal{O}$ by the Lebesgue differentiation theorem. Thus $(AF(R),d_W)$ is compact.  
\end{proof}

Now we will prove that any asymptotic flow can be approximated by a smooth flow which is divergence-free on a slightly smaller region. This is a standard construction which we provide for completeness as it will be used in the next subsection. Essentially all we need to do is to convolve the asymptotic flow with an appropriate smooth bump function. For this, given a region $R$ and $\epsilon>0$ we define\symindex{Chapter 5!$R_{\epsilon}$ - the domain $R$ with an $\epsilon$ boundary removed}
\begin{equation}\label{eq:R_epsilon}
    R_{\epsilon}:=\{x \in R~:~d(x, \partial R)\geq\epsilon\}.
\end{equation}
We will denote the smooth asymptotic flows on a region $R$ by $AF^\infty(R)\subset AF(R)$\symindex{Chapter 5!$AF^\infty(R)$ - the space of smooth asymptotic flows on $R$}. Given $f\in AF(R),g\in AF(R_{\epsilon})$, we say that $d_W(f, g)$ is the distance between $f$ and $g$, where $g$ is extended to be $0$ on
$R\setminus R_{\epsilon}$.
\begin{prop}\label{prop:bump_function}
Fix $f\in AF(R)$. For all $\epsilon>0$ small enough, there is a smooth asymptotic flow $g\in AF^\infty(R_{\epsilon})$ such that
$$d_W(f, g)< K \sqrt{\epsilon}$$
where $K$ is a constant depending only on $R$.
\end{prop}

\begin{proof} 
Consider a bump function $\phi\in C^{\infty}(B_{\epsilon}(0))$, that is, it is a non-negative smooth function such that $\phi|_{\partial B_{\epsilon}(0)}=0$ and $\int_{B_{\epsilon}(0)}\phi(x)\,\dd x=1$.

Let $g= f\ast \phi|_{R_{\epsilon}}$ and suppose that $\psi$ is a smooth test function with compact support in the interior of $R_\epsilon$. To check that $g$ is divergence-free in the interior of $R_\epsilon$ we look at the integral
\begin{eqnarray*}
\int_{R_{\epsilon}}(\nabla\psi\cdot g)(x)\,\dd x&=&
\int_{R_{\epsilon}}\int_{B_\epsilon( 0)}(\nabla\psi\cdot f)(x-y)\phi(y)\,\dd y\,\dd x\\
&=&\int_{B_\epsilon( 0)}\int_{R_{\epsilon}}(\nabla\psi\cdot f)(x-y)\phi(y)\,\dd x\,\dd y=0.
\end{eqnarray*}
Here the last equality uses that $f$ is divergence-free in the interior of $R$. We have that $g(x)=f\ast \phi(x)$ is an average of elements in $\mathcal O$. Since $\mathcal O$ is convex it follows that $g$ takes values in $\mathcal O$. Finally we estimate $d_W(f,g)$. The amount of mass which we might have to delete or add from $R\setminus R_{\epsilon}$ is bounded by
$6(\text{Vol}(R\setminus R_{\epsilon}))\leq c\epsilon^2$ where $c$ depends only on $\partial R$. Now let $(\mu_1, \mu_2, \mu_3)$ and $(\nu_1, \nu_2, \nu_3)$ be the measures corresponding to $f$ and $g$ respectively. Let $\delta>0$. We have that if $B$ is a box of side length $\delta$ contained in $R_{\epsilon}$ then
$$\nu_i(B)=\int_{B}\int_{B_{\epsilon}(0)} f_i(x- y)\phi(y)\,\dd y\,\dd x= \int_{B_{\epsilon}(0)}\mu_i(B-y)\phi(y)\,\dd y.$$
It follows that $|\nu_i(B)-\mu_i(B)|$ is less than the volume of the annular region around $B$ of radius $\epsilon$. Thus we have that
$$|\nu_i(B)-\mu_i(B)|< C\epsilon\delta^2$$ where $C$ is independent of $\epsilon$ and $\delta$. Partition $R_{\epsilon}$ into boxes $B_1, B_2, \ldots B_M$; where $M\sim \delta^{-3}$. By Lemma \ref{lem:covered} we have that
$$\m W^{1,1}_1(\mu_i, \nu_i)< M(10\delta^4 + C \epsilon \delta^2) + c\epsilon^2\sim \delta^{-3}(10\delta^4 + C \epsilon \delta^2) + c\epsilon^2$$
Since $\delta$ is a free parameter, we can take $\delta \sim \sqrt{\epsilon}$ to complete the proof.
\end{proof}

\subsection{Boundary values of asymptotic flows}\label{sec: boundary values of asymptotic flows}

In this section we define the boundary values of asymptotic flows on $\partial R$. In fact we do something slightly more general, and define the restriction of an asymptotic flow (or tiling flow in the next subsection) on a whole class of surfaces, namely \symindex{Chapter 5!$\m S(R)$ - compact piecewise smooth surfaces contained in $R$}
\begin{align*}
    \m S(R) = \{\text{compact piecewise smooth surfaces contained in }R\}
\end{align*}
Note that $R$ is closed, so $\partial R \in \m S(R)$. This general set up will make things easier to prove. We also use the trace operator for other surfaces in Section \ref{sec:generalized_patching}. Recall that $AF^\infty(R)\subset AF(R)$ is the smooth asymptotic flows. \termindex{Chapter 5!trace operator smooth asymptotic flows}\symindex{Chapter 5!$T : AF^\infty(R)\times \m S(R) \to \mc M^s(R)$ - the trace operator on smooth asymptotic flows}

\begin{definition}\label{def: trace on smooth flows}
We define the \textit{trace operator} on smooth asymptotic flows 
\begin{align*}
    T : AF^\infty(R)\times \m S(R) \to \mc M^s(R).
\end{align*}
by
 \begin{align*}
     T(f,S)(x) = \langle f(x), \xi_S(x)\rangle \dd \sigma_S(x),\qquad x\in S,
 \end{align*}
 where $\dd \sigma_S$ denotes surface area measure on $S$ and $\xi_S(x)$ denotes the $L^2$ unit normal vector to $S$ at $x$.
\end{definition}
We show that $T(\cdot,S)$ extends to a uniformly continuous map $(AF(R),d_W)\to (\mc M^s(R), \m W_1^{1,1})$ for all $S\in \m S(R)$. We do this in three main steps: 
\begin{itemize}
    \item Show that $T(\cdot,S)$ is uniformly continuous on $AF^\infty(R)$ when $S$ is a small patch. 
    \item Extend this result to $AF(R)$ by approximation and compatibility results. Since we don't know if $AF^\infty(R)$ is dense in $AF(R)$, this requires slightly more care. 
    \item Extend the uniform continuity result to general $S\in \m S(R)$ by putting together the patches. 
\end{itemize}

\begin{prop}\label{prop: uniform continuity of boundary operator_smooth_1}
	Suppose $S\in \m S(R)$ and is such that there exists a nonzero vector $v$ and a parameter $\theta>0$ for which $S+tv$ is contained in $\text{Int}(R)$ and disjoint for all $0\leq t \leq \theta$. Then for all $\epsilon>0$ there exists $\delta>0$ such that for all  $f, g\in AF^\infty(R)$ with $d_W(f, g)<\delta$, we have that $$\m W^{1,1}_1(T(f, S), T(g, S))<\epsilon.$$
\end{prop}

\begin{proof}
Fix two parameters $\gamma_1, \gamma_2>0$ which we will specify at the end of the proof. Partition $S$ into patches $\alpha_1,\ldots \alpha_M$ such that 
\begin{itemize}
	\item $\alpha_i$ is a smooth surface with piecewise smooth boundary for all $i=1,...,M$.
	\item $\alpha_i$ has diameter at most $\gamma_1$ for all $i$, and $M\leq C \gamma_1^{-2}$ for some constant $C$ depending on $S$.
	\item Let $\alpha_i(t):=\alpha_i+ tv$ for $0\leq t\leq \theta$. For all $i=1,...,M$, $\alpha_i(t)\cap \alpha_i(s) = \emptyset$ for $s\neq t$.
\end{itemize}
Let 
$\mu_f=T(f, S)$, $\mu_g=T(g, S)$ and define $\Delta>0$ by
\begin{equation}\label{eq:Delta_1}
	\Delta := \sup_{1\leq i\leq M} \bigg|\mu_f(\alpha_i)-\mu_g(\alpha_i)\bigg|.
\end{equation}
By the two-dimensional version of Lemma \ref{lem:covered},
\begin{equation}\label{eq:boundary_dist_Delta_bound_1}
	\m W^{1,1}_1(T(f, S),T(g, S)) \leq M(10\gamma_1^3 + \Delta) \leq 10 C \gamma_1 + C \gamma_1^{-2} \Delta. 
\end{equation}
Note that the power of $\gamma_1$ is $3$ instead of $4$ because $S$ is two-dimensional. It remains to find a bound for $\Delta$ in terms of $d_W(f,g)$. 

If $h\in AF^\infty(R)$, then $h$ is divergence-free and hence its flux through any closed surface is zero. This implies that there exists a threshold $\theta_{\gamma_2}\in (0,1)$ such that for all $0\leq t\leq \theta_{\gamma_2}$ and $h\in AF^\infty(R)$,
\begin{align}
	\sup_{1\leq i\leq M}\bigg|\int_{\alpha_i} \langle h(x), \xi(x)\rangle\, \dd \sigma_{\alpha_i}(x) - \int_{\alpha_i(t)} \langle h (x), \xi(x) \rangle\, \dd \sigma_{\alpha_i(t)}(x)\bigg| \leq \gamma_2.\label{equation: surface measures close_1}
\end{align}
Here $\xi(x)$ is the normal vector on the surfaces $\alpha_i(t)$ with appropriate orientation and we are applying the divergence theorem to the boundary of the region $U_i=\cup_{s=0}^t\alpha_i(t)$. Since $h$ takes values in the compact set $\mc O$, the threshold $\theta_{\gamma_2}$ can be taken independent of $h$. Applying this to $f,g\in AF^\infty(R)$ it follows that
\begin{align*}
	\sup_{1\leq i\leq M}\bigg|\int_{\alpha_i} \langle f(x)-g(x), \xi(x)\rangle\, \dd \sigma_{\alpha_i}(x) - \int_{\alpha_i(t)} \langle f (x)- g(x), \xi(x) \rangle\, \dd \sigma_{\alpha_i(t)}(x)\bigg| \leq 2\gamma_2
\end{align*}
for all $0\leq t\leq \theta_{\gamma_2}$. Observe that the first term in the inequality is $\mu_{ f}(\alpha_i)- \mu_{ g}(\alpha_i)$. Integrating over $t\in (0,\theta_{\gamma_2})$, 
\begin{align*}
	\bigg| \bigg| \int_{0}^{\theta_{\gamma_2}}\bigg(\int_{\alpha_i(t)} \langle f (x)-g(x), \xi(x)\rangle\, \dd \sigma_{\alpha_i(t)} (x)\bigg)\, \dd t\bigg| - \theta_{\gamma_2} (\mu_f(\alpha_i)-\mu_g(\alpha_i))\bigg| \leq 2\theta_{\gamma_2} \gamma_2. 
\end{align*}
Taking the supremum over $i=1,...,M$ we get that
\begin{align*}
	\Delta\leq 2\gamma_2 + \frac{1}{\theta_{\gamma_2}}\sup_{1\leq i \leq M} \bigg| \int_{0}^{\theta_{\gamma_2}}\bigg(\int_{\alpha_i(t)} \langle f (x)-g(x), \xi(x)\rangle\, \dd \sigma_{\alpha_i(t)} (x)\bigg)\, \dd t\bigg|.
\end{align*}
Suppose the supremum on the right hand side of the equation is achieved by the index $i$, and let $\alpha(t):=\alpha_i(t)$ to simplify notation. Then plugging this into \eqref{eq:boundary_dist_Delta_bound_1} gives
\begin{equation}\label{eq:trace_distance_bound_1}
	\m W^{1,1}_1(T(f, S), T(g, S)) 
 \end{equation}
$$\leq 10 C \gamma_1 + 2C\gamma_1^{-2} \gamma_2+ C \gamma_1^{-2} \theta_{\gamma_2}^{-1} \bigg| \int_{0}^{\theta_{\gamma_2}}\bigg(\int_{\alpha(t)} \langle f (x)-g(x), \xi(x) \rangle\, \dd \sigma_{\alpha(t)} (x)\bigg)\, \dd t\bigg|.$$
Since $\alpha_i$ is smooth, by an appropriate change of variables we can rewrite the integral above as an integral over $U = \cup_{t=0}^{\theta_{\gamma_2}} \alpha(t)$.
\begin{align*}
    \int_{0}^{\theta_{\gamma_2}}\bigg(\int_{\alpha(t)} \langle f (x)-g(x), \xi(x) \rangle\, \dd \sigma_{\alpha(t)} (x)\bigg)\, \dd t &= \int_{U} \langle f(x)-g(x), \xi(x)\rangle \varphi(x) \, \dd x_1 \dd x_2 \dd x_3\\&= \sum_{k=1}^3 \int_{U} (f_k(x)-g_k(x)) \xi_k(x) \varphi(x) \, \dd x_1 \, \dd x_2 \, \dd x_3. 
\end{align*}
Here $\varphi(x)$ is the factor coming from the Jacobian in the change of variables. Since $\alpha(t)$ is smooth for all $t\in[0,\theta_{\gamma_2}]$, $\varphi(x)$ and $\xi(x)$ are both smooth functions on $U$, so $\psi_k(x):=\xi_k(x) \varphi(x)$ is a smooth and therefore Lipschitz function on $U$. Let $\lambda$ denote the maximum of Lipchitz constants of $\psi_k$, $k=1,2,3$. Then by the dual definition of the Wasserstein metric (Proposition \ref{prop: dual definition wasserstein}), 
\begin{align*}
    \int_{0}^{\theta_{\gamma_2}}\bigg(\int_{\alpha(t)} \langle f (x)-g(x), \xi(x) \rangle\, \dd \sigma_{\alpha(t)} (x)\bigg)\, \dd t \leq \lambda\, d_W(f\mid_U, g\mid_U).
\end{align*}
By Lemma \ref{lem:constant_order_box_Wass_bound}, there is a constant $C(U)$ such that if $d_W(f,g)<\delta$ then $d_W(f\mid_U,g\mid_U) < \delta + C(U) \delta^{1/2}$. Therefore substituting this in to Equation \eqref{eq:trace_distance_bound_1} gives that if $d_W(f,g)<\delta$ then
\begin{align*}
	\m W^{1,1}_1(T(f, S), T(g, S)) \leq 10 C \gamma_1 + 2C\gamma_1^{-2} \gamma_2+ C \gamma_1^{-2} \theta_{\gamma_2}^{-1} \lambda(\delta+ C(U) \, \delta^{1/2}).
\end{align*}
Taking $\gamma_2=\gamma_1^3$ completes the proof.
\end{proof}

We will now prove that perturbing $S$ by a small amount does not change the trace very much.

\begin{prop}\label{prop: uniform continuity of boundary operator_smooth_2}
	Let $S\in \m S(R)$ be a surface satisfying the conditions of Proposition \ref{prop: uniform continuity of boundary operator_smooth_1} for some vector $v$ and threshold $\theta>0$.
	For all $\epsilon>0$ there exists $\delta>0$ such that for all  $t<\delta$ and $f\in AF^\infty(R)$, we have that $$\m W^{1,1}_1(T(f, S), T(f, S+tv))<\epsilon.$$
\end{prop}

\begin{proof} This proof is much simpler than that of Proposition \ref{prop: uniform continuity of boundary operator_smooth_1}. As in that proof, we take parameters $\gamma_1, \gamma_2>0$ to be fixed later and choose a partition of $\alpha_1\ldots, \alpha_M$ of $S$ such that $\alpha_i$ has diameter less than $\gamma_1$ for all $i$, and $M\leq C\gamma_1^{-2}$ for some $C>0$ independent of $\gamma_1$. In addition, by choosing a larger constant $C$ if necessary we can assume that 
$$\sigma_{S}(\alpha_i)\leq C\gamma_1^2.$$
As in \eqref{equation: surface measures close_1}, there exists a threshold $\theta_{\gamma_2}\in (0,1)$ such that for all $0\leq t\leq \theta_{\gamma_2}$ and $f\in AF^\infty(R)$,
\begin{align}\label{equation: surface measures close_3}
	\sup_{1\leq i\leq M}\bigg| T(f, S)(\alpha_i)- T(f, S+tv)(\alpha_i(t)) \bigg| \leq \gamma_2.
\end{align}
To get an upper bound for the distance between the traces on each patch, we give a method for redistributing, adding, and deleting mass to transform $T(f, S)\mid_{\alpha_i}=T(f,\alpha_i)$ into $T(f,\alpha_i(t))$. Note that these are both are signed measures absolutely continuous with respect to $\sigma_S$ and $\sigma_{S+tv}$ respectively, and both have densities bounded between $-1$ and $1$. We can transform one to the other by (1) adding $\gamma_2$ flow, (2) moving flow distance at most $t+\gamma_1$. There is at most $4 \sigma_S(\alpha_i)$ total flow from both measures. Hence
$$\m W^{1,1}_1(T(f, S)|_{\alpha_i}, T(f, S+tv)|_{\alpha_i(t)})\leq \gamma_2+ 4(t+\gamma_1)(\sigma_{S}(\alpha_i))\leq \gamma_2+ 4C(t+\gamma_1)\gamma_1^2.$$
As in the proof of Lemma \ref{lem:covered}, combining the contributions of the patches,
\begin{eqnarray*}
 \m W^{1,1}_1(T(f, S), T(f, S+tv))\leq M\left(\gamma_2+ 4C(t+\gamma_1)\gamma_1^2\right)\leq C\gamma_1^{-2}\gamma_2+ 4C^2(t+\gamma_1).
\end{eqnarray*}
As before, by setting $\gamma_2=\gamma_1^3$, the result follows.
\end{proof}

Using Proposition \ref{prop: uniform continuity of boundary operator_smooth_1} and Proposition \ref{prop: uniform continuity of boundary operator_smooth_2}, when $S\in \m S(R)$ satisfies the conditions of Proposition \ref{prop: uniform continuity of boundary operator_smooth_1} for a vector $v$ and threshold $\theta>0$, we can extend the definition of $T(S,\cdot)$ to all of $AF(R)$ as follows. Given $t>0$, there exists $K$ such that for all $k\geq K$, $S(t)\subset R_{1/k}$. Here recall that 
\begin{align*}
    R_{1/k} = \{x\in R ~:~ d(x,\partial R) \geq 1/k\}.
\end{align*}
By Proposition \ref{prop:bump_function}, for any $f\in AF(R)$ we can find a sequence $g_k\in AF^\infty(R_{1/k})$ such that $d_W(f, g_k)\to 0$ as $k\to \infty$. For $t>0$, we define
\begin{equation}\label{eq: AF(R) trace}
    T(f,S(t)) := \lim_{k\to \infty} T(g_k,S(t))
\end{equation}
where the limit is taken with respect to the metric $d_W$ on flows supported in $R$. Note that if $R'\subset R$, then the projection map $AF^\infty(R) \to AF^\infty(R')$ given by $f\mapsto f\mid_{R'}$ is continuous. By Proposition \ref{prop: uniform continuity of boundary operator_smooth_1}, $T$ is uniformly continuous on $AF^\infty(R_{1/k})$ for any $k>0$, so the limit in Equation \eqref{eq: AF(R) trace} is independent of the approximating sequence $(g_k)_{k\geq 1}$ and converges to $T(f,S+tv)$ if $f\in AF^\infty(R)$. Since $(AF(R),d_W)$ is compact (Theorem \ref{thm:AF_compact}), $T(\cdot, S(t)):AF(R)\to \mc M^s(R)$ is a uniformly continuous map for $t>0$. 

Further, by a variant of Lemma \ref{lem:boundedlimits} where 3-dimensional Lebesgue measure is replaced by the surface area measure $\sigma_{S(t)}$, for any $f\in AF(R)$, $T(f,S(t))$ is a signed measure absolutely continuous to $\sigma_{S(t)}$ with density bounded between $-1$ and $1$. By Proposition \ref{prop: uniform continuity of boundary operator_smooth_2}, for any $\epsilon>0$, there exists $\delta>0$ so that if $t,s>0$ and $|t-s|<\delta$ then
\begin{equation}\label{eq: prop 5 for AF(R)}
    \m W_1^{1,1}(T(f,S(s)), T(f,S(t))< \epsilon.
\end{equation}
Thus we take another limit in the $d_W$ topology to define 
\begin{align*}
    T(f,S):= \lim_{t\to 0} T(f,S(t)). 
\end{align*}
Since $AF(R)$ is compact and the extension above is continuous, we get analogs of Proposition \ref{prop: uniform continuity of boundary operator_smooth_1} and Proposition \ref{prop: uniform continuity of boundary operator_smooth_2} for $AF(R)$.

\begin{prop}\label{prop: uniform cont for patches on AF(R)}
    If $S\in \m S(R)$ is a surface satisfying the conditions of Proposition \ref{prop: uniform continuity of boundary operator_smooth_1} for a vector $v$ and threshold $\theta>0$, then 
    \begin{enumerate}
        \item Using the extension described above, 
        \begin{align*}
        T(\cdot,S):(AF(R),d_W)\to (\mc M^s(R), \m W_1^{1,1})
    \end{align*}
    is uniformly continuous. 
        \item Given $\epsilon >0$ there exists $\delta>0$ such that for all $t<\delta$ and $f\in AF(R)$, \begin{align*}
        \m W_{1}^{1,1}(T(f,S),T(f,S(t)) < \epsilon.
    \end{align*}
     \item  For any $f\in AF(R)$, $T(f,S)$ is a signed measure absolutely continuous with respect to the surface area measure $\sigma_S$ with density bounded between $-1$ and $1$.
    \end{enumerate}
\end{prop}
We now prove a compatibility result for traces on overlapping surfaces. Using this, we can extend the continuity theorems to the trace operator for any $S\in \m S(R)$ by cutting $S$ in patches which satisfy the conditions of Proposition \ref{prop: uniform continuity of boundary operator_smooth_1}.

\begin{lemma}\label{lemma: compatibility of Wasserstein_1}
	Suppose that $S,S'\in \m S(R)$ are surfaces satisfying the conditions of Proposition \ref{prop: uniform continuity of boundary operator_smooth_1} such that $S'\subset S$. Then for any $f\in AF(R)$, then 
	$$T(f,S)\mid_{S'} = T(f, S').$$ 
\end{lemma} 
\begin{proof}
If $f\in AF^\infty(R)$, then the result follows immediately from the form given in Definition \ref{def: trace on smooth flows}. For general $f\in AF(R)$ this follows from Proposition \ref{prop: uniform cont for patches on AF(R)}.
\end{proof}

Finally we put together the pieces to prove uniform continuity of $T(\cdot,S)$ on $AF(R)$ for any $S\in \m S(R)$. \symindex{Chapter 5!$T: AF(R)\times \m S(R)\to \mc M^s(R)$ - the trace operator on asymptotic flows}\termindex{Chapter 5!trace operator for asymptotic flows}

\begin{prop}\label{prop:uniformly_continuous_AF(R)}
For all $S\in \m S(R)$, $T(\cdot,S)$ extends to a uniformly continuous map from $(AF(R),d_W)$ to $(\mc M^s(R),\m W_1^{1,1})$. Further, for all $f\in AF(R)$, $T(f, S)$ is a signed measure on $S$ absolutely continuous to the surface measure $\sigma_S$ with density bounded between $-1$ and $1$.
\end{prop}
\begin{rem}
    In particular this holds for $S=\partial R$.
\end{rem}
\begin{proof}
    Fix $S\in \m S(R)$. We can cover $S$ with finitely many open surfaces $S_1,...,S_k$ which all satisfy the conditions of Proposition \ref{prop: uniform continuity of boundary operator_smooth_1}. By Lemma \ref{lemma: compatibility of Wasserstein_1}, if $S_i\cap S_j\neq \emptyset$, then for all $f\in AF(R)$, 
    \begin{align*}
        T(f,S_i)\mid_{S_i\cap S_j} = T(f,S_j)\mid_{S_i\cap S_j}
    \end{align*}
    Therefore we can define the trace operator $T(S,\cdot)$ for any $S\in \m S(R)$ by 
    \begin{align*}
        T(f,S) \mid_{S_i} = T(f,S_i). 
    \end{align*}
    By Proposition \ref{prop: uniform cont for patches on AF(R)}, $T(f,S)$ is uniformly continuous as a function of $f\in AF(R)$ and has the desired form. 
\end{proof}

With this machinery, we can define the space of asymptotic flows with a fixed boundary value.

\begin{definition}\label{def:boundary_asymptotic_flow}
    We say that $b\in \mc M^s(R)$ is a \textit{boundary asmyptotic flow} if $b \in T(AF(R),\partial R)$. Further, we define $AF(R,b)$ to be the space of asymptotic flows on $R$ with boundary value $b$, i.e.\ $f\in AF(R)$ such that $T(f,\partial R) = b$.
\end{definition}

As a corollary of Proposition \ref{prop:uniformly_continuous_AF(R)} and Theorem \ref{thm:AF_compact} we get the following. 

\begin{cor}\label{cor: AF(R,b) compact}
    The metric space $(AF(R,b),d_W)$ is compact.
\end{cor}

\subsection{Boundary values of tiling flows}\label{sec: boundary values of tiling flows}

Next we define the trace operator on tiling flows, and show that it is compatible with the definition for asymptotic flows. Suppose $f\in TF_n(R)$ and that $S\in \m S(R)$ is a surface which intersects the lattice $\frac{1}{n}\m Z^3$ transversely, i.e.\ $S$ does not contain any vertices of $\frac{1}{n} \m Z^3$ (if $S$ contains a vertex, we translate the lattice slightly so that it does not). As usual let $e$ denote an edge from $\frac{1}{n}\m Z^3$ oriented from even to odd, and let $\xi(x)$ denote the normal vector to $S$ at $x$. \symindex{Chapter 5!$T: TF(R)\times \m S(R)\to \mc M^s(R)$ - trace operator for tiling flows}\termindex{Chapter 5!trace operator for tiling flows}
\begin{definition}\label{def: trace for tiling flows}
    If $f\in TF_n(R)$ and $S\in \m S(R)$ is a surface intersecting $\frac{1}{n}\m Z^3$ transversely, we define 
    \begin{align*}
            T(f,S) = 
            \sum_{e} \frac{2\,\text{sign}\langle \xi(x),e\rangle}{n^2} f(e) \delta(e\cap S).
    \end{align*}
    Note that since $S$ is transverse to $\frac{1}{n} \m Z^3$, if $e\cap S$ is nonempty it is a single point. 
\end{definition}

Using Definition \ref{def: trace for tiling flows} for tiling flows and Definition \ref{def: trace on smooth flows} for asymptotic flows, the final goal of this section is to show that 
\begin{align*}
    T(\cdot, S): (AF(R)\cup TF(R),d_W)\to (\mc M^s(R), \m W_1^{1,1})
\end{align*}
is uniformly continuous for any $S\in \m S(R)$ (see Theorem \ref{thm: boundary_value_uniformly_continuous}). The sequence of results in this section building up to this mirrors the sequence of results in the previous section. The discrete setting makes things slightly more complicated. The main new step is that we start by proving a result for the trace on planes, and extend to more  general surfaces by approximating them with planes. Throughout, we assume that any surface $S$ we consider intersects $\frac{1}{n}\m Z^3$ transversely. Any time it does not, the trace is defined by perturbing the lattice slightly so that it does and then using Definition \ref{def: trace for tiling flows}.

\begin{rem}
    With this definition of the boundary value of a flow, we place a Dirac mass at the intersection of an edge $e$ with the surface $\partial R$. Another definition, which would differ by $O(1/n)$ and therefore does not affect considerations in the scaling limit, would be to place a Dirac mass for an edge $e\in \tau$ on its endpoint vertex in $\partial R_n$. In this alternate definition, all tilings of the same region $R_n$ have exactly the same boundary value. In our definition, they might differ by $O(1/n)$ if different edges incident to a vertex $v\in \partial R_n$ are used.
\end{rem}

\begin{prop}\label{prop: tiling flow boundary on planes}
Suppose $P\in \m S(R)$ is a compact piece of a plane with normal vector $\xi$, and there exists a threshold $\theta>0$ such that $P(t) = P + t\xi$ is contained in $\text{Int}(R)$ for all $t\in [0,\theta]$. Let $f_n\in TF_n(R)$ be a sequence of tiling flows such that $d_W(f_n,f)\to 0$ as $n\to \infty$ for some $f\in AF^\infty(R)$. Then
\begin{align*}
   \lim_{n\to \infty} \m W_1^{1,1}(T(f_n,P),T(f,P)) = 0.
\end{align*}
\end{prop}
\begin{rem}
The conditions here could be rephrased as saying that $P\in \m S(R)$ is contained in a plane and satisfies the conditions of Proposition \ref{prop: uniform continuity of boundary operator_smooth_1} with $v=\xi$. 
\end{rem}
\begin{proof}
As in Proposition \ref{prop: uniform continuity of boundary operator_smooth_1}, fix two parameters $\gamma_1,\gamma_2>0$ and partition $P$ into patches $\alpha_1,...,\alpha_M$ such that 
\begin{itemize}
    \item $\alpha_i$ is a smooth surface with piecewise smooth boundary for all $i=1,\ldots,M$;
    \item $\alpha_i$ has diameter at most $\gamma_1$ for all $i=1,\ldots,M $, and $M\leq C \gamma_1^{-2}$ for some constant $C$ depending on $P$;
    \item Let $\alpha_i(t):= \alpha_i + t \xi$. For all $i=1,\ldots,M$, $\alpha_i(t)\cap \alpha_i(s) = \emptyset$ for $s\neq t$. 
\end{itemize}
We define 
\begin{align*}
    \Delta_n = \sup_{1\leq i \leq M} \bigg\lvert T(f_n,P)(\alpha_i) - T(f,P)(\alpha_i)\bigg\rvert.
\end{align*}
By the two-dimensional version of Lemma \ref{lem:covered},
\begin{align*}
    \m W_1^{1,1}(T(f_n,P),T(f,P)) \leq 10 C \gamma_1 + C\gamma_1^{-2} \Delta_n.
\end{align*}
Let $U_i(s) = \cup_{t=0}^{s} \alpha_i(t)$ be the parallelopiped region between $\alpha_i=\alpha_i(0)$ and $\alpha_i(s)$. Given $\gamma_2$, we can find $\theta_{\gamma_2}$ small enough so that the number of edges from $\frac{1}{n} \m Z^3$ hitting $\partial U_i(t) \setminus (\alpha_i(t)\cup \alpha_i)$ for any $i$ is less than $\gamma_2 n^2 + K' n$, with constant $K'$ depending on the length of $\partial \alpha_i$. 
Since the magnitude of $f_n$ is of order $1/n^2$, for $t<\theta_{\gamma_2}$ the flow of $f_n$ through $\partial U_i(t) \setminus (\alpha_i(t)\cup \alpha_i)$ is bounded by $O(n^{-1}) + \gamma_2$. Since any $f_n\in TF_n(R)$ is divergence-free as a discrete vector field, there exists a constant $K>0$ so that for any $f_n$ and $t\in (0,\theta_{\gamma_2})$,
\begin{equation}\label{eq: tiling flow div free bound}
    \sup_{1\leq i\leq M} \bigg| T(f_n,\alpha_i)(\alpha_i) - T(f_n,\alpha_i(t))(\alpha_i(t)) \bigg| \leq \gamma_2 + K/n.
\end{equation}
Possibly choosing a smaller $\theta_{\gamma_2}$, the same result holds for $f$ without the $K/n$ in the upper bound. By the triangle inequality,
\begin{align*}
    \sup_{1\leq i\leq M} \bigg| T(f_n,\alpha_i)(\alpha_i) - T(f,\alpha_i)(\alpha_i) - T(f_n,\alpha_i(t))(\alpha_i(t)) + T(f,\alpha_i(t))(\alpha_i(t))\bigg|\leq 2\gamma_2 + K/n.
\end{align*}
As in Proposition \ref{prop: uniform continuity of boundary operator_smooth_1}, integrating over $t\in(0,\theta_{\gamma_2})$ and solving for $\Delta_n$ gives 
\begin{align*}
    \Delta_n \leq 2 \gamma_2 + K/n + \frac{1}{\theta_{\gamma_2}} \sup_{1\leq i \leq M}\bigg|\int_{0}^{\theta_{\gamma_2}} (T(f_n,\alpha_i(t))(\alpha_i(t))- T(f,\alpha_i(t))(\alpha_i(t)))\, \dd t \bigg| .
\end{align*}
Let $\alpha_i$ be the patch where the supermum is achieved, and let $\alpha:=\alpha_i$ to simplify notation. Then \begin{align*}
    \m W_1^{1,1}(T(f_n,P),T(f,P)) \leq 10 C \gamma_1 + 2 C \gamma_1^{-2}\gamma_2 +  C \gamma_1^{-2} K n^{-1} \\ +  \frac{C \gamma_1^{-2}}{\theta_{\gamma_2}} \bigg|\int_{0}^{\theta_{\gamma_2}} T(f_n,\alpha(t))(\alpha(t))- T(f,\alpha(t))(\alpha(t))\, \dd t \bigg|.
\end{align*}
Finally we bound the integral in the last term. Let $U = U_i(\theta_{\gamma_2})$ to simplify notation. Recall that $\alpha(t) \subset P + t v$ is a piece of a plane, and has constant unit normal vector $\xi$. By Definition \ref{def: trace on smooth flows}, $T(f,\alpha(t))(x) = \langle f(x), \xi\rangle$. Therefore letting $\sigma_{\alpha(t)}$ denote the surface area measure on $\alpha(t)$, and applying change of variables, 
\begin{align*}
    \int_0^\theta T(f,\alpha(t))(\alpha(t)) \, \dd t &= \int_0^\theta \int_{\alpha(t)} \langle f(x),\xi\rangle \, \dd \sigma_{\alpha(t))}(x) \dd t = \int_U \langle f(x), \xi\rangle \dd x_1\, \dd x_2\, \dd x_3 \\
    &= \sum_{j=1}^3 \xi_j \mu_{j}(U),
\end{align*}
where $\xi = (\xi_1,\xi_2,\xi_3)$ and $(\mu_1,\mu_2,\mu_3)$ is the triple of measures corresponding to $f$. On the other hand, for the tiling flow $f_n$,
\begin{align*}
    \int_0^\theta T(f_n,\alpha(t))(\alpha(t)) \, \dd t = \sum_{j=1}^3 \xi_j \, \mu_j^n(U),
\end{align*}
where $(\mu_1^n,\mu_2^n,\mu_3^n)$ is the triple of measures corresponding to $f_n$. Therefore 
\begin{align*}
    &\m W_1^{1,1}(T(f_n,P),T(f,P)) \\
    \leq &10 C \gamma_1 + 2 C \gamma_1^{-2}\gamma_2 +  C \gamma_1^{-2} K n^{-1} + C \gamma_1^{-2} \theta_{\gamma_2}^{-1} \sum_{j=1}^3 |\xi_j| |\mu_j^n(U)-\mu_j(U)|.
\end{align*}
By Lemma \ref{lem:constant_order_box_Wass_bound}, $d_W(f_n,f)\to 0$ implies that $|\mu_j^n(U)-\mu_j(U)|\to 0$ as $n\to \infty$ for $j=1,2,3$. Taking $n\to \infty$ gives 
\begin{align*}
    \limsup_{n\to \infty} \m W_1^{1,1}(T(f_n,P),T(f,P)) \leq 10 C \gamma_1 + 2 C \gamma_1^{-2}\gamma_2.
\end{align*}
Setting $\gamma_2=\gamma_1^{3}$ and taking $\gamma_1\to 0$ completes the proof.
\end{proof}

Next we prove a version of Proposition \ref{prop: uniform continuity of boundary operator_smooth_2} for tiling flows and small patch surfaces as in Proposition \ref{prop: uniform continuity of boundary operator_smooth_1}.
\begin{prop}\label{prop: surface perturbation for tiling flows}
    Suppose that $S\in \m S(R)$ satisfies the conditions of Proposition \ref{prop: uniform continuity of boundary operator_smooth_1}. For all $\epsilon>0$ there exists $\delta>0$ and $N>0$ such that for all $t<\delta$, all $n\geq N$, and all $f\in TF_n(R)$, 
    \begin{align*}
        \m W_1^{1,1}(T(f,S), T(f,S(t)) ) < \epsilon.
    \end{align*}
\end{prop}
\begin{proof}
    The proof is analogous to the proof of Proposition \ref{prop: uniform continuity of boundary operator_smooth_2}. Again we take parameters $\gamma_1,\gamma_2>0$ to be fixed later and a partition $\alpha_1,....,\alpha_M$ of $S$ into patches of diameter at most $\gamma_1$ for all $i$, and such that $M\leq C \gamma_1^{-2}$ for some constant $C$ independent of $\gamma_1$. Analogous to Equation \eqref{eq: tiling flow div free bound}, given $\gamma_2$ we can find a threshold $\theta_{\gamma_2}>0$ such that for all $0\leq t\leq \theta_{\gamma_2}$ and all $f\in TF_n(R)$,
    \begin{align}\label{eq:first t bound}
        \sup_{1\leq i\leq M}\bigg\lvert T(f,S)(\alpha_i) - T(f, S(t))(\alpha_i(t))\bigg\rvert \leq \gamma_2 + K/n.
    \end{align}
    Using this, we get an upper bound for the distance by giving a method for redistributing, adding, and deleting mass to transform $T(f,S)\mid_{\alpha_i}$ into $T(f,S(t))\mid_{\alpha_i(t)}$. Both measures are a sum of delta masses of weights $2/n^2(\pm 5/6)$ or $2/n^2(\pm 1/6)$. The number of delta masses in $\alpha_i$ or $\alpha_i(t)$ is bounded above by $\text{area}(\alpha_i) n^2$. Since $\alpha_i$ has diameter bounded by $\gamma_1$, there is a constant $A>0$ independent of $\alpha_i$ such that $\text{area}(\alpha_i)\leq A \gamma_1^{2}$. Hence the total mass in each patch is bounded between $-2A\gamma_1^2$ and $2A\gamma_1^2$. Hence adding $\gamma_2+ K/n$ mass plus moving at most $8 A \gamma_1^2$ mass distance at most $t+\gamma_1$, we get the bound
    \begin{align*}
        \m W_1^{1,1}(T(f,S)\mid_{\alpha_i}, T(f,S(t))\mid_{\alpha_i(t)}) \leq \gamma_2 + K/n + 8 A \gamma_1^2 (t+\gamma_1)
    \end{align*}
    Summing over $i$ we get that 
    \begin{align*}
        \m W_1^{1,1}(T(f,S), T(f,S(t))) \
        &\leq M(\gamma_2 + K/n + 8A\gamma_1^2 (t+\gamma_1)) \\
        &= C \gamma_1^{-2} \gamma_2 + CK \gamma_1^{-2} n^{-1} + 8AC (t+ \gamma_1).
    \end{align*}
    Take $\gamma_2=\gamma_1^{3}$ and $\gamma_1,t$ small enough so that \begin{align}\label{eq:second t bound}
        (C+8AC)\gamma_1 + 8ACt < \epsilon /2.
    \end{align}
    Then take $n$ large enough so that 
    \begin{align*}
        CK \gamma_1^{-2} n^{-1}< \epsilon/2,
    \end{align*}
    and the result follows with $\delta = \min\{\theta_{\gamma_1^3}, \frac{1}{8AC}(\epsilon/2-(C+8AC)\gamma_1)\}$. Here the first term in the minimum comes from \eqref{eq:first t bound} and the second comes from \eqref{eq:second t bound}.
\end{proof}

By approximation we can extend Proposition \ref{prop: tiling flow boundary on planes} to any surface $S\in S(\m R)$. For technical reasons, we first show this for $S$ contained strictly in the interior of $R$. Note that we also remove the condition that the limiting flow is smooth.
\begin{prop}\label{prop: trace for interior surfaces}
Suppose that $S\in \m S(R)$ is contained strictly in the interior of $R$. Let $f_n\in TF_n(R)$ be a sequence of tiling flows such that $d_W(f_n,f)\to0$ as $n\to \infty$ for some $f\in AF(R)$. Then 
\begin{align*}
    \lim_{n\to\infty} \m W_1^{1,1}(T(f_n,S),T(f,S)) = 0.
\end{align*}
\end{prop}
\begin{proof}
    As usual, let $\gamma_1,\gamma_2>0$ be small parameters to be fixed later. Since $S$ is contained strictly in the interior of $R$, we can cover $S$ by very small patch surfaces $\alpha_1,...,\alpha_M$ so that:
    \begin{itemize}
        \item Each $\alpha_i$ is smooth with piecewise smooth boundary,
        \item The diameter of $\alpha_i$ is at most $\gamma_1$.
        \item There is a constant $C>0$ such that $M\leq C \gamma_1^{-2}$.
        \item For all $i$, $\alpha_i$ satisfies the conditions of Proposition \ref{prop: uniform continuity of boundary operator_smooth_1} for some threshold $\theta>0$ with vector $v= \xi(q_i)$, where $\xi(q_i)$ is the normal vector to the surface at $q_i$ for some $q_i\in \alpha_i$ with the property that the distance between $q_i$ and any other $x\in \alpha_i$ is at most $\gamma_1$.
        \item Let $P_i$ denote the tangent plane to $\alpha_i$ at $q_i$, and let $\pi_i\subset P_i$ be the patch of the plane corresponding to projecting $\alpha_i$ onto $P_i$. Potentially making $\theta>0$ or $\gamma_1$ smaller, we can assume that $\pi_i$ also satisfies the conditions of Proposition \ref{prop: uniform continuity of boundary operator_smooth_1} for $v=\xi(q_i)$. This is where we are using the fact that $S$ is contained in the interior of $R$.
    \end{itemize}
    As usual we denote $\alpha_i(t) = \alpha_i + t \xi(q_i)$ and $\pi_i(t) = \pi_i + t \xi(q_i)$.
    Note by Definition \ref{def: trace for tiling flows} and Lemma \ref{lemma: compatibility of Wasserstein_1} that 
    \begin{align*}
        T(f_n,S)\mid_{\alpha_i} = T(f_n,\alpha_i) \qquad \text{and} \qquad T(f,S)\mid_{\alpha_i} = T(f, \alpha_i) \qquad i=1,...,M.
    \end{align*}
    Define 
    \begin{align*}
        \Delta_n := \sup_{1\leq i \leq M} \bigg\lvert T(f_n,S)(\alpha_i) - T(f,S)(\alpha_i)\bigg\rvert.
    \end{align*}
    By the two-dimensional version of Lemma \ref{lem:covered},
    \begin{equation}\label{eq: tiling vs asymp from lemma 1}
        \m W_1^{1,1}(T(f_n,S),T(f,S)) \leq M(10\gamma_1^3+\Delta_n) \leq 10C \gamma_1 + C\gamma_1^{-2} \Delta_n.
    \end{equation}
    Since $f_n\in TF_n(R)$ is discrete divergence-free and $S$ is compact and piecewise smooth, given $\gamma_2$ there exists $K>0$ depending on $S$ and a threshold $\theta_{\gamma_2}$ such that for all $0\leq t\leq \theta_{\gamma_2}$,
    \begin{equation}\label{eq: difference bound on surface tiling flow}
        \sup_{1\leq i\leq M} \bigg| T(f_n,\alpha_i)(\alpha_i) - T(f_n, \alpha_i(t))(\alpha_i(t))\bigg| \leq \gamma_2 + K/n.
    \end{equation}
    By Proposition \ref{prop: uniform cont for patches on AF(R)}, up to making $\theta_{\gamma_2}$ smaller, for all $0\leq t\leq \theta_{\gamma_2}$,
    \begin{equation}\label{eq: difference bound on surface asymp flow}
        \sup_{1\leq i\leq M} \bigg| T(f,\alpha_i)(\alpha_i) - T(f, \alpha_i(t))(\alpha_i(t))\bigg| \leq \gamma_2.
    \end{equation}
    Combining Equations \eqref{eq: difference bound on surface tiling flow} and \eqref{eq: difference bound on surface asymp flow}, and as in Proposition \ref{prop: uniform continuity of boundary operator_smooth_1} integrating over $t\in (0,\theta_{\gamma_2})$ then solving for $\Delta_n$ gives 
    \begin{align*}
        \Delta_n \leq 2 \gamma_2 + K/n + \frac{1}{\theta_{\gamma_2}} \sup_{1\leq i\leq M} \bigg|\int_{0}^{\theta_{\gamma_2}} T(f_n,\alpha_i(t))(\alpha_i(t))- T(f,\alpha_i(t))(\alpha_i(t))\, \dd t \bigg|.
    \end{align*}
    Let $i$ be the index where the supremum is achieved, and let $\alpha(t):=\alpha_i(t)$ and $\pi(t):=\pi_i(t)$. We now bound 
    \begin{align*}
        T(f_n,\alpha(t))(\alpha(t))- T(f,\alpha(t))(\alpha(t))
    \end{align*}
    using four terms. Let $(g_m)\in AF^\infty(R_{\epsilon_m})$, $\epsilon_m\to0$ as $m\to \infty$, be a sequence of smooth asymptotic flows such that $\lim_{m\to \infty} d_W(g_m,f) = 0$. Since $\alpha(t),\pi(t)$ are contained strictly in the interior of $R$ for $0\leq t\leq \theta$, we can assume they are all contained in $R_\epsilon$ for some $\epsilon>0$. In particular, for $m$ large enough $\epsilon_m<\epsilon$ and hence $g_m$ is defined on $\pi(t)$. We have for any $0\leq t\leq \theta_{\gamma_2}$, 
    \begin{align*}
         T(f_n,\alpha(t))(\alpha(t))- T(f,\alpha(t))(\alpha(t)) = & T(f_n,\alpha(t))(\alpha(t))- T(f_n,\pi(t))(\pi(t))\\ + & T(f_n,\pi(t))(\pi(t))- T(g_m,\pi(t))(\pi(t)) \\+ & T(g_m,\pi(t))(\pi(t))- T(f,\pi(t))(\pi(t)) \\+&
         T(f,\pi(t))(\pi(t)) - T(f,\alpha(t))(\alpha(t)).
    \end{align*}
    Consider the region $V(t)$ with boundary $\alpha(t)\cup \pi(t)$ plus sides to enclose it. Since $\pi(t)$ is the tangent plane to $\alpha(t)$ at $q(t) = q_i+tv$, the height of the sides needed to enclose this region is bounded by $C_2\gamma_1^2$, where since $S$ is compact, $C_2>0$ is a constant depending only on $S$ ($C_2$ is basically the maximum curvature at a smooth point on $S$). On the other hand, the length of the boundary of $\alpha(t)$ is bounded by $C_3\gamma_1$ for some constant $C_3>0$. Therefore since $f_n$ is divergence-free, for some constant $K>0$,
    \begin{equation}\label{eq: approximation by plane}
        |T(f_n,\alpha(t))(\alpha(t))- T(f_n,\pi(t))(\pi(t))|\leq C_2 C_3 \gamma_1^3 +K/n.
    \end{equation}
    The analogous result holds for $f$ (without the $K/n$ term), controlling the fourth term above. By Proposition \ref{prop: uniform cont for patches on AF(R)}, 
    \begin{align*}
        \lim_{m\to \infty} |T(g_m,\pi(t))(\pi(t))- T(f,\pi(t))(\pi(t))| = 0.
    \end{align*}
    As in Proposition \ref{prop: tiling flow boundary on planes}, if $\mu_j^n$ denotes the component measures of $f_n$, $\nu_j^m$ denotes the component measures of $g_m$, and $\mu_j$ denotes the component measures of $f$, and letting $U= \cup_{t=0}^{\theta_{\gamma_2}} \pi(t)$, we have 
$$        \bigg|\int_{0}^{\theta_{\gamma_2}} T(f_n,\pi(t))(\pi(t))- T(g_m,\pi(t))(\pi(t)) \,\dd t\bigg| $$ $$\leq \sum_{j=1}^3 |\xi_j(q)| |\mu_j^n(U)-\nu_j^m(U)| 
        \leq \sum_{j=1}^3 |\xi_j(q)| (|\mu_j^n(U)-\mu_j(U)|+ |\mu_j(U)-\nu_j^m(U)|).
 $$
    Taking the limit as $m\to 0$ gives 
    \begin{align*}
       \limsup_{m\to \infty} \bigg|\int_{0}^{\theta_{\gamma_2}} T(f_n,\pi(t))(\pi(t))- T(g_m,\pi(t))(\pi(t)) \,\dd t \bigg|\leq \sum_{j=1}^3 |\xi_j(q)| |\mu_j^n(U)-\mu_j(U)|.
    \end{align*}
    Therefore 
    \begin{align*}
    \bigg|\int_{0}^{\theta_{\gamma_2}} T(f_n,\alpha(t))(\alpha(t))- T(f,\alpha(t))(\alpha(t))\, \dd t \bigg| \\ \leq 2C_2 C_3 \theta_{\gamma_2} \gamma_1^3 + \sum_{j=1}^3 |\xi_j(q)| |\mu_j^n(U)-\mu_j(U)| + K\theta_{\gamma_2}/n.
    \end{align*}
    Plugging back in to Equation \eqref{eq: tiling vs asymp from lemma 1}, we get 
    \begin{align*}
        \m W_{1}^{1,1}(T(f_n,S), T(f,S)) \leq &10 C \gamma_1 + 2C \gamma_1^{-2} \gamma_2 + C \gamma_1^{-2} K n^{-1} + 2C C_2 C_3 \gamma_1 + C \gamma_1^{-2} K n^{-1} + \\
        &C \gamma_1^{-2} \theta_{\gamma_2}^{-1} \sum_{j=1}^3 |\xi_j(q)||\mu_j^n(U)-\mu_j(U)|.
    \end{align*}
    Take $\gamma_2 = \gamma_1^3$. Then taking $\gamma_1$ small makes terms 1, 2, and 4 small. Taking $n$ large makes terms 3 and 5 small and by Lemma \ref{lem:constant_order_box_Wass_bound} also makes term 6 small. 
\end{proof}

Next we remove the condition that $S$ is contained in the interior of $R$.

\begin{prop}\label{prop: continuity of tiling flow boundary}
    For any $S\in \m S(R)$ and any sequence of tiling flows $f_n \in TF_n(R)$ such that $d_W(f_n,f)\to 0$ as $n\to \infty$ for some $f\in AF(R)$, 
    \begin{align*}
        \lim_{n\to \infty} \m W_1^{1,1}(T(f_n,S), T(f,S))= 0.
    \end{align*}
\end{prop}
\begin{proof}
    We can cover $S$ with finitely many surfaces $S_1,...,S_M$ which satisfy the conditions of Proposition \ref{prop: uniform continuity of boundary operator_smooth_1} for vectors $v_1,...,v_k$ and a threshold $\theta>0$. We can do this so that $d = \max_{1\leq j\leq M} \text{diam}(S_j)$ and there is a constant $C$ independent of $d$ such that $M = C d^{-2}$. Fix $\epsilon>0$. There exists $\delta>0$, $N>0$ such that for all $j=1,...,M$, and all $0\leq t\leq \delta$, by Proposition \ref{prop: uniform cont for patches on AF(R)}, 
    \begin{align*}
        \m W_{1}^{1,1}(T(f,S_j), T(f,S_j(t))) < \epsilon 
    \end{align*}
    and by Proposition \ref{prop: surface perturbation for tiling flows}, for $n\geq N$ and $0\leq t\leq \delta$,
    \begin{align*}
        \m W_{1}^{1,1}(T(f_n,S_j), T(f_n,S_j(t))) < \epsilon. 
    \end{align*}
    On the other hand, by Proposition \ref{prop: trace for interior surfaces}, for all $j=1,..,M$ and all $t>0$, for $n$ large enough
    \begin{align*}
        \m W_{1}^{1,1}(T(f_n,S_j(t)),T(f,S_j(t)))< \epsilon.
    \end{align*}
    Hence by the triangle inequality, for all $j=1,...,M$
    \begin{align*}
        \m W_{1}^{1,1}(T(f,S_j), T(f_n,S_j)) < 3\epsilon.
    \end{align*}
    By the two-dimensional version of Lemma \ref{lem:covered}, for $n$ large enough, 
    \begin{align*}
        \m W_{1}^{1,1}(T(f,S), T(f_n,S))\leq M (10 d^3 + 3 \epsilon) \leq 10 C d + 3 d^{-2} \epsilon.
    \end{align*}
    Taking $d = \epsilon^{1/3}$ would complete the proof.
\end{proof}

Finally we can prove the main theorem about boundary values that we will refer to later in paper.
\begin{thm}\label{thm: boundary_value_uniformly_continuous}
For any $S\in \m S(R)$, the trace operator 
\begin{align*}
    T(\cdot,S): (AF(R)\cup TF(R),d_W) \to (\mc M^s(R), \m W_{1}^{1,1})
\end{align*}
is uniformly continuous. In particular this holds for $S = \partial R$.
\end{thm}
\begin{proof}
    By Theorem \ref{thm:AF_compact}, $(AF(R),d_W)$ is compact. Since $TF_n(R)$ is finite for each $n$, Theorem \ref{thm:formal_fine_mesh} implies that the $d_W$ limit points of $TF(R)$ are contained in $AF(R)$. Therefore $(AF(R)\cup TF(R),d_W)$ is compact. 
    
    On the other hand, for any $S\in \m S(R)$, Proposition \ref{prop:uniformly_continuous_AF(R)} and Proposition \ref{prop: continuity of tiling flow boundary} combine to show that $T(\cdot,S)$ is a continuous map from $(AF(R)\cup TF(R),d_W)$ to $(\mc M^s(R), \m W_1^{1,1})$. Therefore by compactness $T(\cdot, S)$ is uniformly continuous. 
\end{proof}

\subsection{Properties of Ent}\label{subsection: properties of Ent}

As above, $AF(R)$ denotes the space of asymptotic flows on $R$, and $AF(R,b)$ denotes the asymptotic flows on $R$ with boundary value $b$. Both are equipped with the Wasserstein metric on flows $d_W$ (see Section \ref{section: wasserstein for tiling flows}). Here we use the properties of the mean-current entropy function $\ent$ from Section~\ref{sec:entropy} to prove things about $\Ent$, the entropy functional on asymptotic flows given by 
\begin{align*}
    \Ent(f) = \frac{1}{\text{Vol}(R)} \int_{R} \ent(f(x)) \, \dd x.
\end{align*}
As Corollaries of Lemma \ref{lemma: entropy_concave}, Theorem \ref{theorem: entropy is strictly concave} and Lemma \ref{lemma: entropy_continuous} respectively we get that
\begin{cor}\label{cor: Ent_concave}
    The entropy functional $\Ent$ is concave on $AF(R)$. Further, $\Ent$ is strictly concave when restricted to the space of asymptotic flows which are valued in $\mc O\setminus \mc E$.
\end{cor}
\begin{cor}\label{cor: Ent_ae}
     If $f_n \to f$ almost everywhere in $R$, then $\Ent(f) = \lim_{n\to \infty} \Ent(f_n)$. 
\end{cor}

From this, we show 
\begin{prop}\label{proposition: Upper semicontinuous}
    The functional $\Ent:AF(R) \to [0,\infty)$ is upper semicontinuous in the Wasserstein topology induced by $d_W$. 
\end{prop}
\begin{proof}
	Let $(f_n)_{n\geq 1}$ be a sequence of flows in $AF(R)$ such that $d_W(f_n, f) \to 0$ as $n\to \infty$ for some $f\in AF(R)$. For any $g\in AF(R)$, we can define its approximation
	$g_{\epsilon}$ given by 
	$$g_{\epsilon}(x):=\frac{1}{\text{Vol}{B_{\epsilon}(x)}}\int_{B_{\epsilon}(x)}g(y)\,\dd y.$$
Here we say that $g(y) = 0$ if $y\not\in R$. As $\epsilon\to 0$, $d_W(g_{\epsilon},g)\to 0$. 

While $g_\epsilon$ is not an asymptotic flow because it is not divergence-free, it is still valued in $\mc O$ and thus $\Ent(g_\epsilon) := \frac{1}{\text{Vol}(R)} \int_R \ent(g_\epsilon(x))\, \dd x$ still makes sense. By the Lebesgue differentiation theorem, $g_{\epsilon}$ converges to $g$ almost everywhere as $\epsilon\to 0$. By Corollary \ref{cor: Ent_ae},
$$\lim_{\epsilon\to 0}\text{Ent}(g_{\epsilon})=\text{Ent}(g).$$

By Lemma \ref{lemma: entropy_concave}, for any $x\in R$,
\begin{align*}
    \ent(g_\epsilon(x)) = \ent\bigg(\frac{3}{4\pi \epsilon^3} \int_{B_\epsilon(x)} g(y)\, \dd y\bigg) \geq \frac{3}{4\pi \epsilon^3} \int_{B_\epsilon(x)} \ent(g(y))\,\dd y.
\end{align*}
Therefore there is a constant $C$ (proportional to $\text{Area}(\partial R)/\text{Vol}(R)$ and independent of $\epsilon$) such that
$$\text{Ent}(g_{\epsilon}) + C\epsilon \geq \text{Ent}(g).$$
Since $d_W(f_n,f)\to 0$ as $n\to \infty$, by Corollary \ref{cor: weak convergence wasserstein}, $f_{n,\epsilon}$ converges pointwise to $f_{\epsilon}$. By Corollary \ref{cor: Ent_ae},
$$\limsup_{n\to \infty}\text{Ent}(f_n)\leq \limsup_{n\to \infty}\text{Ent}(f_{n,\epsilon}) + C\epsilon= \text{Ent}(f_{\epsilon}) + C\epsilon.$$
Taking $\epsilon$ to zero, we get that $$\limsup_{n\to \infty}\text{Ent}(f_n) \leq \text{Ent}(f),$$
hence $\Ent$ is upper semicontinuous. 
 \end{proof}

\begin{rem}
It is not difficult to see that $\text{Ent}$ is not continuous. Indeed consider the flows $f_n\in AF([0,1]^3)$  given by
$$f_{n}(x_1, x_2, x_3)=
\begin{cases}\eta_2 &\text{ if }x_1\in(\frac{2k}{2n},\frac{2k+1}{2n} )\text{ for some }0\leq k\leq n-1\\
-\eta_2 &\text{ if }x_1\in(\frac{2k+1}{2n},\frac{2k+2}{2n} )\text{ for some }0\leq k\leq n-1.\end{cases}$$
Then $f_{n}$ converges to the constant zero vector field but $\text{Ent}(f_n)=0$ while $\text{Ent}(0)>0$. 
\end{rem}

Our main goal is to show that there exists a unique $\Ent$ maximizer in $AF(R,b)$ under some mild conditions on the pair $(R,b)$. Standard analytic arguments are enough to show existence and a weak form of uniqueness. Let $\mathfrak e_1,...,\mathfrak e_{12}$ denote the twelve closed edges of $\mathcal O$ which make up $\mc E$. 
\begin{prop}\label{prop: weak uniqueness}
    There exists $f \in AF(R,b)$ such that $\Ent(f) = \sup_{g\in AF(R,b)} \Ent(g)$. Further, given $f_1, f_2\in AF(R,b)$, define 
    $$A = \{x\in R : f_1(x) \neq f_2(x)\},\qquad B =  \bigcup_{i=1}^{12} \{x\in R : f_1 (x),f_2(x) \in \mathfrak e_i\}.$$
    If $f_1, f_2$ are both $\Ent$ maximizers, then $A \subseteq B$. 
\end{prop}
\begin{rem}
The problem is that $\ent$ is only strictly concave on $\mc O\setminus \mc E$, not all of $\mc O$. The same problem arises in two dimensions, and is addressed in \cite{gorin2021lectures} and \cite{Savin}. 
\end{rem}
\begin{proof} Since $(AF(R,b),d_W)$ is compact (Theorem \ref{thm:AF_compact}) and $\Ent$ is upper semicontinuous (Proposition \ref{proposition: Upper semicontinuous}), the existence of the maximizer follows.

    To prove weak uniqueness, recall that $\ent(s) = 0$ if and only if $s\in \mc E$. If $f_1,f_2$ are distinct maximizers then $A$ has positive measure. If $A \cap (R\setminus B)$ has positive measure, then by strict convexity of $\Ent$ on flows valued in $\mc O\setminus \mc E$ (Corollary \ref{cor: Ent_concave}), 
    \begin{align*}
        \Ent\bigg( \frac{f_1 + f_2}{2}\bigg) > \Ent(f_1 ) + \Ent( f_2),
    \end{align*}
    which would contradict the claim that $f_1, f_2$ are maximizers. Therefore $A \subseteq B$.
\end{proof}

We adapt an argument of V.~Gorin in \cite[Proposition 7.10]{gorin2021lectures} to prove uniqueness under the mild condition that the pair $(R,b)$ is \textit{semi-flexible} as defined in Definition \ref{def: flexible} below. We call this \textit{semi-flexible} since it is a weaker condition than \textit{flexible}, which will be defined at the beginning of Section \ref{sec:ldp}. 

\begin{definition}\label{def: frozen}
Fix a boundary asymptotic flow $b$ on $R$. A point $x\in R$ with boundary condition $b$ is \textit{frozen} if for all open sets $U\ni x$ and all entropy {maximizers} $f\in AF(R,b)$, there are points $y\in U$ such that $f(y) \in \mc E$. A point $x\in R$ with boundary condition $b$ is \textit{always frozen} if for all open sets $U\ni x$ and \textit{all} $g\in AF(R,b)$, there are points $y\in U$ such that $f(y) \in \mc E$.\termindex{Chapter 7!frozen points in a domain}
\end{definition}
\begin{definition}\label{def: flexible}
    The pair $(R,b)$ is \textit{semi-flexible} if there are no always frozen points in $\text{Int}(R)$. I.e., $(R,b)$ is semi-flexible if for all $x\in \text{Int}(R)$, there exists an extension $g\in AF(R,b)$ and an open set $U\ni x$ such that $g(U)\subset \mc O\setminus \mc E$. If $(R,b)$ is not semi-flexible, we say $(R,b)$ is \textit{rigid}.
\end{definition}\termindex{Chapter 7!semi-flexible, rigid}
\begin{rem}
    The weak uniqueness statement in Proposition \ref{prop: weak uniqueness} can be rephrased as saying that entropy maximizers are unique on the complement of the frozen points. In particular the task that remains is to show that a region (i.e.\ the set of frozen points) cannot both be frozen and have multiple tilings. 
\end{rem}

\begin{thm}\label{thm: unique maximizer}
If $(R,b)$ is semi-flexible, then there is a unique $\Ent$ maximizer in $AF(R,b)$. 
\end{thm}

\begin{rem}
We do not know of an example of a three-dimensional region $R\subset \m R^3$ with boundary value $b$ such that the entropy maximizer for $(R,b)$ is not unique. However see Problem \ref{prob: region with non unique max}, which includes a two-dimensional, non-planar example where the maximizer is not unique. 
\end{rem}

To prove Theorem \ref{thm: unique maximizer}, we show that an equivalent definition of $(R,b)$ {semi-flexible} is that $b$ has an extension $f_0$ valued in $\mc O\setminus \mc E$ on $\text{Int}(R)$ (Lemma \ref{lem: nonwhere edge extension}). After that, the key step is to show that if a maximizer takes values in $\mc E$, we can perturb it by $f_0$ to get a flow which does not take edge values and has more entropy (Lemma \ref{lem: Gorin}). In particular we have the corollary that even if uniqueness fails for $(R,b)$, it holds if $b$ is replaced by (say) $.999b$.
\begin{cor}
Given any boundary asymptotic flow $b$ on $R$ and any $\delta\in (0,1)$ there is a unique entropy maximizer in $AF(R, \delta b)$.
\end{cor}
\begin{rem}\label{rem:aztec_is_semiflexible}
    It is also not hard to see directly that $(R,\delta b)$ is semi-flexible, and in fact flexible, see Definition \ref{def:fully_flexible} and Remark \ref{rem:aztec_is_flexible}. 
\end{rem}

\begin{lemma}\label{lem: nonwhere edge extension}
    The pair $(R,b)$ is semi-flexible if and only if there exists $f_0\in AF(R,b)$ such that $f_0$ is valued in $\mc O\setminus \mc E$.
\end{lemma}

\begin{proof}
    The reverse implication is clear, since for any $x\in \text{Int}(R)$, taking $U$ small enough so that $U\subset \text{Int}(R)$, $f_0$ is an extension such that $f_0(y)\not\in \mc E$ for all $y\in U$.
    
    If $(R,b)$ is semi-flexible, for all $x\in \text{Int}(R)$ there exists an open set $U_x\ni x$ and $f_x\in AF(R,b)$ such that $f_x(y) \in \mc O\setminus \mc E$ for all $y\in U_x$. If $U_x'\subset U_x$ is a smaller open set, then clearly the same property holds for $U_x'$.
    
    Let $\{V_i\}_{i\in \m N}$ be the collection of open balls centered at rational points in $R$ with rational radii. For any pair $(x,U_x)$ we can find $V_i$ such that $x\in V_i$ and $V_i\subset U_x$. Therefore for each $i\in \m N$, there exists $g_i \in AF(R,b)$ such that $g_i$ is valued in $\mc O\setminus \mc E$ on $V_i$. Hence the flow 
    \begin{align*}
        f_0 := \sum_{i=1}^\infty \frac{1}{2^i} g_i
    \end{align*}
    is valued in $\mc O\setminus \mc E$ everywhere in $\text{Int}(R)$ as desired.
\end{proof}

We follow the same strategy as in \cite[Proposition 7.10]{gorin2021lectures} to prove Theorem \ref{thm: unique maximizer} using the nowhere-edge-valued extension $f_0$. The key step is:

\begin{lemma}\label{lem: Gorin}
    Suppose that $(R,b)$ is semi-flexible, and let $\mc V\subset \mc E$ denote the vertices of $\partial \mc O$. If $f\in AF(R,b)$ maximizes $\Ent$, then up to a set of measure zero $f$ does not take values in $\mc E \setminus \mc V$.
\end{lemma}

\begin{proof}
    Suppose for contradiction that $f$ is an $\Ent$ maximizer in $AF(R,b)$ which takes values in $\mc E\setminus \mc V$ on a set $A$ of positive measure, and that $f_0$ is an extension of the form guaranteed by Lemma \ref{lem: nonwhere edge extension}. We will contradict the claim that $f$ is a maximizer by showing that perturbing $f$ by $f_0$ increases $\Ent$. 

    By Theorem \ref{thm: extremal_entropy}, if $s=(s_1,s_2,s_3)$ is contained in a face of $\partial \mc O$ then $\ent(s)$ is equal to the entropy function for two dimensional lozenge tilings, namely
    \begin{align*}
        \ent(s) = \ent_{\text{loz}}(|s_1|,|s_2|,|s_3|) = \frac{1}{\pi} \bigg( L( \pi |s_1|) + L(\pi|s_2|) + L(\pi |s_3|)\bigg),
    \end{align*}
    where $L(\theta) = \int_0^{\theta} \log (2 \sin t)\, \dd t$ (\cite{cohn2001variational}, see Section \ref{sec:extreme}). As in the proof in two dimensions \cite[Proposition 7.10]{gorin2021lectures}, note from this formula that if $s\in \mc E \setminus \mc V$ and $t$ is contained in a face of $\partial \mc O$ adjacent to the edge containing $s$, then for $\epsilon>0$ small enough
    \begin{align*}
        \ent( \epsilon t + (1-\epsilon) s ) > c \epsilon \log (1/\epsilon)
    \end{align*}
    for some constant $c>0$ depending on $s,t$. More generally, $t\in \text{Int}(\mc O)$ can be written as a weighted average of the six brickwork patterns. Simplifying a bit, this means that $t$ can be written as a weighted average of $t_1$, $t_2$ in the faces adjacent to the edge containing $s$ (this takes into account four brickwork patterns), and $t_3$ in the edge diagonally opposite the edge containing $s$ (this takes into account the remaining two).
    \begin{align*}
        t = \alpha t_1 + \beta t_2 + \gamma t_3, \qquad \alpha + \beta + \gamma = 1. 
    \end{align*}
    By strict concavity of $\ent$ on $\mc O \setminus \mc E$ (Theorem \ref{theorem: entropy is strictly concave}), 
    \begin{equation}
        \ent(\epsilon t + (1-\epsilon) s ) > \alpha\, \ent(\epsilon t_1 + (1-\epsilon) s ) + \beta\, \ent(\epsilon t_2 + (1-\epsilon) s ) + \gamma\, \ent(\epsilon t_3 + (1-\epsilon) s ).
    \end{equation}
    We can use the two-dimensional result directly to bound the first two terms from below. For the third term, we note that $\epsilon t_3 + (1-\epsilon) s\in \text{Int}(\mc O)$, and for $\epsilon$ small enough this whole quantity can be written as an average of mean currents on the faces adjacent to the edge containing $s$. Using strict concavity of $\ent$ on $\mc O\setminus \mc E$ we can again apply the lower bound from the two-dimensional result. In summary, for $\epsilon >0$ small enough, there is a constant $c>0$ depending on $s,t$ so that
    \begin{equation}\label{eq:epsilon log 1/epsilon}
        \ent(\epsilon t + (1-\epsilon) s ) > c \epsilon \log(1/\epsilon).
    \end{equation}
    We now consider the perturbation 
    \begin{align*}
        (1-\epsilon) f + \epsilon f_0 \in AF(R, b).
    \end{align*}
    Let $M = \sup_{s\in \mc O} \ent(s)$ (this is finite because $\ent$ is continuous). For all $x\in R \setminus A$, since $\ent$ is concave on all of $\mc O$ (Lemma \ref{lemma: entropy_concave}) and non-negative we have 
    \begin{align*}
        \ent( (1-\epsilon)\, f(x) + \epsilon \, f_0(x) ) -\ent(f(x)) \geq \epsilon \, ( \ent(f_0(x)) - \ent(f(x)) \geq -M \epsilon. 
    \end{align*}
    Therefore 
    \begin{align*}
        \int_{R\setminus A} \ent((1-\epsilon) f(x) + \epsilon f_0(x))\, \dd x - \int_{R\setminus A} \ent(f(x))\, \dd x \geq - M \epsilon \text{Vol}(R\setminus A).
    \end{align*}
    On the other hand by Equation \eqref{eq:epsilon log 1/epsilon}, for $\epsilon$ small enough there exists $A'\subset A$ of positive measure and a fixed constant $c>0$ such that for all $x\in A'$, 
    \begin{align*}
        \ent( (1-\epsilon) f(x) + \epsilon f_0(x)) - \ent(f(x))= \ent( (1-\epsilon) f(x) + \epsilon f_0(x)) > c \epsilon \log (1/\epsilon). 
    \end{align*}
    Therefore 
    \begin{align*}
        \Ent( (1-\epsilon) f + \epsilon f_0) -\Ent(f) \geq \frac{-M \epsilon \text{Vol}(R\setminus A) + \text{Vol}(A') c \epsilon \log (1/\epsilon)}{\text{Vol}(R)}.
    \end{align*}
    For $\epsilon>0$ small enough this implies $\Ent( (1-\epsilon) f + \epsilon f_0) >\Ent(f)$ and contradicts the claim that $f$ is an entropy maximizer. 
\end{proof}

\begin{proof}[Proof of Theorem \ref{thm: unique maximizer}]
    Suppose that $f_1, f_2$ are maximizers of $\Ent$ in $AF(R,b)$. By Lemma \ref{lem: Gorin}, they cannot take values in $\mc E\setminus \mc V$. By Proposition \ref{prop: weak uniqueness}, they can only differ on frozen points, so
    \begin{align*}
        \{x\in R : f_1(x) \neq f_2(x) \} \subseteq \{x \in R : f_1(x), f_2(x) \in \mc V\}.
    \end{align*}
    On the other hand $\frac{1}{2} (f_1 + f_2)$ is also a maximizer. If there is a point where $f_1, f_2$ take different values in $\mc V$, then $\frac{1}{2}(f_1 + f_2)$ would take an edge value contradicting Lemma \ref{lem: Gorin}. Therefore $f_1 = f_2$. 
\end{proof}

\section{Large deviation principles}\label{sec:ldp}

Here we put together the results of the previous sections to prove the main results of this paper, namely two versions of a \textit{large deviation principle (LDP)} for fine-mesh limits of random dimer tilings of regions $R\subset \m R^3$ with some fixed limiting boundary value $b$, in the topology induced by the Wasserstein metric on flows $d_W$ introduced in Section \ref{sec:asympflows}. Section \ref{sec: main results intro} also includes a discussion of our results and a brief description of what a large deviation principle is in general. For more background information, see e.g.\ \cite{dembo2009large} or \cite{varadhan2016large}. Here we give a slightly more detailed informal description of the main theorems and an outline of the section before getting to formal theorem statements in Sections \ref{sec:sb-statements} and \ref{sec:statements-hb}. We use results here from throughout the paper, but a lot of the notation in this section was originally introduced in Section \ref{sec:asympflows}. 

For the large deviation principles, we only work with the boundary flows $b$ which are (i) \textit{boundary asymptotic flows} meaning that $b$ has an extension $g$ to $R$ which is an asymptotic flow (Definition \ref{def:boundary_asymptotic_flow}) and (ii) \textit{extendable outside} meaning there exists $\epsilon>0$ such that $b$ extends to a divergence-free measurable vector field valued in $\mc O$ on an $\epsilon$ neighborhood of $R$ (Definition \ref{def:extendable}). Analogous extendability conditions are also required in the large deviation principle for dimer tilings in 2D \cite{cohn2001variational}; see Remark \ref{rem:not_extendable_problem}.

In both versions of the LDP we prove, we look at measures supported on dimer tilings of finite regions in $\frac{1}{n} \m Z^3$ that cover $R$ (we call these \textit{free-boundary tilings of $R$ at scale $n$}, see Definition \ref{def:free_boundary_tilings}). We can require that the boundary values of these flows converge as $n\to \infty$ to the fixed boundary value $b$ with either a \textit{soft constraint} or a \textit{hard constraint} on the tilings. 

The large deviation principle for dimer tilings in two dimensions \cite{cohn2001variational} uses a hard constraint. In three dimensions, new subtleties arise from the fact that $\ent$ can be nonzero on $\partial \mc O$, and the analogous hard boundary large deviation principle is not true in full generality (see the discussion in Section \ref{sec: main results intro} or Example \ref{ex:hb_false}). Instead we prove two versions of an LDP, one with soft boundary constraint and one with hard boundary constraint that holds under an additional condition.

A \textit{soft constraint} means that we choose a sequence of good ``thresholds" $(\theta_n)_{n\geq 1}$ with $\theta_n\to 0$ as $n\to \infty$, and look at uniform measures $\rho_n$ on free-boundary tilings of $R$ at scale $n$ with boundary values within $\theta_n$ of $b$ in the Wasserstein metric $\m W_1^{1,1}$ that we use to compare boundary values. The \textit{soft boundary large deviation principle} (SB LDP) says that $\rho_n$ satisfy an LDP, as long as $\theta_n$ goes to $0$ slowly enough. This is stated precisely in Theorem \ref{thm:sb-ldp}. 

A \textit{hard constraint} means that we choose a fixed sequence of discrete tileable regions $R_n\subset \frac{1}{n}\m Z^3$ approximating $R$, and define $\overline \rho_n$ to be uniform measure on tilings of $R_n$. This is a stronger constraint on the boundary conditions.
We show that the measures $(\overline{\rho}_n)_{n\geq 1}$ satisfy an LDP under two conditions: (i) the region $R_n$ is tileable for all $n$ and (ii) the region and boundary value pair $(R,b)$ is \textit{flexible} meaning that for every $x\in \text{Int}(R)$, there exists $g$ extending $b$ and an open set $U\ni x$ such that $\overline{g(U)}\subset \text{Int}(\mc O)$ or equivalently, there exists $f_0\in AF(R,b)$ such that for every compact set $D\subset \text{Int}(R)$, $\;\overline{f_0(D)}\subset \text{Int}(\mc O)$ (see Definition \ref{def:fully_flexible} and Lemma \ref{lem:fully_flexible}). We call this the \textit{hard boundary large deviation principle} (HB LDP), and it is stated precisely in Theorem \ref{thm:hb-ldp}.

The condition $(R,b)$ \textit{flexible} is strictly stronger than $(R,b)$ \textit{semi-flexible}, which says that for every point $x\in \text{Int}(R)$, there is an extension $g$ and an open set $U\ni x$ such that $g(U)\subset \mc O\setminus \mc E$, or equivalently that $b$ has an extension $f_0$ which is valued in $\mc O\setminus \mc E$ (see Definition \ref{def: flexible} and Lemma \ref{lem: nonwhere edge extension}). Recall that if $(R,b)$ is not semi-flexible we call it \textit{rigid}. 

If $(R,b)$ is \textit{semi-flexible}, then $\Ent(\cdot)$ has a unique maximizer in $AF(R,b)$ (Theorem \ref{thm: unique maximizer}). Whenever this holds, as a corollary of either LDP we show that ``random dimer tilings" of $R$ with boundary values converging to $b$ \textit{concentrate} in the fine-mesh limit on the unique deterministic limiting flow which maximizes $\Ent(\cdot)$ in $AF(R,b)$. This result holds for ``random dimer tiling" defined by sampling from any sequence of measures (i.e.\ $\rho_n$ or $\overline{\rho}_n$) for which an LDP holds, see Corollaries \ref{cor:concentration} and \ref{cor:concentration-hb}. 

We summarize the conditions needed for each of the theorems in the following table. Note that in all cases we have the basic assumptions that $R\subset \m R^3$ is a compact region which is the closure of a connected domain, $\partial R$ is piecewise smooth, and $b$ is a boundary asymptotic flow which is extendable outside. 

\begin{center} \begin{tabular}{|c|c|c|c|}
    \hline
     $(R,b)$ & SB LDP & Unique $\Ent$ maximizer in $AF(R,b)$ & HB LDP\\
\hline
rigid & {yes} & {not known in general} & {no} \\
\hline
semi-flexible & {yes} & {yes} & {no}\\
\hline
flexible & {yes} & {yes} & {yes}\\
\hline
\end{tabular}
\end{center}
We remark that the ``no" entries in this table are statements that are provably not true. In particular, there exists $(R,b)$ semi-flexible for which the hard boundary LDP is false; see Example \ref{ex:hb_false} or the discussion in Section \ref{sec: main results intro}. See Problem \ref{prob: region with non unique max} for discussion of the ``not known" entry.

In Section \ref{sec:sb-statements}, we give the precise definitions, conditions, and statement for the soft boundary LDP (Theorem \ref{thm:sb-ldp}), and in Section \ref{sec:statements-hb}, we do the same for the hard boundary LDP (Theorem \ref{thm:hb-ldp}), and explain why the hard boundary LDP can be false for $(R,b)$ just semi-flexible (Example \ref{ex:hb_false}). In both cases, we prove \textit{concentration} when $(R,b)$ is semi-flexible and the LDP holds (so in hard boundary case, $(R,b)$ must be flexible) as a corollary (Corollaries \ref{cor:concentration} and \ref{cor:concentration-hb}) and show that proving the LDP is equivalent to proving corresponding upper and lower bounds statements (Theorems \ref{thm:lower} and \ref{thm:upper} for the soft boundary LDP and Theorems \ref{thm:lower-hb} and \ref{thm:upper-hb} for the hard boundary LDP). The rest of the section is dedicated to proving the upper and lower bounds.

The proofs of the lower bounds are somewhat involved. In Section \ref{sec:pc_approx} we show that if $b$ is extendable outside, then any $g\in AF(R,b)$ can be approximated by a piecewise-constant asymptotic flow on a region slightly larger than $R$ (Proposition \ref{prop:pc_approx}). This is where we use the extendable outside condition. Building on this, in Section \ref{sec: shining light} we show that any asymptotic flow can be approximated by the tiling flow of a free-boundary tiling (Theorem \ref{thm:shininglight}). Combined with the patching theorem (Theorem \ref{patching}), this is all we need to prove the \textit{soft boundary} lower bound (Theorem \ref{thm:lower}), so we prove this in Section \ref{sec:lower}. In Section \ref{sec:generalized_patching} we state and prove a more powerful \textit{generalized patching} theorem (Theorem \ref{thm:generalized_patching}) and use this to prove the hard boundary lower bound (Theorem \ref{thm:lower-hb}). 

In Section \ref{sec:upper} we prove both upper bounds (Theorem \ref{thm:upper} and \ref{thm:upper-hb}). To do this, we prove the soft boundary upper bound (Theorem \ref{thm:upper}) and note that this implies the hard boundary upper bound (Theorem \ref{thm:upper-hb}).

\subsection{Statement and set up: soft boundary LDP}\label{sec:sb-statements}

Let $R\subset \m R^3$ be a compact region which is the closure of a connected domain, with $\partial R$ piecewise smooth. Recall from Section \ref{sec:asympflows} that for each $n$, $TF_n(R)$ is the set of all scale $n$ free boundary tiling flows on $R$. The fine-mesh limits of these with respect to the Wasserstein metric on flows (Theorem \ref{thm:formal_fine_mesh}) are the asymptotic flows $AF(R)$. The space of asymptotic flows with fixed boundary value $b$ is denoted by $AF(R,b)$. For any compact, piecewise smooth surface $S\subset R$, $T(\cdot,S)$ denotes the \textit{trace operator} which takes an asymptotic or tiling flow to its boundary value on $S$ (see Sections \ref{sec: boundary values of asymptotic flows}, \ref{sec: boundary values of tiling flows}). Recall (Definition \ref{def:boundary_asymptotic_flow}) that $b$ is a \textit{boundary asymptotic flow on $R$} if there exists $g\in AF(R)$ such that $T(g,\partial R) = b$. We restrict our attention to boundary asymptotic flows $b$ which are also \textit{extendable outside}.

\begin{definition}\label{def:extendable}
    A boundary asymptotic flow $b$ on $R$ is \textit{extendable outside} if there exists $\epsilon>0$ such that $b$ extends to a divergence-free measurable vector field on an $\epsilon$ neighborhood of $R$. 
\end{definition}\termindex{Chapter 8!boundary conditions which are extendable outside}

\begin{rem}\label{rem:not_extendable_problem}
The assumption that the boundary asymptotic flow is extendable outside is inherent in \cite{cohn2001variational}. The Lipschitz condition in \cite[Theorem 1.1]{cohn2001variational} implies that there is extension of the flow in $\mathbb R^2$. Such a strong hypothesis is not necessary. However it is easy to build boundary asymptotic flows which are not extendable outside, and some of our current techniques do not work in such cases. Let $R=[-1,1]^2\setminus [0,1]^2$ and consider the flow
$f\in \text{AF}(R)$ given by 
$$f(x)=\begin{cases}(3/4,0) &\text{ if }x\in [-1,0]\times[0,1]\\
(0,3/4) &\text{ if }x\in [0,1]\times[-1,0]\\
(0,0) &\text{ if }x\in [-1,0]\times[-1,0].
\end{cases}$$
Any extension of such a flow close to the origin will have to be valued outside $\mathcal O_2$ by the divergence-free condition. We need $b$ to be extendable outside in our arguments to construct a piecewise-constant approximation $\widetilde g$ of any flow $g\in AF(R,b)$, where $\widetilde g$ is supported on a set $\widetilde R\supset R$ (Proposition \ref{prop:pc_approx}). This is an intermediate step in showing that any $g\in AF(R,b)$ can be approximated by a free-boundary tiling $\tau\in T_n(R)$ for $n$ large enough (Theorem \ref{thm:shininglight}). If $R$ is convex, then $b$ is automatically extendable and thus we don't need to add a condition. 
\end{rem}

The version of the LDP we present in this section has \textit{soft} boundary conditions in the discrete. The sequence of probability measures $(\rho_n)_{n\geq 1}$ which we show satisfy an LDP are uniform probability measures on tiling flows at scale $n$ with boundary values conditioned to be in a sequence of neighborhoods around $b$ which shrink as $n\to \infty$.

Recall that the metric on boundary values of flows is $\m W_1^{1,1}$. To define $\rho_n$, we first define the following sets.
\begin{definition}
Let $b$ be a boundary asymptotic flow and fix a threshold $\theta>0$. We denote the set of scale $n$ tiling flows on $R$ with boundary values within $\theta$ of $b$ by 
$$TF_n(R,  b, \theta) := \{f_{\tau}\in TF_n(R) : \m W_1^{1,1}(T( f_\tau,\partial R), b) <\theta\}.$$
\end{definition}
Note that if $\theta$ is too small, $TF_n(R, b, \theta)$ might be empty. However, it will follow from Theorem \ref{thm:shininglight} that given a fixed $\theta$, if $n$ is large enough then $TF_n(R, b, \theta)$ is nonempty (Corollary \ref{cor:good_thresholds_exist}). 

\termindex{Chapter 8!admissible threshold sequence}\symindex{Chapter 8!$(\theta_n)_{n\geq 1}$ - admissible threshold sequence}We say a sequence of thresholds $(\theta_n)_{n\geq 1}$ is \textbf{admissible} if $\theta_n\to 0$ as $n\to \infty$, but sufficiently slowly so that $TF_n(R, b, \theta_n)$ is nonempty for all $n$. When the threshold sequence $\theta_n$ is understood, we define 
\begin{align*}
    TF(R,b):= \cup_{n\geq 1} TF_n(R,b,\theta_n).
\end{align*}
We define a sequence of probability measures $\rho_n$ using an admissible sequence of thresholds.
\begin{definition}
For all $n\geq 1$, $\rho_n$ is the uniform probability measure on $TF_n(R, b, \theta_n)$. Further, we define $\mu_n$ to be the counting measure on $TF_n(R, b, \theta_n)$ and $Z_n$ to be its partition function, so that $\rho_n = \frac{1}{Z_n} \mu_n$. \symindex{Chapter 8!$\rho_n$, $\mu_n,Z_n$}
\end{definition}
\begin{rem}\label{rem:uniform}
If $\text{Unif}_n$ denotes the uniform probability measure on $TF_n(R)$, then $\rho_n$ is the conditional distribution
\begin{align*}
    \rho_n(\cdot) = \text{Unif}_n(\cdot \mid D_{b, \theta_n})
\end{align*}
where $D_{b,\theta_n}$ is the event that the boundary value of a flow is within $\theta_n$ of $b$.
\end{rem}\termindex{Chapter 8!soft boundary large deviation principle (SB LDP)}

\begin{thm}[Soft boundary large deviation principle]\label{thm:sb-ldp} Let $R\subset \m R^3$ be a compact region which the closure of a connected domain, with piecewise smooth boundary $\partial R$. Let $b$ be a boundary asymptotic flow which is extendable outside. 
    
    \symindex{Chapter 8!$v_n=n^3 \Vol(R)$ is roughly the number of vertices of $\frac{1}{n}\Z^3$ in $R$}There exists a sequence of admissible thresholds $(\theta_n)_{n\geq 1}$ such that the uniform probability measures $(\rho_n)_{n\geq 1}$ on $TF_n(R, b, \theta_n)$ satisfy a large deviation principle in the topology induced by $d_W$ with good rate function $I_{{b}}(\cdot)$ and speed $v_n = n^3 \Vol(R)$. Namely for any $d_W$-Borel measurable set $A$,
    \begin{equation}\label{eq:ldp-prob-sb}
    -\inf_{ g\in A^\circ} I_{ b}( g) \leq \liminf_{n\to \infty} v_n^{-1} \log \rho_n(A) \leq \limsup_{n\to \infty} v_n^{-1} \log \rho_n(A) \leq -\inf_{ g\in \overline{A}} I_{ b}( g) 
\end{equation}
Further, the rate function $I_{ b}( g) = C_{ b} - \Ent( g)$ if $g$ is an asymptotic flow, where $C_{ b} = \max_{ f\in AF(R, b)} \Ent(f)$. If $g$ is not an asymptotic flow then $I_b(g) = \infty$.\symindex{Chapter 8!$I_b$ - rate function}
\end{thm}

\begin{rem}
The existence of a sequence of thresholds for which the theorem holds follows from Theorem \ref{thm:lower}. The only requirement is that $(\theta_n)_{n\geq 0}$ goes to $0$ sufficiently slowly.
\end{rem}

\begin{rem}
The weaker, analogous theorem with free boundary values in the limit would also hold, i.e.\ there is a large deviation principle for the sequence of uniform measures $(\text{Unif}_n)_{n\geq 1}$ on $TF_n(R)$ from Remark \ref{rem:uniform}. The rate function in this case is also of the form $C - \Ent(\cdot)$, with $C = \max_{f\in AF(R)} \Ent(f)$. In fact, since $s=0$ is the unique maximizer of $\ent(\cdot)$, $C = \Ent(0)$, where $0$ is the constant zero flow. 
\end{rem}

Under the additional condition that the pair $(R,b)$ is semi-flexible (see Definition \ref{def: flexible}), $\Ent$ has a unique maximizer in $AF(R,b)$ (Theorem \ref{thm: unique maximizer}). In this case, Theorem \ref{thm:sb-ldp} implies a concentration or weak law of large numbers result for fine-mesh limits of $\rho_n$-random tiling flows. 

\begin{cor}\label{cor:concentration}
Fix $\epsilon>0$. Assume that $(R,b)$ is semi-flexible so that $\Ent$ has a unique maximizer in $AF(R,b)$ which we denote by $f_{\max}$. Define the event
\begin{align*} 
A_{\epsilon} = \{f : d_W(f,f_{\max})>\epsilon\}.
\end{align*}
Then
\begin{align*}
    \rho_n(A_{\epsilon}) \leq C^{-n^3}
\end{align*}
where $C>1$ is a constant depending only on $b$ and $R$. In other words, for any $\epsilon>0$, the probability that a tiling flow at scale $n$ sampled from $\rho_n$ (i.e., with boundary value conditioned to be in a shrinking interval around $b$) differs from the entropy maximizer by more than $\epsilon$ goes to $0$ exponentially fast as $n\to \infty$ with rate $n^3$. 
\end{cor}
\begin{proof}
    Cover $AF(R,b)$ by open neighborhoods $B_g$ around each $g\in AF(R,b)$ so that if $g\neq f_{\max}$ then $\Ent(h)<\Ent(f_{\max})$ for all $h\in \overline{B}_g$, and $B_{f_{\max}}$ is the $\epsilon$-neighborhood of $f_{\max}$. Since $AF(R,b)$ is compact, this has a finite subcover $B_{1},...,B_{k}$, where $B_i$ is a neighborhood of $g_i$. Without loss of generality, $B_1 = B_{f_{\max}}$. By Theorem \ref{thm:sb-ldp}, for $n$ large enough,
    \begin{align*}
        \rho_n(A_{\epsilon}) \leq \sum_{i=2}^k \rho_n(B_{i}) \leq \sum_{i=2}^k \exp(v_n (\Ent(f_i)-\Ent(f_{\max}))),
    \end{align*}
    where $f_i$ is the entropy-maximizer in $\overline{B}_i$. Since $\Ent(f_i) - \Ent(f_{\max}) < 0$ for all $i\neq 1$, this completes the proof.
\end{proof}

Recall that $\mu_n = Z_n \rho_n$ is counting measure on $TF_n(R, b, \theta_n)$. We define notation for Wasserstein open balls\symindex{Chapter 8!$A_{\delta}(g)$ - Wasserstein open balls of radius $\delta$ around $g$}, namely
\begin{align*}
    A_\delta(g) = \{h : d_W( h, g) <\delta\}.
\end{align*}
By \cite[Lemma 2.3]{varadhan2016large}, the large deviation principle for $(\rho_n)_{n\geq 1}$ (Theorem \ref{thm:sb-ldp}) is implied by local upper and lower bound statements (Theorem \ref{thm:lower} and \ref{thm:upper}), plus a property called \textit{exponential tightness}, namely that for any $\alpha<\infty$, there exists a compact set $K_\alpha$ such that, for any closed set $C$ disjoint from $K_\alpha$,
    \begin{equation}\label{eq:exponential_tightness}
        \limsup_{n\to \infty} v_n^{-1} \log \rho_n(C) \leq -\alpha.
    \end{equation}
By Corollary \ref{cor: AF(R,b) compact}, $(AF(R,b),d_W)$ is compact. The space $TF(R,b)$ is countable, and by Theorems \ref{thm:formal_fine_mesh} and Theorem \ref{thm: boundary_value_uniformly_continuous}, the limit points of $TF(R,b)$ are contained in $AF(R,b)$. Therefore $(AF(R,b)\cup TF(R,b),d_W)$ is compact, from which exponential tightness follows. To prove the soft boundary large deviation principle (Theorem \ref{thm:sb-ldp}), it remains to prove the following upper and lower bound theorems.

\begin{thm}[Soft boundary lower bound]\label{thm:lower}
 For any $ g\in AF(R, b)$,
\begin{align*}
    \lim_{\delta\to 0}\liminf_{n\to \infty} v_n^{-1} \log \mu_n(A_{\delta}( g)) \geq \Ent( g).
\end{align*}
\end{thm}
\begin{thm}[Soft boundary upper bound]\label{thm:upper}
 For any $ g\in AF(R, b)$,
\begin{align*}
   \lim_{\delta\to 0}\limsup_{n\to \infty} v_n^{-1} \log \mu_n(A_{\delta}( g)) \leq \Ent( g).
\end{align*}
\end{thm}

\subsection{Statement and set up: hard boundary LDP}\label{sec:statements-hb}

This section parallels Section \ref{sec:sb-statements}, but the LDP we prove is for measures $(\overline{\rho}_n)_{n\geq 1}$ defined with a \textit{hard boundary constraint} in the discrete, instead of the soft constraint used to define the measures $(\rho_n)_{n\geq 1}$ in Section \ref{sec:sb-statements}. 

Again let $R\subset \m R^3$ be a compact region which is the closure of a connected domain with $\partial R$ piecewise smooth, and assume that $b$ is a boundary asymptotic flow which is extendable outside. Unlike the soft boundary LDP, we add the condition that the pair $(R,b)$ is \textit{flexible}. 
\begin{definition}\label{def:fully_flexible} \termindex{Chapter 8!flexible}
    A pair $(R,b)$ is \textit{flexible} if for all $x\in \text{Int}(R)$, there exists $g\in AF(R,b)$ and an open set $U\ni x$ such that $\overline{g(U)}\subset \text{Int}(\mc O)$. 
\end{definition}
By completely analogous arguments to the proof of Lemma \ref{lem: nonwhere edge extension}, we have the following equivalent definition of $(R,b)$ flexible. 
\begin{lemma}\label{lem:fully_flexible}
    A pair $(R,b)$ is {flexible} if and only if there exists $f_0\in AF(R,b)$ such that for every compact set $D\subset \text{Int}(R)$, \, $\overline{f_0(D)}\subset \text{Int}(\mc O)$.
\end{lemma}
\begin{rem}\label{rem:aztec_is_flexible}
It is not hard to see directly that the flexible definition given in Definition \ref{def:fully_flexible} is satisfied for the 3D regions in the introduction built out of aztec diamonds. On each 2D aztec diamond $R_a= R \cap \{z = a\}$ and each point $x\in R_a$, consider a rectangle inscribed in $R_a$ containing $x$ and with edges parallel to the coordinate axes. Then the flow which is $0$ inside the rectangle and linear parallel to the adjacent edge of the rectangle in the four triangles is a 2D asymptotic flow with the right boundary conditions. Averaging these flows for different rectangles gives a flow valued in $\text{Int}(\mc O)$ on $R_a$, and combining them gives a flow $f$ valued in $\text{Int}(\mc O)$ everywhere in $\text{Int}(R)$ (in fact, $f$ will be valued in the middle slice of $\text{Int}(\mc O)$ where the third coordinate is zero). 
\end{rem}

Let $R_n\subset \frac{1}{n} \m Z^3$ be tileable regions approximating $R$ in Hausdorff distance from outside, i.e., the union of lattice squares in $R_n$ covers $R$ and all lattice squares intersect $R$. 
We define $\overline{\rho}_n$ to be uniform measure on tilings of $R_n$. 
If $(R,b)$ is flexible and the boundary values of tilings of $R_n$ converge to $b$ in $\m W_1^{1,1}$ as $n\to \infty$, we prove that $(\overline{\rho}_n)_{n\geq 1}$ satisfy an LDP.

\begin{thm}[Hard boundary large deviation principle]\label{thm:hb-ldp} Let $R\subset \m R^3$ be a compact region which the closure of a connected domain, with piecewise smooth boundary $\partial R$. Let $b$ be a boundary asymptotic flow which is extendable outside, and assume that $(R,b)$ is flexible. 

Let $R_n\subset \frac{1}{n}\m Z^3$ be a sequence of tileable regions approximating $R$ in Hausdorff distance and with boundary values  converging to $b$ in $\m W_1^{1,1}$. Define $\overline{\rho}_n$ to be the uniform probability measure on tilings of $R_n$. 

The measures $(\overline{\rho}_n)_{n\geq 1}$ satisfy a large deviation principle in the topology induced by $d_W$ with good rate function $I_{{b}}(\cdot)$ and speed $v_n = n^3 \Vol(R)$. Namely for any $d_W$-Borel measurable set $A$,
    \begin{equation}\label{eq:ldp-prob-hb}
    -\inf_{ g\in A^\circ} I_{ b}( g) \leq \liminf_{n\to \infty} v_n^{-1} \log \overline{\rho}_n(A) \leq \limsup_{n\to \infty} v_n^{-1} \log \overline{\rho}_n(A) \leq -\inf_{ g\in \overline{A}} I_{ b}( g)
\end{equation}
Further, the rate function $I_{ b}( g) = C_{ b} - \Ent( g)$ if $g$ is an asymptotic flow, where $C_{ b} = \max_{ f\in AF(R, b)} \Ent(f)$. If $g$ is not an asymptotic flow then $I_b(g) = \infty$.
\end{thm}\termindex{Chapter 8!hard boundary large deviation principle (HB LDP)}

\begin{rem}
    The large deviation principle in \cite{cohn2001variational} has hard boundary conditions, where the regions $R_n$ approximate $R$ from \textit{within}, i.e.\ $R_n\subset R$. We instead assume that $R_n\supset R$, and that our regions approximate $R$ from outside.
\end{rem}

The \textit{flexible} condition on $(R,b)$ is needed for the \textit{generalized patching theorem} (Theorem \ref{thm:generalized_patching}). We do not know the exact condition on $(R,b)$ needed for Theorem \ref{thm:hb-ldp} to hold, however there do exist regions $(R,b)$ which are just semi-flexible but for which the hard boundary LDP fails. 
\begin{exmp}\label{ex:hb_false}
    As discussed in the introduction, there exists {semi-flexible} region and boundary condition pairs $(R,b)$ with $R\subset \m R^3$ for which the hard boundary large deviation principle is false. The region $R$ is a ``tilted tube," and the boundary value $b$ takes values in a face of $\partial \mc O$. This construction is related to the measures with boundary mean current discussed in Section \ref{sec:extreme}.
    Recall the definition of the \textit{slabs}
    \begin{align*}
        L_c = \{(x_1,x_2,x_3): x_1 + x_2 + x_3 = 2c \text{ or } 2c+1\}. 
    \end{align*}
    Each slab is the union of two planes. Any tiling sampled from a measure with mean current $(s_1,s_2,s_3)\in \partial \mc O$ with $s_1,s_2,s_3\geq 0$ breaks into a sequence of complete dimer tilings of the slabs (Proposition \ref{prop: Sc lattice description}). 
    
    Let $B_n = [-n,n]^3$. Let $A_n(0) = L_0 \cap B_n$, let $A_n(c) = L_c \cap [B_n+(0,0,2c)]$, and finally let $A_{n,n} = \cup_{i=-n}^n A_n(i)$. The region $R$ is then defined so that $\frac{1}{n} A_{n,n}$ is \textit{a} sequence of discrete regions approximating it. We choose the boundary value $b$ to be a constant mean current $s=(s_1,s_2,s_3)\in \partial \mc O$ with $s_1,s_2,s_3>0$. Note that the constant asymptotic flow $g(x) =s\in AF(R,b)$, and by Theorem \ref{thm: extremal_entropy} has $\Ent(g) = \ent_{\text{loz}}(s)>0$. Here are two options we could choose for the sequence of discrete regions:
    \begin{itemize}
        \item We define one sequence of regions $R_n^1$ where for each $c$ such that $L_c$ intersects $A_{n,n}$, $R_n^1\cap L_c=S_c$ is a region such that a lozenge tiling of $S_c$ has slope $s$ along $\partial S_c$. 
        \item We define another sequence of regions $R_n^2$ by alternating between frozen brickwork lozenge tilings. Choose a sequence of ratios $(s_1^n,s_2^n,s_3^n)$ converging to $(s_1,s_2,s_3)$ as $n\to\infty$. We partition the group of indices $c$ such that $L_c$ intersects $A_{n,n}$, into three groups with sizes proportional to $s_1^n,s_2^n,s_3^n$. For $i=1,2,3$, for each $c$ in the $i^{th}$ group, we define $R_n^2\cap L_c$ to be the region tileable by the $\eta_i$ lozenge brickwork tiling. 
    \end{itemize}
    By results for 2D lozenge tilings and the relationships established in Section \ref{sec:extreme}, the hard boundary LDP would hold for the sequence $R_n^1$. On the other hand, $R_n^2$ is frozen for all $n$, so the number of free boundary tilings of $R_n^2$ is $1$ for all $n$. While this unique tiling does approximate the constant flow $g(x)=s$, the corresponding lower bound for the LDP does not hold.

    More generally, one might conjecture that the hard boundary LDP fails when there exist regions in $\text{Int}(R)$ where the $\Ent$-maximizing flow is valued in the faces of $\partial \mc O$. However we do not know if this is a necessary or sufficient condition for the hard boundary LDP to fail, or if there exists a region $R\subset \m R^3$ where the $\Ent$-maximizing flow takes face values in the interior of $R$ but not all of $R$. See Problem \ref{prob:limit_shape_values}.
\end{exmp}

The analogous concentration or weak law of large numbers result that held for the soft boundary measures $\rho_n$ (Corollary \ref{cor:concentration}) also holds for the hard boundary measures $\overline{\rho}_n$. Note that flexible implies semi-flexible, so the maximizer $f_{\max}\in AF(R,b)$ is unique by Theorem \ref{thm: unique maximizer}.
\begin{cor}\label{cor:concentration-hb}
    Assume that $(R,b)$ is flexible and $\overline{\rho}_n$ are as in Theorem \ref{thm:hb-ldp}. Let $f_{\max}$ denote the unique maximizer of $\Ent$ in $AF(R,b)$. Define the event
\begin{align*} 
A_{\epsilon} = \{f : d_W(f,f_{\max})>\epsilon\}.
\end{align*}
Then
\begin{align*}
    \overline{\rho}_n(A_{\epsilon}) \leq C^{-n^3}
\end{align*}
where $C>1$ is a constant depending only on $R$ and $b$. In other words, for any $\epsilon>0$, the probability that a tiling flow at scale $n$ sampled from $\overline{\rho}_n$ (i.e., corresponding to a tiling of $R_n$) differs from the entropy maximizer by more than $\epsilon$ goes to $0$ exponentially fast as $n\to \infty$ with rate $n^3$.
\end{cor}
\begin{proof}
    Analogous to the proof of Corollary \ref{cor:concentration} with $\rho_n$ replaced by $\overline{\rho}_n$ and Theorem \ref{thm:sb-ldp} replaced by Theorem \ref{thm:hb-ldp}.
\end{proof}

Like Theorem \ref{thm:sb-ldp}, by \cite[Lemma 2.3]{varadhan2016large}, to prove Theorem \ref{thm:hb-ldp} it suffices to show that the measures $(\overline{\rho}_n)_{n\geq 1}$ satisfy local upper and lower bound statements (Theorems \ref{thm:lower-hb} and \ref{thm:upper-hb}) plus the \textit{exponential tightness property} stated for $\rho_n$ in Equation \eqref{eq:exponential_tightness}, which follows by analogous straightforward arguments for $\overline{\rho}_n$. 

We let $\overline{Z}_n$ denote the partition function of $\overline{\rho}_n$ and $\overline{\mu}_n = \overline{Z}_n \overline{\rho}_n$ the corresponding counting measures.\symindex{Chapter 8!$\overline{\rho}_n,\overline{\mu}_n,\overline{Z}_n$}

\begin{thm}[Hard boundary lower bound]\label{thm:lower-hb}
 For any $ g\in AF(R, b)$,
\begin{align*}
    \lim_{\delta\to 0}\liminf_{n\to \infty} v_n^{-1} \log \overline{\mu}_n(A_{\delta}( g)) \geq \Ent( g).
\end{align*}
\end{thm}
\begin{thm}[Hard boundary upper bound]\label{thm:upper-hb}
 For any $ g\in AF(R, b)$,
\begin{align*}
   \lim_{\delta\to 0}\limsup_{n\to \infty} v_n^{-1} \log \overline{\mu}_n(A_{\delta}( g)) \leq \Ent( g).
\end{align*}
\end{thm}

\subsection{Piecewise constant approximation}\label{sec:pc_approx}

The goal of this section is to show that if $b$ is extendable outside, then any $g\in AF(R,b)$ is well-approximated in the Wasserstein metric on flows by an asymptotic flow $\widetilde g$ which is piecewise-constant, taking constant values on a mesh $\mc X$ of small tetrahedra covering $R$ (see Remark \ref{rem: why tetrahedra} for why tetrahedra). 

\begin{prop}\label{prop:pc_approx}
    Fix $\epsilon>0$, and suppose that $b$ is a boundary asymptotic flow which is extendable outside. For any $g\in AF(R,b)$, there exists $\delta>0$ and a $\delta$-mesh of tetrahedra $\mc X$ covering $R$ with the following properties. Let $\widetilde{R}=\cup_{X\in \mc X} X$. There exists a flow $\widetilde{g}$ satisfying:
    \begin{itemize}
        \item $\widetilde g\in AF(\widetilde{R})$;
        \item $d_W(g,\widetilde{g})<\epsilon$;
        \item For each $X\in \mc X$, $\widetilde{g}\mid_X = \widetilde g_X$ is constant;
        \item $\widetilde g$ is valued strictly in $\text{Int}(\mc O)$;
        \item $\widetilde g$ takes only rational values. 
    \end{itemize}
\end{prop}\symindex{Chapter 8!$\mathcal X$ - tetrahedral mesh}
\symindex{Chapter 8!$\widetilde g$ - usually denotes the piecewise approximation of $g$}\termindex{Chapter 8!piecewise constant approximation}\termindex{Chapter 8!tetrahedral mesh}
\begin{rem}
    We need the condition that $b$ is extendable outside so that we can take $\widetilde{R}$ to contain $R$. If $b$ is not extendable outside, a similar construction works, but the resulting piecewise-constant flow will be an asymptotic flow on a region $R'$ slightly smaller than $R$ instead.
\end{rem}

\begin{rem}\label{rem: why tetrahedra}
\textbf{Tetrahedral mesh.} The fact that the mesh in this construction is built out of tetrahedra is necessary to ensure that $\widetilde{g}$ is divergence-free (needed for $\widetilde{g}$ to be an asymptotic flow). This is because a divergence-free flow on a polyhedron with $F$ faces is determined by its flow through $F-1$ of them. Since we have $3$ free parameters to specify $\widetilde{g}$ on one polyhedron, we need $F-1\leq 3$. The only polyhedra that satisfy this are tetrahedra. 
\begin{figure}
    \centering
        \includegraphics[scale=0.28]{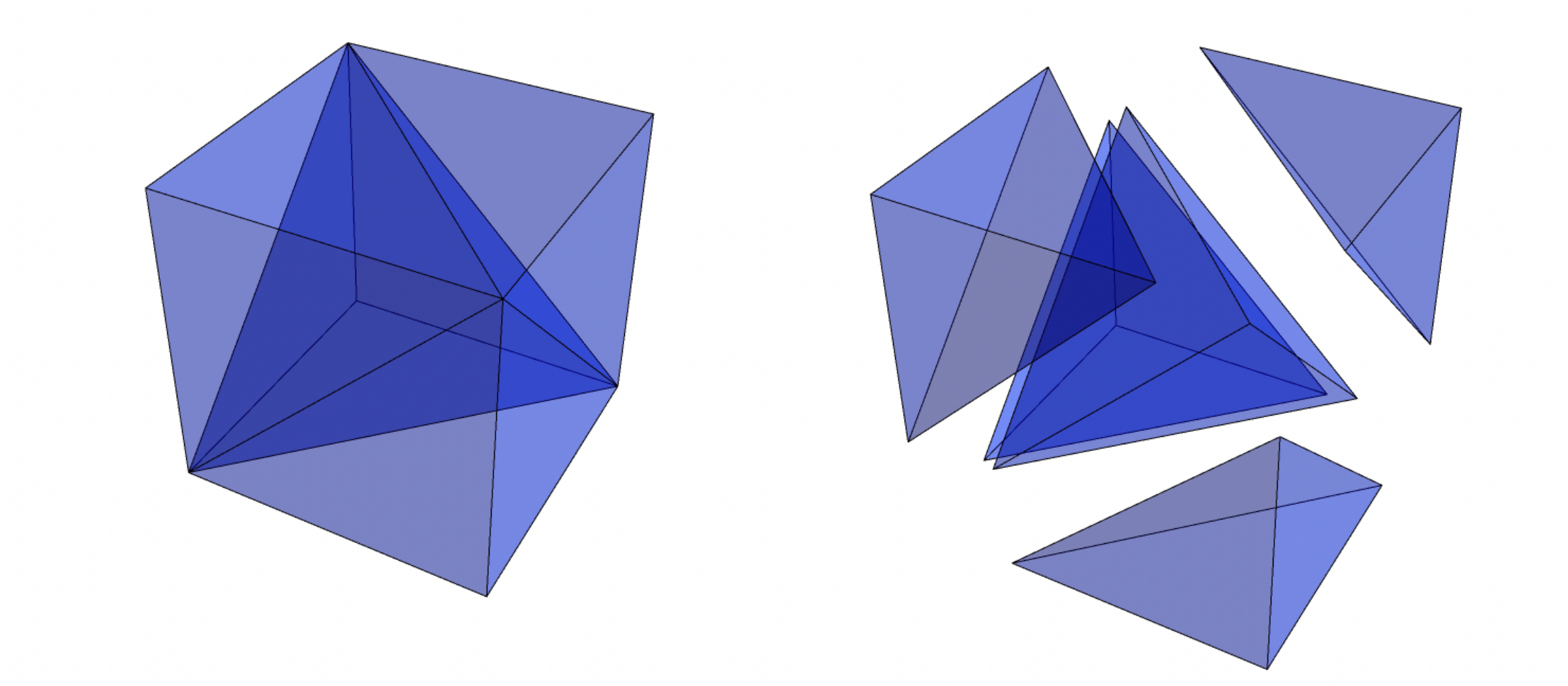} \caption{A cube cut into one regular tetrahedron and four right-angled tetrahedra. The second picture shows the same tetrahedra moved apart.}
    \label{fig:cube cut into tetra}
\end{figure}

However, regular tetrahedra alone do not tile 3-space\footnote{Over 2,000 years ago, Aristotle (mistakenly) claimed in \textit{De Caelo}, Book III Part 8 \cite{aristotle} that regular tetrahedra do tile 3-space. It took around 1,000 years for the mistake to be fixed, see \cite{tetrahedra} for a detailed account of the story.}, so we cannot take all elements of the mesh to be identical. Instead, 3-space can be tiled by regular tetrahedra and right-angled tetrahedra. To see this, note that cubes tile 3-space, and a cube can be cut into four right-angled tetrahedra and one regular tetrahedron (see Figure \ref{fig:cube cut into tetra}). The faces of the regular tetrahedron have normal vectors of the form $(\pm 1, \pm 1, \pm 1)$,  while the right-angled tetrahedra have four coordinate plane faces and one face with normal vector $(\pm 1, \pm 1, \pm 1)$. For technical reasons (see the proof of Theorem \ref{thm:shininglight}, in particular Lemma \ref{lem:gluing}), our arguments are simplified by assuming that the faces of the tetrahedra are always contained in one of these two types of planes. In the proof of Proposition \ref{prop:pc_approx} we will also use this to say that the possible normal vectors to the tetrahedra can be assumed to form a finite set. We assume throughout that our tetrahedral mesh is built out of regular and right-angled tetrahedra.
\end{rem}

\begin{proof}[Proof of Proposition \ref{prop:pc_approx}]

    Since $b$ is extendable outside, there exist $\alpha_0>0$ such that $g$ can be extended to $g'\in AF(R^{\alpha_0})$, where $R^{\alpha_0}$ is
    \begin{align*}
        R^{\alpha_0} = \{x\in \mathbb R^3 : d(x,R)\leq \alpha_0\}. 
    \end{align*}
    Given this, for any $0<\alpha<\alpha_0$, we can approximate $g$ by a \textit{continuous} asymptotic flow $g_\alpha\in AF(R^{\alpha_0-\alpha})$:
    \begin{align*}
        g_\alpha(x) := \frac{1}{|B_\alpha(x)|} \int_{B_\alpha(x)} g'(y) \, \dd y, \qquad x\in R^{\alpha_0-\alpha}.
    \end{align*}
    As $\alpha\to 0$, $d_W(g,g_\alpha\mid_{R})\to 0$. We construct a piecewise-constant, divergence-free approximation $u$ of $g_\alpha$, then modify it to construct $\widetilde{g}$ also satisfying the last two conditions. 
        
    For any fixed $\delta<\alpha_0-\alpha$ (to be specified more precisely later), we take a $\delta$-tetrahedral mesh $\mc X$ built from regular and right-angled tetrahedra (see Remark \ref{rem: why tetrahedra}) such that $X\cap R\neq \emptyset$ for all $X\in \mc X$. Let $\widetilde{R} = \cup_{X\in \mc X} X$ and note that $R\subset \widetilde{R} \subset R^{\alpha_0-\alpha}$. 
  
    Consider one tetrahedron $X\in \mc X$. Let $\zeta_1,\zeta_2,\zeta_3,\zeta_4$ denote the faces of $X$ and let $n_1,n_2,n_3,n_4$ denote their outward pointing normal vectors. Define a vector $u_X$ by 
    \begin{align*}
        u_X \cdot n_i = \frac{1}{\text{area}(\zeta_i)} \int_{\zeta_i} \langle g_\alpha, n_i\rangle \, \dd A \qquad i = 1,2,3.
    \end{align*}
    Since $g_\alpha$ is divergence-free on $X$, 
    \begin{align*}
        u_X \cdot n_4 = \frac{1}{\text{area}(\zeta_4)} \int_{\zeta_4}\langle g_\alpha,n_4\rangle \,\dd A.
    \end{align*}
    Define $u(x):= u_X$ for $x\in X$.

    It remains to show that (up to multiplying by a constant $\lambda\leq 1$ but very close to $1$) $\lambda u\in AF(\widetilde R)$ and bound $d_W(\lambda u, g_\alpha\mid_{\widetilde{R}})$. 

    Since $g_\alpha$ is continuous and ${R}^{\alpha_0-\alpha}$ is compact, $g_\alpha$ is uniformly continuous on $R^{\alpha_0-\alpha}$. Thus given any $\beta>0$ there exists $\theta>0$ such that $|x-y|<\theta$ implies $|g_\alpha(x)-g_\alpha(y)|<\beta$. 

    Fixing $\beta$, we now require that $\delta<\theta$ so that uniform continuity implies that for any $X\in \mc X$ and point $x\in X$, we have that $|g_\alpha(x) - \text{avg}_X g_\alpha|<\beta$. The normal vectors $n_1,n_2,n_3$ to three faces of $X$ are linearly independent but not necessarily orthogonal. However since all $X\in \mc X$ are of one of five forms (see Figure \ref{fig:cube cut into tetra}), there is a constant $K>0$ independent of $X\in \mc X$ so that
    \begin{align*}
        |u_X - \text{avg}_X g_\alpha| \leq K\sum_{i=1}^3 |\text{avg}_{\zeta_i} (g_\alpha\cdot n_i) - \text{avg}_{X} (g_\alpha\cdot n_i)| < 3 K\beta. 
    \end{align*}

     Therefore for all $x\in \widetilde{R}$,
    \begin{equation}
        |u(x) - g_\alpha(x) | < (3K + 1)\beta.
    \end{equation}
    Replacing $u$ by $\lambda u$ with $\lambda = 1-(3K+1)\beta-\beta$, the new flow $\lambda u\in AF(\widetilde R)$ and in fact is valued in $\text{Int}(\mc O)$. Further by Proposition \ref{prop: wass vs sup}, there is another constant $C>0$ such that 
    \begin{equation}\label{eq: approximation to average}
        d_W(\lambda u,g_\alpha\mid_{\widetilde{R}})< C(6K+3) \beta.
    \end{equation}
     By the triangle inequality,
    \begin{align*}
        d_W(\lambda u, g) \leq d_W(\lambda u, g_\alpha\mid_{\widetilde{R}})+d_W(g_\alpha\mid_{\widetilde R},g)\leq 
        d_W(\lambda u,g_\alpha\mid_{\widetilde{R}}) + d_W(g_\alpha\mid_{\widetilde{R}\setminus R}, 0) + d_W(g,g_\alpha\mid_{R}),
    \end{align*}
    where in the second inequality, we use that $g$ is supported in $R$ and a basic property of the Wasserstein distance. The first term is controlled by Equation \eqref{eq: approximation to average}. The second is bounded by a fixed constant times $\delta$,  and the third is bounded by a fixed constant times $\alpha$. (Since $g_\alpha$ is defined by averaging $g$ over balls of size $\alpha$, to transform $g_\alpha$ into $g$, the maximum distance that flow needs to be moved is $\alpha$, and the total flow is of constant order. This gives a bound on the Wasserstein distance of order $\alpha$.) Therefore taking $\alpha,\beta>0$ small enough and correspondingly taking $\delta< \max\{\alpha_0-\alpha,\theta\}$, $d_W(\lambda u,g)$ can be made arbitrarily small. 

    The flow $\lambda u\in AF(\widetilde R)$ and is valued in $\text{Int}(\mc O)$. Finally we modify $\lambda u$ as follows to construct $\widetilde g$ which also takes rational values. To do this, we solve the linear constraint problem to make the values of the flow rational without breaking the divergence-free condition.
        
        Let $M$ be the number of tetrahedra in the mesh $\mc X$. Enumerate the faces of the tetrahedra by $a_1,...,a_m$. Choose a unit normal vector $n_i$ for each face. For any flow $f$, let $F(f) = (F_1(f),...,F_m(f))$, where $F_i(f) = \int_{a_i} \langle f, n_i\rangle \, \dd x$. Note that if $v$ is a piecewise-constant flow on the mesh, then $F(v)$ determines $v$. 
        
        If $F(v)$ corresponds to a divergence-free piecewise-constant flow $v$, then it satisfies a matrix $A$ of $M$ linear constraints of the form
\begin{align*}
    \pm F_{k_1}(v)\pm F_{k_2}(v)\pm F_{k_3}(v)\pm F_{k_4}(v) = 0
\end{align*}
for $a_{k_j}$, $j\in \{1,2,3,4\}$ the faces of a tetrahedron $X\in \mc X$ (the signs are determined by the normal vector orientation, the four terms should all be for flow oriented out of $X$). Thus $F(v)$ solves $A \,F(v) = 0$. Since $A$ has integer entries, there is a rational basis for the space of solutions $Y$ of $A \, Y = 0$. Any other solution can be written as a linear combination of the rational ones, so rational solutions are dense. 

Thus we can find $\widetilde g$ such that $\widetilde g_X$ takes all rational values and $|\widetilde g_X - (1-\delta_1)u_X|$ is as small as needed. Applying Proposition \ref{prop: wass vs sup} again completes the proof. 
\end{proof}

\subsection{Existence of tiling approximations}\label{sec: shining light}

Building on the approximation result in the previous section, we now show that if $b$ is extendable outside then any ${g}\in AF(R,b)$ can be approximated in Wasserstein distance by a tiling flow. More precisely:
\begin{thm}\label{thm:shininglight}
Fix $\delta>0$ and suppose $b$ is a boundary asymptotic flow which is extendable outside. For any ${g}\in AF(R,b)$, there exists $n(\delta)$ such that if $n\geq n(\delta)$, then there is a free boundary tiling $\tau\in T_n(R)$ such that $f_\tau\in A_\delta(g)$.
\end{thm}

The two dimensional analog of this theorem (i.e.\ \cite[Prop.\,3.2]{cohn2001variational}) is the statement that any asymptotic height function can be approximated by the height function of a tiling. In particular, one can choose the maximal height function (analog of $f_\tau$) less than the given asymptotic height function (analog of $g$). There is no analogous notion of ``maximal" tiling flow, so our argument in three dimensions is more complicated, and relies on an explicit construction.

We call the explicit construction in the proof of Theorem \ref{thm:shininglight} the ``shining light construction." The first step is to build piecewise-linear ``channels." We give a method for tiling the channels and show that we can glue them together to construct a tiling of the whole region. The channels are tubular neighborhoods of the flow lines of a tiling flow approximating a piecewise-constant flow as constructed in Proposition \ref{prop:pc_approx}.  We call it the ``shining light construction" because we imagine the flow as beams of light bending through the channels.

Before proving Theorem \ref{thm:shininglight}, we note that the existence of an admissible sequences of thresholds $(\theta_n)_{n\geq 1}$ follows as a straightforward corollary. 
\begin{cor}\label{cor:good_thresholds_exist}
    For any boundary asymptotic flow $b$ which is extendable outside and any threshold $\theta>0$, $TF_n(R, b, \theta)$ is nonempty for $n$ large enough. In particular, admissible sequences of thresholds $(\theta_n)_{n\geq 1}$ exist for any boundary asymptotic flow $b$ which is extendable outside. 
\end{cor}
\begin{proof}
    Recall that $T(\cdot, \partial R): AF(R)\to \mc M^s(R)$ is the boundary value operator and choose $g\in T^{-1}(b)$. By Theorem \ref{thm:shininglight}, for any $n\geq n(\delta)$ there exists a tiling $\tau \in T_n(R)$ with $d_W(f_\tau,g)<\delta$. Since $T$ is uniformly continuous (Theorem \ref{thm: boundary_value_uniformly_continuous}), we can choose $\delta>0$ so that $d_W(g,f_\tau)<\delta$ implies $d_W(b,T(f_\tau)) <\theta$.
\end{proof}

We now proceed to the explicit construction. Recall that $\eta_i$ is the $i^{th}$ positively-oriented unit coordinate vector and that $e_i$ denotes the edge in $\m Z^3$ connecting the origin to $\eta_i$. Similarly, $-e_i$ is the edge connecting the origin to $-\eta_i$. 

Let $\tau_1$ denote the brickwork tiling where all tiles are $-\eta_1$ bricks. To prove Theorem \ref{thm:shininglight}, we show that we can construct a tiling $\tau$ so that the flow corresponding to the double dimer tiling $(\tau, \tau_1)$ is close to the flow $g+\eta_1$. A double dimer tiling consists of a collection of oriented infinite paths, finite loops, and double edges. See Section \ref{section:tiling_flows} and Section \ref{subsection: flows for double dimer}. 

Since $\tau_1$ consists of only $-\eta_1$ tiles, for any other tiling $\tau$, $(\tau,\tau_1)$ consists of only infinite paths and double edges (i.e.\ no finite loops). The double dimer flow $f_{(\tau,\tau_1)} = f_{\tau} - f_{\tau_1}$ is $0$ whenever the tilings agree, and otherwise points in the direction of the oriented infinite path. 

For $x\in \m Z^3$, let $\tau(x)$ denote the tile at $x$ in $\tau$. We say that a tiling $\tau$ of $\m Z^3$ is \textit{periodic} if there exist even integers $r_1,r_2,r_3>0$ such that $\tau(x)$ is equal to its translates $\tau(x+r_1\eta_1)=\tau(x+r_2\eta_2)=\tau(x+r_3 \eta_3)$ for all $x \in \m Z^3$. For periodic tilings, we can define a notion of the \textit{mean current of a tiling}, denoted $s(\tau)$, as the average direction of the tiles in any $r_1\times r_2\times r_3$ box.\termindex{Chapter 8!periodic tiling}

We give a method for constructing a periodic tiling $\tau_v$ of $\mathbb{Z}^3$ of a fixed, rational mean current $v\in \mc O$. This construction will serve as a building block in the proof of Theorem \ref{thm:shininglight}. 

\noindent \textbf{Construction of tiling $\tau_v$.} \symindex{Chapter 8!$\tau_v$}

First we give a construction for $v=(v_1,v_2,v_3)\in \partial \mc O\cap \m Q^3$ with $v_1\geq 0$, then we adapt this to the general case. 

Here we view dimer tiles $a$ in a tiling as vectors directed from even to odd. When we subtract a tiling, we reverse the direction of its dimers. Since $v$ is rational and has nonzero norm, we can find a sequence of tiles ${a}_1,...,{a}_r\in \{{\eta}_1, \text{sign}(v_2) {\eta}_2, \text{sign}(v_3) {\eta}_3\}$ such that ${a}_1+...+{a}_r + r \eta_1$ is parallel to ${v} + {\eta}_1$.

Below by a \textit{plane with normal vector (1,1,1)}, we mean a collection of cubes in $\m Z^3$ with coordinates $\{(x_1,x_2,x_3): x_1 + x_2 + x_3 = c\}$ for some constant $c\in \m Z$. We analogously define a plane with normal vector $(\pm 1, \pm 1, \pm 1)$ to be the modification of this with appropriate signs.

Choose a plane $C_0$ with normal vector $\xi = (1, \text{sign}(v_2), \text{sign}(v_3))$. Let $C_k$ denote $C_0 + (0,0,k)$ for all $k\in \mathbb Z$. Further, assume that the cubes on $C_0$ are even, so that edges parallel to one of $\{\eta_1,\text{sign}(v_2)\eta_2,\text{sign}(v_3)\eta_3\}$ connect cubes on $C_0$ to cubes on $C_1$. Since $v\in \partial \mathcal O$, any tiling with mean current $v$ splits into perfect dimer tilings of \textit{slabs} $L_k$ which consist of unions of adjacent planes $L_k:=C_{2k}\cup C_{2k+1}$, $k\in \mathbb Z$ (see Section \ref{sec:extreme}). By Proposition \ref{prop: Sc lattice description}, each slab is a copy of the 2-dimensional hexagonal lattice. There are three 3D dimer tilings of $C_{2k}\cup C_{2k+1}$ consisting of only one type of dimer, and these correspond to the three brickwork lozenge tilings using one type of lozenge. (See Figure \ref{figure: lozenge dimer} for a review of the correspondence between 3D dimers and lozenges.)

Restricted to each slab $L_k$, $\tau_v$ will be one of the three brickwork lozenge tilings. On $L_0=C_0\cup C_1$, $\tau_v$ will be the ${a}_1$ brickwork lozenge tiling. On $L_1=C_2\cup C_3$, $\tau_v$ will be the ${a}_2$ brickwork lozenge tiling. We continue this by repeating the periodic sequence $a_1,...,a_r$ forwards and backwards in $k$ to choose the tile type for $\tau_v$ on all other slabs $L_k=C_{2k}\cup C_{2k+1}$. 

The reference tiling $\tau_1$ consists of all $-\eta_1$ tiles, which connect $C_{2k}$ to $C_{2k-1}$. Subtracting $\tau_1$, the tiles in $-\tau_1$ connect $C_{2k-1}$ to $C_{2k}$, meaning that they connected the ``odd" half of $L_{k-1}$ to the ``even" half of $L_k$. Hence in the double dimer tiling $(\tau_v,\tau_1)$, every tile is on an infinite path. Along each infinite path, $(\tau_v,\tau_1)$ consists of the periodic sequence of tiles parallel to $\ldots a_1,\eta_1, a_2,\eta_1,\dots,\eta_1,a_r,\ldots$. In particular, all infinite paths are parallel to ${v}+{\eta}_1$. This completes the construction for $v\in \partial \mc O\cap \m Q^3$, $v_1\geq 0$.
\begin{figure}
    \centering
    \includegraphics[scale=0.4]{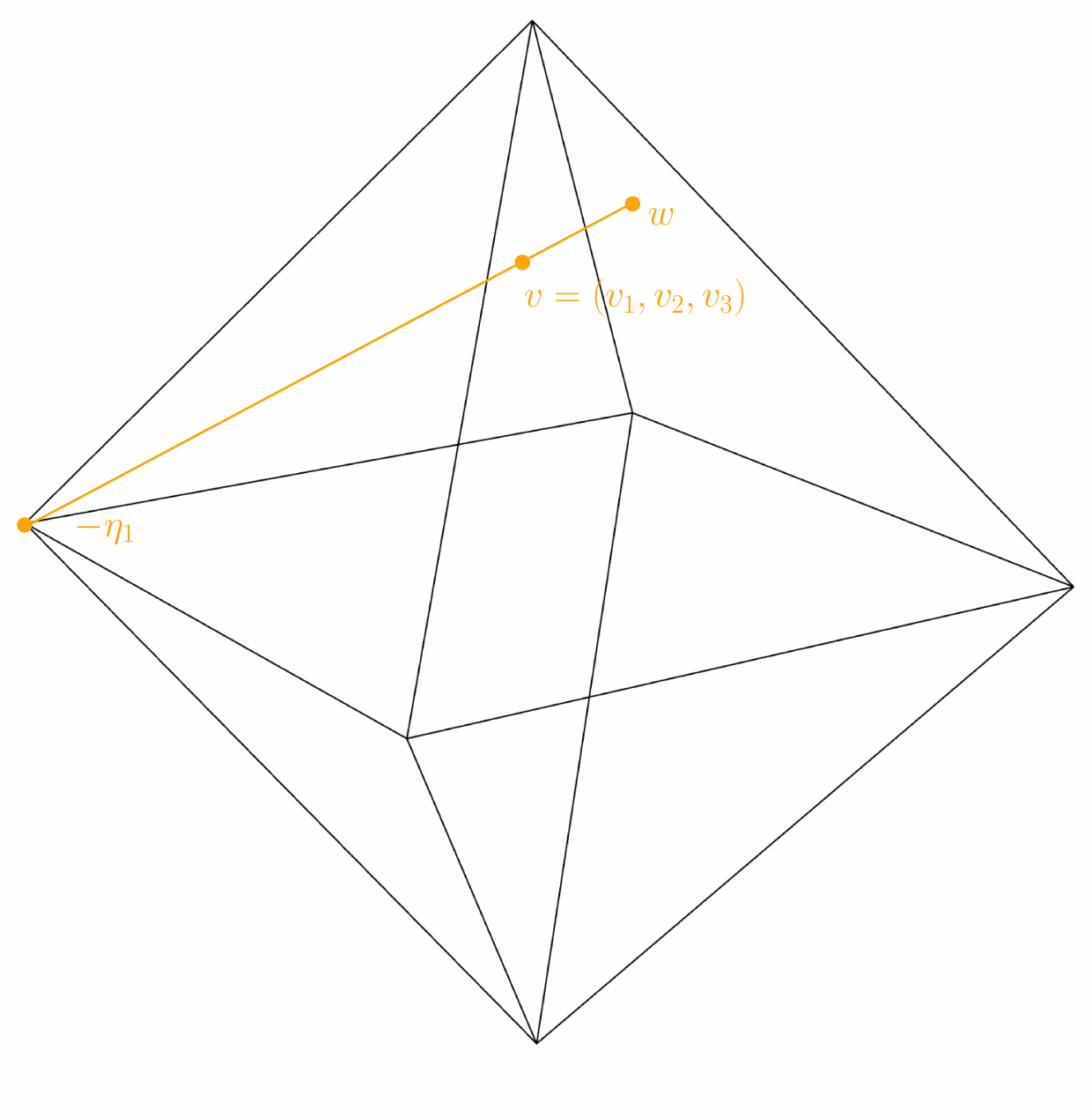}
    \caption{Above is an example of $v= (v_1,v_2,v_3)$ and its relationship to $w(v) = w$.}
    \label{fig:relationship between v and w}
\end{figure}

Now we extend the construction to any $v\in \mc O\cap \m Q^3$, $v\neq -\eta_1$. Let $p_v$ be the line through $-\eta_1 = (-1,0,0)$ and $v=(v_1,v_2,v_3)$. Let $w = w(v)$ be the intersection of $p_v$ with $\{u=(u_1,u_2,u_3) \in \partial \mc O : u_1\geq 0\}$. See Figure \ref{fig:relationship between v and w}. The relationship between $v,w$ will be sufficiently important that we record it as a definition. 
\begin{definition}\label{def:relationship between v and w}
Fix $v\in \mc O$, $v\neq -\eta_1$, and let $p_v$ be the line through $-\eta_1$ and $v$. We define $w(v)$ to be the intersection of $p_v$ with the part of $\partial \mc O$ with non-negative first coordinate. \symindex{Chapter 8!$w(v)$}
\end{definition}

Note that if $v$ is rational, $w=w(v)\in \partial \mc O$ is rational. Since the first coordinate of $w$ is non-negative, we can construct $\tau_w$ as described above. In $(\tau_w,\tau_1)$, every tile is along an infinite path. On the other hand, in $(\tau_1,\tau_1)$ none of the tiles are along infinite paths. To construct $(\tau_v,\tau_1)$, we interpolate between these two options by choosing an intermediate density of infinite paths.

If the line $p_v(t)$ is parameterized with constant speed so that $p_v(1) = (-1,0,0)$ and $p_v(0) = w$, let $a\in[0,1]$ be such that $p_v(a) = v$. Since $v$ rational, $w$ and $a$ are also rational. Thus there exist periodic patterns of cubes in $C_0$ with density $a$. To construct $(\tau_v,\tau_1)$, we fix a choice of periodic pattern of cubes on $C_0$ with density $a$. We delete all the infinite paths in $(\tau_w,\tau_1)$ which do not go through one of the chosen cubes on $C_0$ and replace them with tiles parallel to $-\eta_1$. The resulting tiling is $\tau_v$.

The tilings $\tau_v$ have a few important properties which we highlight. 
\begin{lemma}\label{lem:tau_v_props}
   Let $v\in \mathcal O; v\neq -\eta_1$. \begin{enumerate}
        \item $\tau_v$ has mean current $v$;
        \item Let $w = w(v)$ be as in Definition \ref{def:relationship between v and w}. The infinite paths in $(\tau_v,\tau_1)$ are parallel to $w+\eta_1$.
        \item For any rational plane $P$, the restriction of $\tau_v$ to $P$ is doubly periodic, with period depending on $r$ (the number of tiles $a_1,...,a_r$ used to approximate $w(v)$), the choice of periodic pattern of cubes in $C_0$ and $P$. 
    \end{enumerate}
\end{lemma}

\begin{rem}
Note that $\tau_v$ is not uniquely determined by $v$. It depends on the sequence of tiles $a_1,...,a_r$ used to approximate ${v}$, and on the periodic pattern of initial sites on $C_0$. 
\end{rem}

We now show that pieces of $\tau_v,\tau_u$ can ``glued" along a plane $P$, as long as $ v, u$ have the same flow through $P$. The amount of space $k$ we need to glue is a constant depending only on 
the period of the tilings $\tau_u$ and $\tau_v$.

\begin{lemma}\label{lem:gluing}
  Suppose that $u,v\in \mathcal{O} \cap \mathbb{Q}^3; u, v\neq \eta_1$ and $\tau_u, \tau_v$ are tilings as in Lemma \ref{lem:tau_v_props}. Suppose that $P$ is a coordinate plane or plane with normal vector of the form $(\pm 1, \pm 1, \pm 1)$. In both cases we denote the normal vector by $ n_P$. Let $r$ be such that  $\tau_u$ and $\tau_v$ are periodic in $P$ with fundamental domain an $r\times r$ parallelogram.
  If $ v\cdot n_P =  u \cdot  n_P$, then there is an even integer $k>0$ (depending on $r$ and $P$) such that $\tau_v$ restricted to the left half-space of $P$ can be connected to $\tau_u$ restricted to the right half-space of $P_k=P+k \,n_P$ for some $k=O(r)$. Further, the connecting tiling $\tau$ is also periodic in $P$ with fundamental domain an $r\times r$ parallelogram.
 
\end{lemma}
\begin{rem}
    We restrict to these two types of planes $P$ since the the faces of tetrahedra in the mesh used to define the piecewise constant approximation (Proposition \ref{prop:pc_approx}) are contained in one of these two types of planes; see Remark \ref{rem: why tetrahedra}. The analogous result for other planes also holds, but we do not need it. 
\end{rem}
\begin{proof}
    By Lemma \ref{lem:tau_v_props}, $\tau_u,\tau_v$ are periodic, so there exists $r>0$ finite and determined by $\tau_u,\tau_v,P$ such that $\tau_u$ is periodic on $P$ with fundamental domain $R_0\subset P$, where $R_0$ an $r\times r$ parallelogram contained in $P$, and similarly $\tau_v$ is periodic on $P_k$ with fundamental domain also an $r\times r$ parallelogram in $P_k$. Let $R$ be the parallelopiped region parallel to $n_P$ between one fundamental domain $R_0\subset P$ and another $R_k\subset P_k$.

    Let $R/\sim$ be $R$ with opposite faces other than $R_0$ and $R_k$ paired (i.e., $R/\sim$ is a 2-dimensional torus crossed with an interval). Given the periodicity of $\tau_u,\tau_v$, to that show the region between $P$ and $P_k$ is tileable with $\tau_u\mid_{P}$ and $\tau_v\mid_{P_k}$, it suffices to show that $R/\sim$ is tileable with $\tau_u\mid_{R_0}$ and $\tau_{v}\mid_{R_k}$. 

    To show that this region is tileable we use the same techniques as in Section \ref{sec:patching plus hall}. In other words, first we show that $R/\sim$ with $\tau_u\mid_{R_0}$ and $\tau_{v}\mid_{R_k}$ is balanced, and then use Hall's matching theorem (\ref{thm:halls matching 2}). The setting here is more elementary than what we consider in Section \ref{sec:patching plus hall}, since here the tilings defining the boundary conditions are completely periodic. 

    Since $k$ is even, any perpendicular slice of $R/\sim$ is a fundamental domain for $\tau_u$ or $\tau_v$, the condition $v\cdot n_P = u\cdot n_P$ is equivalent to
    \begin{equation}\label{eq:total_flow}
    \sum_{e \text{ intersecting }R_0} v_{\tau_u}(e) \cdot n_P = \sum_{e \text{ intersecting }R_k} v_{\tau_v}(e) \cdot n_P.
\end{equation}
    Equation \eqref{eq:total_flow} is in turn equivalent to $R/\sim$ with boundary conditions $\tau_u\mid_{R_0}$ and $\tau_{v}\mid_{R_k}$ being balanced. 

    Since the region is balanced, by Hall's matching theorem (Theorem \ref{thm:halls matching 2}) it is not tileable if and only if there exists a counterexample region $U$ which is a strict subset. Since $U$ is a strict subset, $U$ has a nonempty interior boundary $S\subset \partial U$ (i.e., $S$ is the part of $\partial U$ which is strictly between $R_0$ and $R_k$, see Section~\ref{sec: minimal} for more details). Let $T =\partial U\setminus S$. By Proposition \ref{prop: white black imbalance},
    \begin{align*}
        \text{imbalance}(U) = \frac{1}{6}\bigg(\text{white}(T) -\text{black}(T) - \text{area}(S)\bigg). 
    \end{align*}
    Since $\text{area}(T)\leq 2r^2$, $\text{white}(T)\leq 2r^2$. Since the region is tileable with boundary condition from just one of the tilings, $S$ must connect $R_0$ and $R_k$, if $k> r$, by Proposition \ref{prop:quadratic growth}, there is a universal constant $\kappa$ such that $\text{area}(S)\geq \kappa k r$. Therefore by Proposition \ref{prop: white black imbalance},
    \begin{align*}
        \text{imbalance}(U) \leq \frac{2r^2 - \kappa k r}{6}.
    \end{align*}
    Choosing $k= cr$ for some constant $c>2/\kappa$, $U$ is not a counterexample which contradicts the assumption that the region is not tileable. This completes the proof. 
\end{proof}

Using the tilings $\tau_v$ as our building blocks and their gluing properties to put them together, we now proceed to prove Theorem \ref{thm:shininglight}.

\begin{proof}[Proof of Theorem \ref{thm:shininglight}]

Choose a scale $\epsilon>0$ so that the piecewise-constant approximation $\widetilde g$ from Proposition \ref{prop:pc_approx} on an $\epsilon$-scale tetrahedral mesh $\mc X = \{X_1,...,X_M\}$ satisfies $d_W(g,\widetilde g) < \delta/2$ and hence $A_{\delta/2}(\widetilde g) \subset A_\delta(g)$. We assume that all $X\in \mc X$ are regular or right-angled tetrahedra so that all their faces are contained in coordinate planes or planes with normal vector $(\pm 1, \pm 1, \pm 1)$. Recall that $\widetilde g\in AF(\widetilde {R})$ and that $R\subset \widetilde R$, so any free boundary tiling of $\widetilde R$ can be restricted to a free boundary tiling of $R$. To prove the theorem, it suffices to construct $\tau \in T_n(\widetilde R)$ with $d_W(f_\tau, \widetilde g) <\delta/2$.
	
\textit{Constructing channels.} We construct \textit{channels} $C_1,...,C_K$ which are disjoint, partition $\widetilde R$ and will be nicely chosen tubular neighborhoods of a modification of the flow lines of $\widetilde g + \eta_1$. For any $X_j\in \mc X$, let $v_j := \widetilde g \mid_{X_j}$. Recall Definition \ref{def:relationship between v and w}, which relates a vector $v$ with $w(v)$, which is the direction of the infinite paths in a periodic tiling $\tau_v$. For each channel $C_i$, the intersection $C_i\cap X_j$ will be a tube parallel to $$w(v_j)+\eta_1.$$
Since $\widetilde g$ is valued strictly in $\text{Int}(\mc O)$, $v_j\neq -\eta_1$ for all $X_j\in \mc X$, and hence $w(v_j)$ is well-defined everywhere. As shorthand, we let $w(\widetilde g)$ be the piecewise-constant flow equal to $w(v_j)$ on $X_j$. The definitions are made so that if $\tau_{v_j}$ is a periodic tiling built by the construction earlier in this section, the infinite paths in $(\tau_{v_j},\tau_1)$ move parallel to the direction of the channel on $X_j$. The values of $\widetilde g$ change on the boundaries $\partial X_j$ of tetrahedra in the mesh. We choose the channels $C_i$ to be thin enough as follows so that, viewing $C_i$ as a sequence of open tubes, each end of the tube $C_i\cap X_j$ is contained in a single plane (i.e., each end is contained in a single face of $\partial X_j$).

Project the corners and edges of $X_j$ onto $\partial X_j$ along $w(v_j)+\eta_1$. Call these projections $\gamma_j$. The set $\gamma_j\subset \partial X_j$ are the ones along a flow line of $w(\widetilde g) + \eta_1$ that goes through an edge of $X_j$. In fact, since the projection is linear, $\gamma_j$ is a collection of lines which divide the faces of $X_j$ into between $1$ and $3$ sections. 

We further divide $X_j$ by taking into account the iterated projections of the corners and edges of all the other tetrahedra in the mesh. In other words for all $j$, if $X_j,X_k$ share a face $F=X_k\cap X_j$, then we project $\gamma_j\cap F$ onto $\partial X_k$ by orthogonal projection along $w(v_k)+\eta_1$.  We iterate this for all tetrahedra until there is a projection of $\gamma_j$ on $\partial X_\ell$ for all $\{\ell,j\}$ pairs. See the left figure in Figure \ref{fig: tetra and channel}.

The result is that for each $j\in \{1,...,M\}$, each triangular face of $\partial X_j$ is partitioned into between $1$ and $3^M$ pieces, all with piecewise-linear boundaries, and $X_j$ is partitioned into tubes parallel to $w(v_j)+ \eta_1$ with these pieces as their ends. Each sequence of successive tubes glued on their intersections with $\partial X_j$ is a \textit{channel}\termindex{Chapter 8!channel}. The collection of channels $C_1,...,C_K$ is pairwise disjoint and covers $\widetilde R$. Since $w(\widetilde g)+\eta_1$ has positive first coordinate everywhere, each channel connects a patch on $\partial \widetilde R$ to another. 

\textit{Tiling a channel.} Fix $n$ large and a choice of channel $C$. We construct a scale $n$ tiling of $C$ which has only $-\eta_1$ tiles in a neighborhood of $\partial C$ of constant-order width in $n$, and use this to say that we can put together the tilings of the channels together to construct one tiling of the whole region. 

Let $T_1,...,T_m$ be the sequence of tubes of the form $X_j\cap C$ in order from one intersection of $C$ with $\partial \widetilde R$ to the other. Let $v_1,...,v_m$ be the corresponding values of $\widetilde g$ on the tubes.  Consider the $\mathbb{Z}^3$ tilings $\tau_{v_1},...,\tau_{v_m}$ constructed earlier in this section using the reference tiling $\tau_1$. Recall that for each $v_i$, all tiles in $(\tau_{v_i},\tau_1)$ are either double tiles (which must be $-\eta_1$ tiles) or are on an infinite path, and that all infinite paths follow the same periodic sequence. Let $r_j$ be the period of $\tau_{v_j}$, i.e.\ if $\tau(x)$ denotes the tile at $x$ in $\tau$, then for $j=1,...,m$, $r_j$ is such that translates $\tau_{v_j}(x+r_j \eta) = \tau_{v_j}(x)$ for all unit vectors $\eta$.

The main operation we will use is that for any infinite path $\ell\subset (\tau_{v_j},\tau_1)$, we can modify $\tau_{v_j}$ by ``shifting" all the tiles along $\ell$, i.e.\ by removing all the tiles on $\tau_{v_j}$ along $\ell$ and replacing them with $-\eta_1$ tiles. We refer to this as \textit{deleting} the path $\ell$. The idea is to delete paths that would exit the channel before hitting $\partial \widetilde R$, and then to bound the number of paths that we delete.

Let $\pi_j\subset \partial X_j$ be the starting end of $T_j$, so that $T_j$ is a tube connecting $\pi_j$ to $\pi_{j+1}$. For each $j$, we start by restricting $\tau_{v_j}$ to $T_j$. Any infinite path $\ell\subset (\tau_{v_i},\tau_1)$ which enters $T_j$ in $\pi_j$ must exit through $\partial T_j\setminus \pi_{j}$, since paths in $(\tau_{v_j},\tau_1)$ always have direction with positive $\eta_1$ component. 

\begin{figure}
    \centering
	\includegraphics[scale=0.4]{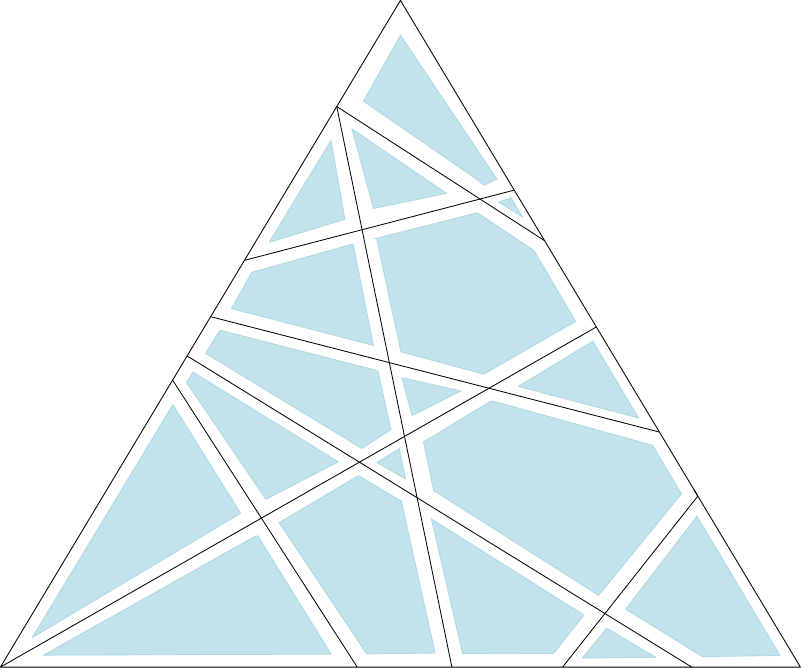}$\qquad\qquad$\includegraphics[scale=1.3]{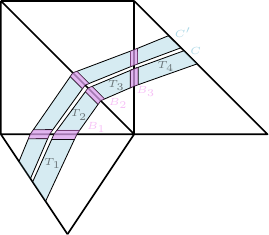}\\
	\caption{On the left is a face of one tetrahedron. The segments are the ends of channels, the smaller blue regions are places where we do not delete infinite paths. The width of the white area is $O(r/n)$. The figure on the right is a 2D schematic showing two channels $C, C'$, with the tubes $T_j$ and connector regions $B_j$ labeled along $C$. The width of the white area between $C$ and $C'$ is $O(r/n$).}
    \label{fig: tetra and channel}
\end{figure}

First, we delete all infinite paths $\ell\subset (\tau_{v_j},\tau_1)$ which do not enter $T_j$ through $\pi_j$ and exit for the first time through $\pi_{j+1}$, replacing the tiles of $\tau_{v_j}$ along these paths with $-\eta_1$ tiles. Note that this includes deleting all infinite paths which do not intersect $T_j$.

By Lemma \ref{lem:tau_v_props}, the infinite paths in $(\tau_{v_j},\tau_1)$ have asymptotic direction $w(v_j)$, with oscillation bounded by the length of the periodic sequence used to construct $\tau_{v_j}$, which is $O(r_j)$. Since the direction of the tube $T_j$ is also $w(v_j)$, any infinite path in $(\tau_{v_j},\tau_1)$ that enters $T_j$ through $\pi_j$ and exits through $T_j\setminus \pi_{j+1}$ is within $O(r_j)$ distance counted as number of edges in $\frac{1}{n} \m Z^3$ of $\partial C$ along $T_j$. Similarly, any path which enters $T_j$ through $\partial T_j\setminus \pi_j$ would also remain within $O(r_j)$ distance in number of $\frac{1}{n} \m Z^3$ edges of $\partial C$ along $T_j$. Let $r=\text{lcm}_j r_j$, then in summary, the paths that we delete which intersect $T_j$ are all contained in an neighborhood of $\partial C\cap T_j$ of width $O(r)$ in edge distance in $\frac{1}{n} \m Z^3$, corresponding to a neighborhood of Euclidean width $O(r/n)$ (recall that $r$ is a constant independent of $n$). 

Second (to avoid issues with corners and edges of tetrahedra, and to isolate channels from each other), we delete all infinite paths which are within a Euclidean neighborhood of width $1000/n$ of $\partial T_j\setminus \{\pi_j\cup \pi_{j+1}\}$ (i.e., $1000$ lattice cubes in $\frac{1}{n} \m Z^3$). By the same logic as above, these are still contained in an $O(r/n)$-width neighborhood of $\partial C$. We call this tiling $\tau_{v_j}'$.

Third, we want to glue the tiling on $T_j$ to the tiling on $T_{j+1}$. To do this, we cut out a neighborhood of width $O(r)$ in $\frac{1}{n}\m Z^3$ lattice cubes around $\pi_{j+1}$ (the face shared by $T_j$ and $T_{j+1}$) which we call the \textit{connector region} $B_j$. Let $\alpha_j, \alpha_{j+1}$ be the ends of the connector region (i.e.\ translates of $\pi_{j+1}$), see the purple region in right side figure in Figure \ref{fig: tetra and channel}. Let $P_j$ be the plane containing $\alpha_j$ and $P_{j+1}$ be the plane containing $\alpha_{j+1}$. 

Since $\tau_{v_j}$ and $\tau_{v_{j+1}}$ are periodic tilings of periods $r_j, r_{j+1}$, and since $B_j$ has width $O(r)$, by Lemma \ref{lem:gluing} we can construct a tiling $\sigma_j$ of $\m Z^3$ such that it agrees with $\tau_{v_j}$ in the left half-space of $P_j$ and $\tau_{v_{j+1}}$ on the right half-space of $P_{j+1}$, and fills in the region in-between in a periodic way with period $O(r)$. We can do this for all $j=1,...,m$. 

Overlaying $(\sigma_j,\tau_1)$, we again get a collection of infinite paths and double tiles. First, we delete all infinite paths in $\sigma_j$ which were deleted to construct $\tau_{v_j}'$ and $\tau_{v_{j+1}}'$ from $\tau_{v_j}$ and $\tau_{v_{j+1}}$. Second, we delete any infinite paths which exit $C$ between $\alpha_j$ and $\alpha_{j+1}$ (i.e., any path which exits $C$ along the connector $B_j$). Since $B_j$ has length $O(r)$ and since $\sigma_j$ is periodic with period $O(r)$, again any infinite path which exits in $B_j$ is contained in an $O(r)$ neighborhood of $\partial C$ along $B_j$. 

Finally, we can glue together the tilings $\sigma_1,...,\sigma_m$ by going back and deleting any infinite path in $(\sigma_j,\tau_1)$ which connects to one which would have been deleted in $(\sigma_i,\tau_1)$ for all other $i\neq j$. Since the number of tubes $m$ is constant, in the end we have a tiling $\tau$ of $C$ where we have deleted infinite paths of $(\tau,\tau_1)$ in a neighborhood of width at most constant-order in $n$ (concretely $1000 + O(r)$, where $r$ is constant in $n$) in distance measuring in edges of $\frac{1}{n} \m Z^3$, corresponding to a neighborhood of Euclidean width of $O(r/n)$.

Since all channels $C$ are tiled so that they have only $-\eta_1$ tiles in a neighborhood of $\partial C$, we can put them together. Therefore we have constructed a tiling $\tau\in T_n(\widetilde R)$. 

\textit{Bounding the final distance.} To emphasize the dependence on $n$, let $\tau^n$ be the tiling at scale $n$ constructed above and let $\tau_1^n$ be the $-\eta_1$ brickwork pattern at scale $n$. On one hand, the total flow of $\widetilde g+ \eta_1$ over any $X_j\in \mc X$ is 
\begin{align*}
    \text{vol}(X_j) (v_j + \eta_1).
\end{align*}
We claim that the double dimer flow $f_{(\tau^n,\tau_1^n)}= f_{\tau^n}-f_{\tau_1^n}$ has the same total flow, up to an $O(n^{-1})$ error. To explain the order of error, recall that for a scale $n$ tiling flow, each $1/n^3$-volume lattice cube has flow of order $1/n^3$. The error comes from the region around the boundary of the channels where some infinite paths were deleted and replaced with $-\eta_1$ tiles. The number of lattice sites on the boundary of the channel is order $n^2$, and the region has width constant order in $n$ in lattice cubes from $\frac{1}{n} \m Z^3$, so the amount of deleted flow in this region has order $n^2/n^3 = 1/n$. Therefore there is a constant $K$ such that 
$$\bigg|  \text{vol}(X_j)(\widetilde g_j +\eta_1) -\sum_{e\in \frac{1}{n}\m Z^3,\, e\cap X_j\neq \emptyset} (f_{\tau^n} - f_{\tau_1^n})(e)\bigg| < K n^{-1}.$$
There is also a constant $C= C(\widetilde R)$ such that $M = C \epsilon^{-3}$ (recall that $M=|\mc X|$ is the number of tetrahedra in the mesh). By Lemma \ref{lem:covered} applied to the partition $X_1,...,X_M$ of $\widetilde R$, 
\begin{align*}
	d_W(f_{\tau^n}-f_{\tau_1^n}, \widetilde g +  \eta_1) < 3M ( 10 \epsilon^{4} + K n^{-1}) < 30 C \epsilon + 3 C K \epsilon^{-3} n^{-1}.
\end{align*}
Taking $n$ large enough so that $1/n<\epsilon^4$, this becomes a bound which is a constant times $\epsilon$. A few applications of the triangle inequality and the ``mass shift" property of Wasserstein distance, i.e.\ that $\m W_1^{1,1}(\mu,\nu) = \m W_1^{1,1}(\mu+\rho,\nu+\rho)$ (see Lemma \ref{lemma: wasserstein mass shift}), give that 
\begin{align*}
    d_W(f_{\tau^n}, \widetilde g) < d_W(f_{\tau^n}- f_{\tau_1^n},\widetilde g+\eta_1) + d_W(f_{\tau_1^n},-\eta_1).
\end{align*}
Since $d_W(f_{\tau_1^n},-\eta_1)\to 0$ as $n\to \infty$, we can make this as small as needed as $n\to \infty$. Therefore we can choose $\epsilon$ small enough and $n$ large enough given $\delta$ so that $\tau^n\in T_n(\widetilde R)$ has $d_W(f_{\tau^n},\widetilde g)<\delta/2$. Restricting $\tau^n$ to $R$ completes the proof. 
\end{proof}

\subsection{Soft boundary lower bound}\label{sec:lower}

With the machinery developed in the previous section we can now prove the \textit{soft boundary lower bound}, namely Theorem \ref{thm:lower}. In particular we show that for $(R,b)$ with $b$ extendable outside, then for any $g\in AF(R, b)$,
\begin{align*}
    \lim_{\delta\to 0}\liminf_{n\to \infty} v_n^{-1} \log \mu_n(A_{\delta}( g)) \geq \Ent( g).
\end{align*}

Recall that $\mu_n$ is counting measure on $TF_n(R, b, \theta_n)$, the set of free boundary tiling flows on $R$ at scale $n$ with boundary values within $\theta_n$ of $b$. The main idea of the proof is to show that from the one free boundary tiling flow $f_\tau\in A_\delta(g)\cap TF_n(R)$ constructed in previous section (Theorem \ref{thm:shininglight}), we can use the patching theorem (Theorem \ref{patching}) to show that there are actually many tiling flows in $A_\delta(g)$.

\begin{proof}[Proof of Theorem \ref{thm:lower}]

By Proposition \ref{prop:pc_approx}, there exists $\delta_1>0$ such that there is a $\delta_1$-tetrahedral mesh $\mc X=\{X_1,...,X_M\}$, region $\widetilde R= \cup_{X\in \mc X} X$ containing $R$, and an asymptotic flow $\widetilde g\in AF(\widetilde R)$ taking constant values on tetrahedra in $\mc X$ with $d_W(g,\widetilde g)<\delta/2$ so that 
\begin{align*}
    A_{\delta/2}(\widetilde g) \subset A_\delta(g). 
\end{align*}
Let $\widetilde g_i:=\widetilde g\mid_{X_i}$. Computing directly, 
\begin{align*}
    \Ent(\widetilde g) = \frac{1}{\Vol(R)} \sum_{i=1}^M \Vol(X_i) \ent(\widetilde g_i). 
\end{align*}
On the other hand by Proposition \ref{proposition: Upper semicontinuous}, 
\begin{align*}
    \Ent(\widetilde g) = \Ent(g) + o_{\delta}(1).
\end{align*}
Using the shining light construction from the proof of Theorem \ref{thm:shininglight}, for any $n$ large enough there exists a tiling $\tau \in T_n(R)$ such that $f_\tau\in A_{\delta/2}(\widetilde g)$ has a particular form. Let $C_1,...,C_K$ denote the channels in the shining light construction. For each tetrahedron $X$ and channel $C$ that intersect, $X\cap C$ is a tube. As in the proof of Theorem \ref{thm:shininglight}, $\tau\mid_{X\cap C}$ is periodic at a scale independent of $n$, and has mean current $\widetilde g_X \in \text{Int}(\mc O)$ on $X\cap C$ outside a neighborhood of $\partial(X\cap C)$ of width constant order in $n$. 

We choose $\epsilon \ll \delta_1$, and partition the interior of $X\cap C$ (where $\tau$ has mean current in $\text{Int}(\mc O)$) into small cubes with side length $\leq \epsilon$. For each $i=\{1,...,M\}$, call the pieces of the partition contained in $X_i$
$$\{Q_1^{i},...Q_{k_i}^i\}_{i=1}^M.$$
For any $(i,k)$ pair, $\tau\mid_{Q_k^i}$ is periodic. Recall that $Q_k^i$ has diameter $<\epsilon$. For $\epsilon_1\ll \epsilon$, for $n$ large enough $\tau\mid_{\partial Q_k^i}$ is $\epsilon_1$-nearly-constant with value $\widetilde g_i$ (see Definition \ref{def:nearly constant}), so it satisfies the conditions for the outer boundary condition in the patching theorem (Theorem \ref{patching}). Fix $c\in (0,1)$. For each $(i,k)$ pair, we choose an EGM $\mu_{i,k}$ of mean current $\widetilde g_i$ (these exist by Corollary \ref{cor: EGMs exist!}). Since $\widetilde g_i\in \text{Int}(\mc O)$, a sample from $\mu_{i,k}$ satisfies the conditions of Theorem \ref{patching} with probability going to $1$ as $n\to \infty$ (Corollary \ref{cor: ergodic implies nearly constant}). Therefore by Theorem \ref{patching}, for $n$ large enough, with probability $(1-c)$, $\tau$ restricted to $\partial Q_k^i$ can be patched with a sample $\sigma$ from $\mu_{i,k}$ on an annulus of width $cn$. By the ergodic theorem, given any $\epsilon_2>0$, for $n$ large enough we can assume that \begin{equation}\label{eq:wass bound ergodic theorem}
d_W(f_\sigma\mid_{Q_k^i},\widetilde g_i\mid_{Q_k^i})<\epsilon_2
\end{equation}
with probability $1-c$. Therefore with probability $1-2c$, Equation \eqref{eq:wass bound ergodic theorem} holds and $\sigma$ can be patched with $\tau$.

Let $\pi_{i,k,n}$ denote uniform measure on tilings $\sigma$ of $Q_k^i$ at scale $n$ with $\sigma\mid_{\partial Q_k^i} = \tau$ and satisfying Equation \eqref{eq:wass bound ergodic theorem}, and let $Z_n(Q_k^i)$ be its partition function. The additional constraint that Equation \eqref{eq:wass bound ergodic theorem} is satisfied does not change the exponential order of the number of tilings, hence by Proposition \ref{thm:epsilon_nearly_constant_vs_EGM}, we get the following consequences for entropy:
\begin{align*}
    (1+O(c)) h(\mu_{i,k}) \leq n^{-3} \Vol(Q_{k}^i)^{-1} H(\pi_{i,k,n}) = n^{-3}\Vol(Q_k^i)^{-1} \log Z_n(Q_k^i). 
\end{align*}
By Lemma \ref{lem:covered} applied to the partition $\{Q_1^i,....,Q_{k_i}^i\}_{i=1}^M\cup \{R\setminus \cup_{i=1}^M \cup_{k=1}^{k_i} Q_k^i\}$, if $\sigma\in T_n(R)$ is a free boundary tiling of $R$ whose restrictions to each $Q_{k}^i$ are in the support of $\pi_{i,k,n}$ then using Equation \eqref{eq:wass bound ergodic theorem},
$$d_W( f_\tau,  f_\sigma) < 3 \epsilon^{-3}(10 \epsilon^4 + \epsilon_2) + 3 C(\partial R) \epsilon$$
where $C(\partial R)$ is a constant depending only on $R$. In particular, choosing $\epsilon_2=\epsilon^4$ and taking $\epsilon$ sufficiently small, $f_\sigma \in A_{\delta/2}(\widetilde g)$. By Theorem \ref{thm: boundary_value_uniformly_continuous} (uniform continuity of the boundary value operator $T(\cdot,\partial R)$), we can choose $\epsilon$ small enough so that the boundary values $T(f_\sigma,\partial R)$ are within $\theta_n$ of $b$ for all $\sigma$ in the support of $\pi_{i,k,n}$. Therefore for $n$ large enough, 
\begin{align*}
    \mu_n(A_{\delta/2}(\widetilde g)) \geq \prod_{i=1}^M \prod_{k=1}^{k_i} Z_n(Q_k^i) \geq \prod_{i=1}^M \prod_{k=1}^{k_i} \exp\bigg( n^3 \Vol(Q_k^i) h(\mu_{i,k}) (1+O(c))\bigg).
\end{align*}
Recall that $v_n^{-1} = n^{-3} \Vol(R)^{-1}$. Since $\mu_{i,k}$ is an EGM of mean current $\widetilde g_i\in \text{Int}(\mc O)$, $h(\mu_{i,k}) = \ent(\widetilde g_i)$ by Theorem \ref{theorem: existence of gibbs}. Rearranging and taking into account the $O(\epsilon)$ proportion of each tetrahedron $X\in \mc X$ that is not included in the patched regions, 
\begin{align*}
    v_n^{-1}\log \mu_n(A_{\delta/2}(\widetilde g)) &\geq \sum_{i=1}^M \sum_{k=1}^{k_i} \frac{\Vol(Q_k^i)}{\Vol(R)} \ent(\widetilde g_i) (1+O(c)) \\ &\geq \sum_{i=1}^M \frac{\Vol(X_i)}{\Vol(R)} \ent(\widetilde g_i) (1+ O(c))(1-O(\epsilon))\\
    &= \Ent(\widetilde g) (1+ O(c))(1-O(\epsilon)). 
\end{align*}
Recall that $\epsilon$ is the size of the patched regions, and $c<\epsilon$ is the patching error. All of $\epsilon,c,\epsilon_2$ are much smaller than $\delta>0$, and go to $0$ as $\delta\to 0$. In particular for any fixed $\delta>0$, and $\epsilon>0$ fixed small enough given $\delta>0$, there is a $n(\delta)$ such that for all $n>n(\delta)$, 
\begin{align*}
    \liminf_{n\to \infty} v_{n}^{-1}\log \mu_n(A_{\delta}(g)) \geq \liminf_{n\to \infty} v_{n}^{-1}\log \mu_n(A_{\delta/2}(\widetilde g)) \geq  \Ent(\widetilde g) + O(\epsilon) = \Ent(g) + o_\delta(1).
\end{align*}
Taking $\delta\to 0$ completes the proof.
\end{proof}

\subsection{Generalized patching and hard boundary lower bound}\label{sec:generalized_patching}

To prove the hard boundary lower bound (Theorem \ref{thm:lower-hb}), we need one more tool. Recall that $\overline{\rho}_n$ is the uniform probability measure on tilings of a fixed region $R_n\subset \frac{1}{n} \m Z^3$.

The shining light construction (Theorem \ref{thm:shininglight}) shows that for any $\delta>0$ and $g\in AF(R,b)$, for $n$ large enough there exists a free-boundary tiling $\tau\in T_n(R)$ such that $d_W(f_\tau,g)<\delta$. For hard boundary conditions, we need to know that every $g\in AF(R,b)$ can be approximated by a tiling of the \textit{fixed} region $R_n$. To do this, we prove a \textit{generalized patching theorem} (Theorem \ref{thm:generalized_patching}).

Let $B_n = [-n,n]^3$. Recall that the patching theorem (Theorem \ref{patching}) says that if two tilings $\tau_1,\tau_2$ are \textit{nearly constant} with value $s\in \text{Int}(\mc O)$ (Definition \ref{def:nearly constant}), then for any $c>0$ there is $n$ large enough that we can patch together $\tau_2\mid_{B_{(1-c)n}}$ and $\tau_1\mid_{\m Z^3\setminus B_n}$ by tiling the width-$c n$ annulus between them. 

In Section \ref{sec:patching plus hall} where the regular patching theorem (Theorem \ref{patching}) was proved, all our tools were combinatorial, and we thought of tilings $\tau$ of $\m Z^3$ without rescaling to $\frac{1}{n} \m Z^3$. Here we look at the tileability of more general ``annular regions," where the tilings are rescaled to live in $\frac{1}{n} \m Z^3$. 

Let $R\subset \m R^3$ be a compact set which is the closure of a connected domain and has piecewise-smooth boundary $\partial R$ (i.e., the sort of region to which our LDP applies). For any small $c>0$, we define 
\begin{align*}
    R^c = \{x\in R: d(x,\partial R)\geq c\}.
\end{align*}
The set $R\setminus R^c$ is an annular region. On the discrete side, given a free-boundary tiling $\tau\in T_n(R)$, it restricts to $\tau'\in T_n(R^c)$ which is a free boundary tiling of $R^c$. We let $R_n^c\subset \frac{1}{n} \m Z^3$ be the region covered by $\tau'$. Given another free-boundary tiling $\sigma\in T_n(R)$, let $R_n\subset \frac{1}{n}\m Z^3$ be the region covered by $\sigma$. The region $A= R_n\setminus R_n^c\subset \frac{1}{n} \m Z^3$ is the type of annular region we study here. 

Let $(R,b)$ be flexible and suppose that $R_n\subset \frac{1}{n} \m Z^3$ is a sequence of regions satisfying the conditions of the hard boundary LDP (Theorem \ref{thm:hb-ldp}), i.e., tileable regions with boundary conditions converging to $b$. To prove any $g\in AF(R,b)$ can be approximated by $f_\tau$ with $\tau$ a tiling of $R_n$, we show that we can \textit{patch together} suitable tilings on annuli of the form $R_n\setminus R_n^c$, with hard boundary condition on the outside.

\begin{definition}
    We say that a flow $g\in AF(R,b)$ is \textit{flexible} if $g$ satisfies the condition that for any compact set $D\subset \text{Int}(R)$, $\overline{g(D)}\subset \text{Int}(\mc O)$. 
\end{definition}
The pair $(R,b)$ is \textit{flexible} (see Definition \ref{def:fully_flexible} and Lemma \ref{lem:fully_flexible}) if and only if there exists $g\in AF(R,b)$ which is flexible. 

\begin{thm}[Generalized patching theorem]\label{thm:generalized_patching}
    Fix $c>0$. Let $(R,b)$ be flexible, with $b$ a boundary asymptotic flow which is extendable outside. Let $R_n\subset \frac{1}{n} \m Z^3$ be a sequence of tileable regions approximating $R$ in Hausdorff distance and with boundary values on $\partial R$ converging to $b$ in $\m W_1^{1,1}$. Let $\sigma_n$ be a sequence of tilings of $R_n$. 

    Let $\tau_n\in T_n(R)$ be a sequence of tilings such that $d_W(g,f_{\tau_n}) \to 0$ as $n\to \infty$ for $g\in AF(R,b)$ flexible. Let $\tau_n'$ be $\tau_n$ restricted to a free boundary tiling of $R^c$, and let $R_n^c\subset \frac{1}{n} \m Z^3$ be the cubes covered by $\tau_n'$. 
    
    For $n$ large enough, $R_n\setminus R_n^c$ is tileable. 
\end{thm}
\begin{rem}
   The \textit{flexible} condition here is analogous to the $s\in \text{Int}(\mc O)$ condition in the original patching theorem (Theorem \ref{patching}). The generalized patching theorem is the reason the hard boundary LDP requires that $(R,b)$ is flexible. 
\end{rem}

Before we prove this, we explain how it can be used to prove the hard boundary lower bound (Theorem \ref{thm:lower-hb}). First, from the generalized patching theorem, it is straightforward to prove the fixed boundary analog of Theorem \ref{thm:shininglight}. 
\begin{cor}\label{cor:fixed bc}
    Suppose that $(R,b)$ is flexible and $b$ is a boundary asymptotic flow which is extendable outside. Fix a sequence of tileable regions $R_n$ approximating $R$ in Hausdorff distance and with boundary values on $\partial R$ converging to $b$ in $\m W_1^{1,1}$. For any $\delta>0$ and any $g\in AF(R,b)$, there is $n$ large enough such that there exists a tiling $\tau$ of $R_n$ with $d_W(f_{\tau},g) < \delta$.
\end{cor}
\begin{proof}
    Since $(R,b)$ is flexible, there exists $g_0\in AF(R,b)$ which is flexible. For any $\epsilon > 0$, the new flow $g_\epsilon = \epsilon g_0 + (1-\epsilon) g$ satisfies $d_W(g,g_\epsilon) < C \epsilon$ for some constant $C>0$. Taking $\epsilon$ small enough, we can guarantee that $d_W(g,g_\epsilon)< \delta/2$. 

    Since $g_\epsilon$ is flexible and $R_n$ is tileable, by Theorem \ref{thm:generalized_patching} for $n$ large enough there exists a tiling $\tau$ of $R_n$ such that $d_W(g_\epsilon, f_\tau)<\delta/2$. By the triangle inequality $d_W(g,f_\tau)< \delta$, which completes the proof.
\end{proof}
Adding Corollary \ref{cor:fixed bc} as the first step, the proof of Theorem \ref{thm:lower-hb} (hard boundary lower bound) is the same as the proof of Theorem \ref{thm:lower} (soft boundary lower bound).
\begin{proof}[Proof of Theorem \ref{thm:lower-hb}]
    By Corollary \ref{cor:fixed bc}, given any $\delta >0$, for $n$ large enough and any $g\in AF(R,b)$ we can find a tiling $\tau$ of $R_n$ such that $d_W(f_\tau,g)<\delta$. Further, this tiling is of the form given in the shining light construction (proof of Theorem \ref{thm:shininglight}), other than in an annulus of width $c\in(0,1)$ where $c$ can be taken arbitrarily small. The remainder of the proof is identical to the proof of Theorem \ref{thm:lower}, where we use the regular patching theorem to patch in samples from ergodic Gibbs measures of appropriate mean currents. 
\end{proof}

\begin{figure}
    \centering
    \includegraphics[scale=0.6]{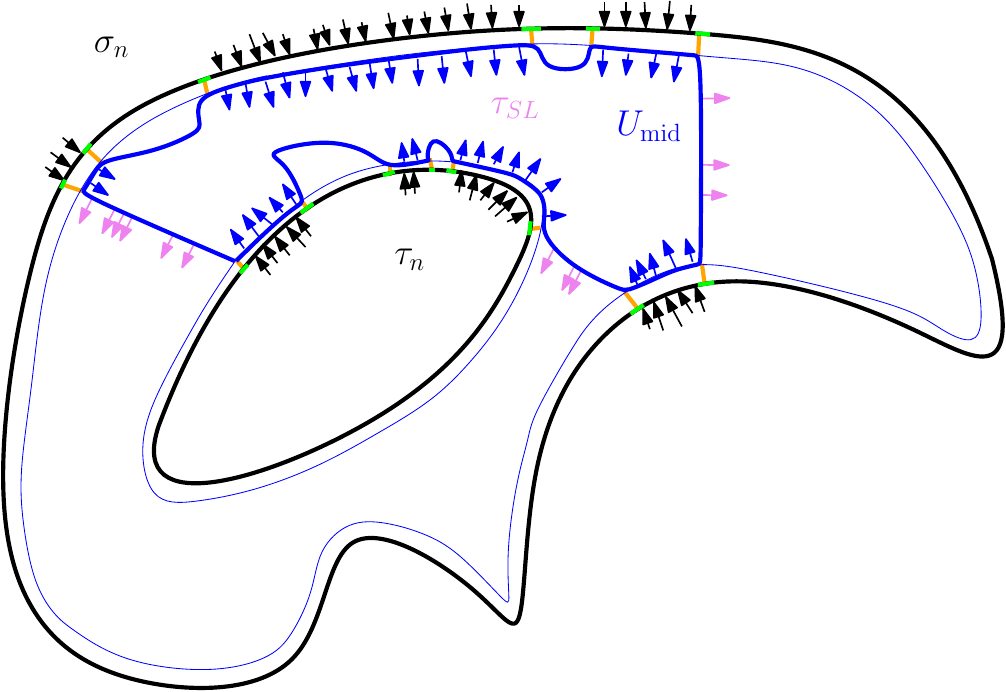}
    \caption{2D schematic picture for the proof of the generalized patching theorem. See further explanation in the proof sketch below.}
    \label{fig:generalized_patching}
\end{figure}
It remains to prove the generalized patching theorem (Theorem \ref{thm:generalized_patching}). The proof is structurally analogous to the proof of the regular patching theorem for cubes $B_n = [-n,n]^3$ (Theorem \ref{patching}), and relies on a sequence of lemmas. We give an outline of the main ideas to explain where each of the lemmas is used, accompanied by the schematic picture in Figure \ref{fig:generalized_patching}. We then state and prove each of the lemmas, followed by a proof of Theorem \ref{thm:generalized_patching}. 

Note that there are some superficial changes between the results here and their analogs in Section \ref{sec:patching plus hall}, since here our regions $R_n\subset \frac{1}{n} \m Z^3$ instead of $B_n\subset \m Z^3$. Basically this corresponds to a change in units. We introduce a few new pieces of notation to make it easier to work with the tilings $\tau$ of $\frac{1}{n} \m Z^3$ instead of $\m Z^3$. 
\begin{itemize}
    \item If $Q$ is a discrete surface built out of lattice squares in $\frac{1}{n}\m Z^3$, we define $\text{area}_n(Q)$ to be the number of lattice squares on $Q$. This is $n^2$ times the Euclidean area of $Q$. If $Q$ is a surface not built of lattice squares, we can still use $\text{area}_n(Q)$ to mean the Euclidean area of $Q$ times $n^2$.
    \item If $\ell$ is a discrete curve built out of edges of squares in $\frac{1}{n} \m Z^3$, we defined $\text{length}_n(\ell)$ to be number of lattice edges in $L$. This is $n$ times the Euclidean length of $\ell$. 
    \item If $\tau$ is a tiling of $\frac{1}{n}\m Z^3$, then the tiling flow $f_\tau$ and pretiling flow $v_\tau$ are typically \textit{rescaled} so that for $e\in \frac{1}{n}\m Z^3$, $v_\tau(e) = \pm 1/n^3$ or $0$, and $f_\tau(e) = \pm5/6n^3$ or $\pm 1/6n^3$. We define $\widetilde v_\tau$ to be \textit{unrescaled flow} $\widetilde v_\tau(e) = \pm 1$ or $0$ for $e\in \frac{1}{n}\m Z^3$, and similarly $\widetilde f_\tau(e) = \pm 5/6$ or $\pm 1/6$. 
\end{itemize}
\textbf{Proof sketch.} The proof of Theorem \ref{thm:generalized_patching} uses a mixture of combinatorial results like in Section \ref{sec:patching plus hall} and more analytic results about Wasserstein distance, which are for rescaled tiling flows. These pieces of notation make it easier to go between these points of view, and to explain the purely combinatorial arguments in a way more analogous to Section \ref{sec:patching plus hall}.

The main combinatorial tool is again Hall's matching theorem (Theorem \ref{thm:halls matching 2}), which says that if $A =R_n\setminus R_n^c$ is not tileable, then there exists a counterexample set $U\subset A$ which proves it. The \textit{interior boundary} $S$ of $U$ is without loss of generality a minimal monochromatic discrete surface with number of squares from $\frac{1}{n} \m Z^3$ on the surface bounded above and below by a constant times $n^2$ (Lemma \ref{lem:general_annulusbound}). The first step is to \textit{indent} slightly and let $A_{\text{mid}}\subset A$ be a slightly smaller annulus where $U$ is well-behaved (Lemma \ref{lem:generalized_indenting}). We then define $U_{\text{mid}} = U\cap A_{\text{mid}}$. 

We find a \textit{test tiling} $\tau_{\text{SL}}$ using a shining light construction (Definition \ref{def:shininglight_measures}) and show that it satisfies a flow bound (Lemma \ref{lemma:shininglight_flowbound}). This is where we use the condition that $g$ is flexible. Using Lemma \ref{lemma:shininglight_flowbound}, we show that there is a constant $K \in (0,1)$ such that $\text{imbalance}(U_{\text{mid}}') \leq - K n^2 + O(n)$, where $U_{\text{mid}}'$ is $U_{\text{mid}}$ plus a few cubes from the rest of $U$ (determined by $\tau_{\text{SL}}$). This ``slack" corresponds to flow from $\tau_{\text{SL}}$ which exits through the boundary of $U_{\text{mid}}$ in the interior, see the pink arrows in Figure \ref{fig:generalized_patching}.

It remains to bound the imbalance in $U_{\text{shell}} := U\setminus U_{\text{mid}}'$. This we break into two pieces: 
\begin{itemize}
    \item Regions of $U_{\text{shell}}$ contained in nice ``cylinders" connecting $\partial A$ to $\partial A_{\text{mid}}$. These are the regions where there are black and blue arrows in Figure \ref{fig:generalized_patching}, the sides of the cylinder are the orange regions which we call the ``ribbon surface."
    \item The rest of $U_{\text{shell}}$, which we call the ``leftover region." 
\end{itemize}
Up to error related to the orange area in Figure \ref{fig:generalized_patching} (the ribbon surface), we show that the imbalance of $U_{\text{shell}}$ is the same as the imbalance in the cylinder regions. Finally we relate the imbalance in the cylinders to the flux of tiling flows $f_{\sigma_n}$, $f_{\tau_n}$, and $f_{\tau_{\text{SL}}}$, which we can bound using Wasserstein convergence considerations and Lemma \ref{lem:constant_order_box_Wass_bound}, up to an error proportional to the green area in Figure \ref{fig:generalized_patching}.

This completes the sketch, and we now proceed to the lemmas. We first note that Lemma \ref{annulusbound} for cubes has an analog for general regions. The only difference is that the constants $c_1,c_2$ change since they can depend on the regions. 
\begin{lemma}\label{lem:general_annulusbound}
    Let $R_n\subset \frac{1}{n} \m Z^3$ be a region of diameter in lattice squares at least $n$ such that $R_n$ are regions approximating a fixed region $R\subset \m R^3$ with $\partial R$ a piecewise smooth surface. Define $A = R_n\setminus R_n^c$ for some $c>0$. Suppose that $S\subset A$ is a monochromatic minimal discrete surface in $\frac{1}{n}\m Z^3$ with connects the inner and outer boundaries of $A$. Then there exist constant $c_1,c_2$ independent of $S$ and $n$ such that $$ c_1 n^2 \geq \text{area}_n(S) \geq c_2 n^2$$
    where $c_2 \sim c^2$.
\end{lemma}
\begin{proof}
This proof is completely analogous to the proof of Lemma \ref{annulusbound}.     

For the upper bound, we use that $S$ is minimal to get a bound which is a constant times the surface area of $\partial R_n$. Since $R_n\subset \frac{1}{n}\m Z^3$ are regions approximating $R$, and since $\partial R$ is piecewise smooth, $\text{area}_n(\partial R_n)$ is bounded by a constant times $n^2$, where this constant is determined by $\partial R$. 
        
For the lower bound, we apply Proposition \ref{prop:quadratic growth} to a point $p\in S$ which is distance at least $c/3$ from $\partial A$ to get that 
    \begin{align*}
        \text{area}_n(S) \geq \kappa ((c/3)\, n)^2 = c_2 n^2,
    \end{align*}
    where $\kappa$ is a universal constant coming from the isopertimetric inequality. From this expression we see that $c_2$ is order $c^2$.
\end{proof}

\begin{lemma}[Generalized indenting lemma]\label{lem:generalized_indenting}
Let $A = R_n \setminus R_{n}^c$, and fix $\beta>0$ small. Let $S$ be a minimal monochromatic surface connecting inner and outer boundaries of $A$. There exists $\epsilon<c$ independent of $S$ such that the following hold for $n$ large enough:
    \begin{enumerate}
        \item Let $A_{a,b}$ denote the annulus between layers $\partial R_n^a$ and $\partial R_n^b$. Then $\text{area}_n(S\cap A_{\epsilon,2\epsilon}) < \beta n^2$. 
        \item There exists $a\in (\epsilon, 2\epsilon)$ such that $\text{length}_n(\partial(\partial R_n^a \cap S)) \leq (\beta/\epsilon) n$. 
    \item A ``ribbon surface" $\gamma$ for $S\cap \partial R_n^a$ is a surface connecting $\partial R_n^a$ to $\partial R_n$ with boundary $\partial (S\cap \partial R_n^a)$ on $\partial R_n^a$. The ``{ribbon area}" is the minimal $\text{area}_n$ of a ribbon surface, and is bounded by $2\beta n^2$. 
    \end{enumerate}
    The analogous bounds hold for some $a'\in (c-2\epsilon, c -\epsilon)$.
\end{lemma}
\begin{rem}
    For two surfaces $A,B$, the set $A\cap B$ is either a surface (2-dimensional), a curve (1-dimensional), or a combination of the two. In any of these cases, we take $\partial (A\cap B)$ to mean that we take union of the curve part of $A\cap B$ and the boundary of the surface part of $A\cap B$. 
\end{rem}
\begin{proof}

By Lemma \ref{lem:general_annulusbound}, $c_2 n^2 \leq \text{area}_n(S)\leq c_1 n^2$. For a given $b>0$, divide a band of the form $A_{0,b}$ in half. Each time we divide, by Lemma \ref{lem:general_annulusbound} both the halves of $S$ have area bounded below by a fixed constant times $b^2$. On the other hand, one of the halves of $S$ can have at most $1/2$ the original area. Iterating this, we can find $\epsilon>0$ small enough so that the outer band after we split has area at most $\beta n^2$.

By the pigeonhole principle, there exists a layer $an$ between $\epsilon n$ and $2\epsilon n$ where 
\begin{align*}
    \text{length}_n(\partial(\partial R_n^a \cap S)) \leq (\beta n^2)/(\epsilon n)=(\beta/\epsilon) n.
\end{align*}
Given this, we can find a ribbon surface $\gamma$  with $\text{area}_n(\gamma)\leq (\beta/\epsilon)n (2\epsilon n) = 2\beta n^2$, and hence the ribbon area is bounded by $2 \beta n^2$. 
\end{proof}

Suppose that $g\in AF(R,b)$ is flexible. For any $\delta>0$, we can find a piecewise-constant flow $\widetilde g$ with $d_W(g,\widetilde g)<\delta$ (Proposition \ref{prop:pc_approx}). More precisely, $\widetilde g$ is piecewise-constant on a mesh of small tetrahedra $\mc X$. The region $\widetilde R = \cup_{X\in \mc X} X$ contains $R$, and $\widetilde g\in AF(\widetilde R)$. In the proof of Theorem \ref{thm:shininglight}, to construct a tiling approximation $\tau\in T_n(R)$ of $g$ for $n$ large, we construct a tiling approximation $\widetilde \tau \in T_n(\widetilde R)$ of $\widetilde g$ for $n$ large, and restrict it to $R$. 

The flexible condition passes from $g$ to $\widetilde g$ as follows. For any compact set $D\subset \text{Int}(R)$, $g$ flexible means that $\overline{g(D)}\subset \text{Int}(\mc O)$. In particular there is a constant $k_D\in[0,1)$ such that if $x\in D$ then $$|g(x)|_1\leq k_D < 1,$$ 
where $|\cdot|_1$ denotes the $L^1$ norm. When $g$ is flexible, we can choose $\widetilde g$ so that for any compact set $D\subset \text{Int}(R)$, for all tetrahedra $X\in \mc X$ such that $X\subset D$, $|\widetilde g_X|_1\leq k_D$. When this holds, we say that $\widetilde g$ is a \textit{piecewise constant approximation of $g$ inheriting the flexible condition.}

For any $\epsilon>0$ and $D\subset \text{Int}(R)$ compact, for $\delta>0$ small enough we can find $\widetilde g$ which is piecewise-constant on a scale $\delta$ tetrahedral mesh $\mc X$, has $d_W(g,\widetilde g)<\epsilon$, and inherits the flexible condition on $D$, i.e.\ $|\widetilde g_X|\leq k_D$ for all $X\in \mc X$ such that $X\subset D$. 

In the shining light construction in the proof of Theorem \ref{thm:shininglight}, we cut the tetrahedra into tubes (with thin space between them), and construct the tiling approximation $\widetilde \tau$ of $\widetilde g$ by filing the tubes in $X$ with periodic tilings of mean current approximating $\widetilde g_X$ on $X$, and fill the thin area in between with all $-\eta_1$ tiles. There is some maximum period $r$ that we use to construct the periodic tilings on the tubes, and $r$ is independent of $n$. For all $n$, at all the sites where $\widetilde \tau$ does not have mean current $\widetilde g$, $\widetilde \tau$ looks locally like the $-\eta_1$ brickwork pattern. The width of the region containing the places where $\widetilde \tau$ looks like the $-\eta_1$ brickwork pattern is $O(r)$ and therefore independent of $n$.  

\begin{definition}[Shining light measures]\label{def:shininglight_measures}
Let $g\in AF(R,b)$ and let $\widetilde g$ be an approximation as discussed above. For each $n$, let $\widetilde \tau_n\in T_n(\widetilde R)$ be a tiling produced by the shining light construction with $\widetilde g$, where the periodic tilings in the tubes have maximum period $r$. Fix a large constant $C$ which is a multiple of $r$. We define a \textit{sequence of shining light measures} $\lambda_n$ for $g$ using $\widetilde g$ so that for each $n$, $\lambda_n$ is uniform measure on tilings of the form $(\widetilde \tau_n + x)\mid_R\in T_n(R)$ for $x\in \frac{1}{n}\m Z^3$ with $|x| \leq C$. 
\end{definition}

We include averaging over translations in the definition of $\lambda_n$, since this allows to use flexible condition on $g$ to get bounds on the expected flux. Using this, we prove a lemma analogous to Lemma \ref{lem: flow bound} from Section \ref{sec:patching plus hall}. Instead of a result for ergodic measures of mean current $s\in \text{Int}(\mc O)$, this lemma is for a sequence of shining light measures $\lambda_n$ for a flexible flow $g$. This lemma is why the \textit{flexible} condition is needed for generalized patching and hence for the hard boundary LDP. 

\begin{lemma}[Shining light flow bound]\label{lemma:shininglight_flowbound}
    Let $D\subset \text{Int}(R)$, and let $g\in AF(R,b)$ be flexible. Let $\widetilde g$ be piecewise constant approximation of $g$ on a tetrahedral mesh $\mc X$ such that $\widetilde g$ inherits the flexible condition, and such that the union of tetrahedra $\mc D\subset \mc X$ covering $D$ is still contained in $\text{Int}(R)$.
    
     Let $S$ be a monochromatic black surface in $\frac{1}{n}\m Z^3$ with boundary $\partial S$ and $\lambda_n$ be a sequence of shining light measures as in Definition \ref{def:shininglight_measures} for $\widetilde g$. Let $\Theta_D$ be the collection of odd cubes adjacent to $S\cap D$. Let $N = \text{area}_n(S\cap D)$ be the number of squares from $\frac{1}{n} \m Z^3$ on $S\cap D$. Then there is a constant $K_D\in (0,1)$ independent of $S$ such that for all $n$ large enough,
     \begin{align*}
         \m E_{\lambda_n}[|\text{flux}(\widetilde v_\tau,S\cap D)|]\geq K_{D} |\Theta_{D}| +O(n^{-1}N) + O(\text{length}_n(\partial(S\cap D)).
     \end{align*}
     The constant $K_D$ depends only on $D$ and $g$. In particular, it is independent of $\widetilde g$, as long as $\widetilde g$ inherits the flexible condition and is constructed on a small enough mesh $\mc X$. 
\end{lemma}
\begin{rem}
    Recall that for $\tau$ a tiling in $\frac{1}{n}\m Z^3$, the flow $\widetilde v_\tau$ is the \textit{non-rescaled} pretiling flow. 
\end{rem}
\begin{proof}
    Restricted to any tetrahedron $X\in \mc X$, $\lambda_n$ samples a tiling which is periodic with mean current $\widetilde g_X$ up an $O(n^{-1})$ error. In particular, there are probabilities $p_1(X),....,p_6(X)$ such that the probability of seeing a tile of type $i$ in $\tau$ restricted to $X$ sampled from $\lambda_n$ is $p_i(X) +O(n^{-1})$ for $i=1,...,6$. Let $N_1,...,N_6$ denote the corresponding six types of squares $f$ on $S$. Let $N_i(X)$ denote the number of each type of square in $S$ restricted to $X$. 
    
    Recall that $\mc D\subset \mc X$ is the collection of mesh tetrahedra $X\in \mc X$ such that $X\cap D\neq \emptyset$. By assumption $\cup_{X\in \mc D} X\subset \text{Int}(R)$. Since $\widetilde g$ inherits the flexible condition, there is a constant $k_D\in [0,1)$ such that $|\widetilde g_X|_1\leq k_D$ for all $X\in \mc D$. This is related to the probabilities since 
    \begin{equation}\label{eq:prob_to_flow}
        \widetilde g_X = (p_1(X) - p_2(X), p_3(X)-p_4(X),p_5(X)-p_6(X)) 
    \end{equation}
    Since $S$ is monochromatic, the flux of $\widetilde v_\tau$ across $S$ is equal to minus the number of tiles in $\tau$ which cross $S$. Thus the expected value of the flux (up to sign) can be split as the sum of indicator functions $\mathbbm{1}_f$, where $f$ is a face of $S \cap D$, and $\mathbbm{1}_f(\tau)$ is $1$ if $f$ is crossed by a tile in $\tau$ and $0$ otherwise. Then
    \begin{align*}
        \m E_{\lambda_n}[|\text{flux}(\widetilde v_\tau,S\cap D)|] \geq \sum_{X\in \mc D} \sum_{f\in S\cap X\cap D} \m E_{\lambda_n}[\mathbbm{1}_f(\tau)] = \sum_{X\in \mc D} \sum_{i=1}^6 N_i(X) (p_i(X) + O(n^{-1})). 
    \end{align*}
    The total number of squares on the surface is $N = \sum_{X\in \mc D}\sum_{i=1}^6 N_i(X) = \text{area}_n(S\cap D)$. By Lemma \ref{lem: counting squares}, $N_1+ N_2, N_3+N_4, N_5+N_6$ are all equal to $N/3$ up to an error of $O(\text{length}_n(\partial(S\cap D))$. Let 
    \begin{align*}
        p_D = \min_{X\in \mc D} \max \{\min\{p_1(X),p_2(X)\},\min\{p_3(X),p_4(X)\},\min\{p_5(X),p_6(X)\}\}.
    \end{align*}
    Since $|\widetilde g_X|_1 \leq k_D<1$, for all $X\in \mc D$ at least four of $p_1(X),...,p_6(X)$ must be nonzero, including one from each pair. Combined with Equation \eqref{eq:prob_to_flow}, one can easily check from this that $p_D \geq (1-k_D)/6>0$. On the other hand by the same arguments as in Lemma \ref{lem: flow bound}, 
    \begin{align*}
        \m E_{\lambda_n}[|\text{flux}(\widetilde v_\tau,S\cap D)|] \geq p_D N /3 + O(n^{-1}N) + O(\text{length}_n(\partial (S\cap D))). 
    \end{align*}
    Since $|\Theta_D|\leq \text{area}_n(S\cap D)=N$, this completes the proof. The constant $K_D = p_D / 3= (1-k_D)/18$. As $k_D$ is determined by just $g$ and $D$, we note that this is independent of the choice of $\widetilde g$, as long as $\widetilde g$ inherits the flexible condition and is constructed on a small enough mesh $\mc X$. 
\end{proof}

Equipped with these lemmas, we can now give the proof of the generalized patching theorem. 

\begin{proof}[Proof of Theorem \ref{thm:generalized_patching}]
    If $A = R_n\setminus R_n^c$ is not tileable, then by Theorem \ref{thm:halls matching 2} there exists a counterexample region $U\subset A$ which has only odd cubes along its interior boundary $S$, but has 
    \begin{align*}
        \text{imbalance}(U) = \text{even}(U) - \text{odd}(U) > 0. 
    \end{align*}
    By Corollary \ref{cor:minimal_tileable}, we can assume that the interior boundary $S\subset \partial U$ is a minimal monochromatic discrete surface in $\frac{1}{n} \m Z^3$. 

    By Lemma \ref{lem:generalized_indenting}, for any $\beta>0$ we can find $\epsilon$ and inner and outer layers $a_+\in (\epsilon, 2\epsilon)$ and $a_-\in (c-2\epsilon, c-\epsilon)$ such that for $a=a_+$ or $a=a_-$,
    \begin{equation}\label{eq:length_generalized}
        \text{length}_n(\partial (S\cap \partial R_n^a)) \leq (\beta/\epsilon)n
    \end{equation}
    and further such that there is a ribbon surface $\gamma$ for $\partial (S\cap \partial R_n^a)$ such that 
    \begin{equation}\label{eq:ribbon}
        \text{area}_n(\gamma) \leq 2\beta n^2. 
    \end{equation}
    We define the middle annulus $A_{\text{mid}} = R_n^{a_+}\setminus R_n^{a_-}$. Let $U_{\text{mid}}=U\cap A_{\text{mid}}$. 
    
    We now choose a compact set $D\subset \text{Int}(R)$. We can assume that $D$ is contained in $A_{\text{mid}}$. We can find a piecewise-constant approximation $\widetilde g$ of $g$ on a tetrahedral mesh $\mc X$ which inherits the flexible condition. We let $\mc D\subset \mc X$ be the collection of tetrahedra $X$ such that $X\cap D\neq \emptyset$. We can assume that the mesh scale of $\mc X$ is small enough so that the union of all $X\in \mc D$ is still contained in $\text{Int}(R)$. 
        
    By analogous pigeonhole principle arguments in the indenting lemmas, we can further assume that $D$ has width in squares from $\frac{1}{n} \m Z^3$ of at least $cn/2$ and has $\text{length}_n(\partial (S\cap D)) = O(n)$.
    
    Let $N = \text{area}_n(S\cap D)$. By Lemma \ref{lemma:shininglight_flowbound}, we can find a sequence of shining light measures $\lambda_n$ satisfying for $n$ large enough, 
    \begin{align*}
        \m E_{\lambda_n}[|\text{flux}(\widetilde v_\tau,S\cap D)|]\geq K_D |\Theta_D|+O(N n^{-1}) + O(\text{length}_n(\partial(S\cap D)). 
    \end{align*}
    By Lemma \ref{lem:general_annulusbound}, $|\Theta_D|\geq N/6 = \text{area}_n(S\cap D)/6\geq ( c_2 /24 ) n^2$. Therefore 
    \begin{align*}
        \m E_{\lambda_n}[|\text{flux}(\widetilde v_\tau,S\cap D)|] \geq \frac{K_D c_2}{24}n^2 + O(n).
    \end{align*}
    In particular, we can sample a tiling $\tau_{\text{SL}}$ from $\lambda_n$ such that 
    \begin{equation}\label{eq:test_tiling_SL}
        |\text{flux}(\widetilde v_{\tau_{\text{SL}}},S\cap D)| \geq \frac{K_D c_2}{24}n^2 + O(n).
    \end{equation}
    Let $U_{\tau_{\text{SL}}}$ be the cubes covered by $\tau_{\text{SL}}$ restricted to $U_{\text{mid}}$. This is a tileable region, so 
    \begin{align*}
        \text{imbalance}(U_{\tau_{\text{SL}}}) = 0. 
    \end{align*}
    Let $U_{\text{mid}}'\subset U$ be $U_{\tau_{\text{SL}}}$ minus even cubes in $A\setminus U$ with a face on $S$ which are connected to $U$ by $\tau_{\text{SL}}$. By Equation \eqref{eq:test_tiling_SL}, 
    \begin{equation}\label{eq:imbalance_mid}
        \text{imbalance}(U_{\text{mid}}') \leq -\frac{K_D c_2}{24}n^2 + O(n).
    \end{equation}
    Let $U_{\text{shell}} = U\setminus U_{\text{mid}}'$. It remains to bound the imbalance in $U_{\text{shell}}$.
    
    Consider the set $\alpha = U\cap \partial A_{\text{mid}}$. 
    For each connected component $\alpha_i$ of $\alpha$, we form a closed surface using a cylinder ribbon surface component $\gamma_i$ of $\gamma$ and corresponding patch $\alpha_i'\subset \partial A$. Let $\alpha'$ be the union of the $\alpha_i'$ components. 

     Let $V_i$ be the region enclosed by $\alpha_i,\gamma_i$, and $\alpha_i'$, and let $V = \cup_i V_i$. 

     The regions $U_{\text{shell}}\setminus V$ are the \textit{leftover regions}. Let $W$ be a connected component of the leftover region. By construction, $\partial W$ intersects at most one of $\partial R_n$ or $\partial R^c_n$. Thus the boundary condition on $\partial W$ comes from only one tiling, either $\tau_n$ or $\sigma_n$. 

     Suppose it comes from $\sigma_n$, i.e.\ that $\partial W\cap \partial R_n\neq \emptyset$ (the version where it comes from $\tau_n$ is identical, we just make a choice for concreteness). Since $\sigma_n$ can be extended to a tiling of all of $R_n$, we can extend $\sigma_n$ to a tiling covering $W$. Let $W_\sigma$ be the region covered by the tiles from $\sigma_n$ which intersect $W$. Clearly $\text{imbalance}(W_\sigma)=0$, and $W_\sigma\cap \partial R_n = W\cap \partial R_n$. If a tile in $W_\sigma$ crosses $\partial W$, then either
     \begin{itemize}
         \item It crosses $\partial W \cap S$, in which case $W_\sigma$ contains an even cube which is not contained in $W\subset U$. 
         \item It crosses $\partial W\cap \gamma$ (recall that $\gamma$ is the ribbon surface). In this case $W_\sigma$ could have an odd cube which is not in $W$. However the number of these added cubes over all components $W$ of $\text{U}_{\text{shell}}\setminus V$ is bounded by $\text{area}_n(\gamma) \leq 4\beta n^2$.
     \end{itemize}
    Therefore 
    \begin{align*}
        \text{imbalance}(U_{\text{shell}}\setminus V) \leq 4 \beta n^2. 
    \end{align*}
    We now show that 
    \begin{align*}
        \text{imbalance}(U_{\text{shell}}\cap V) \leq \text{imbalance}(V) + 4\beta n^2. 
    \end{align*}
    The $4\beta n^2$ term again comes from the ribbon area. We use ideas analogous to those above. Let $Y$ be a component of $V\setminus (U_{\text{shell}}\cap V)$. First note that $Y$ intersects at most one of $\partial R_n$ and $\partial R_n^c$, so its boundary condition comes from only one tiling. 

    Assume $\sigma_n$ is the tiling which defines the boundary condition on $Y$, and extend it to a tiling which covers $Y$. Let $Y_\sigma$ be the region covered by $\sigma_n$ tiles which intersect $Y$. Clearly $\text{imbalance}(Y_\sigma) = 0$. If a tile in $Y_\sigma$ crosses $\partial Y$, then it is in one of two cases: 
    \begin{itemize}
        \item It crosses $\partial Y \cap S$. Since $Y\subset A\setminus U$, in this case $Y_\sigma$ contains an odd cube which is not in $Y$. This makes the imbalance of $Y$ larger. 
        \item It crosses $\partial Y \cap \gamma$ (recall $\gamma$ is the ribbon surface). In this case $Y_\sigma$ could have an even cube which is not in $Y$. However the number of these added cubes over all components $Y$ of $V\setminus (\text{U}_{\text{shell}}\cap V)$ is bounded by $\text{area}_n(\gamma) \leq 4\beta n^2$.
    \end{itemize}
    Therefore in summary,
    \begin{align*}
        \text{imbalance}(U_\text{shell}) \leq 8\beta n^2 + \text{imbalance}(V).
    \end{align*}
    We now relate imbalance to flux. As the proof of Proposition \ref{prop: white black imbalance} (where we relate black and white surface area to imbalance), given a set $V$, we apply the divergence theorem to the reference flow $\widetilde r(e) =1/6$ for all $e$ in $\frac{1}{n}\m Z^3$ oriented even to odd to get that 
    \begin{align*}
        \text{imbalance}(V) = \text{flux}(\widetilde r, \partial V) = \frac{1}{6}\bigg(\text{white}(\partial V)-\text{black}(\partial V)\bigg).
    \end{align*}
    We apply this to the set $V=\cup_i V_i$. By Equation \eqref{eq:ribbon}, the $\gamma_i$ contribute at most $4\beta n^2$. Therefore 
    \begin{equation}\label{eq:imbalance_as_flux_r}
        \text{imbalance}(U_\text{shell})  \leq \sum_i \text{flux}(\widetilde r, \partial V_i) \leq 12\beta n^2 + |\text{flux}(\widetilde r,\alpha) - \text{flux}(\widetilde r,\alpha')| .
    \end{equation}
    For the second inequality, we orient $\alpha,\alpha'$ to always both have inward-pointing normal vector (i.e. inward on $\partial A$ and inward on $\partial A_{\text{mid}}$), meaning one has the opposite normal vector as when we compute flux for $\partial V_i$. This is why we get a minus sign.     
    
    The non-rescaled flow $\widetilde f_\tau$ is the divergence-free version of the pretiling flow $\widetilde v_\tau$; related by the equation $\widetilde f_\tau(e) = \widetilde v_\tau(e) - \widetilde r(e)$ for all edges $e$ oriented even to odd. The rescaled version has $f_\tau = \frac{1}{n^3} \widetilde f_\tau$.
    
    Since the boundary condition on $\alpha$ is given by $\tau_{\text{SL}}$ and the boundary conditions on $\alpha'$ are given by $\tau_n$ on the inner boundary and $\sigma_n$ on the outer boundary, none of the tiles from the corresponding tilings cross $\alpha,\alpha'$ and hence
    \begin{align*}
        \text{flux}(\widetilde v_{\tau_{\text{SL}}},\alpha) = \text{flux}(\widetilde v_\ast,\alpha') = 0,
    \end{align*}
    where $\widetilde v_\ast$ is equal to $\widetilde v_{\tau_n}$ on the inner boundary of $\partial A$ and is equal to $\widetilde v_{\sigma_n}$ on the outer boundary of $\partial A$. Therefore
    \begin{equation}\label{eq:imbalanace_as_flux}
         \text{imbalance}(U_\text{shell})  \leq 12\beta n^2 + |\text{flux}(\widetilde f_{\tau_{\text{SL}}},\alpha) - \text{flux}(\widetilde f_{\ast},\alpha')|,
    \end{equation}
    where $\widetilde f_{\ast} = \widetilde f_{\tau_n}$ on the inner boundary of $\partial A$ and $\widetilde f_{\ast}=\widetilde f_{\sigma_n}$ on the outer boundary of $\partial A$.

    It remains to bound these flux differences, and this is where we use information about the boundary conditions. First note that for any surface $X$ and any tiling $\tau$ of $\frac{1}{n}\m Z^3$, the flux of the non-rescaled $\widetilde f_\tau$ and the rescaled $f_\tau$ are related by:
    \begin{equation}\label{eq:differ_by_n2}
        \text{flux}(\widetilde f_\tau,X) = n^2 \text{flux}(f_\tau,X).  
    \end{equation}
    We have that the \textit{rescaled versions} of the tiling flows $f_{\tau_n}$ and $f_{\tau_{\text{SL}}}$ (rescaled, so without the tildes) converge as $n\to \infty$ to the $g\in AF(R,b)$ given in the theorem statement, that is, 
    \begin{align*}
        \lim_{n\to \infty} d_W( f_{\tau_n},g) = \lim_{n\to \infty}d_W( f_{\tau_{\text{SL}}}, g) =0.
    \end{align*}
    Recall that $T(\cdot, X)$ denotes the trace operator which takes a flow to its restriction to a surface $X$. By Theorem \ref{thm: boundary_value_uniformly_continuous}, for $X$ fixed and any $f_1,f_2\in AF(R)$, given any $\delta>0$ there exists $\delta_1$ such that if $d_W(f_1,f_2)< \delta_1$ then $\m W_1^{1,1}(T(f_1,X),T(f_2,X))<\delta$. Recall also that $T(g,\partial R) = b$, and that we are given that $T(f_{\sigma_n},\partial R)$ converges to $b$ in $\m W_1^{1,1}$. Given these facts, we can choose $n$ large enough to guarantee the following: 
    \begin{itemize}
        \item For the outer boundary $\partial R$, 
        \begin{equation}\label{eq:bound1}
        \m W_1^{1,1}(T( f_{\sigma_n},\partial R), b) < \delta
        \end{equation}
        \begin{equation}\label{eq:bound2}
            \m W_1^{1,1}(T( f_{\tau_{\text{SL}}},\partial R), b) < \delta.
        \end{equation}
        \item For the inner boundary $\partial R^c$,
        \begin{equation}\label{eq:bound3}
            \m W_1^{1,1}(T( f_{\tau_n},\partial R^c),T(g,\partial R^c))<\delta
        \end{equation}
        \begin{equation}\label{eq:bound4}
            \m W_1^{1,1}(T( f_{\tau_{\text{SL}}},\partial R^c),T(g,\partial R^c))<\delta.
        \end{equation}
        \item Finally, let $\partial R_{\text{mid}}=\partial (R^{a_+}\setminus R^{a_-})$ be the piecewise smooth surface approximated by $\partial A_{\text{mid}}= \partial(R_n^{a_+}\setminus R_n^{a_-})$. Then  
        \begin{equation}\label{eq:bound5}
            \m W_1^{1,1}(T( f_{\tau_{\text{SL}}},\partial R_{\text{mid}}),T(g,\partial R_{\text{mid}}))<\delta.
        \end{equation}
    \end{itemize}   

    Recall also that boundary value flows correspond to measures, and note that $\text{flux}(f,X) = T(f,X)(X)$ is the total mass of the measure $T(f,X)$ on $X$. In particular, for a surface $X$ and tiling $\tau$ of $\frac{1}{n}\m Z^3$ and $B\subset X$,
    \begin{align*}
        \text{flux}(f_\tau, B) = T(f_\tau,X)(B).
    \end{align*}
Lemma~\ref{lem:constant_order_box_Wass_bound} applied to measures $\mu,\nu$ supported on a surface $X$ says that if $\m W_1^{1,1}(\mu,\nu)<\delta$, then for any $B\subset X$,  
    \begin{align*}
        \m W_1^{1,1}(\mu\mid_B,\nu\mid_B) \leq \delta + \delta^{1/2} (C(B)+1),
    \end{align*}
    where $\delta^{1/2} C(B)$ is bounded by $2$ times the difference of the area of B and the $\delta^{1/2}$ neighborhood of $B$ within $X$; equivalently, by the area of the annulus of width $\delta^{1/2}$ with inner boundary $\partial B$ (see Remark \ref{rem:constant_explained}). 
    
To use this, we relate $\alpha,\alpha'$ which are contained in the discrete surfaces $\partial A_{\text{mid}}$ and $\partial A$ built out of $\frac{1}{n}\m Z^3$ lattice squares, to $B,B'$ on the piecewise smooth surfaces $\partial R_{\text{mid}}$ and $\partial R\cup\partial R^c$ respectively.

By Equation \eqref{eq:length_generalized}, $\text{length}_n(\partial \alpha) \leq (2\beta/\epsilon)n$. Correspondingly the Euclidean length is bounded as $\text{length}(\partial \alpha)\leq (2\beta/\epsilon)$. Given this, we can cover $\partial \alpha$ with a collection of cubes $\mc C$ with Euclidean side length $3\epsilon$, with $|\mc C|\leq 2\beta/(3\epsilon^2)$. Since the Euclidean width between $\partial A$ and $\partial A_{\text{mid}}$ is less than $2\epsilon$, $\mc C$ also covers $\partial \alpha'\subset \partial A$.

The Hausdorff distances between $\partial A$ and $\partial R\cup \partial R^c$ and between $\partial A_{\text{mid}}$ and $\partial R_{\text{mid}}$ are both bounded by $2/n$. There are corresponding sets $B\subset \partial R_{\text{mid}}$ and $B'\subset \partial R\cup \partial R^c$ which differ from $\alpha, \alpha'$ respectively by Hausdorff distance at most $2/n$. Thus for $n$ large enough, $\mc C$ also covers $B,B'$. Since $\partial R\cup \partial R^c$ and $\partial R_{\text{mid}}$ are piecewise smooth, there is a constant $C'$ such that the area of either surface restricted to one of the cubes in $\mc C$ is at most $C' \epsilon^2$. Since $|\mc C| \leq 2\beta/(3\epsilon^2)$, there is some constant $C$ such that if $\delta^{1/2}\leq\epsilon$, then
\begin{align}
    \label{eq:C(B) bound}\delta^{1/2} C(B) &\leq C \beta\\
    \label{eq:C(B') bound}\delta^{1/2}C(B') &\leq C \beta.
\end{align}
We have that $\text{length}(\partial \alpha)\leq 2\beta/\epsilon$, so the number of $\frac{1}{n} \m Z^3$ lattice points along $\partial \alpha$ is bounded by a constant times $n$ (the constant here depends on $\beta/\epsilon$). Since the Hausdorff distance between $\alpha$ and $B$ is bounded by $2/n$, for any tiling $\tau$ of $\frac{1}{n} \m Z^3$, the flux of $f_\tau$ through a surface is proportional to the number of $\frac{1}{n}\m Z^3$ lattice points on the surface times $\frac{1}{n^2}$. Thus for any tiling $\tau$, since $f_\tau$ is divergence-free,
\begin{align}\label{eq:alphaB bound}
    |\text{flux}(f_\tau,\alpha) - \text{flux}(f_\tau,B)| \leq O(n^{-1}).
\end{align}
By Equation \eqref{eq:C(B') bound}, since $\text{length}(\partial B')\leq C(B')$, also have that $\text{length}(\partial B')\leq C \beta/\epsilon$. Since the Hausdorff distance between $\alpha'$ and $B'$ is bounded by $2/n$, the number of $\frac{1}{n}\m Z^3$ lattice points on a surface between them is also bounded by a constant times $n$ (the constant here depends on $\beta/\epsilon$). Thus we analogously get that for any tiling $\tau$ of $\frac{1}{n} \m Z^3$,
\begin{align}\label{eq:alphaB' bound}
    |\text{flux}(f_\tau,\alpha') - \text{flux}(f_\tau,B')| \leq O(n^{-1}).
\end{align}
Therefore by Lemma \ref{lem:constant_order_box_Wass_bound}, for $\delta$ such that $\delta^{1/2}<\epsilon$, Equations \eqref{eq:bound1},   \eqref{eq:bound3} to relate $f_*$ to $g$ on $B'$, plus Equation  \eqref{eq:C(B') bound} where we determine the constant $C(B')$, and finally Equation \eqref{eq:alphaB' bound} to relate $f_*$ on $\alpha'$ to $f_*$ on $B'$, we get that
 \begin{align}
    |\text{flux}(f_{*},\alpha') - \text{flux}(g,B')| \leq \delta + \delta^{1/2} + C\beta + O(n^{-1}),
\end{align}
where $f_{*}$ is $f_{\sigma_n}$ on the outer boundary of $\partial A$ and $f_{\tau_n}$ on the inner boundary. Similarly, using Equations \eqref{eq:bound2}, \eqref{eq:bound4}, and \eqref{eq:bound5} to relate $f_{\tau_\text{SL}}$ to $g$, plus Equations \eqref{eq:alphaB bound}, \eqref{eq:C(B) bound} for the constant $C(B)$ and to relate $f_{\tau_\text{SL}}$ on $\alpha, B$, the test tiling $\tau_{\text{SL}}$ satisfies analogous bounds on both $\alpha$ and $\alpha'$:
\begin{align}
    |\text{flux}(f_{\tau_{\text{SL}}},\alpha') - \text{flux}(g,B')| &\leq \delta +\delta^{1/2}+ C\beta + O(n^{-1})\\
     |\text{flux}(f_{\tau_{\text{SL}}},\alpha) - \text{flux}(g,B)| &\leq \delta +\delta^{1/2} + C\beta + O(n^{-1}).
\end{align}
Since $g$ is divergence-free and takes values with norm bounded between $-1$ and $1$, and $B,B'$ differ from $\alpha,\alpha'$ by Hausdorff distance bounded by $2/n$, Equation \eqref{eq:ribbon} implies that
    \begin{equation}\label{eq:bound6}
        |\text{flux}(g,B) - \text{flux}(g,B')| \leq 4\beta + O(n^{-1}).
    \end{equation}

Combining Equation \eqref{eq:differ_by_n2} with the above, 
    \begin{align}
        &|\text{flux}(\widetilde f_{\tau_{\text{SL}}},\alpha) - \text{flux}(\widetilde f_{\ast},\alpha')| \\
        &\leq \bigg[|\text{flux}(f_{\tau_{\text{SL}}},\alpha)-\text{flux}(g,B)| +|\text{flux}(g,B)-\text{flux}(g,B')| +|\text{flux}(g,B')-\text{flux}(f_{*},\alpha')|\bigg] n^2 \\
        &\leq (2 \delta + 2\delta^{1/2} + 2 C \beta + O(n^{-1}) + 4\beta )n^2.
    \end{align}

    Combining this with Equation \eqref{eq:imbalance_mid} and Equation \eqref{eq:imbalanace_as_flux}, we get that
    \begin{align*}
        \text{imbalance}(U) &= \text{imbalance}(U_{\text{mid}}') + \text{imbalance}(U_{\text{shell}}) \\
        &\leq -\frac{K_D c_2}{24}n^2 + (16+2C)\beta n^2 + 2\delta n^2 + 2\delta^{1/2} n^2 + O(n).
    \end{align*}
    The factor $K_D c_2/24$ and the constant $C$ are fixed independent of $\delta,\beta$. Implicit here is also the parameter $\epsilon$ that we indent by.
    
    Taking $\epsilon$ small enough, we can make $\beta$ as small as needed. The parameter $\delta>0$ is related to the distance between tiling flows and their limits and is required to satisfy $\delta^{1/2}<\epsilon$, but this can be guaranteed for $n$ large enough. Therefore for $n$ large enough, $\text{imbalance}(U)$ will be non-positive, and hence $U$ is not a counterexample. This completes the proof.
\end{proof}

\subsection{Upper bounds}\label{sec:upper}

To complete the proof of the large deviation principles, we need prove the upper bounds, namely Theorem \ref{thm:upper} for the soft boundary LDP and Theorem \ref{thm:upper-hb} for the hard boundary LDP. We show that the soft boundary upper bound implies the hard boundary one, and then prove the soft boundary one.
\begin{lemma}
    Theorem \ref{thm:upper} implies Theorem \ref{thm:upper-hb}.
\end{lemma}
\begin{proof}
    Recall that $\mu_n$ is counting measure on $TF_n(R,b,\theta_n)$ for some sequence of thresholds $(\theta_n)_{n\geq 1}$ with $\theta_n\to 0$ as $n\to \infty$ sufficiently slowly.
    
    On the other hand $\overline{\mu}_n$ is counting measure on tilings of fixed regions $R_n$ with scale $n$ tileable boundary value $b_n$ such that $b_n\to b$ as $n\to \infty$.
    
    We choose the sequence of thresholds $\theta_n$ so that there exists $N$ such that if $n\geq N$ then $\m W_1^{1,1}(b_n,b)<\theta_n$. In this case, for any $g\in AF(R,b)$, for $n$ large enough,
    \begin{align*}
        \overline\mu_n(A_{\delta}( g))\leq \mu_n(A_{\delta}( g)).
    \end{align*}
    Therefore if Theorem \ref{thm:upper} holds, then for all $g\in AF(R,b)$,
    \begin{align*}
   \lim_{\delta\to 0}\limsup_{n\to \infty} v_n^{-1} \log \overline\mu_n(A_{\delta}( g))\leq \lim_{\delta\to 0}\limsup_{n\to \infty} v_n^{-1} \log \mu_n(A_{\delta}( g)) \leq \Ent( g).
\end{align*}
This completes the proof of Theorem \ref{thm:upper-hb} from Theorem \ref{thm:upper}.
\end{proof}

It remains to prove Theorem \ref{thm:upper}, namely that for any $ g\in AF(R, b)$,
\begin{align*}
   \lim_{\delta\to 0}\limsup_{n\to \infty} v_n^{-1} \log \mu_n(A_{\delta}( g)) \leq \Ent( g).
\end{align*}
The main idea is ``coarse graining", i.e.\ that on a very small box, a uniform random tiling of $R$ looks approximately like a random tiling sampled from a $\threeeven$-invariant Gibbs measure of mean current $s$, where $s$ is the expected mean current on the box. 

\begin{proof}[Proof of Theorem \ref{thm:upper}]
In this proof, we assume without loss of generality that $R$ is contained in the unit cube $B = [0,1]^2$ (this is just to avoid complicating the proof with an extra scaling parameter).

Let $\pi_{n,\delta}$ be the uniform probability measure on the set of tilings $\tau$ with tiling flow $f_\tau\in A_\delta(g) \cap TF_n(R)$ and satisfying $\m W_1^{1,1}(T(f_\tau),b)<\theta_n$. The purpose of this is so that the partition function of $\pi_{n,\delta}$ is $Z_{n,\delta} = \mu_{n}(A_\delta(g))$. 

Tile $\m R^3$ by translated copies of $B$, each with a translated copy of $R$ inside it. Let $\Lambda_n = \frac{1}{n} \m Z^3 \cap B $. We define a $\threeeven$-invariant measure $\nu_n$ on tilings of $\m Z^3$ as follows (this measure can sample tilings with some double tiles and some untiled sites). We take an independent sample from $\pi_{n,\delta}$ on each copy of $R$, and then average over translations by $x\in \text{even}(\Lambda_n)$. The measure $\nu_n$ samples tilings that are perfect matchings on the interior of each copy of $R$. All sites in each copy of $B\setminus R$ which are not covered by a tile connecting it to a copy of $R$ are empty. Two copies of $R$ might intersect on their boundaries (e.g.\ in the case $R= B$), in which case $\nu_n$ can sample double tiles. However the fraction of possible sites where $\nu_n$ samples double tiles is bounded by the fraction of sites in $\partial B$, namely
\begin{align*}
    \frac{6n^2}{ n^3} = \frac{6}{n}.
\end{align*}
We define a subsequential limit
\begin{align*}
    \nu:= \lim_{j\to \infty} \nu_{n_j}.
\end{align*}
Note that $\nu$ is a $\threeeven$-invariant measure on tilings of $\m Z^3$, allowed to have untiled sites. It can written as a weighted average of a measure on dimer tilings and the empty ensemble. 

Let $\nu_{n,0}$ be defined analogously to $\nu_n$, but without averaging over translations. Let $\nu_{n,x}$ be the version where all tilings are translated by a fixed $x\in \text{even}(\Lambda_n)$. The Shannon entropy of $\pi_{n,\delta}$ is 
\begin{align*}
 n^{-3}\text{Vol}(R)^{-1} \log Z_{n,\delta} = v_n^{-1} \log \mu_n(A_\delta(g)).
\end{align*}
By construction, 
$$|\Lambda_{n}|^{-1} H_{\Lambda_{n}}(\nu_{n,0})=  n^{-3} \log \mu_n(A_\delta(g)).$$ 
For any other $x\in \text{even}(\Lambda_n)$, a sample on $B$ contains pieces from up to $8$ samples of $\pi_{n,\delta}$. By subadditivity of $H$,
\begin{align*}
    H_{\Lambda_n}(\nu_{n,x})\geq H_{\Lambda_n}(\nu_{n,0}) \qquad \forall \, x\in \text{even}(\Lambda_n).
\end{align*}
The specific entropy of $\nu_n$ can be computed using any sequence of boxes $\Delta_M$ with $|\Delta_M|\to\infty$ as $M\to \infty$. In particular, we can choose $\Delta_M =\Lambda_{Mn}$ so that
\begin{align*}
    h(\nu_n) = \lim_{M\to \infty} |\Lambda_{Mn}|^{-1}H_{\Lambda_{Mn}}(\nu_n).
\end{align*}
On each of the $M^3$ copies of $\Lambda_n$ in $\Lambda_{Mn}$, $\nu_n$ samples an independent draw from $\pi_{n,\delta}$. Thus 
\begin{align*}
    h(\nu_n) \geq \lim_{M\to \infty} M^3 |\Lambda_{Mn}|^{-1}H_{\Lambda_{n}}(\nu_{n,0}) = |\Lambda_n|^{-1} H_{\Lambda_n}(\nu_{n,0}) = n^{-3} \log \mu_n(A_\delta(g)). 
\end{align*}
On the other hand, since $h$ is upper-semicontinuous, 
\begin{align*}
    \limsup_{j\to \infty} h(\nu_{n_j}) \leq h(\nu).
\end{align*}
Therefore 
\begin{align*}
    \limsup_{j\to \infty} v_{n_j}^{-1}\mu_{n_j}(A_\delta(g)) \leq \frac{1}{\Vol(R)} h(\nu),
\end{align*}
and we have reduced the problem to bounding $h(\nu)$. Define $\varphi_n$ to be the $\nu_n$-expected flow, namely
\begin{align*}
    \varphi_n := Z_{n,\delta}^{-1} \sum_\tau f_\tau, 
\end{align*}
where the sum is over tilings $\tau$ in the support of $\pi_{n,\delta}$. We define a subsequential limit
$$\varphi:= \lim_{j\to \infty} \varphi_{n_j}.$$
Up to taking additional subsequences we can assume that the subsequences for $\varphi_n$ and $\nu_n$ are the same. Note that $\varphi\in AF(R)$. Since $A_\delta(g)$ is convex, Theorem \ref{thm:formal_fine_mesh} implies that $\varphi\in \overline{A_\delta(g)}$. Therefore 
$$\frac{1}{\Vol(R)} \int_{R} \ent(\varphi(x)) \, \dd x = \Ent(\varphi) \leq \sup_{h\in \overline{A_\delta(g)}} \Ent(h) = \Ent(g) + o_\delta(1).$$
The last equality uses that $\Ent$ is upper semi-continuous in the Wasserstein topology (Proposition \ref{proposition: Upper semicontinuous}). This reduces the problem to showing that $\frac{1}{\Vol(R)} h(\nu)$ is bounded by $\Ent(\varphi)$. To this end, we partition $B$ into a collection $\mc C$ of $k^3$ smaller cubes of size $1/k^3$. We define a new flow $\alpha_k$ supported in $R$ by, for all $C\in \mc C$ such that $C\cap R\neq \emptyset$, 
\begin{align*}
    \alpha_k(x) = \frac{1}{|C\cap R|} \int_{C\cap R} \varphi(y) \, \dd y \qquad \forall x\in C\cap R.
\end{align*}
For $x\not\in R$, $\alpha_k(x) = 0$. Since $\ent$ is concave (Lemma \ref{lemma: entropy_concave}), by Jensen's inequality, 
\begin{align*}
    \int_{C\cap R} \ent(\alpha_k(x)) \dd x \geq \int_{C\cap R} \ent(\varphi(x)) \dd x.
\end{align*}
On the other hand, $\alpha_k$ converges to $\varphi$ a.s. and $|\alpha_k|\leq 1$, so $\alpha_k$ converges to $\varphi$ in $L^1$, hence by Corollary \ref{cor: Ent_ae},
\begin{align*}
    \lim_{k\to\infty} \Ent(\alpha_k) = \Ent(\varphi). 
\end{align*}
Therefore it is sufficient to show that $\Ent(\alpha_k)$ is an upper bound for all $k$. We now define $\nu_{n,C}$ to be $\nu_{n}$ but averaged only over the translations $x\in \text{even}(\Lambda_n\cap C)$ (equivalently, conditioned on the origin being in $C$). For each $C\in \mc C$, let $\nu_C$ be a subsequential limit of $\nu_{n,C}$. The measures $\nu_C$ are $\threeeven$-invariant, and we can choose the subsequences so that 
\begin{align*}
    \nu = k^{-3} \sum_{C\in \mc C} \nu_C.
\end{align*}
Therefore 
\begin{align*}
    h(\nu) = k^{-3} \sum_{C\in \mc C} h(\nu_C). 
\end{align*}
Note that if $C\cap R = \emptyset$, then $\nu_C$ is the empty ensemble and hence in that case $h(\nu_C)=0$. When $C\cap R\neq \emptyset$, then $\nu_C$ splits as the sum of an empty ensemble (corresponding to selecting the origin in $C\setminus C\cap R$) and a measure on dimer tilings (corresponding to selecting the origin in $C\cap R$). In a slight abuse of notation we refer to the mean current of $\nu_C$ as the mean current of its component which is a measure on dimer tilings. To bound $h(\nu_C)$ when $C\cap R\neq \emptyset$ we compute this mean current. Recall from Section \ref{section:measures-currents} that $s_0(\tau)$ is the vector of the tile at the origin in $\tau$, and the mean current of a $\threeeven$-invariant measure $\mu$ can be computed as
\begin{align*}
    s(\mu) = \int_{\Omega} s_0(\tau)\, \dd \mu(\tau) = \m E_{\mu}[s_0(\tau)].
\end{align*}
We can also compute the mean current by looking at the expected tile direction over a set of points instead of just looking at the origin. Let $E(\Lambda_n)$ denote the edges in $\Lambda_n$ oriented from even to odd. We can similarly define 
$$s_{n,C}(\tau) := |\text{even}(\Lambda_n\cap C\cap R)|^{-1} \sum_{e\in E(\Lambda_n\cap C\cap R)} f_\tau(e)e.$$ 
Here note that we intersect with $R$ because if the origin is chosen in $C\setminus C\cap R$ we get the empty ensemble, and hence the mean current is not defined. This is the average direction of $f_\tau$ over $\Lambda_n\cap C\cap R$, and by $\threeeven$-invariance $s(\mu)$ can also be computed 
\begin{align*}
    s(\mu) = |\text{even}(\Lambda_n\cap C\cap R)|^{-1} \sum_{x\in \text{even}(\Lambda_n\cap C\cap R)} \int_{\Omega} s_0(\tau+x) \,  \dd\mu(\tau) = \int_{\Omega} s_{n,C}(\tau)\, \dd \mu(\tau). 
\end{align*}
Using this, we compute that 
\begin{align*}    \m E_{\nu_{n,C}}[s_{n,C}(\tau)] = Z_{n,\delta}^{-1}  \sum_{\tau} \sum_{e\in E(\frac{1}{n} \m Z^3\cap C\cap R)} |C\cap R|^{-1} f_\tau(e)e = \text{avg}_{C\cap R}(\varphi_n).
\end{align*}
Since $\nu_C$ is a subsequential limit of $\nu_{n,C}$ (up to choice of another subsequence),
\begin{align*}
    s_C:= s(\nu_C) = \int_{\Omega} s_0(\tau) \,\dd \nu_C(\tau)= 
    \lim_{j\to \infty}
    \int_{\Omega} s_0(\tau) \,\dd \nu_{n_j,C}(\tau)&=\lim_{j\to \infty}
    \int_{\Omega} s_{n_j,C}(\tau) \,\dd \nu_{n_j,C}(\tau)
    \\
    &=\lim_{j\to \infty} \text{avg}_{C\cap R}(\varphi_{n_j}).
\end{align*}
On the other hand, since $\varphi_{n_j}\to \varphi$ in $L^1$, 
\begin{align*}
    s_{C} = \limsup_{j\to\infty} \text{avg}_C(\varphi_{n_j}) = \text{avg}_C(\varphi). 
\end{align*}
Finally we relate $h(\nu_C)$ to $\ent(s_C)$. Recall that $\ent(s):=\max_{\rho\in \mc P^s} h(\rho)$, where $\mc P^s$ is the space of $\threeeven$-invariant probability measures on dimer tilings of mean current $s$. The measure $\nu_C$ is a sum of an empty ensemble (corresponding to the origin being chosen in $C\setminus R$) which has zero entropy and a $\threeeven$-invariant on dimer tilings (corresponding to the origin being chosen in $C\cap R$) which has mean current $s_C$. Thus
\begin{align*}
    h(\nu_C) \leq \ent(s_C).
\end{align*}
Therefore for all $k>1$,
\begin{align*}
    \limsup_{j\to \infty} v_{n_j}^{-1} \mu_{n_j}(A_\delta(g)) &\leq \frac{1}{\Vol(R)} h(\nu) = \frac{1}{\Vol(R)} \frac{1}{k^3} \sum_{C \in \mc C} h(\nu_C) \\ &\leq \frac{1}{\Vol(R)} \frac{1}{k^3} \sum_{C\in \mc C}\Vol(R\cap C) \ent(s_C) = \Ent(\alpha_k). 
\end{align*}
Taking $k\to \infty$, this shows that 
\begin{align*}
    \limsup_{n_j\to \infty} v_{n_j}^{-1} \mu_{n_j}(A_\delta(g)) \leq \Ent(\varphi) = \Ent(g) + o_\delta(1).
\end{align*}
Since this holds for any convergent subsequence $n_j$, taking $\delta\to 0$ completes the proof. 
  
\end{proof}
 
\section{Open problems}\label{sec:open problems} \label{sec:open}

We mentioned in the introduction that there is literature exploring the local move connectivity problem, considering moves such as the ``flip'' and ``trit'' illustrated below.
\begin{figure}[H]
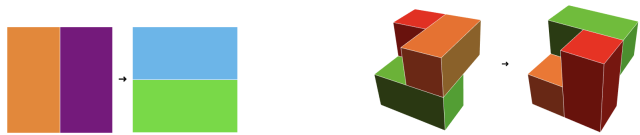

    \centering
        \includegraphics[scale=0.25]{flip3d.pdf}$\qquad\qquad$\includegraphics[scale=0.2]{trit.pdf}
    \caption{Flip and trit.}
    \label{fig:flip_trit}
\end{figure}
Both the flip and the trit amount to finding a cycle in $\mathbb Z^3$ (of length $4$ or $6$ respectively) that alternates between membership and non-membership in $\tau$, and then swapping the members and non-members. Generally, a {\em cycle swap} is a swap of an alternating cycle of length $k$, and a {\em $k$-swap} is a cycle swap for which the cycle has length $k$. It is clear that any two perfect matchings of the same region can be connected by a sequence of such swaps (simply by applying swaps to all of the cycles contained in the union of the two perfect matchings). But it is in general not so clear whether one can get from any matching to any other using only $k$-swaps for small $k$.

\begin{problem}\label{prob: local moves on box}
Is there a finite $K$ such that for any positive $j$, $m$ and $n$ (at least one of which is even) it possible to get from any dimer configuration of an $j \times m \times n$ box to any other via sequence of $k$-swaps with $k \leq K$?  Is this possible using only flips and trits?
\end{problem}

The examples we have presented in Section~\ref{sec:localmoves} already show the answer to both questions is no if one replaces boxes with general simply connected regions, such as those that can be tiled with alternating slabs of brickwork, each oriented a different direction. If we think in terms of the non-intersecting path interpretation from Section~\ref{subsec:history}, we can see that the existence of taut patterns like the ones shown there are an obstruction to local move connectedness.

Progress was made on Problem \ref{prob: local moves on box} just after the first draft of this paper was released in \cite{localdimer}. In particular their results show that any dimer tiling of a $j\times m \times n$ box (for $jmn$ even, $j,m,n\geq 2$) admits at least one flip or trit \cite[Theorem 1]{localdimer}. See Section~\ref{sec:localmoves} for further description of their results. 

\begin{problem}\label{prob: markov chain mixing time}
What can be said about the convergence rate of the mixing algorithm described in Section~\ref{sec:uniform sampling}? Is there a more efficient way to sample random perfect matchings of 3D regions?
\end{problem}

\begin{problem}\label{prob: unique EGM}
Is there a unique ergodic Gibbs measure corresponding to each mean current in the interior of $\mathcal O$?
\end{problem}

\begin{problem}\label{prob: infinite paths}
If $\nu_1$ and $\nu_2$ are ergodic Gibbs measures of the same mean current, and $(\tau_1, \tau_2)$ is sampled uniformly from $(\nu_1 \otimes \nu_2)$, are there necessarily infinitely many infinite paths in the union of $\tau_1$ and $\tau_2$?
\end{problem}

\begin{problem}
What can be said about the {\em typical fluctuations} of the flow associated to a uniformly random perfect matching of a simple region such as a cube or torus? Do they converge to a natural Gaussian process?
\end{problem}

In 2D, Kenyon showed that domino tiling height functions converge in law to the Gaussian free field \cite{kenyon2000conformal}. This suggests that the discrete gradients of the height functions should converge (at least in some sense) to the gradient of the Gaussian free field. The dual of the discrete gradient (i.e., the discrete flow) should converge in some sense to the dual of the gradient of the Gaussian field---which can be shown to be equivalent to the field obtained by projecting vector-valued white noise orthogonally onto the space of divergence-free fields.  It seems reasonable to conjecture that the same holds in any dimension.

\begin{problem} Is it the case for $d \geq 3$ that the discrete divergence-free flows obtained from uniformly random perfect matchings (on a torus or box, say, or in the $\mathbb Z^3$ Gibbs measure setting) converge in the fine mesh limit to the Gaussian random generalized flow obtained by projecting vector-valued white noise onto the space of divergence-free flows? \end{problem}

\begin{problem}\label{prob: region with non unique max}
Does there exist a three-dimensional region $R\subset \m R^3$ and a boundary condition $b$ for which the $\Ent$ maximizer for $(R,b)$ is not unique? We have shown that such a system would have to be ``rigid'' in the sense defined in the introduction (i.e., there is an interior point $x$ such that for any neighborhood $U$ of $x$ the set $\overline{g(U)}$ must intersect one of the edges of $\mathcal O$).  But we have not ruled out the existence of multiple $\Ent$ maximizers. 

In fact there do exist two dimensional surfaces $R$ where the corresponding $\Ent$ maximizer is not unique. Consider the ``slanted cylinder" below, where the left and right edges are glued following the numbers in the diagram. Here are two possible tilings of the slanted cylinder.
\begin{figure}[H]
    \centering
    \includegraphics[scale=0.6]{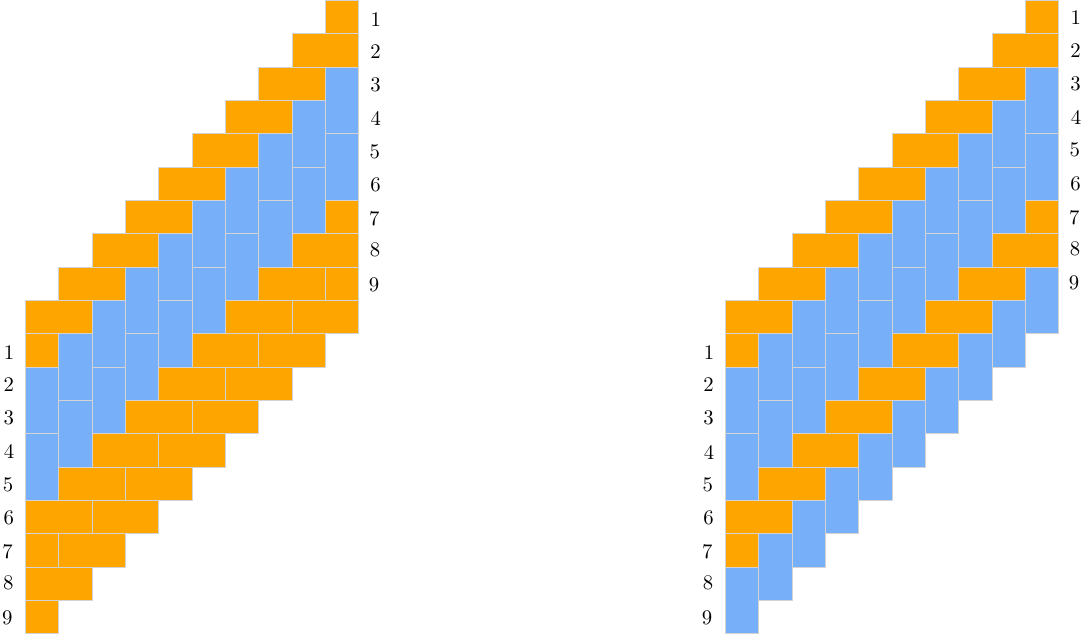}
    \caption{Two tilings of the slanted cylinder. The left and right edges are glued. }
    \label{fig:slanted_cylinder}
\end{figure}
Any tiling of the slanted cylinder consists of a choice of north (N) or east (E) tile for each diagonal, so if the cylinder has height $m$ then it has $2^m$ distinct tilings. Since there is only one choice to make on each of the diagonal ``stripes'' (deciding whether to color it blue or orange) the entropy per site tends to zero as the width of the cylinder tends to infinity, and the functions obtained as fine-mesh limits of these constructions are all maximizers of $\Ent$. A slanted cylinder can also be realized as an induced subgraph of $\mathbb Z^3$ as shown below.

\begin{figure}[H]
    \centering
    \includegraphics[scale=0.8]{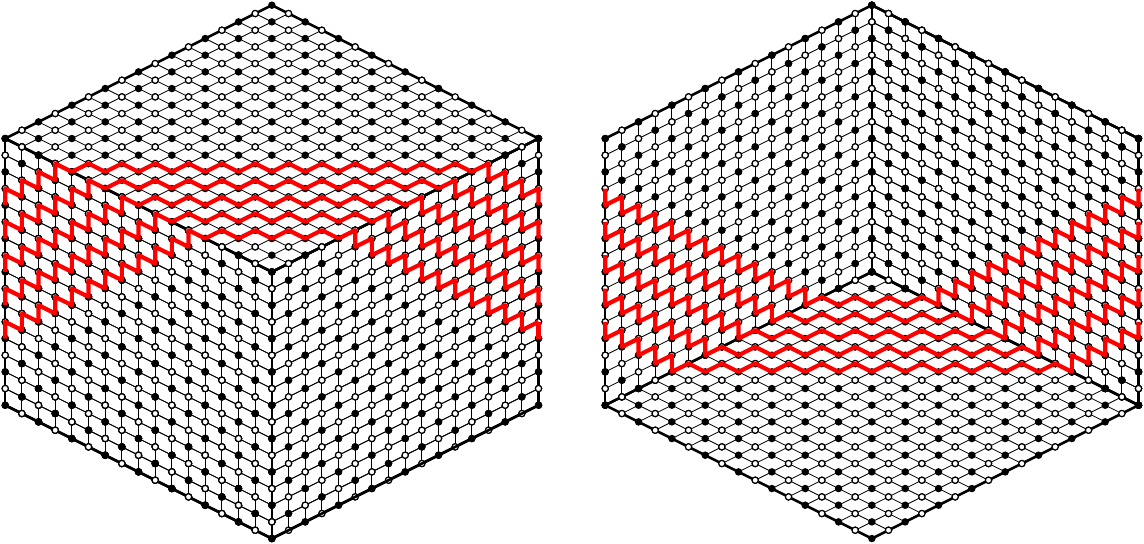}
    \caption{Front and back sides of the surface of a cube, with five ``stripes'' wrapping around it, whose vertices correspond to the squares in Figure~\ref{fig:slanted_cylinder} and form a slanted cylinder embedded in $\mathbb Z^3$. Hall's matching theorem implies that every perfect matching of the set of vertices hit by these stripes is obtained by choosing one of the two possible perfect matchings within each stripe.}
    \label{fig:stripedcube}
\end{figure}

If we try to take a fine mesh limit of this example, we get a region with zero volume in $\mathbb R^3$. The question is whether this kind of phenomenon can arise for regions with non-zero volume that are ordinary subsets of $\mathbb R^3$ (as opposed to, say, 3D analogs of the slanted cylinder). ``Thickening'' the ribbon-like example above (by taking the union of multiple ribbon layers, taken on different concentric cubes) does not seem to work, as a tiling obtained that way need not be locally frozen (trit moves may be possible at the corners).

\end{problem}

\begin{problem}\label{prob:limit_shape_values}
Is there a region $R\subset \m R^3$ and boundary asymptotic flow $b$ where the $\Ent$ maximizing flow takes values on a face of $\partial \mc O$ within a strict subset of the interior of $R$? Or within a strict subset of all of $R$? (Does this happen on the boundary of the Aztec octahedron?) For an example of $(R,b)$ where the limit shape takes values in a face of $\partial \mc O$ on all of $R$ (i.e., not only a strict subset of the region), see Example \ref{ex:hb_false}. 
\end{problem}

\begin{problem}
    Given a region $R\subset \m R^3$ and a flow $b$ on $\partial R$, is there an elegant way to describe the conditions under which $AF(R,b)$ is nonempty? In other words, under what conditions does $b$ admit an extension to $R$ which is an asymptotic flow (measurable, divergence-free, and valued in $\mc O$)? Recall that if $R\subset \m Z^3$ is a discrete region and $b$ is a discrete vector field on $\partial R\subset \m Z^3$, then Hall's matching theorem or the min cut, max flow principle say that $b$ is extendable if and only if there is no \textit{counterexample region} $U\subset R$ such that $S = \partial U\cap R$ is a type of discrete minimal surface, and any extension of $b$ would be required to have too much flow across $S$. Is there a continuum version of Hall's matching theorem and the min cut, max flow principle that characterizes when $b$ on $\partial R$ can be extended to an asymptotic flow---i.e., a statement that $b$ is extendable as long as there is no ``minimal surface" cut $S$ such that any extension of $b$ would be required to have too much flow across $S$? See e.g.\ \cite{Strang2010} for discussion of related problems.

    A particularly simple case of interest is that where $R$ is a polyhedron and the boundary value $b$ is constant on the faces of the polyhedron.
\end{problem}

\begin{problem} Let us try to generalize Aztec prism example from the introduction. Suppose $R\subset \m R^3$ is a prism of the form $S \times [0,1]$ (where $S$ is a two-dimensional region) and $b$ is equal to $0$ on the top and bottom faces of the prism. Alternatively, one may identify the top and bottom of the prism, to obtain $S$ cross a circle. We expect that one can show from basic symmetry that the $\Ent$ minimizing flow $g$ has zero flow in the vertical direction, that its restriction to a slice $S \times \{x \}$ does not depend on $x$.  Understanding the behavior within this slice is then a two-dimensional flow problem. Is this behavior the same as what one would see for the corresponding two-dimensional dimer model on the slice?
\end{problem}

\begin{problem}
What can be said about the interfaces between frozen regions on the boundaries of limit shapes (such as those apparent in the figures in the introduction)? How large do the fluctuations tend to be?
\end{problem} 

\begin{problem}
The 2D Aztec diamond has four frozen regions (one for each vertex) and the 3D Aztec octahedron appears to have twelve frozen regions (one for each edge). One might guess that in the $k$-dimensional analog we would see $4 \binom{k}{2}$ frozen regions, one for each co-dimension-two boundary simplex. Can anything along these lines be proved, either in 3D or in higher dimensions?
\end{problem} 

\begin{problem}
In two dimensions, the large deviation theory \cite{cohn2001variational} can be generalized to many other types of random height function models \cite{AST_2005__304__R1_0}, even though for most of these models we cannot compute $\ent$ explicitly. For example, instead of having height differences constrained to $\{3/4,-1/4 \}$ as in the 2D dimer model, they could be constrained to some other set, like $\{-1,1\}$ or $\{-1,0,1\}$. That raises a natural question for us. To what other discrete divergence-free flow models in 3D (or in higher dimensions) can the results of this paper be extended? For example, what if instead of restricting the even-to-odd flows to lie in $\{5/6,-1/6\}$ we restrict them to $\{-1,1\}$ or to some other set? Would the max-flow-min-cut theory available in these settings allow us to complete the steps that relied on Hall's matching theorem in this paper? Could the ``chain swapping'' arguments used in this paper be adapted to establish the strict concavity of $\ent$ in these settings?
\end{problem}

As we mentioned earlier, given a lattice flow $v$ on $\mathbb Z^3$ one can define a discrete ``curl'' that assigns to each oriented plaquette---which corresponds to an oriented edge of the dual lattice---the flow of $v$ around that plaquette.  One can then define the {\em vector potential} $A_\tau$ on the dual lattice of $\mathbb Z^3$ whose curl corresponds to the flow $f_\tau$ on $\mathbb Z^3$, though $A_\tau$ is {\em a priori} only determined up to the addition of a vector field with curl zero. Restricting the flow $f_\tau$ to take values in $\{5/6,-1/6\}$ is then equivalent to restricting the curl of $A_\tau$ to lie in $\{5/6,-1/6\}$.

Readers familiar with lattice gauge theory (see \cite{chatterjee2019yang} for a survey) can tell a similar story about a constrained lattice {\em connection} with gauge group $U(1)$ (the complex unit circle) as follows.  Fix some small constant $\alpha \in (0,\pi)$ and constrain the holonomy around every plaquette (oriented clockwise as one looks from the even to the odd incident cube) to lie in $\{e^{5 \alpha i / 6}, e^{-\alpha i /6} \}$.  Then define a domino to be a pair of cubes separated by a plaquette with holonomy $e^{5 \alpha i / 6}$. Since the product of oriented holonomies around a single cube is zero, each interior cube belongs to exactly one domino, and (up to boundary conditions) one expects a uniformly random constrained connection to correspond to a uniformly random 3D domino tiling. 

\begin{problem} \label{prob:gauge}
Can our large deviation theory be extended to any other types of holonomy-constrained random connections, Abelian or otherwise? Are there other aspects of gauge theory for which this perspective is useful?
\end{problem}

Although we have not explained this in detail, we believe that all the arguments of this paper will still apply to the setting where the edges are periodically ``weighted’’ in the manner described in \cite{kenyon2006dimers}.  For example, one might consider a weighting that strongly favors edges whose vertices have the form $(x,y,z)$ and $(x,y,z+1)$ where $z$ is even. If the weight is strong enough, one can use a standard Peierls argument to show if we are given two independent samples from the minimal-specific-free-energy ergodic Gibbs measure, then there are a.s.\ no infinite paths in their union.

\begin{problem}
If we allow periodic weights, as in \cite{kenyon2006dimers}, what can we say about the phase diagram? Are there some choices of weights for which the double dimer 
model a.s.\ contains no infinite paths and others for which it a.s.\ contains infinitely many infinite paths? Are there any other possibilities? Can one say, even on a rough qualitative level, how similar the function $\ent$ described here (and its periodically-edge-weighted analogs) will be to the surface tension functions described in \cite{kenyon2006dimers} (which are interesting algebraic geometry constructions with finitely many singular cusps)? In this generalized setting, can one say anything about the magnitude of the typical fluctuations of a random flow, or how such fluctuations might depend on the edge weights?
\end{problem}

\bibliographystyle{alpha}
\bibliography{dimer}

\printindex

\end{document}